\documentclass[a4paper,leqno,11pt]{amsart}

\usepackage[top=1.5in,bottom=1.3in,left=1.3in,right=1.3in,marginparwidth=1.5cm]{geometry}

\usepackage{times}
\usepackage{tensor}
\usepackage[all]{xy}

\usepackage{amsmath, amssymb, amsfonts, latexsym, mdwlist, amsthm}
\usepackage{mathrsfs}
\usepackage{subfig}
\usepackage{graphicx}
\usepackage{wrapfig}
\usepackage{longtable}
\usepackage{enumerate}

\usepackage{mathtools} %for \xhookrightarrow

\makeindex%[title=Index of Notations]

\usepackage[colorlinks=true, citecolor=blue, urlcolor=blue, linkcolor=blue, pagebackref]{hyperref}

\def\A{\ensuremath{\mathbb{A}}}
\def\C{\ensuremath{\mathbb{C}}}
\def\D{\ensuremath{\mathbb{D}}}
\def\G{\ensuremath{\mathbb{G}}}
\def\H{\ensuremath{\mathbb{H}}}

\def\P{\ensuremath{\mathbb{P}}}
\def\Q{\ensuremath{\mathbb{Q}}}
\def\R{\ensuremath{\mathbb{R}}}
\def\Z{\ensuremath{\mathbb{Z}}}

\def\cA{\ensuremath{\mathcal A}}
\def\AA{\ensuremath{\mathcal A}}
\def\cB{\ensuremath{\mathcal B}}
\def\cC{\ensuremath{\mathcal C}}
\def\CC{\ensuremath{\mathcal C}}
\def\cD{\ensuremath{\mathcal D}}

\def\cE{\ensuremath{\mathcal E}}
\def\cF{\ensuremath{\mathcal F}}
\def\cG{\ensuremath{\mathcal G}}

\def\cL{\ensuremath{\mathcal L}}
\def\cM{\ensuremath{\mathcal M}}
\def\cN{\ensuremath{\mathcal N}}
\def\cO{\ensuremath{\mathcal O}}
\def\cP{\ensuremath{\mathcal P}}
\def\cQ{\ensuremath{\mathcal Q}}
\def\cT{\ensuremath{\mathcal T}}
\def\cW{\ensuremath{\mathcal W}}
\def\cX{\ensuremath{\mathcal X}}
\def\cY{\ensuremath{\mathcal Y}}
\def\cZ{\ensuremath{\mathcal Z}}

\def\uu{\ensuremath{\mathbf u}}
\def\vv{\ensuremath{\mathbf v}}
\def\ww{\ensuremath{\mathbf w}}

\def\bA{\ensuremath{\mathbf A}}

\def\bZ{\ensuremath{\mathbf Z}}

\def\llambda{\ensuremath{\boldsymbol{\lambda}}} 
 
\def\eeta{\ensuremath{\boldsymbol{\eta}}} 

\def\fM{\mathcal M}
\def\fP{\mathfrak P}

\def\ii{\mathfrak i}
\def\sI{\ensuremath{\mathscr I}}

\def\rH{\ensuremath{\mathrm H}}

\def\ov{\ensuremath{\overline{v}}}

\def\us{\underline{\sigma}}
\def\uvs{\underline{\varsigma}}
\def\utau{\underline{\tau}}
\def\ut{\underline{\tau}}

\def\tH{\ensuremath{\widetilde{H}}}
\def\tF{\ensuremath{\widetilde{F}}}
\def\tM{\ensuremath{\widetilde{M}}}

\def\Tr{\mathrm{Tr}}
\def\Art{\mathop{{\mathrm{Art}}}\nolimits}

\def\Br{\mathop{\mathrm{Br}}\nolimits}
\def\ch{\mathop{\mathrm{ch}}\nolimits}

\def\Coh{\mathop{\mathrm{Coh}}\nolimits}

\def\cone{\mathop{\mathrm{cone}}\nolimits}
\def\cok{\mathop{\mathrm{Coker}}\nolimits}
\def\Db{\mathop{\mathrm{D}^{\mathrm{b}}}\nolimits}
\def\Dperf{\mathop{\mathrm{D}_{\mathrm{perf}}}\nolimits}
\def\perf{\mathop{{\mathrm{perf}}}\nolimits}
\def\Dqc{\mathop{\mathrm{D}_{\mathrm{qc}}}\nolimits}
\def\qc{\mathop{{\mathrm{qc}}}\nolimits}

\def\dim{\mathop{\mathrm{dim}}\nolimits}
\def\Dqc{\mathop{\mathrm{D}_{\mathrm{qc}}}\nolimits}

\def\Ext{\mathop{\mathrm{Ext}}\nolimits}
\def\Fix{\mathop{\mathrm{Fix}}}

\def\GL{\mathop{\mathrm{GL}}\nolimits}
\def\Hdg{\mathrm{Hdg}}

\def\Hom{\mathop{\mathrm{Hom}}\nolimits}
\def\lHom{\mathop{\mathcal Hom}\nolimits}

\def\RHom{\mathop{\mathrm{RHom}}\nolimits}
\def\id{\mathop{\mathrm{id}}\nolimits}

\def\im{\mathop{\mathrm{Im}}\nolimits}

\def\Ker{\mathop{\mathrm{Ker}}\nolimits}

\def\mod{\mathop{\mathrm{mod}}\nolimits}
\def\min{\mathop{\mathrm{min}}\nolimits}

\def\num{\mathop{\mathrm{num}}\nolimits}

\def\Qcoh{\mathop{\mathrm{Qcoh}}\nolimits}
\def\rk{\mathop{\mathrm{rk}}}
\def\Sch{\mathop{{\mathrm{Sch}}}}

\def\Spec{\mathop{\mathrm{Spec}}}

\def\Tor{\mathop{\mathrm{Tor}}\nolimits}
\def\Tr{\mathop{\mathrm{Tr}}\nolimits}

\def\Stab{\mathop{\mathrm{Stab}}\nolimits}

\def\Tot{\mathop{\mathrm{Tot}}\nolimits}
\def\len{\mathop{\mathrm{length}}\nolimits}

\def\st{\mathrm{st}}

\def\Stor{{\ensuremath{S\text{-}\mathrm{tor}}}}
\def\Stf{{\ensuremath{S\text{-}\mathrm{tf}}}}
\def\Ctor{{\ensuremath{C\text{-}\mathrm{tor}}}}

\def\Ctf{{\ensuremath{C\text{-}\mathrm{tf}}}}

\def\Ku{\operatorname{Ku}}

\def\uPP{\ensuremath{\widetilde{\mathbb{P}}}}
\def\uXX{\ensuremath{\widetilde{\cX}}}

\def\oLambda{\ensuremath{\overline{\Lambda}}}

\def\Knum{K_{\mathrm{num}}}
\def\Ku{\mathcal{K}u}

\def\ZK{\ensuremath{Z_K}}
\def\Zc{\ensuremath{Z_{\Ctor}}}

\newcommand{\set}[1]{\left\{#1\right\}}
\newcommand{\gen}[1]{\left\langle#1\right\rangle}
\newcommand{\sth}{\;\vline\;} % 'such that' set notation bar
\newcommand{\res}[2]{#1_{#2}} % Arend - changed macro to match new convention
\def\norm#1{\left\|#1\right\|}

\def\blank{\underline{\hphantom{A}}}

\def\into{\ensuremath{\hookrightarrow}}
\def\onto{\ensuremath{\twoheadrightarrow}}

\def\wGL2{\ensuremath{\widetilde{\mathrm{GL}_2^+(\R)}}}

\newtheorem{Thm}{Theorem}[section]
\newtheorem{Prop}[Thm]{Proposition}
\newtheorem{PropDef}[Thm]{Proposition and Definition}
\newtheorem{Lem}[Thm]{Lemma}
\newtheorem{Cor}[Thm]{Corollary}

\newtheorem*{Ques*}{Question}

\theoremstyle{definition}
\newtheorem{Def}[Thm]{Definition}
\newtheorem{Rem}[Thm]{Remark}

\newtheorem{Ex}[Thm]{Example}

\newtheorem{Setup}{Setup}[part]

\makeatletter
\newtheoremstyle{italicsname}% <name>
 {3pt}% <Space above>
 {3pt}% <Space below>
 {\itshape}% <Body font>
 {}% <Indent amount>
 {\itshape}% <Theorem head font>
 {.}% <Punctuation after theorem head>
 {.5em}% <Space after theorem heading>
 {\thmname{#1}\thmnumber{\@ifnotempty{#1}{ }#2}%
 \thmnote{ {\the\thm@notefont(#3)}}}% <Theorem head spec (can be left empty, meaning `normal')>
\makeatother
\theoremstyle{italicsname}
\newtheorem{innerstep}{Step}
\newenvironment{step}[1]
 {\renewcommand\theinnerstep{#1}\innerstep}
 {\endinnerstep}

\newtheorem{innerclaim}{Claim}
\newenvironment{claim}[1]
 {\renewcommand\theinnerclaim{#1}\innerclaim}
 {\endinnerclaim}

\newcommand{\xrightarrowdbl}[2][]{%
 \xrightarrow[#1]{#2}\mathrel{\mkern-14mu}\rightarrow
}

\newcommand{\cMpug}{{\cM}_{\rm pug}} 
\newcommand{\scM}{{s\cM}}

\newcommand{\CH}{\mathrm{CH}}

\numberwithin{equation}{section}

\makeatletter
\newcommand{\mylabel}[2]{#2\def\@currentlabel{#2}\label{#1}}
\makeatother

\def\Ann{\mathrm{Ann}}
\newcommand{\Cat}{\mathrm{Cat}}
\newcommand{\PrCat}{\mathrm{PrCat}}
\newcommand{\QCoh}{\mathrm{QCoh}}
\newcommand{\rb}{\mathrm{b}}
\newcommand{\pr}{\mathrm{pr}}
\def\qc{\mathop{{\mathrm{qc}}}\nolimits}
\newcommand{\hf}{\mathrm{hf}}
\newcommand{\Dpc}{\mathrm{D}_{\rm pc}}
\newcommand{\Ind}{\mathrm{Ind}}

\newcommand{\pc}{\mathrm{pc}}

\def\fq{\mathfrak q}
\def\frm{\mathfrak m}

\newcommand{\cHom}{\mathcal{H}\!{\it om}}

\newcommand{\rS}{\mathrm{S}}

\newcommand{\rL}{\mathrm{L}}

\def\abs#1{\left\lvert#1\right\rvert}

\def\citestacks#1{\cite[\href{https://stacks.math.columbia.edu/tag/#1}{Tag #1}]{stacks-project}}

\def\Rtf{{\ensuremath{R\text{-}\mathrm{tf}}}}

\newcommand{\rD}{\mathrm{D}}
\newcommand{\Dpug}{\mathrm{D}_{\rm pug}}

\newcommand{\pug}{\mathrm{pug}}
\newcommand\uHom{\underline{\mathrm{Hom}}}
\newcommand{\Sets}{\mathrm{Sets}}
\newcommand{\Gpds}{\mathrm{Gpds}}
\newcommand{\op}{\mathrm{op}}
\newcommand{\Quot}{\mathrm{Quot}}
\newcommand{\rR}{\mathrm{R}}

\DeclareMathOperator{\colim}{\mathrm{colim}}

\newcommand\stv[2]{\left\{#1\,\colon\,#2\right\}}

\begin{document}

\title[Stability conditions in families]{Stability conditions in families}

\author[A.~Bayer, M.~Lahoz, E.~Macr\`i, H.~Nuer, A.~Perry, P.~Stellari]{Arend Bayer, Mart\'i Lahoz, Emanuele Macr\`i, Howard Nuer,\\ Alexander Perry, Paolo Stellari}

\address{A.B.: School of Mathematics and Maxwell Institute,
University of Edinburgh,
James Clerk Maxwell Building,
Peter Guthrie Tait Road, Edinburgh, EH9 3FD,
United Kingdom}
\email{arend.bayer@ed.ac.uk}
\urladdr{\url{http://www.maths.ed.ac.uk/~abayer/}}

\address{M.L.: Universit\'{e} Paris Diderot -- Paris~7, B\^{a}timent Sophie Germain, Case 7012, 75205 Paris Cedex 13, France}
\email{marti.lahoz@ub.edu}
\urladdr{\url{http://www.ub.edu/geomap/lahoz/}}
\curraddr{Departament de Matem\`atiques i Inform\`atica,
Universitat de Barcelona, Gran Via de les Corts Catalanes, 585, 08007 Barcelona, Spain}

\address{E.M.: Department of Mathematics, Northeastern University, 360 Huntington Avenue, Boston, MA 02115, USA}
\email{emanuele.macri@universite-paris-saclay.fr}
\urladdr{\url{https://www.imo.universite-paris-saclay.fr/~macri/}}
\curraddr{Universit\'e Paris-Saclay, CNRS, Laboratoire de Math\'ematiques d'Orsay, Rue Michel Magat, B\^at. 307, 91405 Orsay, France}

\address{H.N.: Department of Mathematics, Northeastern University, 360 Huntington Avenue, Boston, MA 02115, USA}
\email{hnuer@technion.ac.il}
\urladdr{\url{https://sites.google.com/site/howardnuermath/home}}
\curraddr{Department of Mathematics, Technion, Israel Institute of Technology, Amado 914, Haifa 32000, Israel}

\address{A.P.: Department of Mathematics, Columbia University, 2990 Broadway, New York, NY 10027, USA}
\email{arper@umich.edu}
\urladdr{\url{http://www-personal.umich.edu/~arper/}}
\curraddr{Department of Mathematics, University of Michigan, 530 Church Street, Ann Arbor, MI
48109, USA}

\address{P.S.: Dipartimento di Matematica ``F.~Enriques'', Universit{\`a} degli Studi di Milano, Via Cesare Saldini 50, 20133 Milano, Italy}
\email{paolo.stellari@unimi.it}
\urladdr{\url{https://sites.unimi.it/stellari}}

\keywords{Bridgeland stability conditions, base change for semiorthogonal decompositions, base change for t-structures, relative moduli spaces, non-commutative varieties, K3 surfaces, cubic fourfolds}
\subjclass[2010]{14A22, 14D20, 14F05, 14J28, 14M20, 14N35, 18E30}
\thanks{
A.~B.~was supported by the ERC Starting Grant ERC-2013-StG-337039-WallXBirGeom, by the ERC Consolidator Grant ERC-2018-CoG-819864-WallCrossAG, and by the NSF Grant DMS-1440140 while the author was in residence at the MSRI in Berkeley, during the Spring 2019.
M.~L.~was supported by a Ram\'on y Cajal fellowship and partially by the Spanish MINECO research project PID2019-104047GB-I00.
E.~M.~was partially supported by the NSF grant DMS-1700751, by a Poincar\'e Chair from the Institut Henri Poincar\'e and the Clay Mathematics Institute, by the Institut des Hautes \'Etudes Scientifiques (IH\'ES), by a Poste Rouge CNRS at Universit\'e Paris-Sud, by the ERC Synergy Grant ERC-2020-SyG-854361-HyperK, and by the National Group for Algebraic and Geometric Structures, and their Applications (GNSAGA-INdAM).
H.~N.~was partially supported by the NSF postdoctoral fellowship DMS-1606283, by the NSF RTG grant DMS-1246844, and by the NSF FRG grant DMS-1664215.
A.~P.~was partially supported by the NSF postdoctoral fellowship DMS-1606460 and the NSF grant DMS-2002709.
P.~S.~was partially supported by the ERC Consolidator Grant ERC-2017-CoG-771507-StabCondEn, by the research project PRIN 2017 ``Moduli and Lie Theory'', and by research project FARE 2018 HighCaSt (grant number R18YA3ESPJ)}

\begin{abstract}
We develop a theory of Bridgeland stability conditions and moduli spaces of semistable objects for a family of varieties.
Our approach is based on and generalizes previous work by Abramovich--Polishchuk, Kuznetsov, Lieblich, and Piyaratne--Toda.
Our notion includes openness of stability, semistable reduction, a support property uniformly across the family, and boundedness of semistable objects.
We show that such a structure exists whenever stability conditions are known to exist on the fibers.

Our main application is the generalization of Mukai's theory for moduli spaces of semistable sheaves on K3 surfaces to moduli spaces of Bridgeland semistable objects in the Kuznetsov component associated to a cubic fourfold.
This leads to the extension of theorems by Addington--Thomas and Huybrechts on the derived category of special cubic fourfolds, to a new proof of the integral Hodge conjecture, and to the construction of an infinite series of unirational locally complete families of polarized hyperk\"ahler manifolds of K3 type.

Other applications include the deformation-invariance of Donaldson--Thomas invariants counting Bridgeland stable objects on Calabi--Yau threefolds, and a method for constructing stability conditions on threefolds via degeneration.
\end{abstract}

\maketitle
\setcounter{tocdepth}{1}
\tableofcontents

%%%%%%%%%%%%%%%%%%%%%%%%%%%%%%%%%%%%%

\section{Introduction}\label{sec:intro}

Stability conditions on triangulated categories were introduced by Bridgeland in \cite{Bridgeland:Stab} and further studied by Kontsevich and Soibelman in \cite{Kontsevich-Soibelman:stability};
originally based on work by Douglas \cite{Douglas:stability} in string theory, they have found many applications to algebraic geometry via wall-crossing.
In this paper we develop the corresponding theory for derived categories in a family of varieties over a base.
Our definition is guided by the requirement that it should come with a notion of \emph{relative moduli spaces}, but also be flexible enough to allow deformations.

\subsection*{Stability conditions over a base}

Let $X\to S$ be a flat family of projective varieties over some base scheme $S$.
Consider a collection of stability conditions $\sigma_s$ on the derived categories $\Db(X_s)$ of the fibers.
When does this form a well-behaved family $\us$ of stability conditions?
Our proposed answer is the main content of this article.
It can be paraphrased as follows;
see Definitions~\ref{def:familyfiberstabilities} and~\ref{def:fiberwisesupport} for the precise formulation.

\begin{Def} \label{MainDef:Intro}
	A collection of numerical stability conditions 
\[ \us = \left(\sigma_s = (Z_s, \cP_s)\right)_{s \in S}
\]
on the fibers is a \emph{stability condition on $\Db(X)$ over $S$} if it satisfies
 the following conditions: 
 \begin{enumerate}[{\rm (1)}]
 	\item \label{enum:introZlocallyconstant}
 	The central charge is locally constant in families of objects.
 	\item \label{enum:introstableisopen}
 	Geometric stability is open in families of objects.
 	\item \label{enum:introHNstructure}
 	After base change $C \to S$ to any Dedekind scheme $C$,
 	the stability conditions $\sigma_c$ on the fibers over $C$ are induced by a \emph{HN structure} on $\Db(X_C)$.
 	\item \label{enum:introsupport}
 	Each stability condition $\sigma_s$ satisfies the support property, an inequality on classes of semistable objects, in a form that is uniform across all fibers.
 	\item \label{enum:introboundedness}
 	The set of semistable objects in the fibers of $X \to S$ satisfies 
 	a boundedness condition.
 \end{enumerate} 
\end{Def}

Our precise setup is quite general, see \hyperref[MainSetup]{Main Setup}; 
for instance $S$ could be defined over a non-algebraically closed field, or it could be a scheme of mixed characteristic.
In the definition, we consider \emph{all} points of $S$, closed or non-closed.
We work with numerical stability conditions, as they allow base change under field extensions, see Theorem~\ref{thm:base-change-stability-condition}.
This gives a stability condition $\sigma_t$ on $X_t$ for any point $t$ over $S$.

For a base change $T \to S$, the notion of a family of objects over $T$ appearing in~\eqref{enum:introZlocallyconstant} and~\eqref{enum:introstableisopen} is formalized by Lieblich's notion of a relatively perfect object $E \in \Db(X_T)$ \cite{Lieblich:mother-of-all};
the fibers $E_t$ of such an object lie in $\Db(X_t)$.

Condition \eqref{enum:introstableisopen} is the glue tying together the stability conditions on different fibers;
for example, we will see in Proposition~\ref{prop:ProductStabilityCondition} that in the case where $X$ is a product $X_0 \times S$, it forces the collection of stability conditions on $X_0$ parametrized by $s \in S$ to be constant.

The concept of a Harder--Narasimhan (HN) structure in \eqref{enum:introHNstructure} will be introduced in Section~\ref{sec:defnHNoveraCurve} of this paper (see Definition~\ref{def:HNstructure_C} for a precise definition) and studied for the remainder of Part~\ref{part:HNStrCurve}.
The condition \eqref{enum:introHNstructure} requires that the hearts $\cA_c$ of the stability conditions $\sigma_c$ ``integrate'' to a global heart $\cA_C$;
in the product case, such an $\cA_C$ has been constructed by Abramovich--Polishchuk \cite{AP:t-structures,Polishchuk:families-of-t-structures}.
Moreover, every object in $\Db(X_C)$ is required to have a global HN filtration;
this combines the classical notion of generic HN filtrations with the existence of semistable reductions, and will thus imply the valuative criterion of properness for relative moduli spaces.

Let us elaborate on conditions \eqref{enum:introsupport} and \eqref{enum:introboundedness}.
We fix a finite rank free abelian group $\Lambda$, and a morphism $v_s \colon \Knum(\Db(X_s)) \to \Lambda$ for every $s \in S$, such that $v_t(E_t)$ is locally constant in families.
Then $\us$ is a stability condition with respect to $\Lambda$ if each $Z_s$ factors as $Z \circ v_s$ for some fixed $Z \in \Hom(\Lambda, \C)$.
In \eqref{enum:introsupport}, we require that there is a quadratic form $Q$ on $\Lambda_\R$ such that 	\begin{enumerate}
 		\item the kernel $(\ker Z)\subset\Lambda$ is negative definite with respect to $Q$, and
 		\item for every $s \in S$ and for every $\sigma_s$-semistable object
 		$E \in \cD_s$, we have $Q(v_s(E)) \geqslant 0$.
 	\end{enumerate} 
Our boundedness condition \eqref{enum:introboundedness} says that given $\vv \in \Lambda$, there is a finite type family parametrizing the union for all $s \in S$ of objects $E \in \Db(X_s)$ that are $\sigma_s$-semistable with $v_s(E) = \vv$.

Our first main result is a version of Bridgeland's Deformation Theorem.

\begin{Thm}[Theorem~\ref{thm:deformfamiliystability}]\label{MainThm:IntroDeformation}
	The space $\Stab_\Lambda(\Db(X)/S)$ of stability conditions on $\Db(X)$ over $S$ with respect to $\Lambda$ is a complex manifold, and the forgetful map 
	\[
	\cZ \colon \Stab(\Db(X)/S) \to \Hom(\Lambda,\C),
	\]
	is a local isomorphism.
\end{Thm}

In the absolute case where $S$ is a point, 
the support property \eqref{enum:introsupport} is enough to prove Bridgeland's Deformation Theorem for stability conditions \cite{Bridgeland:Stab}. 
In our setting, the main step is to show that openness of geometric stability is preserved under small deformations of the central charge $Z$;
this both requires boundedness of semistable objects \eqref{enum:introboundedness} (to obtain boundedness of destabilizing quotients, and thereby constructibility of the unstable locus) and the existence of HN structures over DVRs \eqref{enum:introHNstructure} (to show that the unstable locus is closed under specialization).

Our second main result (which is the content of Part~\ref{part:Tilting}) is that, whenever stability conditions are known to exist fiberwise, they also exist in families.

\begin{Thm}[{Theorem~\ref{mainthm:construction}}]\label{MainThm:IntroConstruction}
Let $g \colon X \to S$ be a polarized flat family of smooth projective varieties.
\begin{enumerate}[{\rm (1)}]
\item%\label{enum:construction1}
If the fibers of $g$ are one- or two-dimensional, then the standard construction of stability conditions on curves or surfaces produces a stability condition $\us$ on $\Db(X)$ over $S$.
\item%\label{enum:construction2}
If the fibers of $g$ are three-dimensional and satisfy the conjectural Bogomolov--Gieseker inequality of \cite{BMT:3folds-BG, BMS:abelian3folds}, then the construction of stability conditions proposed in [ibid.] produces a stability condition $\us$ on
$\Db(X)$ over $S$.
\end{enumerate} 
\end{Thm}

The known construction of stability conditions on smooth projective varieties is based on the operation of tilting bounded t-structures, starting from coherent sheaves, and uses weaker notions of stability similar to slope-stability.
We first generalize Definition~\ref{MainDef:Intro} to allow for a collection of weak stability conditions.
Then we show that this tilting procedure extends to our setup.
The main tool is the derived dual functor, which we show to provide a notion of double dual inside the hearts $\cA_s$ for $s \in S$, and $\cA_C$ for a Dedekind scheme $C$ over $S$.

\subsection*{Relative moduli spaces and properness}

Moduli spaces of semistable objects, and their wall-crossing, are what makes stability conditions useful to algebraic geometers.
Given a class $\vv\in\Lambda$, we denote by $\fM_{\us}(\vv)$ the stack parameterizing $\us$-semistable objects of class $\vv$.
Using results of Lieblich, Piyaratne, and Toda \cite{Lieblich:mother-of-all,Toda:K3,PT15:bridgeland_moduli_properties}, we show the following. 

\begin{Thm}[{Theorem~\ref{thm:modulispacesArtinstacks}}]\label{MainThm:IntroProper}
Let $\us$ be a stability condition on $\Db(X)$ over $S$ and $\vv\in\Lambda$.
Then $\fM_{\us}(\vv)$ is an algebraic stack of finite type over $S$.
In characteristic 0 it admits a good moduli space $M_{\us}(\vv)$ which is an algebraic space proper over $S$.
\end{Thm}

The notion of good moduli space was introduced by Alper in \cite{Alper:GoodModuli}.
The existence of a good moduli space is a consequence of the very recent result \cite{AHLH:good_moduli}.

When moduli stacks of Bridgeland stable objects are well-behaved, associated Donaldson--Thomas invariants have been defined in \cite{PT15:bridgeland_moduli_properties}, based on \cite{Kontsevich-Soibelman:stability,Joyce-Song}.
When $M_{\sigma}(\vv)=M_{\sigma}^{\st}(\vv)$, namely there are no properly semistable objects, the definition is the same as in \cite{Thomas:Casson,Behrend:DT-microlocal}, via virtual classes or weighted Euler characteristic:
\[
\mathrm{DT}_{\sigma}(\vv) := \int_{[M_{\sigma}(\vv)]^{\mathrm{vir}}} 1 = \int_{M_{\sigma}(\vv)} \chi_B \in\Z.
\]

As a first application of Theorem~\ref{MainThm:IntroConstruction} and Theorem~\ref{MainThm:IntroProper}, we can prove the deformation-invariance of DT invariants for those CY threefolds on which Bridgeland stability conditions exist;
in particular, by \cite{Li:Quintic3fold}, for quintic threefolds.

\begin{Cor}[{Theorem~\ref{thm:DTquintic}}]\label{MainCor:IntroDT} Let $X \to S$ be a flat family of Calabi--Yau threefolds, defined over $\C$, with $S$ connected.
If the fibers $X_s$ satisfy the generalized Bogomolov--Gieseker inequality of \cite{BMT:3folds-BG, BMS:abelian3folds}, then the Donaldson--Thomas invariant $\mathrm{DT}_{\sigma_s}(\vv)$ is independent of $s$.
In particular, this holds for smooth quintic threefolds.
\end{Cor}

Another application, pointed out by Koseki \cite[Proposition~3.2]{koseki:2}, is a method to construct stability conditions on threefolds via degeneration, see Proposition~\ref{prop:Koseki}. 

\subsection*{Cubic fourfolds}

In the case of higher-dimensional Fano varieties, moduli spaces of semistable objects often become more useful when we restrict our attention to objects lying in certain semiorthogonal components of $\Db(X_s)$, called \emph{Kuznetsov components}. 

This can also be done in families.
Let $\cD\subset\Db(X)$ be an admissible subcategory that is invariant under tensoring with perfect complexes on $S$ (see Section~\ref{sec:SO} for the precise assumptions).
Then Theorems~\ref{MainThm:IntroDeformation} and \ref{MainThm:IntroProper} hold similarly with $\Db(X)$ replaced by $\cD$.
We also give conditions under which a weak stability condition on $\Db(X)$ over $S$ induces a stability condition on $\cD$ over $S$, see Theorem~\ref{thm:InducingStabilityOverBase} which extends the corresponding results in the absolute case from \cite{BLMS}.

Our main application, treated in Part~\ref{part:CubicFourfolds}, concerns cubic fourfolds.
Let $X\subset\P^5$ be a smooth cubic fourfold over the complex numbers.
We denote by $\Ku(X)$ its \emph{Kuznetsov component}
\[
\Ku(X):= \cO_X^\perp \cap \cO_X(H)^\perp \cap \cO_X(2H)^\perp \subset \Db(X).
\]

Over the moduli space of cubic fourfolds, these categories give a family of polarized non-commutative K3 surfaces \cite{Kuz:fourfold,AT:CubicFourfolds} (see also \cite{Huy:cubics, MS:survey}).
There is no analogue of slope- or Gieseker-stability for such non-commutative K3 surfaces; instead, Bridgeland stability conditions have been constructed in \cite{BLMS}.
The main result of Part~\ref{part:CubicFourfolds} is that our notion of stability over a base exists in this setup.

The first application concerns moduli spaces $M_\sigma(\Ku(X), \vv)$ of $\sigma$-stable objects in $\Ku(X)$ with Mukai vector $\vv$, for $\sigma$ a Bridgeland stability condition in the distinguished connected component $\Stab^\dagger(\Ku(X))$, generalizing a series of results for K3 surfaces \cite{Beauville:HK,Mukai:Symplectic,Mukai:BundlesK3,O'Grady:weight2,Huybrechts:BirationaSymplectic,Yoshioka:Abelian,Toda:K3,BM:proj}.

\begin{Thm}[{Theorem~\ref{thm:YoshiokaMain}}]\label{MainThm:IntroYoshioka}
Let $X$ be a cubic fourfold with Kuznetsov component $\Ku(X)$.
Let $\tH(\Ku(X),\Z)$ be its extended Mukai lattice, together with the Mukai Hodge structure.
Then $\tH_{\Hdg}(\Ku(X),\Z)=\tH_{\mathrm{alg}}(\Ku(X),\Z)$.
Moreover, assume that $\vv\in \tH_{\Hdg}(\Ku(X),\Z)$ is a non-zero primitive vector and let $\sigma\in\Stab^\dagger(\Ku(X))$ be a stability condition on $\Ku(X)$ that is generic for $\vv$.
Then:
\begin{enumerate}[{\rm (1)}]
\item $M_{\sigma}(\Ku(X),\vv)$ is non-empty if and only if $\vv^2\geqslant-2$.
Moreover, it is a smooth projective irreducible holomorphic symplectic variety of dimension $\vv^2 + 2$, deformation-equivalent to a Hilbert scheme of points on a K3 surface.
\item If $\vv^2\geqslant 0$, then there exists a natural Hodge isometry 
\[
\theta\colon H^2(M_\sigma(\Ku(X),\vv),\Z)\xrightarrow{\quad\sim\quad}
\begin{cases}\vv^\perp & \text{if }\vv^2>0\\ \vv^\perp/\Z\vv & \text{if } \vv^2=0,\end{cases}
\]
where the orthogonal is taken in $\tH(\Ku(X),\Z)$.
\end{enumerate}
\end{Thm}

Theorem~\ref{MainThm:IntroYoshioka} has many interesting applications to cubic fourfolds and to families of polarized hyperk\"ahler manifolds. We
invite the reader to jump to Section~\ref{sec:MainResultsCubics} for a comprehensive summary of our related results, and give a shorter overview here.
First of all, we can extend a result by Addington and Thomas, \cite[Theorem~1.1]{AT:CubicFourfolds}:

\begin{Cor}[{Corollary~\ref{cor:completeAT}}]\label{MainCor:IntroAT}
Let $X$ be a cubic fourfold.
Then $X$ has a Hodge-theoretically associated K3 if and only if there exists a smooth projective K3 surface $S$ and an equivalence $\Ku(X) \cong \Db(S)$.
\end{Cor}

A version of the corollary also holds for K3 surfaces with a Brauer twist, extending a result of Huybrechts \cite[Theorem~1.4]{Huy:cubics};
the corresponding Hodge-theoretic condition is the existence of a square-zero class in $\tH_\Hdg(\Ku(X),\Z)$.

A second application is that relative moduli spaces give rise to unirational locally-complete families of polarized hyperk\"ahler manifolds of arbitrarily large dimension and degree over the moduli space of cubic fourfolds:

\begin{Cor}[{Corollary~\ref{cor:locfam20dim}}]\label{MainCor:IntroLocCompleteFam}
For any pair $(a,b)$ of coprime integers, there is a unirational locally complete $20$-dimensional family, over an open subset of the moduli space of cubic fourfolds, of polarized smooth projective irreducible holomorphic symplectic manifolds of dimension $2n+2$, where $n=a^2-ab+b^2$.
The polarization has either degree $6n$ and divisibility $2n$ if $3$ does not divide $n$, or degree and divisibility $\frac{2}{3}n$ otherwise.
\end{Cor}

When $a=1$ and $b=1$, this is nothing but the family of Fano varieties of lines in the cubic fourfold from \cite{BD:fano_lines} (with the Pl\"ucker polarization of degree $6$), while when $a=2$ and $b=1$, we find the family of polarized eightfold from \cite{LLSvS} (with the Pl\"ucker polarization of degree $2$).
For the proof of these two examples, we refer to \cite{LiPertusiZhao:TwistedCubics} (see also \cite{LLMS}).

\subsection*{Strategy of the proof: base change, Quot spaces, HN structures over a curve, and the support property}

The theory of base change for semiorthogonal decompositions has been developed by Kuznetsov \cite{Kuz:base-change}; the corresponding theory for t-structures is originally due to Abramovich and Polishchuk \cite{AP:t-structures,Polishchuk:families-of-t-structures}.
In Part~\ref{part:sod-t-structure-families} we recall and generalize these results to our context.
In particular, Theorem~\ref{Thm-D-bc} addresses an important technical point, base change of t-structures with respect to essentially perfect morphisms.

A second technical point of this paper is a rigorous treatment of Quot spaces for fiberwise t-structures.
This is the main goal of Section~\ref{sec:quotspaces}, in Part~\ref{part:ModuliSpaces}.
The existence of Quot spaces allows us, among other things, to give a proof of Langton-Maruyama's semistable reduction without using completions, see Theorem~\ref{thm:Langton}.
This result is fundamental for us, since it gives examples for HN structures over a curve, and is thus part of the proof of Theorem~\ref{MainThm:IntroConstruction}.  It is worth mentioning that we use the algebraicity of the moduli stack to prove the algebraicity of the Quot functor.  This implication is in the opposite direction from the construction of classical moduli spaces of sheaves, so our results depend heavily on the powerful machinery of \cite{lieblich:remarks-coherent-algebras}. 

The support property in Definition~\ref{MainDef:Intro}.\eqref{enum:introsupport} is studied in detail in Section~\ref{sec:StabCondOverS}.
In combination with boundedness and universal closedness of relative moduli spaces it ensures, for example, that openness of stability is preserved by deformations, by tilting, or when inducing stability conditions on the admissible subcategory $\cD$.
It also allows us to avoid imposing the noetherianity of the heart of the t-structure.

\subsection*{Relation with existing works and open questions}

The theory of deformations of Bridgeland stable objects and relative moduli spaces can be applied to other families of polarized non-commutative K3 surfaces.
In fact, Theorem~\ref{MainThm:IntroYoshioka} and its Corollaries will hold in a polarized family of non-commutative K3 surfaces,
as long as at least one fiber is an actual K3 surface, and stability conditions exist on the whole family. 
For example, the case of Gushel--Mukai fourfolds (see \cite{IlievManivel:GM4,DIM:GM4,DebarreKuznetsov:GM,DebarreKuznetsov:GM2,KP:GM,KP:CatJoins}) has now been established in \cite{GM-stability}. 
Another interesting example to study is the case of Debarre--Voisin manifolds \cite{DV:HK}.

Even in the case of cubic fourfolds, there are other examples of families of polarized hyperk\"ahler manifolds associated to some geometric construction which should be possible to interpret as relative moduli spaces.
For instance, in Corollary~\ref{MainCor:IntroLocCompleteFam}, when $a=2$ and $b=2$ the (singular) relative moduli space is birational to the construction in \cite{Voisin:twisted} by \cite{LiPertusiZhao:EllipticQuintics}.

\subsection*{Structure of the paper}
Part~\ref{part:sod-t-structure-families} and Part~\ref{part:ModuliSpaces} of the paper are mostly review sections; in these sections, we prove results about base change for triangulated categories and t-structures, and existence of moduli spaces in the generality needed in the paper.
In particular, base change results for t-structures for essentially perfect morphisms in Section~\ref{section-bc-t-structure} are new, as well as the complete proof of representability of the Quot functor in Section~\ref{sec:quotspaces}.

Part~\ref{part:HNStrCurve} is the first instance of a (weak) stability condition over a base.
We consider the case where the base is a Dedekind scheme and develop the theory of HN structures over it.
We do not yet introduce the support property, but we do require stronger conditions in the one-dimensional case, which essentially translate into the valuative criterion of properness for relative moduli spaces later on.
In Section~\ref{sec:StabCond} we briefly review Bridgeland stability conditions over a field, and prove a base change result in this setting.
In Section~\ref{sec:defnHNoveraCurve} we present the main definition and results for HN structures. In order to construct examples, 
we first extend this to the less well-behaved, and more technical, notion of weak stability conditions in Section~\ref{sec:defnweakHNstructure}. Then, in Section~\ref{sec:SemistableReduction}, we prove the fundamental result of semistable reduction in our context.
The remaining sections are more technical. Section~\ref{sec:HNstructuresTorsiontheory} shows that a HN structure comes with a well-behaved notion of $C$-torsion, 
and Section~\ref{sec:HNstructuresfibers} gives a criterion for
a collection of stability conditions on the fibers to ``integrate'' to a HN structure.
Finally, in Section~\ref{sec:tilt1}, we give conditions under which a (weak) HN structure can be rotated.

Part~\ref{part:higher-dimensional-bases} includes the main sections in the paper.
Definition~\ref{MainDef:Intro} is presented there (in Section~\ref{sec:FlatFamilyFiberwiseStability} and Section~\ref{sec:StabCondOverS}), together with the proofs of Theorem~\ref{MainThm:IntroDeformation} (in Section~\ref{sec:deformfamiliystability}) and Theorem~\ref{MainThm:IntroProper} (in Section~\ref{sec:StabCondOverS}).
In Section~\ref{sec:inducing}, we discuss how to induce stability conditions on semiorthogonal components, which will be important in applications, e.g.~to cubic fourfolds.

In Part~\ref{part:Tilting} we deal with the construction of stability conditions over a base, thus proving Theorem~\ref{MainThm:IntroConstruction} and its immediate application to Donaldson--Thomas invariants (in Section~\ref{sec:DT}).
Finally, the applications to cubic fourfolds are contained in Part~\ref{part:CubicFourfolds}.

\subsection*{Acknowledgments}
The paper benefited from many useful discussions with Nicolas Addington, Jarod Alper, Tom Bridgeland, Fran\c{c}ois Charles, Olivier Debarre, Johan de Jong, Laure Flapan, Fabian Haiden, Daniel Halpern-Leistner, Brendan Hassett, Jochen Heinloth, Daniel Huybrechts, Alexander Kuznetsov, Chunyi Li, Kieran O'Grady, Arvid Perego, Laura Pertusi, Alexander Polishchuk, Giulia Sacc\`a, Claire Voisin, and Xiaolei Zhao.

Parts of the paper were written while the authors were visiting several institutions.
In particular, we would like to thank Institut des Hautes \'Etudes Scientifiques, Institut Henri Poincar\'e, Northeastern University, Universit\`a degli studi di Milano, University of Edinburgh, and the MSRI in Berkeley for the excellent working conditions.

We would also like to thank the referees for their careful reading of the manuscript and insightful comments.

\section{Setup, notation, and some terminology} \label{sec:setupnotation}

\subsection*{Main Setup}\label{MainSetup}
The main results in this paper work in the following setup:
\begin{itemize}
 \item $S$ is a noetherian \emph{Nagata scheme} which is quasi-projective over a noetherian affine scheme (for the existence of good moduli spaces, we need $S$ to have characteristic $0$);
 \item if $S$ is also integral, regular, and one-dimensional, we call it a \emph{Dedekind scheme}, and often write $C$ instead of $S$;
 \item $X$ is a noetherian scheme of finite Krull dimension;
 \item $g\colon X \to S$ is a flat projective morphism (the notion of projective morphism we use is the one in \citestacks{01W8}); in our constructions of stability conditions we further assume $g$ to be smooth;
 \item $\cD\subset\Db(X)$ is an $S$-linear strong semiorthogonal component of finite cohomological amplitude (see Definitions~\ref{def-strong-sod} and \ref{def:finitecohamplitude}).
\end{itemize}

We assume $S$ to be Nagata since in Parts~\ref{part:ModuliSpaces} and \ref{part:higher-dimensional-bases} we need the valuative criterion for relative moduli spaces of stable objects, or for Quot spaces.
However, we will only consider DVRs that are essentially of finite type over $S$; this is sufficient to deduce universal closedness or properness if $S$ is locally given by a Nagata ring, see Section~\ref{subsection-nagata}.
We refer to \citestacks{032E} for the definition of Nagata rings.
It is preserved by morphisms essentially of finite type, and it is implied by being excellent \citestacks{07QV}.
Examples include localizations of finite type ring extensions of a field, of $\Z$, or more generally of a Dedekind domain with generic point of characteristic zero.

\smallskip

The general assumptions in each part of the paper are the following.
Some results work is greater generality and we will explain this by recalling the precise assumptions in each section.
\begin{description*}
\item[Part~\ref{part:sod-t-structure-families}]\mbox{}
\begin{enumerate}[(1)]
\item Sections~\ref{sec:SO}--\ref{section-bc-t-structure} work with $g \colon X \to S$ a morphism of quasi-compact schemes with affine diagonal and $X$ noetherian of finite Krull dimension; 
\item In Section~\ref{sec:flat-torsion-tf-obj}, $g$ is in addition flat;
\item In Section~\ref{sec:inducingtstructures}, the key results hold when $g$ is smooth and proper and $S$ is noetherian with an ample line bundle.
\end{enumerate}
\item[Part~\ref{part:ModuliSpaces}]\mbox{}
\begin{enumerate}[(a)]
 \item $g$ is a flat, proper, finitely presented morphism of schemes which are quasi-compact with affine diagonal, where $X$ is noetherian of finite Krull dimension;
 \item $\cD\subset\Db(X)$ is an $S$-linear strong semiorthogonal component of finite cohomological amplitude.
\end{enumerate}
\item[Part~\ref{part:HNStrCurve}]\mbox{}
\begin{enumerate}[(a)]
 \item $g \colon \cX \to C$ is as in Part~\ref{part:ModuliSpaces} and it is in addition projective, where $C$ is a Dedekind scheme;
 \item $\cD\subset\Db(\cX)$ is as in Part~\ref{part:ModuliSpaces}.
\end{enumerate}
\item[Part~\ref{part:higher-dimensional-bases}]\mbox{}
\begin{enumerate}[(a)]
 \item $g \colon \cX \to S$ is as in Part~\ref{part:ModuliSpaces} and it is in addition projective, where $S$ is a Nagata scheme which is quasi-projective over a noetherian affine scheme;
 \item In Section~\ref{sec:inducing}, we further assume $g$ to be smooth;
 \item $\cD\subset\Db(\cX)$ is as in Part~\ref{part:ModuliSpaces}.
\end{enumerate}
\item[Part~\ref{part:Tilting}]\mbox{}
\begin{enumerate}[(a)]
 \item $g\colon \cX\to S$ is as in Part~\ref{part:ModuliSpaces} and it is in addition smooth, projective of relative dimension $n\leqslant 3$, where $S$ is a Nagata scheme which is quasi-projective over a noetherian affine scheme.
 \item $\cD\subset\Db(\cX)$ is as in Part~\ref{part:ModuliSpaces}.
\end{enumerate}
\item[Part~\ref{part:CubicFourfolds}]\mbox{}
\begin{enumerate}[(a)]
 \item $S$ is a quasi-projective variety over $\C$;
 \item $g\colon \cX\to S$ is a smooth, projective family of cubic fourfolds.
\end{enumerate}\end{description*}

\subsection*{Derived categories}

For a scheme $X$ we consider various versions of the derived category: 
\begin{itemize}
\item the category of perfect complexes $\Dperf(X)$, 
\item the category of pseudo-coherent complexes with bounded cohomology $\Db(X)$, 
\item the category of pseudo-coherent complexes $\Dpc(X)$, and 
\item the unbounded derived category of $\cO_X$-modules with quasi-coherent cohomology $\Dqc(X)$, 
\item the unbounded derived category of $\cO_X$-modules $\rD(X)$.
\end{itemize}
These categories are related by inclusions 
\begin{equation*}
\Dperf(X) \subset \Db(X) \subset \Dpc(X) \subset \Dqc(X) \subset \rD(X).
\end{equation*}
For background on pseudo-coherent complexes, see~\citestacks{08E4}.
We note that a complex $E \in \rD(X)$ is pseudo-coherent if and only if there exists an open cover $X = \bigcup U_i$ such that $\res{E}{U_i}$ is quasi-isomorphic to a bounded above complex of finitely generated locally free sheaves on $U_i$;
indeed, the forward direction follows directly from the definitions, and the converse follows from~\citestacks{064U}.
If $X$ is noetherian, then pseudo-coherence boils down to a more classical notion: 
$\Dpc(X)$ coincides with the bounded above derived category of $\cO_X$-modules with coherent cohomologies \citestacks{08E8}.
Moreover, in the noetherian case we have equivalences 
$\Db(X) \simeq \Db(\Coh X)$ and $\Dqc(X) \simeq \rD(\Qcoh X)$.

We use the ``classical'' language of triangulated categories, with the following exceptions 
that occur in a few places in Sections~\ref{sec:SO} and \ref{section-bc-t-structure}.
The triangulated categories we consider arise as subcategories of 
the derived categories of schemes, and hence come with natural stable $\infty$-category enhancements.
All colimits in such categories are meant in the $\infty$-categorical sense.
Moreover, we sometimes use the enhanced structure to make natural constructions, e.g.~tensor products, with such categories; see Remark~\ref{Rem-S-linear} for further discussion.

We exclusively use the language of ``classical'' scheme theory.
Using derived algebraic geometry, it would be possible to remove various flatness and transversality 
assumptions.
We have not done so because our results already cover the situations of interest in applications.

Unless otherwise stated, all derived functors (e.g.~pushforward, pullback, tensor 
products) will be denoted as if they are underived.
We denote the $S$-relative derived sheaf Hom by 
\begin{equation*}
\cHom_S(-,-) = g_* \cHom(-,-).
\end{equation*}

Given $E \in \Dqc(S)$ and a morphism $T \to S$, we let $X_T$ denote the base change of $X$, and $E_T$ the pullback of $E$.
In particular, for a closed subset $W \subset S$, we write $i_W \colon X_W \to X$ for the embedding of the fiber over $W$, and $E_W = i_W^* E$ for the derived restriction.
We write $s \in S$ for (closed or non-closed) points of $S$, and $E_s$ for the pullback to the residue field $k(s)$.
In the case of a Dedekind scheme $C$, we write $p \in C$ for a closed point, $c \in C$ for an arbitrary point, $K$ for its fraction field, and $E_p$, $E_C$ or $E_K$ for the corresponding pullbacks.

In Section~\ref{sec:SO} we will consider an \emph{$S$-linear strong semiorthogonal component $\cD \subset \Db(X)$ of finite cohomological amplitude}; we will write $\cD_T$ for its base change to $T$ when that exists, in particular $\cD_s$ and $\cD_W$ for its base change to points or closed subschemes of $S$, and similarly $\cD_p$, $\cD_K$, or $\cD_c$ in the case of a Dedekind scheme $C$.

\newpage
\part{Semiorthogonal decompositions and t-structures in families}
\label{part:sod-t-structure-families}

\section{Semiorthogonal decompositions and base change}
\label{sec:SO}

In this section, following \cite{Kuz:base-change}, we recall some results on base change for semiorthogonal decompositions.

\subsection{Semiorthogonal decompositions}\label{subsec:SO}

\begin{Def}
\label{def-sod}
Let $\cD$ be a triangulated category.
A \emph{semiorthogonal decomposition} 
\begin{equation*}
\cD = \langle \cD_1, \dots, \cD_n \rangle
\end{equation*}
\index{D@$\cD=\langle \cD_1, \dots, \cD_n \rangle$, semiorthogonal decomposition}
is a sequence of full triangulated subcategories $\cD_1, \dots, \cD_n$ of $\cD$ 
--- called the \emph{components} of the decomposition --- such that: 
\begin{enumerate}[{\rm (1)}] 
\item \label{definition-sod-1} 
$\Hom(F, G) = 0$ for all $F \in \cD_i$, $G \in \cD_j$ and $i>j$.
\item \label{definition-sod-2}
For any $F \in \cD$, there is a sequence of morphisms
\begin{equation*}
0 = F_n \to F_{n-1} \to \cdots \to F_1 \to F_0 = F,
\end{equation*}
such that $\cone(F_i \to F_{i-1}) \in \cD_i$ for $1 \leqslant i \leqslant n$.
\end{enumerate}
\end{Def}

\begin{Rem}
\label{remark-projection-functors}
Condition~\eqref{definition-sod-1} of the definition implies the 
``filtration'' in~\eqref{definition-sod-2} and its ``factors'' are unique and functorial.
The functor $\pr_i \colon \cD \to \cD$ given by the $i$-th ``factor'', i.e., 
\begin{equation*}
\pr_i(F) = \cone(F_i \to F_{i-1}),
\end{equation*}
is called the \emph{projection functor} onto $\cD_i$.
\end{Rem}

A full triangulated subcategory $\cC \subset \cD$ is called \emph{right admissible} 
if the embedding functor \mbox{$\gamma \colon \CC \to \cD$} admits a right adjoint $\gamma^!$, 
\emph{left admissible} if $\gamma$ admits a left adjoint $\gamma^*$, 
and \emph{admissible} if $\gamma$ admits both right and left adjoints.
If $\cD = \langle \cD_1, \cD_2 \rangle$ is a semiorthogonal decomposition, 
then $\cD_1$ is left admissible and $\cD_2$ is right admissible.
Vice versa, if $\CC \subset \cD$ is left admissible then there is a semiorthogonal decomposition $\cD = \langle \CC, {^\perp}\CC \rangle$, and if $\CC \subset \cD$ is right 
admissible then there is a semiorthogonal decomposition $\cD = \langle \CC^{\perp}, \CC \rangle$.
Here, 
\begin{align*}
{}^{\perp}\CC & = \{ \, G \in \cD \sth \Hom(G, F) = 0 \text{ for all } F \in \CC \, \}, \\ 
\CC^{\perp} & = \{ \, G \in \cD \sth \Hom(F, G) = 0 \text{ for all } F \in \CC \, \}, 
\end{align*} 
denote respectively the \emph{left} and \emph{right orthogonals} to $\CC \subset \cD$.
These remarks lead to the notion of mutation functors.

\begin{Def}
\label{def-mutation-functor}
Let $\cC \subset \cD$ be an inclusion of triangulated categories.
If $\cC$ is right admissible, then the inclusion $\alpha \colon \cC^{\perp} \to \cD$ admits a left adjoint 
$\alpha^*$, and the functor 
\begin{equation*}
\rL_{\cC} = \alpha \circ \alpha^* \colon \cD \to \cD
\end{equation*} 
is called the \emph{left mutation functor} through $\cC$.
\index{LC@$\rL_{\cC}$, left mutation functor through $\cC$}
Similarly, if $\cC \subset \cD$ is left admissible, then the inclusion $\beta \colon {^\perp}\cC \to \cD$ 
admits a right adjoint $\beta^!$, and the functor 
\begin{equation*}
\rR_{\cC} = \beta \circ \beta^! \colon \cD \to \cD
\end{equation*} 
is called the \emph{right mutation functor} through $\cC$.
\index{RC@$\rR_{\cC}$, right mutation functor through $\cC$}
\end{Def}

\begin{Rem} When they exist, the mutation functors fit into exact triangles
\begin{equation*}
\gamma \circ \gamma^! \to \id_{\cD} \to \rL_{\cC} \quad \text{and} \quad 
\rR_{\cC} \to \id_{\cD} \to \gamma \circ \gamma^*.
\end{equation*} 
\end{Rem}

In our discussion of base change below, we need the following technical notions.

\begin{Def}
\label{def-strong-sod}
A semiorthogonal decomposition $\cD = \langle \cD_1, \dots, \cD_n \rangle$ is 
called \emph{strong} if for each $i$ the category $\cD_i$ is right admissible in $\cD$. We will call a subcategory $\cD' \subset \cD$ a \emph{strong semiorthogonal component} if it is part of a strong semiorthogonal decomposition.
\index{D@$\cD=\langle \cD_1, \dots, \cD_n \rangle$, semiorthogonal decomposition!strong}
\end{Def}

\begin{Rem}
Definition~\ref{def-strong-sod} is slightly different than the one given 
in~\cite[Definition~2.6]{Kuz:base-change}, but is easily seen to be equivalent.
\end{Rem}

\begin{Def} \label{def:finitecohamplitude}
Let $X$ be a scheme and let 
$\cD \subset \Dqc(X)$ be a triangulated subcategory.
If $Y$ is a scheme and $\Phi \colon \cD \to \Dqc(Y)$ is a triangulated functor, 
we say that $\Phi$ has \emph{cohomological amplitude} $[a,b]$ if 
\begin{equation*}
\Phi(\cD \cap \rD_{\qc}^{[p,q]}(X)) \subset \rD_{\qc}^{[p+a, q+b]}(Y) 
\end{equation*}
for all $p,q \in \Z$.
We say $\Phi$ has \emph{left finite cohomological amplitude} if $a$ can be chosen finite, 
\emph{right finite cohomological amplitude} if $b$ can be chosen finite, and 
\emph{finite cohomological amplitude} if $a$ and $b$ can be chosen finite.
We say that a semiorthogonal decomposition 
$\cD = \langle \cD_1, \dots, \cD_n \rangle$ is of \emph{(right or left) finite cohomological amplitude} 
if its projection functors have (right or left) finite cohomological amplitude.
\end{Def}

\subsection{Linear categories and base change} 
\label{subsection-linear-cats}
We will be interested in triangulated subcategories that occur as subcategories of the derived category of a scheme.
A crucial point is that there is a good notion of a ``family'' of such categories, made precise by the notion of a linear category.

\begin{Def}
Let $g \colon X \to S$ be a morphism of schemes.
A triangulated subcategory $\cD \subset \Dqc(X)$ is called \emph{$S$-linear} if it is stable 
with respect to tensoring by pullbacks of perfect complexes on $S$, i.e., if for every $F \in \cD$ and $G \in \Dperf(S)$ we have $F \otimes g^*G \in \cD$.
\index{D@$\cD=\langle \cD_1, \dots, \cD_n \rangle$, semiorthogonal decomposition!$S$-linear}
A semiorthogonal decomposition of $\cD$ is called \emph{$S$-linear} if all of its components are $S$-linear.
\end{Def}

\begin{Rem}
\label{remark-relative-Hom-sod}
By \cite[Lemma~2.7]{Kuz:base-change}, for an $S$-linear semiorthogonal decomposition the condition~\eqref{definition-sod-1} of Definition~\ref{def-sod} is equivalent to the following:
$\cHom_S(F,G) = 0$ for all $F \in \cD_i$, $G \in \cD_j$ and $i > j$.
\end{Rem}

In the rest of this subsection, we review a formalism of base change for linear categories from \cite{Kuz:base-change}.
Along the way we explain how to upgrade results proved in 
\cite{Kuz:base-change} for quasi-projective varieties over a field to more general settings.
Another generalization is discussed in \cite[Section~3]{belmans-sod}; we also refer to \cite[Proposition~3.13]{Antieau-Elmanto:semiorthogonal-descent} for a statement in derived algebraic geometry language.

Our main interest is in linear categories that occur as semiorthogonal components of $\Db(X)$ for a scheme $X$, 
but it is convenient to consider the whole chain of derived categories 
\begin{equation*} 
\Dperf(X) \subset \Db(X) \subset \Dpc(X) \subset \Dqc(X).
\end{equation*} 
Namely, we will see that base change for semiorthogonal decompositions of $\Db(X)$ can 
be obtained by a combination of inducing and restricting semiorthogonal decompositions 
along the above inclusions and base change for $\Dperf(X)$.

\begin{Lem}
\label{lem-sod-Db-perf}
Let $X \to S$ be a morphism of schemes where $X$ is noetherian of finite Krull dimension.
Let 
\begin{equation*}
\Db(X) = \langle \cD_1, \dots, \cD_n \rangle
\end{equation*} 
be a strong $S$-linear semiorthogonal decomposition.
Define $\cD_{i,\perf} = \cD_i \cap \Dperf(X)$.
Then there is an $S$-linear semiorthogonal decomposition 
\begin{equation*}
\Dperf(X) = \langle \cD_{1,\perf}, \dots, \cD_{n, \perf} \rangle .
\end{equation*} 
\end{Lem}

\begin{proof}
An object $E \in \Db(X)$ is called \emph{homologically finite} if 
for all $F \in \Db(X)$ we have $\Hom(E, F[i]) = 0$ for all but finitely many $i \in \Z$.
Let $\Db(X)_{\hf} \subset \Db(X)$ denote the full triangulated subcategory of 
homologically finite objects, and let $\cD_{i, \hf} = \cD_{i} \cap \Db(X)_{\hf}$.
Then by \cite[Proposition~1.10]{orlov-LG} we have a semiorthogonal decomposition 
\begin{equation*}
\Db(X)_{\hf} = \langle \cD_{1,\hf}, \dots, \cD_{n,\hf} \rangle.
\end{equation*}
So it suffices to show that $\Db(X)_{\hf} = \Dperf(X)$.
This follows by an easy modification of the argument of \cite[Proposition~1.11]{orlov-LG}.
\end{proof}

We will often work with quasi-compact, quasi-separated schemes.
The relevance of these conditions for us is the following result.

\begin{Prop}[{\cite{neeman, thomason-trobaugh}}]
\label{Prop-qcqs-cg}
If $X$ is a quasi-compact and quasi-separated scheme, then $\Dqc(X)$ is compactly generated with compact objects the perfect complexes.
\end{Prop}

\begin{Rem}
\label{Rem-qcqs-cg} 
By \cite[Propositions 3.6 and 3.9]{bzfn}, for a scheme $X$ we have $\Dqc(X) = \Ind(\Dperf(X))$ if and only 
if $\Dqc(X)$ is compactly generated with compact objects the perfect complexes; 
in particular, by Proposition~\ref{Prop-qcqs-cg} this holds if $X$ is quasi-compact and quasi-separated.
Here, $\Dqc(X)$ and $\Dperf(X)$ are regarded as $\infty$-categories, and 
$\Ind(\Dperf(X))$ denotes the category of Ind-objects in the sense of \cite[Section~5.3]{lurie-HTT}.
This recipe for recovering $\Dqc(X)$ from $\Dperf(X)$ will be used several times below.

We also note that in \cite{bzfn}, $X$ is called \emph{perfect} if it has affine 
diagonal and $\Dqc(X) = \Ind(\Dperf(X))$.
In particular, a quasi-compact scheme with affine diagonal is perfect.
Below we will often consider such schemes in order to take advantage of the results of \cite{bzfn}.
\end{Rem} 

Next we show that a semiorthogonal decomposition of $\Dperf(X)$ induces one of 
$\Dqc(X)$ and $\Dpc(X)$.
We use the following terminology.
Let $\Phi \colon \cD \to \cD'$ be a functor between triangulated categories.
We say that $\Phi$ is \emph{compatible} with semiorthogonal decompositions 
$\cD = \langle \cD_1, \dots, \cD_n \rangle$ and $\cD' = \langle \cD'_1, \dots, \cD'_n \rangle$ 
if $\Phi(\cD_i) \subset \cD'_i$ for all $i$.
If $\Phi$ is fully faithful, then $\Phi$ is compatible if and only if $\cD_i = \Phi^{-1}(\cD'_i)$, 
see \cite[Lemma~3.3]{Kuz:base-change}.
Moreover, if $\Phi$ is fully faithful and $\cD' = \langle \cD'_1, \dots, \cD'_n \rangle$ is given, 
then $\cD_i = \Phi^{-1}(\cD'_i)$ defines a semiorthogonal decomposition of $\cD$ if and only if 
the image of $\Phi$ is preserved by the projection functors of the decomposition of $\cD'$, 
see \cite[Lemma~3.4]{Kuz:base-change}.

\begin{Lem}
\label{lem-perf-qc}
Let $X \to S$ be a morphism of schemes where $X$ is quasi-compact and quasi-separated.
Let 
\begin{equation}
\label{lem-perf-qc-sodperf}
\Dperf(X) = \langle \cD_1, \dots, \cD_n \rangle 
\end{equation}
be an $S$-linear semiorthogonal decomposition.
\begin{enumerate}[{\rm (1)}] 
\item 
\label{perf-to-qc}
Define $\cD_{i,\qc} \subset \Dqc(X)$ to be the minimal triangulated 
subcategory of $\Dqc(X)$ which is closed under arbitrary direct sums 
and contains $\cD_i$.
Then there is an $S$-linear semiorthogonal decomposition
\begin{equation}
\label{DqcX-sod}
\Dqc(X) = \langle \cD_{1,\qc}, \dots, \cD_{n,\qc} \rangle 
\end{equation} 
whose projection functors are cocontinuous.
\index{Dqc@$\cD_{\qc}$, quasi-coherent component}

\item 
\label{perf-to-pc} 
Assume that the decomposition~\eqref{DqcX-sod} has right finite cohomological amplitude.
Define $\cD_{i,\pc} = \cD_{i, \qc} \cap \Dpc(X)$.
Then there is an $S$-linear semiorthogonal decomposition 
\begin{equation}
\label{Dpc-sod}
\Dpc(X) = \langle \cD_{1, \pc}, \dots, \cD_{n, \pc} \rangle.
\end{equation} 

\item 
\label{qc-amplitude}
Assume that $X$ is noetherian of finite Krull dimension, 
and that the decomposition~\eqref{lem-perf-qc-sodperf} 
is induced via Lemma~\ref{lem-sod-Db-perf} by a strong $S$-linear decomposition 
\begin{equation*}
\label{DbX-sod}
\Db(X) = \langle \cD_{1}^{\rb}, \dots, \cD_{n}^{\rb} \rangle 
\end{equation*}
of right finite cohomological amplitude.
Then the inclusion $\Db(X) \to \Dqc(X)$ is compatible with the given decomposition of $\Db(X)$ and the decomposition~\eqref{DqcX-sod} of $\Dqc(X)$, and the projection functors of $\Dqc(X)$ have the same cohomological amplitude 
as those of $\Db(X)$.
\end{enumerate}
\end{Lem}

\begin{proof}
Using Proposition~\ref{Prop-qcqs-cg}, 
the argument of \cite[Proposition~4.2]{Kuz:base-change} 
goes through to prove~\eqref{perf-to-qc} in the stated generality.
Alternatively, $\Dqc(X) = \Ind (\Dperf(X))$ by {Remark~\ref{Rem-qcqs-cg}}, 
so the result follows from the more general \cite[Lemma~3.12]{NCHPD}, 
which also shows $\cD_{i, \qc} = \Ind(\cD_i)$.

For~\eqref{perf-to-pc} we must show that the projection functors $\hat{\pr}_i \colon \Dqc(X) \to \Dqc(X)$ 
of the decomposition~\eqref{DqcX-sod} preserve $\Dpc(X)$.
We use the following characterization of pseudo-coherent objects on a quasi-compact and quasi-separated scheme, see \citestacks{0DJN}: 
an object $F \in \Dqc(X)$ is pseudo-coherent if and only if $F = \colim_{k \in \Z} F_k$ where $F_k$ is perfect and for any $n \in \Z$ the map $\tau^{\geqslant n} F_k \to \tau^{\geqslant n} F$ is an isomorphism for $k \gg 0$.
Given such an $F$, by the cocontinuity of $\hat{\pr}_i$ we have 
\begin{equation}
\label{hatpri-F}
\hat{\pr}_i(F) \simeq \colim \hat{\pr}_i(F_k).
\end{equation}
By construction, $\hat{\pr}_i$ restricts to the projection functor of the decomposition 
\eqref{lem-perf-qc-sodperf} of $\Dperf(X)$, so $\hat{\pr}_i(F_k)$ is perfect.
Defining $C_k$ as the cone of $F_k \to F$, the above condition on the truncations of $F_k \to F$ can be rephrased as follows:
for any $n \in \Z$ we have $C_k \in \rD_{\qc}^{\leqslant n}(X)$ for $k \gg 0$.
By the assumption that $\hat{\pr}_i$ has right finite cohomological amplitude, it follows that 
for any $n \in \Z$ we have \mbox{$\hat{\pr}_i(C_k) \in \rD_{\qc}^{\leqslant n}(X)$} for \mbox{$k \gg 0$}.
In other words, for any $n \in \Z$ the map $\tau^{\geqslant n} \hat{\pr}_i(F_k) \to \tau^{\geqslant n} \hat{\pr}_i(F)$ is an isomorphism for $k \gg 0$.
All together, this proves that $\hat{\pr}_i(F)$ satisfies the above criterion for pseudo-coherence.

Now we turn to~\eqref{qc-amplitude}.
To show the compatibility of $\Db(X) \to \Dqc(X)$, we must show that if $F \in \cD_i^{\rb}$ then $\hat{\pr}_i(F) \simeq F$.
Since $F$ is in particular pseudo-coherent, the description~\eqref{hatpri-F} above applies, 
so we need to show $\colim \hat{\pr}_i(F_k) \simeq F$.
By construction both $\hat{\pr}_i$ and the projection functor $\pr_{i}^{\rb}$ of~\eqref{DbX-sod} 
restrict to the same projection functor of the decomposition 
\eqref{lem-perf-qc-sodperf}, hence $\colim \hat{\pr}_i(F_k) \simeq \colim {\pr}_i^{\rb}(F_k)$.
As above, we consider the cone $C_k$ of $F_k \to F$, which we note lies in $\Db(X)$ in our current situation.
Since $\pr_{i}^{\rb}(F) \simeq F$, we obtain an exact triangle 
\begin{equation*}
\pr_i^{\rb}(F_k) \to F \to \pr_i^{\rb}(C_k).
\end{equation*} 
Using the assumption that $\pr_i^{\rb}$ has right finite cohomological amplitude, the argument 
from the previous paragraph shows that for any $n \in \Z$ the map 
$\tau^{\geqslant n} {\pr}^{\rb}_i(F_k) \to \tau^{\geqslant n} F$ is an isomorphism for $k \gg 0$.
Applying \citestacks{0CRK} 
(taking $H$ there to be the cohomology functors on $\rD(X)$), it follows that 
$\colim {\pr}^{\rb}_i(F_k) \simeq F$, as required.

Finally, we show that if $\pr_i^{\rb}$ has cohomological amplitude $[a_i, b_i]$, then 
so does $\hat{\pr}_i$.
In other words, we claim that for any object $F \in \rD_{\qc}^{[p,q]}(X)$ we have 
$\hat{\pr}_i(F) \in \rD_{\qc}^{[p+a_i, q+b_i]}(X)$.
By Lemma~\ref{lemma-qc-colimit-Db} below, 
we can write such an object as filtered colimit $F = \colim F_\alpha$ where $F_\alpha \in \Db(X)^{[p,q]}$.
Then we have 
\begin{equation*}
\hat{\pr}_i(F) \simeq \colim \hat{\pr}_i(F_\alpha) \simeq \colim \pr_i^{\rb}(F_\alpha), 
\end{equation*} 
where the second equivalence holds by the compatibility of $\Db(X) \to \Dqc(X)$ 
shown above.
But $\pr_i^{\rb}(F_\alpha) \in \Db(X)^{[p+a_i, q+b_i]}$, so 
Lemma~\ref{lemma-qc-colimit-Db} gives $\hat{\pr}_i(F) \in \rD_{\qc}^{[p+a_i, q+b_i]}(X)$.
\end{proof}

\begin{Lem}
\label{lemma-qc-colimit-Db}
Let $X$ be a noetherian scheme.
Let $F \in \Dqc(X)$.
Then $F \in \rD_{\qc}^{[a,b]}(X)$ if and only if $F = \colim F_\alpha$ for a filtered system of 
$F_\alpha \in \Db(X)^{[a,b]}$.
\end{Lem}

\begin{proof}
By \citestacks{09T4} we have $\Dqc(X) \simeq \rD(\QCoh X)$, so by \cite[Propositions 1.3.5.21 and 1.4.4.13]{lurie-HA} the truncation functors for the standard t-structure on $\Dqc(X)$ commute with filtered colimits.
This implies that if $F = \colim F_\alpha$ with $F_\alpha \in \Db(X)^{[a,b]}$, 
then $F \in \rD_{\qc}^{[a,b]}(X)$.
Conversely, since $\Dqc(X) = \Ind(\Dperf(X))$ 
by {Remark~\ref{Rem-qcqs-cg}}, 
any $F \in \Dqc(X)$ can be expressed as a filtered colimit 
$F = \colim G_{\alpha}$ where $G_{\alpha} \in \Dperf(X)$.
If $F \in \rD_{\qc}^{[a,b]}(X)$, then we have 
\begin{equation*}
F \simeq \tau^{\geqslant a} \tau^{\leqslant b} F \simeq \colim \tau^{\geqslant a} \tau^{\leqslant b} G_{\alpha} .
\end{equation*} 
Hence we may take $F_{\alpha} = \tau^{\geqslant a} \tau^{\leqslant b} G_{\alpha} \in \Db(X)^{[a,b]}$.
\end{proof} 

Now we discuss base change for semiorthogonal decompositions.
Given morphisms of schemes $g \colon X \to S$ and $\phi \colon T \to S$, we 
use the following notation for the base change diagram: 
\begin{equation*}
\label{diagram-bc}
\xymatrix{
X_{T} \ar[d]_{g'} \ar[r]^{\phi'} & X \ar[d]^{g} \\
T \ar[r]^{\phi} & S 
}
\end{equation*}
Following~\cite{Kuz:base-change}, we say that $\phi \colon T \to S$ is \emph{faithful} with respect to $g \colon X \to S$ if the canonical morphism of functors $\phi^*g_* \to g'_*\phi'^*$ is an isomorphism; 
this holds, for instance, if either $\phi \colon T \to S$ or $g \colon X \to S$ is flat.

\begin{Prop}
\label{prop-sod-bc-perf-qc} 
Let $g \colon X \to S$ be a morphism of {schemes which are quasi-compact with affine diagonal}.
Let 
\begin{equation*}
\Dperf(X) = \langle \cD_1, \dots, \cD_n \rangle 
\end{equation*} 
be an $S$-linear semiorthogonal decomposition.
Let $\phi \colon T \to S$ be a morphism from {a scheme $T$ which is quasi-compact with affine diagonal,}
such that $\phi$ is faithful with respect to $g$.
Define the following categories: 
\begin{enumerate}[{\rm (1)}] 
\item $\cD_{iT} \subset \Dperf(X_T)$ is the minimal 
triangulated category subcategory which is idempotent complete\footnote{Recall that an additive category is idempotent complete if any idempotent $e \colon F \to F$ arises from a splitting $A = \operatorname{Im}(e) \oplus \Ker(e)$.} and 
contains $\phi'^*F \otimes g'^*G$ for all $F \in \cD_{i}$ 
and $G \in \Dperf(T)$.

\item $(\cD_{i,\qc})_T \subset \Dqc(X_T)$ is the 
minimal triangulated subcategory which is closed under arbitrary 
direct sums and contains $\cD_{iT}$.
\end{enumerate}
Then there are $T$-linear semiorthogonal decompositions 
\begin{align}
 \Dperf(X_T) & = \langle \cD_{1T}, \dots, \cD_{nT} \rangle , \\ 
\label{DqcXT0} \Dqc(X_T) & = \langle (\cD_{1,\qc})_T, \dots, (\cD_{n,\qc})_T \rangle , 
\end{align} 
where the projection functors of~\eqref{DqcXT0} are cocontinuous.
Moreover, the functors $\phi'_*$ and $\phi'^*$ are compatible with all of the 
above decompositions.
\end{Prop}

\begin{proof}
In the setting of quasi-projective varieties over a field, this is a combination of \cite[Proposition~5.1 and 5.3]{Kuz:base-change}.
The proof there goes through once we know the following claim: the category $\Dperf(X_T)$ coincides with the minimal idempotent complete triangulated subcategory of $\Dqc(X_T)$ which is idempotent complete and contains $\phi'^*F \otimes g'^*G$ for all $F \in \cD_{i}$ and $G \in \Dperf(T)$.
Note that $X_T$ agrees with the derived fiber product of $X$ and $T$ over $S$, 
since $\phi$ is faithful with respect to $g$.
Hence by \cite[Theorem~1.2]{bzfn} we have an equivalence 
\begin{equation*}
\Dperf(X_T) \simeq \Dperf(X) \otimes_{\Dperf(S)} \Dperf(T), 
\end{equation*}
so the above claim follows from \cite[Lemma~2.7]{NCHPD}.
Alternatively, we note that the proposition follows by combining the above equivalence, 
the fact that $X_T$ is perfect (see \cite[Proposition~3.24]{bzfn}), and \cite[Lemma~3.15 and 3.12]{NCHPD}.
\end{proof} 

\begin{Rem}
\label{Rem-S-linear}
As hinted in the proof, Proposition~\ref{prop-sod-bc-perf-qc} can be formulated in a much more general setting, 
which is sometimes useful.
We briefly summarize the situation; see \cite[Section~2.3]{NCHPD} for details.
Namely, there is an abstract notion of an \emph{$S$-linear $\infty$-category} (called simply an 
``$S$-linear category'' in \cite{NCHPD}), which is a small 
idempotent-complete stable $\infty$-category $\cC$ equipped with a module structure over 
$\Dperf(S)$.
These categories are organized into an $\infty$-category $\Cat_S$.
As an example, the semiorthogonal components $\cD_i$ in Proposition~\ref{prop-sod-bc-perf-qc} 
have a canonical $S$-linear $\infty$-category structure, and hence may be regarded as objects of $\Cat_S$.

Further, there is a notion of a \emph{presentable $S$-linear $\infty$-category} (called simply a 
``presentable $S$-linear category'' in \cite{NCHPD}), which is a presentable 
$\infty$-category equipped with a module structure over $\Dqc(S)$.
These categories also organize into an $\infty$-category $\PrCat_S$.
The operation of taking Ind-categories gives a functor $\Ind \colon \Cat_S \to \PrCat_S$.
For instance, $\cD_{i, \qc} = \Ind(\cD_i)$ is an object of $\PrCat_S$.

There are several advantages of working in this setting. For instance, for any $S$-scheme $T$ 
we get categories $\Dperf(T) \in \Cat_S$ and $\Dqc(T) \in \PrCat_S$, and using the monoidal structures 
on $\Cat_S$ and $\PrCat_S$ we can form for any $\cC \in \Cat_S$ and $\cC' \in \PrCat_S$ new 
categories 
\begin{equation*}
\cC_T = \cC \otimes_{\Dperf(S)} \Dperf(T) \in \Cat_T \quad \text{and} \quad \cC'_T = \cC' \otimes_{\Dqc(S)} \Dqc(T) \in \PrCat_T. 
\end{equation*}
This construction makes sense for \emph{any} $S$-scheme $T$, and plays well with semiorthogonal 
decompositions by \cite[Lemma~3.15 and 3.12]{NCHPD}. 
In case $\cC = \cD_i$ and $\cC' = \cD_{i, \qc}$, this construction recovers the categories from 
Proposition~\ref{prop-sod-bc-perf-qc} \emph{if} $T \to S$ is faithful with respect to $g$ and $T$ is 
{quasi-compact with affine diagonal}. 

Besides the use of this technology in the proof of Proposition~\ref{prop-sod-bc-perf-qc}, the only 
other time we will need it in this paper is in the proof of Theorem~\ref{thm-Dqc-bc} {below}. 
\end{Rem}

\begin{Thm}
\label{theorem-bc-sod}
Let $g \colon X \to S$ be a morphism of {schemes which are quasi-compact with affine diagonal,} 
where $X$ is noetherian of finite Krull dimension. 
Let 
\begin{equation}
\label{DbX-bc-sod}
\Db(X) = \langle \cD_1, \dots, \cD_n \rangle
\end{equation} 
be a strong $S$-linear semiorthogonal decomposition. 
Let $\phi \colon T \to S$ be a morphism from 
{a scheme $T$ which is quasi-compact with affine diagonal,} 
such that $\phi$ is faithful with respect to~$g$. 
Let 
\begin{equation*}
(\cD_{i,\perf})_T \subset \Dperf(X_T) \quad \text{and} \quad (\cD_{i,\qc})_T \subset \Dqc(X_T) 
\end{equation*}
be the categories obtained by combining Lemma~\ref{lem-sod-Db-perf} 
and Proposition~\ref{prop-sod-bc-perf-qc}, and define 
\begin{equation*}
(\cD_{i,\pc})_T = (\cD_{i,\qc})_T \cap \Dpc(X_T) \quad \text{and} \quad 
\cD_{iT} = (\cD_{i,\qc})_T \cap \Db(X_T) . 
\end{equation*} 
Then there are $T$-linear semiorthogonal decompositions 
\begin{align}
\label{DperfXT} \Dperf(X_T) & = \langle (\cD_{1, \perf})_T, \dots, (\cD_{n,\perf})_T \rangle, \\ 
\label{DqcXT} \Dqc(X_T) & = \langle (\cD_{1,\qc})_T, \dots, (\cD_{n,\qc})_T \rangle , 
\end{align} 
where the projection functors of~\eqref{DqcXT} are cocontinuous. 
If the decomposition~\eqref{DbX-bc-sod} has right finite cohomological amplitude, 
then the projection functors of~\eqref{DqcXT} have the same cohomological amplitude 
as those of~\eqref{DbX-bc-sod}. 
In this case, we have a $T$-linear semiorthogonal decomposition 
\begin{equation}
\label{DpcXT} \Dpc(X_T) = \langle (\cD_{1,\pc})_T, \dots, (\cD_{n, \pc})_T \rangle , 
\end{equation}
and if~\eqref{DbX-bc-sod} has finite cohomological amplitude we have a $T$-linear semiorthogonal decomposition 
\index{DT@$\cD_T$, base change of an $S$-linear strong semiorthogonal component $\cD \subset \Db(X)$ via $T\to S$} 
\begin{equation}
\label{DbXT} \Db(X_T) = \langle \cD_{1T}, \dots, \cD_{nT} \rangle. 
\end{equation}
Moreover, the functors $\phi'_*$ and $\phi'^*$ are compatible with all of the above decompositions. 
\end{Thm}

\begin{proof}
Combining Lemma~\ref{lem-sod-Db-perf} and Proposition~\ref{prop-sod-bc-perf-qc} 
gives the decompositions~\eqref{DperfXT} and~\eqref{DqcXT}, and guarantees that 
the projection functors of~\eqref{DqcXT} are cocontinuous. 
To prove the theorem, it suffices to show that if 
the decomposition~\eqref{DbX-bc-sod} has right finite cohomological amplitude, 
then the projection functors of~\eqref{DqcXT} have the same cohomological amplitude 
as those of~\eqref{DbX-bc-sod}. 
Indeed, then the decomposition~\eqref{DpcXT} holds 
by Lemma~\ref{lem-perf-qc}.\eqref{perf-to-pc}. 
Moreover, if~\eqref{DbX-bc-sod} has finite cohomological amplitude, then the projection 
functors of~\eqref{DpcXT} preserve $\Db(X_T) \subset \Dpc(X_T)$ because they are of finite cohomological amplitude, so~\eqref{DbXT} holds. 

Let $\pr_i, \hat{\pr}_i$, and $\hat{\pr}_{iT}$ denote the projection functors for the semiorthogonal decompositions 
of $\Db(X), \Dqc(X)$, and $\Dqc(X_T)$. 
We want to show that 
if $[a_i,b_i]$ is the cohomological amplitude of $\pr_i$, then 
it is also the amplitude of $\hat{\pr}_{iT}$, i.e., for any $F \in \rD_{\qc}^{[p,q]}(X_T)$ we 
have $\hat{\pr}_{iT}(F) \in \rD_{\qc}^{[p+a_i, q+b_i]}(X_T)$. 

If $U \subset T$ is an affine open subset, then Proposition~\ref{prop-sod-bc-perf-qc} also gives a semiorthogonal 
decomposition of $\Dqc(X_U)$, whose projection functors we denote by $\hat{\pr}_{iU}$. 
By the compatibility of the restriction functor $\Dqc(X_T) \to \Dqc(X_U)$ with the semiorthogonal decompositions, 
we have
\begin{equation*}
\hat{\pr}_{iT}(F)_{U} \simeq \hat{\pr}_{iU}(F_{U}) . 
\end{equation*} 
It follows that the claim on the amplitude of $\hat{\pr}_{iT}$ can be checked 
on an affine open cover of~$T$, so we may assume $T$ is affine. 

Since $T$ is affine and $S$ has affine diagonal, the morphism $\phi \colon T \to S$ is affine. 
Therefore so is the base change $\phi' \colon X_T \to X$. 
In particular, $\phi'_* \colon \Dqc(X_T) \to \Dqc(X)$ is conservative and t-exact, which 
implies that an object $F \in \Dqc(X_T)$ satisfies $F \in \rD_{\qc}^{[p,q]}(X_T)$ if and only if 
$\phi'_*(F) \in \rD_{\qc}^{[p,q]}(X)$. 
Further, by compatibility of $\phi'_*$ with the semiorthogonal decompositions, we have 
\begin{equation*}
\phi'_*(\hat{\pr}_{iT}(F)) \simeq \hat{\pr}_i(\phi'_*(F)). 
\end{equation*} 
It follows that to prove the result, we just need to show that $\hat{\pr}_i$ has 
the same cohomological amplitude as $\pr_i$. 
But this holds by Lemma~\ref{lem-perf-qc}.\eqref{qc-amplitude}. 
\end{proof} 

We will often apply the above results in the following way. 
Let $X \to S$ be a morphism of schemes with $X$ noetherian of finite Krull dimension. 
Let $\cD \subset \Db(X)$ be an admissible $S$-linear subcategory. 
Then it follows directly from the definitions that 
\begin{equation}
\label{cD-cDperp}
\Db(X) = \langle \cD, {^\perp}\cD \rangle 
\end{equation}
is a strong $S$-linear semiorthogonal decomposition. 
Hence by Lemmas~\ref{lem-sod-Db-perf} and~\ref{lem-perf-qc} we obtain $S$-linear 
left admissible subcategories $\cD_{\perf} \subset \Dperf(X)$ and $\cD_{\qc} \subset \Dqc(X)$. 

Moreover, {if $\phi \colon T \to S$ is a morphism 
which is faithful with respect to $X \to S$ and all of $X, S$, and $T$ are quasi-compact with 
affine diagonal,} then 
by Theorem~\ref{theorem-bc-sod} we obtain $T$-linear left admissible subcategories 
$(\cD_{\perf})_T \subset \Dperf(X_T)$ and $(\cD_{\qc})_T \subset \Dqc(X_T)$. 
We call the projection functor onto $\cD$ in the decomposition~\eqref{cD-cDperp} 
the \emph{left projection functor of $\cD$}. 
The decomposition~\eqref{cD-cDperp} has finite cohomological amplitude if this 
projection functor does. 
Hence in this case we also obtain by Theorem~\ref{theorem-bc-sod} a $T$-linear 
left admissible subcategory $\cD_{T} \subset \Db(X_T)$. 

If $T \to S$ and $T' \to S$ are morphisms of quasi-compact schemes with affine diagonal which are faithful with respect to $X \to S$, and if $f \colon T' \to T$ is a morphism over $S$, then we consider the base change 
morphism $f' \colon X_{T'} \to X_{T}$. 
It follows from the definitions that pushforward and pullback along $f'$ induce functors 
\begin{equation*}
f'_* \colon (\cD_{\qc})_{T'} \to (\cD_{\qc})_T \quad \text{and} \quad 
f'^* \colon (\cD_{\qc})_T \to (\cD_{\qc})_{T'}. 
\end{equation*} 
Further, if the projection functor of $\cD$ has finite cohomological amplitude, 
then we get a functor $f'_* \colon \cD_{T'} \to \cD_T$ if $f$ is proper, 
and a functor $f'^* \colon \cD_{T} \to \cD_{T'}$ if $f$ has finite $\Tor$-dimension.
We will often refer to pullback as restriction and use the notation $\res{F}{T} = f'^*F$. 

Below we use the following observations. 
The inclusion $\phi \colon U \to S$ of an open subscheme is always faithful 
with respect to $X \to S$ and of finite $\Tor$-dimension; 
{moreover, if $S$ has affine diagonal then so does $U$.} 
Similarly, if $S = \Spec(A)$ and $B$ is a localization of $A$, then $\Spec(B) \to \Spec(A)$ is 
faithful with respect to $X \to S$ and of finite $\Tor$-dimension. 

\begin{Lem}
\label{lem-open-restriction-es} 
Let $X \to S$ be a morphism of {schemes which are quasi-compact with affine diagonal, where}
$X$ {is} noetherian of finite Krull dimension. 
Let $\cD\subset \Db(X)$ be an $S$-linear strong semiorthogonal component.
If $T \to S$ is either: 
\begin{enumerate}[{\rm (1)}] 
\item \label{es-open-inclusion}
the inclusion of a quasi-compact open subset, or 
\item \label{es-localization}
a morphism of affine schemes corresponding to a localization 
of rings $A \to B$, 
\end{enumerate}
then the pullback functors 
\begin{equation*}
\cD_{\qc} \to (\cD_{\qc})_T \quad \text{and} \quad 
\cD \to \cD_T 
\end{equation*}
are essentially surjective. 
For the statement for $\cD \to \cD_T$, we assume the projection functor of $\cD$ has 
finite cohomological amplitude so that $\cD_T$ is defined.
\end{Lem}

\begin{proof}
In both cases~\eqref{es-open-inclusion} and~\eqref{es-localization}, 
the pullback functors 
\begin{equation}
\label{Dqc-Db-restriction}
\Dqc(X) \to \Dqc(X_T) \quad \text{and} \quad 
\Db(X) \to \Db(X_T) 
\end{equation}
are essentially surjective. 
In case~\eqref{es-open-inclusion} this holds for $\Dqc(-)$ by \citestacks{08ED} and for $\Db(-)$ by~\cite[Lemma~2.3.1]{Polishchuk:families-of-t-structures}. 
In case~\eqref{es-localization} essential surjectivity holds by the same arguments. 
The result follows since 
pullback along $T \to S$ is compatible with the semiorthogonal 
decompositions of the source and target categories in~\eqref{Dqc-Db-restriction}.
\end{proof}

\subsection{Relative exceptional collections}\label{subsec:RelativeExceptional}

In the applications later in this paper, we will consider $S$-linear semiorthogonal 
decompositions induced by relative exceptional collections, which are defined 
and studied below. 

\begin{Def} \label{def:relativeexceptional} 
Let $X \to S$ be a morphism of schemes, and 
let $\cD \subset \Dqc(X)$ be an $S$-linear subcategory.
A \emph{relative exceptional object in $\cD$} is an object $E \in \cD$ such that $E$ is perfect and 
$\cHom_S(E, E) = \cO_S$. 
\index{E1@$E_1, \dots, E_m$, relative exceptional objects in $\cD$}
A \emph{relative exceptional collection in $\cD$} is a sequence $E_1, \dots, E_m$ of relative exceptional objects in $\cD$ such that $\cHom_S(E_i, E_j) = 0$ for all $i > j$.
\end{Def}

\begin{Rem}
If $S = \Spec(k)$ is a point, then $\cHom_S(-,-)$ is simply $\RHom(-,-)$ regarded as a 
$k$-complex. Hence the above definition reduces to the usual one in this case. 
\end{Rem}

We note that the relative Hom functor behaves well under base change. 

\begin{Lem}
\label{Lem-cHom-bc} 
Let $g \colon X \to S$ be a quasi-compact and quasi-separated morphism of schemes. 
Let $E \in \rD_{\perf}(X)$ and $F \in \Dqc(X)$. 
Let $\phi \colon T \to S$ be a morphism which is faithful with respect to $g$. 
Then we have 
\begin{equation*}
\cHom_S(E,F)_T \simeq \cHom_T(E_T, F_T). 
\end{equation*} 
\end{Lem}

\begin{proof}
Indeed, using the notation of~\eqref{diagram-bc}, we have 
\begin{equation*}
\cHom_S(E,F)_T \simeq g'_* \phi'^* \cHom_{X}(E, F) 
 \simeq g'_{*} \cHom_{X_T}(E_T, F_T) 
 = \cHom_T(E_T, F_T), 
\end{equation*} 
where the first isomorphism holds by base change~\citestacks{08IB} and the second since $E$ is perfect.
\end{proof} 

The property that a set of perfect objects is a relative exceptional collection can be checked fiberwise: 

\begin{Lem}
Let $X \to S$ be a flat quasi-compact and quasi-separated morphism of schemes. 
Then an object $E \in \Dperf(X)$ is relative exceptional if and only if $E_s \in \Dperf(X_s)$ is exceptional 
for all points $s \in S$, if and only if $E_s \in \Dperf(X_s)$ is exceptional for all closed points $s \in S$. 
Similarly, a sequence $E_1, \dots, E_m \in \Dperf(X)$ is a relative exceptional collection 
if and only its restriction to $X_s$ is an exceptional collection for all points $s \in S$, 
or equivalently for all closed points $s \in S$. 
\end{Lem}

\begin{proof}
Let $E \in \Dperf(X)$. The canonical morphism $\cO_S \to \cHom_S(E, E)$ is an isomorphism if and only if its restriction 
$\kappa(s) \to \cHom_S(E,E)_s$ is an isomorphism for every point $s \in S$, 
or equivalently for every closed point $s \in S$. 
But by Lemma~\ref{Lem-cHom-bc} and the flatness of $g$, this restriction is identified 
with $\kappa(s) \to \cHom_s(E_s, E_s)$, which is an isomorphism if and only if $E_s$ is 
exceptional. This proves the first claim of the lemma, and the second follows similarly. 
\end{proof}

In the next lemma, we will use the following observation. 
If $g \colon X \to S$ is a proper morphism of finite Tor-dimension between 
noetherian schemes, then 
pushforward and pullback give functors $g_* \colon \Db(X) \to \Db(S)$ 
and $g^* \colon \Db(S) \to \Db(X)$. 
The functor $g_*$ is right adjoint to $g^*$. 
If the relative dualizing complex $\omega_g$ of $g$ is a shift of a line bundle on $X$, 
then $g_*$ also admits a right adjoint $g^! \colon \Db(S) \to \Db(X)$ and 
$g^*$ admits a left adjoint $g_! \colon \Db(X) \to \Db(S)$, given by 
\begin{equation*}
g^!(-) = g^*(-) \otimes \omega_g \quad \text{and} \quad g_!(-) = g_*(- \otimes \omega_g). 
\end{equation*} 
Indeed, the first formula holds by Grothendieck duality.
More precisely, \cite[Theorem~5.4]{neeman} applies because $S$ is separated and $g^!$ preserves arbitrary coproducts since $g_*$ sends perfect complexes to perfect complexes by our assumption. 
Moreover, since $\omega_g$ is the shift of a line bundle, $g^!$ preserves bounded derived categories.
The formula for $g_!$ follows from the one for $g^!$. 
The condition that $\omega_g$ is the shift of a line bundle  holds 
for instance if $g \colon X \to S$ is smooth, or, more generally, if $X$ and $S$ are smooth over some common base.  

\begin{Lem} 
\label{Lem-relexcept-adjoint}
Let $g \colon X \to S$ be a proper morphism of finite Tor-dimension between noetherian schemes. 
Let $E \in \Dperf(X)$ be a relative exceptional object. 
Then the functor 
\index{alphaE@$\alpha_E\colon \Db(S) \to \Db(X)$, pullback functor, tensor with a relative exceptional object $E$} 
\begin{equation*}
\alpha_E \colon \Db(S) \to \Db(X), ~ F \mapsto g^*(F) \otimes E. 
\end{equation*} 
is fully faithful, and admits a right adjoint given by 
\begin{equation*}
\alpha_E^! \colon \Db(X) \to \Db(S), ~ G \mapsto \cHom_S(E, G). 
\end{equation*} 
Moreover, if $\omega_g$ is a shift of a line bundle, 
then $\alpha_E$ admits a left adjoint given by 
\begin{equation*}
\alpha_E^* \colon \Db(X) \to \Db(S), ~ G \mapsto g_! \cHom_X(E, G). 
\end{equation*} 
\end{Lem}

\begin{proof}
For $F \in \Db(S)$ and $G \in \Db(X)$, we compute 
\begin{align*}
\Hom(\alpha_E(F), G) & = \Hom(g^*(F) \otimes E, G) \\
& \simeq \Hom(g^*(F), \cHom_X(E, G)) \\ 
& = \Hom(F, \cHom_S(E, G)), 
\end{align*} 
which proves the formula for $\alpha_E^!$. 
Moreover, we have 
\begin{align*}
\alpha_E^! \alpha_E(F) & = \cHom_S(E, g^*(F) \otimes E) \\ 
& = g_* \cHom_X(E, g^*(F) \otimes E) \\ 
& \simeq g_* (g^*(F) \otimes \cHom_X(E, E)) \\ 
& = F \otimes \cHom_S(E,E) %\\ &
\simeq F, 
\end{align*}
where the final line holds because $E$ is relative exceptional. 
This proves that $\alpha_E$ is fully faithful. 
Finally, an argument similar to the one for $\alpha_E^!$ shows the existence and 
claimed formula for $\alpha_E^*$ in the presence of a relative dualizing complex which is a shift of a line bundle. 
%perfect relative dualizing complex. 
\end{proof}

\begin{Rem}
If the assumption in Lemma~\ref{Lem-relexcept-adjoint} guaranteeing the existence 
of $\alpha_E^*$ holds and $X$ is in addition regular of finite Krull dimension, 
then in fact $\alpha_E^*(G) \simeq \cHom_S(G, E)^{\vee}$. 
Indeed, the regularity of $X$ guarantees that $\Db(X) = \Dperf(X)$, 
and then the claim follows easily from Grothendieck duality. 
\end{Rem} 

In the situation of Lemma~\ref{Lem-relexcept-adjoint}, we write $\rL_{E/S}$ and $\rR_{E/S}$ for 
the right and left mutation functors through $\alpha_E(\Db(S)) \subset \Db(X)$, see Definition~\ref{def-mutation-functor}. 
Note that these can be computed via exact triangles 
\begin{equation*}
g^*\cHom_S(E,F) \otimes E \to F \to \rL_{E/S}(F) \quad \text{and} \quad 
\rR_{E/S}(F) \to F \to g^*g_! \cHom_X(E, F). 
\end{equation*} 

\begin{Lem} 
\label{lem-E1Em-sod}
Let $g \colon X \to S$ be a proper morphism of finite Tor-dimension between noetherian schemes. 
Let $E_1, \dots, E_m$ be a relative exceptional collection in $\Dperf(X)$. 
Then there is an $S$-linear semiorthogonal decomposition of finite cohomological amplitude
\begin{equation}
\label{E1Em-sod1} \Db(X) = \langle \cD, \alpha_{E_1}(\Db(S)), \dots, \alpha_{E_m}(\Db(S)) \rangle, \\ 
\end{equation} 
where the left adjoint to the inclusion $\cD \to \Db(X)$ is given by 
\begin{equation}
\label{LE1Em}
\rL_{E_1/S} \circ \rL_{E_2/S} \cdots \circ \rL_{E_m/S} . 
\end{equation} 
If further $g \colon X \to S$ is smooth, then the components appearing in \eqref{E1Em-sod1} are all admissible; in particular, \eqref{E1Em-sod1} is a strong semiorthogonal decomposition. 
\end{Lem}

\begin{proof}
First we claim the sequence $\alpha_{E_1}(\Db(S)), \dots, \alpha_{E_m}(\Db(S))$ is semiorthogonal. 
For this, it suffices to show that the composition $\alpha_{E_i}^! \circ \alpha_{E_j}$ vanishes for $i > j$. 
For $F \in \Db(S)$, we compute 
\begin{align*}
\alpha_{E_i}^! \alpha_{E_j}(F) & = \cHom_S(E_i, g^*(F) \otimes E_j) \\ 
& = g_* \cHom_X(E_i, g^*(F) \otimes E_j) \\ 
& \simeq g_* (g^*(F) \otimes \cHom_X(E_i, E_j)) \\ 
& \simeq F \otimes \cHom_S(E_i, E_j) = 0 , 
\end{align*}
where in the third line we used that $E_i$ is perfect and in the fourth we used the projection formula.
Now the decomposition \eqref{E1Em-sod1} and the formula~\eqref{LE1Em} follow from \cite[Lemma~3.10]{NCHPD}, 
since the subcategory $\alpha_{E_i}(\Db(S)) \subset \Db(X)$ is right admissible by Lemma~\ref{Lem-relexcept-adjoint}. 

If further the morphism $g \colon X \to S$ is smooth, then it follows from Lemma~\ref{Lem-relexcept-adjoint} that the $\alpha_{E_i}(\Db(S))$ are in fact admissible in $\Db(X)$. 
Then again by \cite[Lemma~3.10]{NCHPD} it follows that the subcategory 
$\langle \alpha_{E_1}(\Db(S)), \dots, \alpha_{E_m}(\Db(S)) \rangle$ they generate is admissible. 
Thus $\cD$ is also admissible by Lemma~\ref{lem-smooth-proper-admissible} below. 
\end{proof} 

\begin{Lem}
\label{lem-smooth-proper-admissible}
Let $X \to S$ be a smooth and proper morphism of noetherian schemes.
Let 
\begin{equation*}
\Db(X) = \langle \cD_1, \cD_2 \rangle
\end{equation*} 
be an $S$-linear semiorthogonal decomposition.
Then $\cD_1$ is admissible if and only if $\cD_2$ is admissible. 
\end{Lem}

\begin{proof}
By our assumptions, $\Db(X)$ admits a relative Serre functor $\rS_{\Db(X)/S}$; see the discussion preceding Corollary~\ref{cor:induce-relative-t} for a review of this notion. 
Assume that $\cD_1$ is admissible. 
Then there is a semiorthogonal decomposition $\Db(X) = \langle \cD'_2, \cD_1 \rangle$. Moreover, it follows directly from the definition of the Serre functor that the autoequivalence $\rS_{\Db(X)/S}$ takes $\cD_2$ to $\cD_2'$. 
Since $\cD'_2$ admits a left adjoint, it follows that $\cD_2$ does too. 
Hence $\cD_2$ is admissible. 
The proof that admissibility of $\cD_2$ implies that of $\cD_1$ is similar.
\end{proof}

\section{Local t-structures}
\label{sec:Polishchuk}

In this section, we discuss basic definitions and results about local t-structures. 
We start by recalling the absolute case. 

\subsection{t-structures}\label{subsec:Tstructures}

\begin{Def}\label{def:tstructure}
A \emph{t-structure} $\tau$ on a triangulated category $\cD$ is a pair of full subcategories $(\cD^{\leqslant 0},\cD^{\geqslant 0})$ satisfying the following conditions:
\index{tau@$\tau=(\cD^{\leqslant 0},\cD^{\geqslant 0})$, t-structure on $\cD$}
\begin{enumerate}[{\rm (1)}] 
\item $\cD^{\leqslant 0}[1]\subset\cD^{\leqslant 0}$ and $\cD^{\geqslant 0}[-1]\subset\cD^{\geqslant 0}$;
\item \label{def-t-structure-orthogonality} 
$\Hom(F,G)=0$, for every $F\in\cD^{\leqslant 0}$ and $G\in\cD^{\geqslant 1}$;
\item every object $E\in \cD$ fits into an exact triangle
\[
\tau^{\leqslant 0} E\to E \to \tau^{\geqslant 1} E\to \tau^{\leqslant 0} E[1]
\]
with $\tau^{\leqslant 0} E\in \cD^{\leqslant 0}$ and $\tau^{\geqslant 1}E \in \cD^{\geqslant 1}$.
\end{enumerate}
\end{Def}

Here we used the notation $\cD^{\leqslant n}:=\cD^{\leqslant 0}[-n]$ and $\cD^{\geqslant n}:=\cD^{\geqslant 0}[-n]$, for any $n\in \Z$.
Similarly, for the truncation functors $\tau^{\leqslant n}, \tau^{\geqslant n}$.
Moreover, we let $\cD^{[a,b]} =\cD^{\leqslant b}\cap \cD^{\geqslant a}$, for all $a,b\in\Z \cup \set{ \pm \infty}$, $a\leqslant b$.
 
\begin{Def}\label{def:heart}
The \emph{heart} of a t-structure $(\cD^{\leqslant 0},\cD^{\geqslant 0})$ is 
the abelian category $\AA=\cD^{\leqslant 0}\cap \cD^{\geqslant 0}$.
\index{A@$\cA$, heart of a t-structure $(\cD^{\leqslant 0},\cD^{\geqslant 0})$}
\end{Def}

The cohomology objects of an object $E\in\cD$ with respect to the heart of a t-structure $\cA$ will be denoted by $\rH_\AA^\bullet(E)$.
When $\cD=\Db(X)$ and $\cA=\Coh X$ we simply write $\rH^\bullet(E)$.

\begin{Def}
A t-structure $(\cD^{\leqslant 0},\cD^{\geqslant 0})$ is \emph{bounded} if 
$\cD = \bigcup_{n,m\in\Z} \cD^{\leqslant n}\cap \cD^{\geqslant m}$.
\index{tau@$\tau=(\cD^{\leqslant 0},\cD^{\geqslant 0})$, t-structure on $\cD$!bounded}
\end{Def}

\begin{Rem}
Our terminology for bounded t-structures follows \cite{Bridgeland:Stab}. 
Such a t-structure is called nondegenerate in 
\cite{AP:t-structures,Polishchuk:families-of-t-structures}, and bounded and nondegenerate in \cite{BBD}. 
\end{Rem}

A bounded t-structure is uniquely determined by its heart.

\begin{Prop}[{\cite[Lemma~3.1]{Bridgeland:K3}}]\label{prop:BridgelandCriterion}
If $\AA \subset\cD$ is a full additive subcategory of a triangulated category $\cD$, 
then $\AA$ is the heart of a bounded t-structure on $\cD$ if and only if 
\begin{enumerate}[{\rm (1)}]
\item\label{eq:BridgelandCriterion1} for $F,G\in\AA$, we have $\Hom_\cD (F, G[k]) = 0$, for all $k<0$; and
\item\label{eq:BridgelandCriterion2} for all $E \in \cD$, there are integers $m < n$ and a collection of triangles
\begin{equation*}\label{eqn:collTriang}
\xymatrix@C=.3em{
0_{\ } \ar@{=}[r] & E_{m-1} \ar[rrrr] &&&& E_{m} \ar[rrrr] \ar[dll] &&&& E_{m+1}
\ar[rr] \ar[dll] && \ldots \ar[rr] && E_{n-1}
\ar[rrrr] &&&& E_n \ar[dll] \ar@{=}[r] & E_{\ } \\
&&& A_m \ar@{-->}[ull] &&&& A_{m+1} \ar@{-->}[ull] &&&&&&&& A_n \ar@{-->}[ull]
}
\end{equation*}
with $A_i[i] \in\AA$ for all $i$.
\end{enumerate}
\end{Prop}

\begin{Def}\label{def:TorsionPair} Let $\AA\subset\cD$ be the heart of a t-structure.
\index{T@$(\cT,\cF)$, torsion pair of an abelian category $\cA$}
A pair of additive subcategories $(\cT,\cF)$ of $\AA$ is called a \emph{torsion pair} if
\begin{itemize}
\item for all $T\in\cT$ and for all $F\in\cF$, we have $\Hom(T,F)=0$,
\item for all $E\in\AA$, there exist $T_E\in\cT$, $F_E\in\cF$, and an exact sequence
\[
0\to T_E \to E \to F_E \to 0.
\]
\end{itemize}
\end{Def}

Note that if $(\cT, \cF)$ is a torsion pair, then $\cF = \cT^{\perp}$ is the right orthogonal to $\cT$ in $\cA$; we will call a subcategory $\cT \subset \cA$ a \emph{torsion subcategory} if $(\cT, \cT^{\perp})$ forms a torsion pair. 

Given a torsion pair $(\cT,\cF)$ in $\AA$, we can \emph{tilt} to obtain a new t-structure $(\cD^{\sharp, \leqslant 0}, \cD^{\sharp, \geqslant 0})$
\index{D@$(\cD^{\sharp, \leqslant 0}, \cD^{\sharp, \geqslant 0})$, tilted t-structure at a torsion pair}
with
\[ \cD^{\sharp, \leqslant 0} := \left\{ E \in \cD^{\leqslant 0} \sth \rH^{0}_\cA(E) \in \cT \right\}, \quad
\cD^{\sharp, \geqslant 0} := \left\{ E \in \cD^{\geqslant -1} \sth \rH^{-1}_\cA(E) \in \cF \right\},\]
see \cite{Happel-al:tilting}.
\index{Asharp@$\cA^\sharp= \langle \cF[1], \cT \rangle$, tilted heart at a torsion pair}
Its heart $\AA^\sharp$ can be described as the extension-closure
\[
\AA^\sharp := \langle \cF[1], \cT \rangle.
\]
If $(\cD^{\leqslant 0}, \cD^{\geqslant 0})$ is bounded, then so is $(\cD^{\sharp, \leqslant 0}, \cD^{\sharp, \geqslant 0})$. By \cite[Lemma~1.1.2]{Polishchuk:families-of-t-structures}, 
$(\cD^{\sharp, \leqslant 0}, \cD^{\sharp, \geqslant 0})$ can be obtained by tilting from $(\cD^{\leqslant 0}, \cD^{\geqslant 0})$ if and only if
$\cD^{\leqslant -1} \subset \cD^{\sharp, \leqslant 0}_{} \subset \cD^{\leqslant 0}$.

\begin{Def}
A t-structure $(\cD^{\leqslant 0},\cD^{\geqslant 0})$ is
\begin{enumerate}[{\rm (1)}] 
\item \emph{noetherian} if its heart is noetherian, and
\index{tau@$\tau=(\cD^{\leqslant 0},\cD^{\geqslant 0})$, t-structure on $\cD$!noetherian}
\item \emph{tilted-noetherian} if it can be obtained from a noetherian t-structure $(\cD_0^{\leqslant 0}, \cD_0^{\geqslant 0})$ 
on $\cD$ by tilting, i.e., if $\cD_0^{\leqslant -1} \subset \cD^{\leqslant 0}_{} \subset \cD_0^{\leqslant 0}$. 
\index{tau@$\tau=(\cD^{\leqslant 0},\cD^{\geqslant 0})$, t-structure on $\cD$!tilted-noetherian}
\end{enumerate}
\end{Def}

In~\cite{Polishchuk:families-of-t-structures} a 
tilted-noetherian t-structure is called close to noetherian. 

\begin{Def}
Let $\cD_1$ and $\cD_2$ be a pair of triangulated categories equipped with t-structures.
An exact functor $\Phi\colon \cD_1\to\cD_2$ is called \emph{left} (resp. \emph{right}) \emph{t-exact} if $\Phi(\cD_1^{\geqslant 0})\subseteq \cD_2^{\geqslant 0}$ (resp. $\Phi(\cD_1^{\leqslant 0})\subseteq \cD_2^{\leqslant 0}$).
A \emph{t-exact functor} is a functor which is both left and right t-exact.
\end{Def}

\begin{Rem} \label{rem:conservativeandexact}
Recall that a functor between triangulated categories is \emph{conservative} if $\Phi(E) \cong 0$ implies $E \cong 0$.
Now assume $\Phi \colon \cD_1 \to \cD_2$ is both conservative and t-exact, and write
$\cA_i \subset \cD_i$ for the corresponding hearts. Then $\Phi(E) \in \cA_2$ if and only if $E \in \cA_1$. Similarly, in this case two morphisms 
$A \to B \to C$ in $\cD_1$ form a short exact sequence in $\cA_1$ if and only if $\Phi(A) \to \Phi(B) \to \Phi(C)$ is a short exact sequence in $\cA_2$.
\end{Rem}

\subsection{Local t-structures}\label{subsec:localTstructures} 

In the case of a linear category, we will be interested in t-structures that are local over the base scheme in the following sense. 

\begin{Def}
\label{def-local-t-structure}
Let $X \to S$ be a morphism of {schemes which are quasi-compact with affine diagonal, where 
$X$ is noetherian of finite Krull dimension.} 
Let $\cD\subset \Db(X)$ be an $S$-linear strong semiorthogonal component. 
\begin{enumerate}[{\rm (1)}] 
\item A t-structure on $\cD_{\qc}$ is called \emph{$S$-local} if for every quasi-compact open $U \subset S$, there exists a t-structure on $(\cD_{\qc})_U$ such that the restriction functor $\cD_{\qc} \to (\cD_{\qc})_U$ is t-exact.
\item \label{D-local-t-structure}
Assume the projection functor of $\cD$ has finite cohomological amplitude. 
\index{tau@$\tau=(\cD^{\leqslant 0},\cD^{\geqslant 0})$, t-structure on $\cD$!$S$-local}
Then a t-structure on $\cD$ is called \emph{$S$-local} if for every quasi-compact open $U \subset S$, there exists a t-structure on $\cD_U$ such that the restriction functor $\cD \to \cD_U$ is t-exact. 
\end{enumerate}
\end{Def}

In Definition~\ref{def-local-t-structure} we require 
{$U \subset S$ to be quasi-compact so that the base change categories 
 $(\cD_{\qc})_U$ and}
(if the projection functor has finite cohomological amplitude) $\cD_U$ are defined; 
see the discussion preceding Lemma~\ref{lem-open-restriction-es}. 
{In particular, $U$ may be any affine open in $S$, or an arbitrary open if $S$ is noetherian.} 

Lemma~\ref{lem-open-restriction-es} implies that, given an $S$-local t-structure on $\cD$ or $\cD_{\qc}$, for every quasi-compact open $U \subset S$ the t-structure on $\cD_U$ or $(\cD_{\qc})_U$ is uniquely determined. 
We shall sometimes refer to this as the induced t-structure on $\cD_U$ or $(\cD_{\qc})_U$. 
We denote by $\cA_U \subset \cD_U$ or $(\cA_{\qc})_U \subset (\cD_{\qc})_U$ the heart of the corresponding t-structure.
\index{Aqc@$\cA_{\qc}$, induced heart on the quasi-coherent component $\cD_{\qc}$}

Our primary interest is local t-structures on $\cD$, but parallel to the situation for 
base change of linear categories from~Section~\ref{subsection-linear-cats}, 
when we discuss base change of t-structures in Section~\ref{section-bc-t-structure} 
we will also need to consider $\cD_{\qc}$. 
In the rest of this section, we focus on results that do not require the use of $\cD_{\qc}$. 

Note that for $F \in \cD$ the condition $F \in \cD^{[a,b]}$ can be checked 
locally on $S$, since the cohomology functors $\rH^\bullet_{\cA_U}$ commute 
with restriction and the condition that an object in $\cD$ vanishes can be checked locally. 
This observation has the following consequence. 

\begin{Lem}
\label{lem-tensor-vb}
Let $g\colon X \to S$ be a morphism {schemes which are quasi-compact with affine diagonal, where 
$X$ is noetherian of finite Krull dimension.} 
Let $\cD\subset \Db(X)$ be an $S$-linear strong semiorthogonal component whose projection 
functor is of finite cohomological amplitude, and which is equipped with an $S$-local t-structure. 
Then for any vector bundle $V$ on $S$, the functor $(g^*V \otimes -) \colon \cD \to \cD$ is 
t-exact. 
\end{Lem}

The following gives a relative analogue of condition~\eqref{def-t-structure-orthogonality} 
in Definition~\ref{def:tstructure}; recall from Section~\ref{sec:setupnotation} our 
notation $\cHom_S(-,-)$ for the $S$-relative derived sheaf Hom. 

\begin{Lem}\label{lem:HomS-heart}
Let $X \to S$ be a proper morphism of noetherian schemes {with affine diagonals}, where $X$ has finite Krull dimension and $S$ admits an ample line bundle. 
Let $\cD\subset \Db(X)$ be an $S$-linear strong semiorthogonal component whose projection functor is of finite cohomological amplitude, and which is equipped with an $S$-local t-structure. 
Let $F \in \cD^{\leqslant a} \cap \Dperf(X)$ and $G \in \cD^{\geqslant b}$. 
Then 
\begin{equation*}
\cHom_S(F,G) \in \Db(S)^{\geqslant b-a} . 
\end{equation*} 
\end{Lem}

\begin{proof}
Note that since $g$ is a proper morphism between noetherian schemes and $F$ is perfect, 
the object $\cHom_S(F,G)$ indeed lies in $\Db(S)$. 
Let $q \in \Z$ be the smallest integer so that $\rH_{\Coh S}^q\cHom_S(F,G) \neq 0$. 
We must show $q \geqslant b-a$. 
Note that if $L$ is a line bundle on $S$, then by the projection formula 
and perfectness of $F$ we have 
\begin{equation*}
\cHom_S(F,G) \otimes L \simeq \cHom_S(F, G \otimes g^*L). 
\end{equation*} 
Hence there is a spectral sequence 
\begin{equation*}
E_2^{p,q} = \rH^p(S, \rH_{\Coh S}^q\cHom_S(F,G) \otimes L) \Rightarrow 
\Hom(F, G \otimes g^*L [p+q]). 
\end{equation*} 
For degree reasons, the term $\rH^0(S, \rH_{\Coh S}^q\cHom_S(F,G) \otimes L)$ must 
survive in the spectral sequence. 
By choosing a suitably ample $L$ we can ensure this term, and hence 
also $\Hom(F, G \otimes g^*L [q])$, is nonzero. 
But by Lemma~\ref{lem-tensor-vb} we have $G \otimes g^*L \in \cD^{\geqslant b}$, 
so it follows that $q \geqslant b-a$. 
\end{proof} 

\begin{Thm}
\label{thm:local-t-structure-tensor-ample}
Let $X \to S$ be a morphism 
{schemes which are quasi-compact with affine diagonal, where 
$X$ is noetherian of finite Krull dimension} 
and $S$ admits an ample line bundle $L$. 
Let $\cD\subset \Db(X)$ be an $S$-linear strong semiorthogonal component whose projection functor is of finite cohomological amplitude. 
Then a bounded t-structure on $\cD$ is $S$-local if and only if tensoring with $g^*L$ is left t-exact, or equivalently, if and only if tensoring with $g^*L$ is t-exact. 
In particular, if $S$ is affine then any bounded t-structure on $\cD$ is automatically $S$-local. 
\end{Thm}

\begin{proof}
This is the analogue of \cite[Theorem~2.3.2]{Polishchuk:families-of-t-structures} 
(that builds on \cite[Theorem~2.1.4]{AP:t-structures}) in our setup, 
and follows by the same argument. 
\end{proof} 

\begin{Def} 
Let $X \to S$ be a morphism {schemes which are quasi-compact with affine diagonal, where 
$X$ is noetherian of finite Krull dimension.} 
Let $\cD\subset \Db(X)$ be an $S$-linear strong semiorthogonal component whose projection 
functor is of finite cohomological amplitude. 
An $S$-local t-structure $(\cD^{\leqslant 0}, \cD^{\geqslant 0})$ on $\cD$ is called: 
\begin{enumerate}[{\rm (1)}] 
\item \emph{noetherian locally on $S$} if for every quasi-compact 
open subscheme $U \subset S$, the induced t-structure on $\cD_U$ is noetherian. 
\item \emph{tilted-noetherian locally on $S$} 
if there exists an $S$-local t-structure $(\cD_0^{\leqslant 0}, \cD_0^{\geqslant 0})$ 
on $\cD$ which is noetherian locally on $S$, such that 
$\cD_0^{\leqslant -1} \subset \cD^{\leqslant 0} \subset \cD_0^{\leqslant 0}$. 
\end{enumerate}
\end{Def}

An $S$-local t-structure is tilted-noetherian locally on $S$ if and only if it is obtained by tilting 
(on each quasi-compact open) from an $S$-local t-structure which is noetherian 
locally on~$S$.

\begin{Lem}
\label{lem-noetherian-t-structure}
Let $X \to S$ be a morphism of 
{schemes which are quasi-compact with affine diagonal, where 
$X$ is noetherian of finite Krull dimension.} 
Let $\cD\subset \Db(X)$ be an $S$-linear strong semiorthogonal component whose projection 
functor is of finite cohomological amplitude, and which is equipped with an $S$-local t-structure. The following are equivalent:
\begin{enumerate}[{\rm (1)}]
\item \label{noetherian-t-structure-1}
The t-structure on $\cD$ is noetherian locally on $S$. 
\item 
\label{noetherian-t-structure-2}
For every quasi-compact open $U \subset S$ in a basis for the topology of $S$, 
the induced t-structure on $\cD_{U}$ is noetherian. 
\item \label{noetherian-t-structure-3}
There is an open cover $S = \cup_i U_i$ by quasi-compact opens $U_i \subset S$ 
such that the induced t-structure on $\cD_{U_i}$ is noetherian locally on $U_i$ for each $i$. 

\item \label{noetherian-t-structure-4} 
There is an affine open cover $S = \cup_i U_i$ 
such that the induced t-structure on $\cD_{U_i}$ is noetherian for each $i$. 

\end{enumerate}
Moreover, if $S$ admits an ample line bundle, then the above 
conditions are further equivalent to the following: 
\begin{enumerate}[{\rm (1)}]\setcounter{enumi}{4}
\item%[(\mylabel{noetherian-t-structure-5}{5})] 
\label{noetherian-t-structure-5}
The t-structure on $\cD$ is noetherian. 
\end{enumerate}
\end{Lem}

\begin{proof}
Clearly~\eqref{noetherian-t-structure-1} implies~\eqref{noetherian-t-structure-2}, \eqref{noetherian-t-structure-3}, and \eqref{noetherian-t-structure-4}.

To see~\eqref{noetherian-t-structure-2} implies~\eqref{noetherian-t-structure-1}, 
let $U \subset S$ be an arbitrary quasi-compact open. 
Choose a finite cover $U = \cup U_i$ by quasi-compact opens $U_i \subset S$ in 
the given basis for the topology of $S$. 
Now observe that an increasing sequence $F_1 \subset F_2 \subset \cdots \subset F$ 
in $\cA_U \subset \cD_U$ stabilizes at the $N$-th term if and only if its 
restriction to $U_i$ stabilizes at the $N$-th term for all $i$. 
Since there are finitely many $U_i$, it follows that~\eqref{noetherian-t-structure-2} 
implies~\eqref{noetherian-t-structure-1}. 

Similarly, if~\eqref{noetherian-t-structure-3} holds and $U \subset S$ is a quasi-compact 
open, then we can write $U = \cup_j V_j$ as a union of finitely many $V_j = U \cap U_j$. 
Since the t-structure on $\cD_{V_j}$ is noetherian by assumption 
(note that each $V_j$ is quasi-compact since $S$ has affine diagonal), we conclude as 
above that the t-structure on $\cD_{U}$ is noetherian, i.e., \eqref{noetherian-t-structure-1} holds. 

If $S$ admits an ample line bundle, then $S$ is in particular quasi-compact, 
so \eqref{noetherian-t-structure-1} implies~\eqref{noetherian-t-structure-5}. 
Conversely, we show that if $S$ admits an ample line bundle, 
then~\eqref{noetherian-t-structure-5} implies~\eqref{noetherian-t-structure-2}. 
Note that $S$ has a basis for its topology given by affine open sets 
of the form $U = \set{ x \in S \sth f(x) \neq 0}$, 
where $f$ is a global section of a line bundle $L$ on $S$. 
Hence it suffices to show that for such a $U$, any sequence $F_1 \subset F_2 \subset \cdots \subset F$ 
of inclusions in the induced heart $\cA_U \subset \cD_U$ stabilizes. 
By Lemma~\ref{lem-extend-from-localisation}.\eqref{enum:lem-extend-filtration-from-localisation} below, there is a sequence 
$\tilde{F_1} \subset \tilde{F_2} \subset \cdots \subset \tilde{F}$ 
in $\cA_S$ which restricts to the given sequence $F_1 \subset F_2 \subset \cdots \subset F$ 
in $\cA_U$. Since the sequence in $\cA_S$ must stabilize by assumption, this claim 
implies the result. 

So far we have shown \eqref{noetherian-t-structure-1}, \eqref{noetherian-t-structure-2}, and 
\eqref{noetherian-t-structure-3} are equivalent, imply~\eqref{noetherian-t-structure-4}, 
and if $S$ admits an ample line bundle then they are further equivalent to~\eqref{noetherian-t-structure-5}. 
Since any affine scheme admits an ample line bundle, we thus conclude that~\eqref{noetherian-t-structure-4} 
implies~\eqref{noetherian-t-structure-3}, finishing the proof. 
\end{proof}

The proof relied on the following Lemma, which we state in a more general form that will become useful in Section~\ref{section-bc-t-structure}. 
\begin{Lem}
\label{lem-extend-from-localisation}
Let $g\colon X \to S$ be a morphism of 
{schemes which are quasi-compact with affine diagonal, where 
$X$ is noetherian of finite Krull dimension.} Let $\cD \subset \Db(X)$ be an $S$-linear strong semiorthogonal component whose projection functor is of finite cohomological amplitude. 
Let $f \colon T \to S$ be one, or the composition, of
\begin{enumerate}[{\rm (a)}]
 \item \label{enum:opensubset-extension}
 the inclusion of an open subset
given by the non-vanishing locus of a section $h \in \Gamma(S, L)$ of a line bundle $L$ on $S$, and
 \item \label{enum:localisation-extension} a morphism between affine schemes given by a localization of rings.
\end{enumerate} 
\begin{enumerate}[{\rm (1)}]
 \item\label{enum:lem-extend-morphism-from-localisation} Let $\tilde{F} \in \cD$ be an object with pullback $F = \tilde{F}_T \in \cD_T$. 
Let $\beta \colon G \to F$ be a morphism in $\cD_T$. 
Then there exists a morphism $\tilde{\beta} \colon \tilde{G} \to \tilde{F}$ in $\cD$ such that $g^*\tilde \beta = \beta$. 
\item \label{enum:lem-extend-morphism-in-heart}
Assume that $\cD$ and $\cD_T$ have t-structures such that $f^*$ is t-exact. Let $\tilde{F}\in\cA_S$ and $F=\tilde{F}_T$, and let $\beta\colon G\to F$ be a morphism in $\cA_T$. Then there exists a morphism $\tilde{\beta}\colon \tilde{G}\to \tilde{F}$ in $\cA_S$ with $f^*\tilde \beta = \beta$. 
If $\beta$ is injective, we can choose $\tilde{\beta}$ to be injective; if $\beta$ is instead surjective, we can replace
$\tilde F$ by another lift $\tilde F'$ to make $\tilde \beta$ surjective.
\item \label{enum:lem-extend-filtration-from-localisation}
More generally, any (not necessarily finite) filtration in $\cA_T$
\[
F_1 \subset F_2 \subset \ldots \subset F,
\]
can be lifted to a filtration in $\cA_S$
\[
\tilde{F}_1 \subset \tilde{F}_2 \subset \ldots \subset \tilde{F}.
\]
\end{enumerate}
\end{Lem}
\begin{proof}
To prove \eqref{enum:lem-extend-morphism-from-localisation}, we first use Lemma~\ref{lem-open-restriction-es} and choose an arbitrary lift $\tilde G_0 \in \cD$ of $G$.
In case \eqref{enum:opensubset-extension}, \cite[Lemma~2.1.8]{AP:t-structures} shows that $h^k\beta$ extends to a morphism $G_0 \otimes L^{-k} \to F$, which restricts to $\beta$ after the identification $L_U^{-1} \cong \cO_U$ induced by $h$.
In case \eqref{enum:localisation-extension}, flat base change implies $\Hom(G, F) = \Hom(\tilde G_0, \tilde F)\otimes_{\cO_S} \cO_T$;
thus a similar statement holds with $L$ trivial, and $h \in H^0(\cO_S)$ in the localizing subset;
replacing $\tilde G_0$ by the isomorphic object $\tilde G_0 \otimes (h)$ then yields a morphism as claimed.

To prove \eqref{enum:lem-extend-morphism-in-heart}, we first choose a lift $\tilde G_1 \in \cD$ as in \eqref{enum:lem-extend-morphism-from-localisation}. 
Since $f^*$ is exact, we can replace $\tilde G_1$ by $\tilde G_2 = \rH^0_{\cA_S}(\tilde G_1)$ and obtain a lift of $G$ and $\beta$ in $\cA_S$.
Finally, since $f^*$ is t-exact, the morphism $\beta$ is injective if and only if $f^* \Ker \tilde \beta = 0$;
in this case, we can thus replace $\tilde G_2$ by the image of $\tilde \beta$ to obtain an injective lift of $\beta$ in $\cA_T$.
The case of $\beta$ surjective follows similarly.

Finally, to prove \eqref{enum:lem-extend-filtration-from-localisation}, we first choose lifts $\tilde \beta_i \colon \tilde F'_i \into \tilde F$ given by the previous steps.
Then we obtain a desired filtration of $\tilde F$ by replacing $\tilde F'_i$ with 
\[
\tilde F_i := \bigoplus_{j \leqslant i} \im \tilde \beta_i. \qedhere
\]
\end{proof}

\begin{Rem}
We will see in Theorem~\ref{Thm-D-bc} that such a t-structure on $\cD_T$ always exists.
\end{Rem}

\section{Base change of local t-structures}
\label{section-bc-t-structure} 
In this section we prove results on base changes of local t-structures. 
Namely, given an $S$-linear category $\cD \subset \Db(X)$ with an $S$-local t-structure, 
we construct under suitable hypotheses induced t-structures on the base changes 
of $\cD_{\qc}$ and $\cD$ along a morphism $T \to S$ 
(Theorems~\ref{thm-Dqc-bc} and~\ref{Thm-D-bc}). 
The result for base changes of $\cD$ generalizes the results on ``constant t-structures'' from 
\cite[Section~2]{AP:t-structures} and \cite[Theorem~3.3.6]{Polishchuk:families-of-t-structures}, 
which correspond to the case where $\cD = \Db(X)$ and $S$ is a point. 
In fact, many of the ingredients in our proof come from 
\cite{AP:t-structures, Polishchuk:families-of-t-structures}. 

\subsection{The unbounded quasi-coherent case} 
In this subsection we focus on base change of t-structures in the setting of 
unbounded derived categories of quasi-coherent sheaves. 
In this setting it is possible to prove results with very few hypotheses, because 
it is very easy to construct t-structures, as the following lemma illustrates. 

\begin{Lem}
\label{lem-induced-t-Dqc}
Let $X \to S$ be a morphism of schemes where $X$ is noetherian of 
finite Krull dimension. 
Let $\cD \subset \Db(X)$ be an $S$-linear strong semiorthogonal component endowed with a
t-structure $(\cD^{\leqslant 0}, \cD^{\geqslant 0})$. 
Then there is a t-structure on $(\cD_{\qc}^{\leqslant 0}, \cD_{\qc}^{\geqslant 0})$ on $\cD_{\qc}$ 
where:
\begin{itemize}
\item $\cD_{\qc}^{\leqslant 0}$ is the smallest full subcategory of $\cD_{\qc}$ which contains $\cD^{\leqslant 0}$ and is closed under extensions and small colimits. 
\item $\cD_{\qc}^{\geqslant 0} = \set{ F \in \cD_{\qc} \sth \Hom(G, F) = 0 \text{ for all } G \in \cD^{\leqslant -1} }$. 
\end{itemize} 
This t-structure has the following properties: 
\begin{enumerate}[{\rm (1)}] 
\item 
\label{Dqc-tr-continuous}
The truncation functors commute with filtered colimits. 
\item 
\label{cD-cDqc-exact}
The inclusion $\cD \to \cD_{\qc}$ is t-exact. 
\item 
\label{colim-Dqc-D}
For $F \in \cD_{\qc}$, we have $F \in \cD_{\qc}^{[a,b]}$ if and only if 
$F = \colim F_\alpha$ for a filtered system of objects $F_\alpha \in \cD^{[a,b]}$. 
\item \label{Dqc-local}
Assume {$X$ and $S$ are quasi-compact with affine diagonal,} 
the projection functor of $\cD$ has finite cohomological 
amplitude, and the t-structure on $\cD$ is $S$-local. 
Then the above t-structure on $\cD_{\qc}$ is $S$-local. 
More precisely, if $U \subset S$ is a quasi-compact open, then the t-structure on 
$(\cD_{\qc})_U$ making $\cD_{\qc} \to (\cD_{\qc})_U = (\cD_{U})_{\qc}$ t-exact is obtained 
by applying the above construction to the t-structure on $\cD_{U}$. 
\end{enumerate}
\end{Lem}

\begin{proof}
By \cite[Propositions 1.4.4.11 and 1.4.4.13]{lurie-HA} 
there is a t-structure on $\cD_{\qc}$ with $\cD_{\qc}^{\leqslant 0}$ as described, 
whose truncation functors commute with filtered colimits. 
Then $\cD_{\qc}^{\geqslant 0}$ is the right orthogonal to $\cD_{\qc}^{\leqslant -1}$, which 
is easily seen to be given by the stated formula. 
The t-exactness of $\cD \to \cD_{\qc}$ then follows directly. 
We note that the above argument is essentially the same as \cite[Lemma~2.1.1]{Polishchuk:families-of-t-structures}, but we use \cite{lurie-HA} for the statement about truncation functors. 

By construction $\cD_{\qc} = \Ind(\cD_{\perf})$, so using parts~\eqref{Dqc-tr-continuous} 
and~\eqref{cD-cDqc-exact} of the lemma, the argument of 
Lemma~\ref{lemma-qc-colimit-Db} proves part~\eqref{colim-Dqc-D}. 

Now assume we are in the situation of part~\eqref{Dqc-local}. 
Let $U \subset S$ be a quasi-compact open subset. 
By what we have already shown, the t-structure on $\cD_U$ induces one on $(\cD_U)_{\qc}$. 
Note that $(\cD_U)_{\qc} = (\cD_{\qc})_U$. 
We claim that the restriction functor $\cD_{\qc} \to (\cD_{U})_{\qc}$ is t-exact, 
which will prove that the t-structure on $\cD_{\qc}$ is $S$-local. 
This follows from part~\eqref{colim-Dqc-D} of the lemma, the fact that 
the restriction functor $\cD_{\qc} \to (\cD_{U})_{\qc}$ commutes with colimits, and 
the $S$-locality of the t-structure on $\cD$. 
\end{proof}

\begin{Rem}
In Lemma~\ref{lem-induced-t-Dqc}, if $\cD = \Db(X)$ with the standard t-structure, then 
the induced t-structure on $\cD_{\qc} = \Dqc(X)$ is the standard one. 
\end{Rem}

In \cite{Polishchuk:families-of-t-structures}, the focus is on inducing t-structures in the setting of 
bounded derived categories of coherent sheaves, but the idea of first constructing t-structures on 
unbounded derived categories of quasi-coherent sheaves is used extensively. 
The following theorem can be thought of as an elaboration on this idea; 
see also \cite[Section~5]{DHL-instability} for similar results in the setting of varieties over a field. 

\begin{Thm}
\label{thm-Dqc-bc}
Let $g \colon X \to S$ be a morphism of {schemes which are quasi-compact with affine diagonal, where 
$X$ is noetherian of finite Krull dimension.} 
Let $\cD \subset \Db(X)$ be an $S$-linear strong semiorthogonal component with a t-structure $(\cD^{\leqslant 0}, \cD^{\geqslant 0})$. 
Let $\phi \colon T \to S$ be a morphism from a 
{scheme $T$ which is quasi-compact with affine diagonal,} 
such that $\phi$ is faithful with respect to~$g$.  Let $g^\prime$ and $\phi^\prime$ be the morphisms in the following base change diagram:
\begin{equation*}
\xymatrix{
X_{T} \ar[d]_{g'} \ar[r]^{\phi'} & X \ar[d]^{g} \\
T \ar[r]^{\phi} & S 
}
\end{equation*}
If $T$ is affine, then there is a t-structure $((\cD_{\qc})_T^{\leqslant 0}, (\cD_{\qc})_T^{\geqslant 0})$ on $(\cD_{\qc})_T$ where: 
\index{D@$(\cD_T^{\leqslant 0}, \cD_T^{\geqslant 0})$, base change of a t-structure via $T\to S$!$T$ affine}
\index{AT@$\cA_T$, base change of the heart of a bounded t-structure via $T\to S$!$T$ affine}
\begin{itemize}
\item $(\cD_{\qc})_T^{\leqslant 0}$ is the smallest full subcategory of 
$(\cD_{\qc})_T$ which contains $\phi'^*(\cD^{\leqslant 0})$ and is closed under extensions and small colimits. 
\item $(\cD_{\qc})_T^{\geqslant 0} = \set{ F \in (\cD_{\qc})_T \sth \Hom(G,F) = 0 \text{ for all } G \in \phi'^*(\cD^{\leqslant -1}) }$. 
\end{itemize}
In general, there is a t-structure on $(\cD_{\qc})_T$ 
given by 
\begin{equation} 
\label{DT-t-structure-1} 
(\cD_{\qc})_{T}^{[a,b]} = 
\set{ F \in (\cD_{\qc})_T \sth 
\begin{array}{c}
F_U \in (\cD_{\qc})_U^{[a,b]} \text{ for any flat morphism } \\ 
U \to T \text{ with } U \text{ affine}
\end{array}} . 
\end{equation}
This t-structure has the following properties: 
\begin{enumerate}[{\rm (1)}] 
\item 
\label{DT-tr-functors} 
The truncation functors of $(\cD_{\qc})_T$ commute with filtered colimits. 

\item 
\label{DT-Ui}
For any fpqc cover $\set{U_i \to T}$ of $T$ by affine schemes $U_i$, we have 
\begin{equation}
\label{DT-Ui-formula} 
(\cD_{\qc})_{T}^{[a,b]} = \set{ F \in (\cD_{\qc})_T \sth
F_{U_i} \in (\cD_{\qc})_{U_i}^{[a,b]} \text{ for all } i} . 
\end{equation} 
In particular, for $T$ affine the above prescriptions for t-structures on $(\cD_{\qc})_T$ agree. 

\item 
\label{DT-tensor-S}
For any $G \in \rD_{\qc}^{\leqslant 0}(T)$, the functor 
$(g_T^*(G) \otimes -) \colon (\cD_{\qc})_T \to (\cD_{\qc})_T$ is right t-exact, 
where $g_T \colon X_T \to T$ denotes the projection. 

\item 
\label{DT-f}
Let $T' \to S$ be another morphism {with the same assumptions as $\phi$}, 
let $f \colon T' \to T$ be a morphism of schemes over $S$, and let $f' \colon X_{T'} \to X_{T}$ be the induced morphism. 
\begin{enumerate}[{\rm (a)}]
\item 
\label{DT-f-pullback}
$f'^* \colon (\cD_{\qc})_T \to (\cD_{\qc})_{T'}$ is right t-exact. 
\item 
\label{DT-f-pushforward}
$f'_* \colon (\cD_{\qc})_{T'} \to (\cD_{\qc})_{T}$ is left t-exact. 
\item 
\label{DT-f-flat}
If $f$ is flat, then $f'^* \colon (\cD_{\qc})_T \to (\cD_{\qc})_{T'}$ is t-exact. 
\item 
\label{DT-f-affine} 
If $f$ is affine, then $f'_* \colon (\cD_{\qc})_{T'} \to (\cD_{\qc})_{T}$ is t-exact. 

\end{enumerate}

\item 
\label{DT-S-local}
Assume the projection functor of $\cD$ has finite cohomological amplitude and the t-structure on $\cD$ is $S$-local. 
Then for $T = S$ the above t-structure agrees with the one on $\cD_{\qc}$ from Lemma~\ref{lem-induced-t-Dqc}. 
Moreover, for general $T$ we have 
\begin{equation} 
\label{DT-t-structure} 
(\cD_{\qc})_{T}^{[a,b]} = \set{ F \in (\cD_{\qc})_T \sth 
\begin{array}{c}
\phi'_{U*}(\res{F}{U}) \in \cD_{\qc}^{[a,b]} \text{ for any flat morphism } \\ 
U \to T \text{ with } U \text{ affine}
\end{array}} 
\end{equation}
where $\phi'_U \colon X_U \to X$ is the morphism induced by $U \to T \xrightarrow{\, \phi \,} S$, 
and for any fpqc cover $\set{ U_i \to T }$ by affine schemes $U_i$, 
we have 
\begin{equation}
\label{DT-t-structure-Ui}
(\cD_{\qc})_{T}^{[a,b]} = \set{ F \in (\cD_{\qc})_T \sth
\phi'_{U_i*}(\res{F}{U_i}) \in \cD_{\qc}^{[a,b]} \text{ for all } i} . 
\end{equation} 
\end{enumerate}
\end{Thm}

\begin{proof}
We prove the theorem in several steps. 

\begin{step}{1}
\label{DT-step-T-affine}
If $T$ is affine, the prescription for $((\cD_{\qc})_T^{\leqslant 0}, (\cD_{\qc})_T^{\geqslant 0})$ defines a t-structure on $(\cD_{\qc})_T$ such that parts~\eqref{DT-tr-functors} and~\eqref{DT-tensor-S} of the theorem hold. 
\end{step}

This argument of Lemma~\ref{lem-induced-t-Dqc} shows the prescription defines a t-structure satisfying~\eqref{DT-tr-functors}. 
Since $T$ is affine, the object $\cO_T$ generates $\Dqc(T)^{\leqslant 0}$ under colimits. 
But $g^*_T$ and tensor products commute with colimits, and tensoring with $\cO_{X_T} = g_T^*(\cO_T)$ preserves $\phi'^*(\cD^{\leqslant 0})$. 
So it follows that $g_T^*(G) \otimes (\cD_{\qc})_T^{\leqslant 0} \subset (\cD_{\qc})_T^{\leqslant 0}$ for any $G \in \Dqc(T)^{\leqslant 0}$, i.e., \eqref{DT-tensor-S} holds. 

\begin{step}{2}
\label{DT-step-affine-f} 
If $T$ and $T'$ are affine, part~\eqref{DT-f} of the theorem holds for the t-structures 
from Step~\ref{DT-step-T-affine}. 
\end{step} 

Since $f'^*$ admits a right adjoint, it commutes with colimits. 
It follows that $f'^*((\cD_{\qc})_T^{\leqslant 0}) \subset (\cD_{\qc})_{T}^{\leqslant 0}$, i.e., 
$f'^*$ is right t-exact. 

Since $f'_*$ is right adjoint to $f'^*$, it follows formally that $f'_*$ is left t-exact. 
Note that for any object $F \in (\cD_{\qc})_T$ we have 
$f'_*f'^*(F) \simeq F \otimes f'_*(\cO_{T'})$. 
The morphism $f'$ is affine, being the base change of $f$, 
and hence $f'_*(\cO_{T'})$ is a sheaf. 
Therefore by the property~\eqref{DT-tensor-S} proved in Step~\ref{DT-step-T-affine}, 
it follows that $f'_*(f'^*\phi'^*(\cD^{\leqslant 0})) \subset (\cD_{\qc})_T^{\leqslant 0}$. 
But since $f'$ is affine $f'_*$ preserves colimits (see for instance 
\cite[Proposition~2.5.1.1]{lurie-SAG}), so it follows that 
$f'_*((\cD_{\qc})_{T'}^{\leqslant 0}) \subset (\cD_{\qc})_T^{\leqslant 0}$. 
This proves $f'_*$ is also right t-exact, and hence t-exact. 

Finally, assume $f$ is flat. 
Since the functor $f'_*$ is conservative 
and by the above t-exact, 
the t-exactness of $f'^*$ is equivalent to t-exactness of the functor 
\begin{equation*}
f'_* \circ f'^* \simeq (f'_*(\cO_{T'}) \otimes -) \colon \cD_{T} \to \cD_{T}. 
\end{equation*} 
Note that $f'_*(\cO_{T'}) \simeq g_T^*(f_*\cO_T)$. 
Since $f \colon T' \to T$ is a flat morphism of affine schemes, by 
Lazard's theorem $f_* \cO_T$ is a filtered colimit of finite free $\cO_T$-modules. 
Hence $g_T^*(f_* \cO_T)$ is a filtered colimit of finite free $\cO_{X_T}$-modules. 
Now it follows from the property~\eqref{DT-tr-functors} proved in Step~\ref{DT-step-T-affine} that the above functor is t-exact. 

\begin{step}{3}
\label{DT-step-Ui} 
For any fpqc cover $\set{U_i \to T}$ of $T$ by affine schemes $U_i$, the prescription~\eqref{DT-Ui-formula} defines a t-structure on $(\cD_{\qc})_T$ such that part~\eqref{DT-tr-functors} of the theorem holds. 
\end{step}

Let $U^{\bullet} \to T$ be the \v{C}ech nerve of the map $U = \bigsqcup U_i \to T$, i.e., 
the simplicial scheme which in degree $n$ is given by the $(n+1)$-fold fiber product of 
$U$ over $T$. Then by fpqc descent for $\Dqc(-)$, pullback induces an equivalence 
\begin{equation*}
\Dqc(T) \simeq \Tot(\Dqc({U^\bullet}))
\end{equation*}
where the right side denotes the totalization, i.e., the limit, of the cosimplicial diagram of 
$\infty$-categories $\Dqc({U^\bullet})$. 
Note that $\cD_{\qc}$ is compactly generated (by $\cD_{\perf}$). 
Hence by \cite[Lemma~4.3]{NCHPD} the category $\cD_{\qc}$ is dualizable as an 
object of $\PrCat_S$ (see Remark~\ref{Rem-S-linear}). 
It follows that the functor from $\PrCat_S$ to itself given by tensoring with $\cD_{\qc}$ 
admits a right adjoint (given by tensoring with the dual of $\cD_{\qc}$), and thus commutes 
with limits. In particular, the above equivalence implies that pullback induces an equivalence 
\begin{equation}
\label{Dqc-tot}
(\cD_{\qc})_T \simeq \Tot((\cD_{\qc})_{U^\bullet}) . 
\end{equation} 

For $i_0, i_1, \dots, i_n$, we write 
\begin{equation*}
U_{i_0, \dots, i_n} = U_{i_0} \times_T U_{i_1} \dots \times_T U_{i_n}. 
\end{equation*}
Then $U^n$ is the coproduct of the $U_{i_0, \dots, i_n}$ over all $i_0, \dots, i_n$. 
Hence pullback induces an equivalence 
\begin{equation*}
\Dqc(U^n) \simeq \prod_{i_0, \dots, i_n} \Dqc(U_{i_0, \dots, i_n}) . 
\end{equation*} 
By the observation above, this implies that pullback induces an equivalence 
\begin{equation}
\label{DUn-prod} 
(\cD_{\qc})_{U^n} \simeq \prod_{i_0, \dots, i_n} (\cD_{\qc})_{U_{i_0, \dots, i_n}}. 
\end{equation} 
Combining~\eqref{Dqc-tot} and~\eqref{DUn-prod}, we see that $(\cD_{\qc})_T$ is 
expressed as a limit of the diagram of categories $(\cD_{\qc})_{U_{i_0, \dots, i_n}}$. 

Note that each $U_{i_0, \dots, i_n}$ is affine since $T$ has affine diagonal. 
Hence by Step~\ref{DT-step-T-affine} each term $(\cD_{\qc})_{U_{i_0, \dots, i_n}}$ carries a t-structure whose truncation functors preserve filtered colimits. 
Moreover, the projection morphisms between the $U_{i_0, \dots, i_n}$ are flat. 
Hence by property~\eqref{DT-f-flat} verified in Step~\ref{DT-step-affine-f}, the pullback functors between the $(\cD_{\qc})_{U_{i_0, \dots, i_n}}$ are t-exact. 
By \cite[Chapter I.3, Lemma~1.5.8]{gaitsgory-DAG} their limit $(\cD_{\qc})_T$ thus carries a t-structure given by 
\begin{equation*}
(\cD_{\qc})_{T}^{[a,b]} = \set{ F \in (\cD_{\qc})_T \sth
F_{U_{i_0, \dots, i_n}} \in (\cD_{\qc})_{U_{i_0, \dots, i_n}}^{[a,b]} \text{ for all } i_0, \dots, i_n}, 
\end{equation*} 
whose truncation functors preserve filtered colimits. 
Since the pullback functors between the $(\cD_{\qc})_{U_{i_0, \dots, i_n}}$ are t-exact, 
the condition $F_{U_{i_0, \dots, i_n}} \in (\cD_{\qc})_{U_{i_0, \dots, i_n}}^{[a,b]}$ for all $i_0, \dots, i_n$
is equivalent to $F_{U_i} \in (\cD_{\qc})_{U_i}^{[a,b]}$ for all $i$. 

\begin{step}{4} 
\label{DT-step-any-U}
The prescription~\eqref{DT-t-structure-1} defines a t-structure on $(\cD_{\qc})_T$ such 
that parts~\eqref{DT-tr-functors} and~\eqref{DT-Ui} 
of the theorem hold. 
\end{step}

Take any fpqc cover $\set{U_i \to T}$ of $T$ by affine schemes $U_i$, 
and let $U \to T$ be a flat morphism with $U$ affine. 
Then $\set{U_i \to T} \cup \set{ U \to T}$ is also an fpqc cover. 
By Step~\ref{DT-step-Ui}, both the formula 
\begin{equation}
\label{DT-Ui-formula-2} 
(\cD_{\qc})_{T}^{[a,b]} = \set{ F \in (\cD_{\qc})_T \sth
F_U \in (\cD_{\qc})_U^{[a,b]} \text{ and } F_{U_i} \in (\cD_{\qc})_{U_i}^{[a,b]} \text{ for all } i} 
\end{equation} 
and~\eqref{DT-Ui-formula} define t-structures on $(\cD_{\qc})_T$ such that part~\eqref{DT-tr-functors} of the theorem holds. 
Since the right side of \eqref{DT-Ui-formula-2} is contained in the right side of \eqref{DT-Ui-formula} 
and both define t-structures, they coincide. 
It follows that~\eqref{DT-t-structure-1} defines a t-structure on $(\cD_{\qc})_T$ such that parts~\eqref{DT-tr-functors} and~\eqref{DT-Ui} 
of the theorem hold. 

\begin{step}{5}
Parts~\eqref{DT-tensor-S} and~\eqref{DT-f} of the theorem hold. 
\end{step}

Using~\eqref{DT-Ui-formula}, parts~\eqref{DT-tensor-S}, \eqref{DT-f-pullback}, \eqref{DT-f-flat}, and \eqref{DT-f-affine} 
reduce to the case where $T$ and $T'$ are affine, which were handled in Step~\ref{DT-step-T-affine}. 
Part~\eqref{DT-f-pushforward} follows from \eqref{DT-f-pullback} since $f'_*$ is right adjoint to $f'^*$. 

\begin{step}{6}
Part~\eqref{DT-S-local} of the theorem holds. 
\end{step}

To show the two t-structures on $\cD_{\qc}$ agree, by part~\eqref{DT-Ui} and Lemma~\ref{lem-induced-t-Dqc}.\eqref{Dqc-local} 
it suffices to show that for any affine open $U \subset S$ the t-structure on $(\cD_{\qc})_U$ constructed in this 
theorem agrees with the one induced by the local t-structure on $\cD_{\qc}$ from Lemma~\ref{lem-induced-t-Dqc}.\eqref{Dqc-local}. 
This follows from the description of the induced t-structure on $(\cD_{\qc})_U$ in Lemma~\ref{lem-induced-t-Dqc}.\eqref{Dqc-local}. 

Finally, by~\eqref{DT-t-structure-1} and~\eqref{DT-Ui-formula}, 
the formulas~\eqref{DT-t-structure} and~\eqref{DT-t-structure-Ui} reduce to showing that for $\phi \colon T \to S$ 
a morphism from an affine scheme $T$, we have 
\begin{equation*}
(\cD_{\qc})_T^{[a,b]} = 
\set{ F \in (\cD_{\qc})_T \sth
\phi'_*(F) \in \cD_{\qc}^{[a,b]} }. 
\end{equation*} 
The functor $\phi'_*$ is conservative since $\phi'$ is affine, so this follows from \eqref{DT-f-affine}. 
\end{proof}

\subsection{The bounded coherent case} 

Our next goal is to show that under suitable hypotheses, the base changed t-structures 
constructed in Theorem~\ref{thm-Dqc-bc} induce t-structures at the level of bounded derived 
categories of coherent sheaves. 

A map $A \to B$ of rings is called \emph{perfect} if it is pseudo-coherent and $B$ has finite $\Tor$-dimension over $A$, see \citestacks{067G}. 
A morphism of schemes $X \to Y$ is called \emph{perfect} if 
if there exists an affine open cover $Y = \bigcup_{j \in J} V_j$ and affine open covers 
$f^{-1}(V_j) = \bigcup_{i \in I_j} U_i$ such that the ring map $\cO_Y(V_j) \to \cO_X(U_i)$ is perfect for all $j \in J$, $i \in I_j$. 
For a discussion of this notion, see \citestacks{0685}. 
In other words, 
a morphism $f \colon X \to Y$ is perfect if and only if it is pseudo-coherent and of finite $\Tor$-dimension. 
Note that if $Y$ is locally noetherian, then $f$ is pseudo-coherent if and only if $f$ is locally of finite type \citestacks{0684}.
In particular, if $Y$ is regular of finite Krull dimension, then $f$ is perfect if and only if $f$ is locally 
of finite type. 
We will consider base changes along the following mild generalization of the class of perfect morphisms.

\begin{Def}\label{def:EssPerfect}
A map $A \to B$ of rings is called \emph{essentially perfect} if {it} is a localization of a perfect $A$-algebra. 
A morphism of schemes $X \to Y$ is called \emph{essentially perfect} 
if there exists an affine open cover $Y = \bigcup_{j \in J} V_j$ and affine open covers 
$f^{-1}(V_j) = \bigcup_{i \in I_j} U_i$ such that the ring map $\cO_Y(V_j) \to \cO_X(U_i)$ is essentially 
perfect for all $j \in J$, $i \in I_j$. 
\end{Def} 

\begin{Rem}
As a warning, if $\Spec(B) \to \Spec(A)$ is an essentially perfect morphism of affine schemes, 
then $A \to B$ may not be essentially perfect. 
\end{Rem}

If $\tilde{\cD}$ is a triangulated category with a t-structure $(\tilde{\cD}^{\leqslant 0}, \tilde{\cD}^{\geqslant 0})$ and 
$\cD \subset \tilde{\cD}$ is a triangulated subcategory, we say that $(\tilde{\cD}^{\leqslant 0}, \tilde{\cD}^{\geqslant 0})$ 
induces a t-structure on $\cD$ if setting $\cD^{\leqslant 0} = \tilde{\cD}^{\leqslant 0} \cap \cD$ and 
$\cD^{\geqslant 0} = \tilde{\cD}^{\geqslant 0} \cap \cD$ defines a t-structure. This is equivalent to the truncation functors 
of $\tilde{\cD}$ preserving $\cD$. 
The following analog of \cite[Theorem~2.3.5]{Polishchuk:families-of-t-structures} in our setting, 
which holds by the same argument, is a key ingredient for the general base change result below. 
 
\begin{Thm}
\label{thm-t-structure-finite-map}
Let $g \colon X \to S$ be a morphism of {schemes which are quasi-compact with affine diagonal, where 
$X$ is noetherian of finite Krull dimension.}
Let $\cD\subset \Db(X)$ be an $S$-linear strong semiorthogonal component whose projection 
functor is of finite cohomological amplitude, and which is equipped with an $S$-local t-structure. 
Let $\phi \colon T \to S$ be a finite perfect morphism from a 
{scheme $T$ which is quasi-compact with affine diagonal,} such that $\phi$ is faithful with respect to $g$. 
Then the t-structure on $(\cD_{\qc})_T$ from Theorem~\ref{thm-Dqc-bc} induces a 
t-structure on $\cD_{T}$, which is bounded or tilted-noetherian or noetherian if the 
given t-structure on $\cD$ is. 
\end{Thm}

Now we can prove our main base change result. 

\begin{Thm}
\label{Thm-D-bc} 
Let $g \colon X \to S$ be a morphism of {schemes which are quasi-compact with affine diagonal, where 
$X$ is noetherian of finite Krull dimension.} 
Let $\cD \subset \Db(X)$ be an $S$-linear strong semiorthogonal component whose projection functor is of finite cohomological amplitude. 
Let $(\cD^{\leqslant 0}, \cD^{\geqslant 0})$ be a bounded $S$-local t-structure on $\cD$ which is tilted-noetherian locally on $S$. 
Let $\phi \colon T \to S$ be an essentially perfect morphism from a 
{scheme $T$ which is quasi-compact with affine diagonal,} 
such that $\phi$ is faithful with respect to $g$. 
Then: 
\begin{enumerate}[{\rm (1)}]
\item 
\label{DqcT-to-DT}
The t-structure on $(\cD_{\qc})_T$ from Theorem~\ref{thm-Dqc-bc} induces a bounded $T$-local t-structure on $\cD_{T}$ which is tilted-noetherian locally on $T$. 
\index{D@$(\cD_T^{\leqslant 0}, \cD_T^{\geqslant 0})$, base change of a t-structure via $T\to S$!$T\to S$ essentially of finite type}
\index{AT@$\cA_T$, base change of the heart of a bounded t-structure via $T\to S$!$T\to S$ essentially of finite type}
\item 
\label{D-bc-noetherian}
If the t-structure on $\cD$ is noetherian locally on $S$, so is the induced t-structure on 
$\cD_T$. 
\item 
\label{D-bc-projective}
If $\phi \colon T \to S$ is projective, $L$ is a $\phi$-relatively ample line bundle on $T$, 
and $L_{X_T}$ denotes its pullback to $X_T$, 
then the t-structure on $\cD_T$ satisfies 
\begin{equation*} 
\cD_T^{[a,b]} = \set{ F \in \cD_T \sth
\phi'_{*}(F \otimes L_{X_T}^n) \in \cD^{[a,b]} \text{ for all } n \gg 0 } . 
\end{equation*}
\item 
\label{DbT-f}
Let $T' \to S$ be another morphism with the same assumptions as $\phi$, 
let $f \colon T' \to T$ be a morphism of schemes over $S$, 
and let $f' \colon X_{T'} \to X_{T}$ be the induced morphism. 
\begin{enumerate}[{\rm (a)}]
\item 
\label{DbT-f-pullback}
$f'^* \colon \cD_T \to \cD_{T'}$ is right t-exact. 
\item 
\label{DbT-f-flat}
If $f$ is flat, then $f'^* \colon \cD_T \to \cD_{T'}$ is t-exact. 
\item 
\label{DbT-f-finite-pushforward} 
If $f$ is finite, then $f'_* \colon \cD_{T'} \to \cD_{T}$ is t-exact. \end{enumerate}
\end{enumerate} 
\end{Thm}

\begin{proof}
We prove the theorem in several steps. 

\begin{step}{1}
\label{DbT-step-Pr}
If $S$ is affine and $\phi \colon T = \P^r_S \to S$ is a projective space over $S$, 
then the formula 
\begin{equation}
\label{t-structure-Pr}
\cD_T^{[a,b]} = \set{ F \in \cD_T \sth \phi'_{*}(F(n)) \in \cD^{[a,b]} 
\text{ for all } n \gg 0 }
\end{equation}
defines a bounded $T$-local t-structure on $\cD_T$ 
which is tilted-noetherian locally on $T$, 
and noetherian locally on $T$ if the given t-structure on $\cD$ is. 
\end{step}

In~\cite[Theorem~2.3.6]{AP:t-structures} it is shown that if $\cD = \Db(X)$, 
$X$ is smooth and projective over the spectrum $S$ of a field, and the t-structure on $\cD$ is noetherian, 
then~\eqref{t-structure-Pr} defines a bounded noetherian $T$-local t-structure on $\cD_T$. 
We leave it to the reader to verify that their proof goes through 
in greater generality: in our setup if $S$ is affine and the t-structure on $\cD$ is noetherian, 
then~\eqref{t-structure-Pr} defines a bounded noetherian $T$-local t-structure on $\cD_T$. 
Since by Lemma~\ref{lem-noetherian-t-structure} a t-structure on $\cD$ or $\cD_T$ is noetherian if and only if it is noetherian locally on the base, this proves the desired claim in case the given t-structure on $\cD$ is noetherian locally on $S$. 
From this, the tilted-noetherian case follows as in the proof of~\cite[Lemma~3.3.2]{Polishchuk:families-of-t-structures}.

\begin{step}{2}
\label{DbT-step-Ar}
If $S$ is affine and $\phi \colon T = \A_S^r \to S$ is an affine space over $S$, then 
the bounded $T$-local t-structure on $\cD_{T}$ induced by the $\P_S^r$-local t-structure 
on $\cD_{\P_S^r}$ from Step~\ref{DbT-step-Pr} is given by 
\begin{equation*}
\cD_T^{[a,b]} = \set{ F \in \cD_T \sth\phi'_{*}(F) \in \cD_{\qc}^{[a,b]} }. 
\end{equation*} 
\end{step} 

This holds by the same argument as in \cite[Lemma~3.3.4]{Polishchuk:families-of-t-structures}. 

\begin{step}{3}
\label{DbT-step-affine-perfect} 
If $S$ and $T$ are affine and $\phi \colon T \to S$ is perfect, then the 
t-structure on $(\cD_{\qc})_T$ from Theorem~\ref{thm-Dqc-bc} induces a bounded 
t-structure on $\cD_{T}$, which is noetherian if the t-structure on $\cD$ is. 
\end{step}

The morphism $\phi$ is of finite type (see \citestacks{0682}), so it factors through a closed immersion $T \hookrightarrow \A^r_S$. 
The morphism $T \hookrightarrow \A^r_S$ is pseudo-coherent by \citestacks{0683}, of finite $\Tor$-dimension by \citestacks{068X}, and hence perfect. 
Thus the claim follows from Theorem~\ref{thm-t-structure-finite-map} combined with Step~\ref{DbT-step-Ar}. 

\begin{step}{4} 
\label{DbT-step-affine-eperfect} 
If $S = \Spec(A)$ and $T = \Spec(B)$ are affine and $\phi \colon T \to S$ corresponds to an 
essentially perfect map of rings $A \to B$, then the 
t-structure on $(\cD_{\qc})_T$ from Theorem~\ref{thm-Dqc-bc} induces a bounded 
t-structure on $\cD_{T}$, which is noetherian if the t-structure on $\cD$ is. 
\end{step}

By Step~\ref{DbT-step-affine-perfect} we reduce to the case where $B$ is a localization of $A$. 
Then $\phi \colon T \to S$ is flat, so by Theorem~\ref{thm-Dqc-bc}.\eqref{DT-f-flat} the functor 
$\phi'^* \colon \cD_{\qc} \to (\cD_{\qc})_T$ is t-exact. 
Since the functor $\cD \to \cD_{T}$ is essentially surjective by Lemma~\ref{lem-open-restriction-es}, 
it follows that the t-structure on $(\cD_{\qc})_T$ induces a bounded t-structure on $\cD_T$. 
To prove noetherianity, we first note that any subobject $F_i \subset F$ of a fixed object $F$ is already
defined over the complement $U_i \subset S$ of the zero locus of some $f_i \in A$. Thus we can proceed exactly
as in Lemma~\ref{lem-extend-from-localisation}.\eqref{enum:lem-extend-filtration-from-localisation} to lift any 
possibly infinite filtration
of $F$ in $\cA_T$ to a filtration of $\tilde F$ in $\cA_S$, which proves that $\cA_T$ is noetherian if $\cA_S$ is.

\begin{step}{5}
Parts~\eqref{DqcT-to-DT} and \eqref{D-bc-noetherian} of the theorem hold in general. 
\end{step}

First we show that $(\cD_{\qc})_T$ induces a t-structure on $\cD_T$. 
For any $a \in \Z$, let $\hat{\tau}^{\geqslant a}_{T}$ denote the truncation functor for the 
t-structure on $(\cD_{\qc})_T$. 
Then we must show that for any $F \in \cD_T$ we have $\hat{\tau}^{\geqslant a}_{T}(F) \in \cD_T$. 
Since $T$ is quasi-compact and $\phi$ is essentially perfect, we may choose a finite 
affine open cover $T = \cup U_i$ and affine opens $V_i \subset S$ such that for each $i$ 
the morphism $U_i \hookrightarrow T \xrightarrow{\, \phi \,} S$ factors through $V_i$ and 
the corresponding ring map $\cO_S(V_i) \to \cO_T(U_i)$ is essentially perfect. 
To show $\hat{\tau}^{\geqslant a}_{T}(F) \in \cD_T$ it suffices to show 
$\hat{\tau}^{\geqslant a}_{T}(F)_{U_i} \in \cD_{U_i}$ for each $i$ 
because pseudo-coherence is a local property 
and boundedness can be checked on a finite open cover. 
We have $\hat{\tau}^{\geqslant a}_{T}(F)_{U_i} \simeq \hat{\tau}^{\geqslant a}_{U_i}(F_{U_i})$ because 
the restriction functor $(\cD_{\qc})_{T} \to (\cD_{\qc})_{U_i}$ is t-exact. 
In view of Theorem~\ref{thm-Dqc-bc}.\eqref{DT-S-local}, we thus reduce to the case of the 
affine morphism $U_i \to V_i$, which was handled in Step~\ref{DbT-step-affine-eperfect}. 
This shows that $(\cD_{\qc})_T$ induces a t-structure on $\cD_T$, and a similar argument 
shows that this t-structure on $\cD_T$ is bounded. 

Moreover, the t-structure on $\cD_T$ is $T$-local. Indeed, if $U \subset T$ is a quasi-compact 
open subset, then {$U$ also has affine diagonal} and the morphism $U \to T$ is perfect. 
Hence by what we have already shown, we conclude that $(\cD_{\qc})_U$ induces a t-structure 
on $\cD_U$. 
Since the restriction functor $(\cD_{\qc})_T \to (\cD_{\qc})_U$ is t-exact by Theorem~\ref{thm-Dqc-bc}.\eqref{DT-f-flat}, 
so is the functor $\cD_T \to \cD_U$. Hence the t-structure on $\cD_T$ is $T$-local. 

Further, the t-structure on $\cD_{T}$ is noetherian locally on $T$ 
if the t-structure on $\cD$ is noetherian locally on $S$. 
Indeed, then for each affine $U_i \subset T$ in the affine cover considered above, 
the induced t-structure on $\cD_{U_i}$ is noetherian by Step~\ref{DbT-step-affine-eperfect}. 
So by Lemma~\ref{lem-noetherian-t-structure} the claim holds. 

Finally, it follows directly from the definitions and the result of the 
previous paragraph that the t-structure on $\cD_{T}$ is tilted-noetherian locally on $T$. 

\begin{step}{6}
Part~\eqref{D-bc-projective} of the theorem holds. 
\end{step}
The proof of~\cite[Theorem~3.3.6(ii)]{Polishchuk:families-of-t-structures} goes through 
in our setup. 

\begin{step}{7}
Part~\eqref{DbT-f} of the theorem holds.
\end{step}
These claims follow immediately from the corresponding statements in Theorem~{\ref{thm-Dqc-bc}.\eqref{DT-f}}.
\end{proof}

We make explicit the following immediate consequence of the theorem: 
\begin{Cor} \label{cor:base-change-tstructure-point}
 In the assumptions of Theorem~\ref{Thm-D-bc}, assume also that $g$ is flat, that $S$ is regular of finite Krull dimension. Let $s$ be a point of $S$ and let $\cD_s$ be the base change category to $\Spec \kappa(s)$.
 \index{As@$\cA_s$, base change of the heart of a bounded t-structure to a point $s\in S$!$S$ regular}
\begin{enumerate}[{\rm (1)}] 
 \item \label{enum:base-change-point}
Then $(\cD^{\leqslant 0}, \cD^{\geqslant 0})$ induces a bounded t-structure on $\cD_s$.
\item \label{enum:base-change-generic}
Moreover, if $S$ is irreducible and $s \in S$ the generic point, then base change to $\cD_s$ is t-exact.
\end{enumerate}
\end{Cor}
\begin{proof} 
Let $f \colon \bar s \subset S$ be the inclusion of the closure of $s$. Since $f$ is finite and $S$ is regular, it is automatically perfect \citestacks{068B}. Hence the composition $\Spec \kappa(s) \to \bar s \to S$ is essentially perfect, and part~\eqref{enum:base-change-point} is a special case of Theorem~\ref{Thm-D-bc}. Then part~\eqref{enum:base-change-generic} follows from flatness, i.e., Theorem~\ref{Thm-D-bc}.\eqref{DbT-f-flat}.
\end{proof}

Another important consequence of the theorem is base change for t-structures along field extensions. 
By a common abuse of notation, when we base change to an affine scheme $\Spec(A)$, we denote $\cD_{\Spec (A)}$ by $\cD_{A}$.
 \index{Dl@$\cD_\ell$, base change of $\cD$ to $\Spec(\ell)$}

\begin{Prop}
\label{Prop-D-bc-fields} 
Let $X$ be a noetherian scheme of finite Krull dimension over a field $k$. 
Let $\cD \subset \Db(X)$ be a $k$-linear strong semiorthogonal component whose projection 
is of finite cohomological amplitude, and which is equipped with a bounded tilted-noetherian 
t-structure. 
Let $k \subset \ell$ be a (not necessarily finitely generated) field extension. Then: 
\begin{enumerate}[{\rm (1)}]
\item The t-structure on $(\cD_{\qc})_{\ell}$ from Theorem~\ref{thm-Dqc-bc} induces a bounded 
t-structure on $\cD_{\ell}$. 
\index{Al@$\cA_\ell$, base change of the heart of a bounded tilted-noetherian t-structure to a field extension $k\subset\ell$}
\index{AK@$\cA_K$, base change of the heart of a bounded tilted-noetherian t-structure to $K$ the function field of $S$}
\item \label{Prop-D-bc-fields-exact-bc}
Base change $\cD \to \cD_\ell$ is t-exact. 

\item \label{Prop-D-bc-fields-Efg}
For every object $E \in \cA_\ell$ there exists a subfield $\ell' \subset \ell$, finitely generated over $k$, and an object
$F \in \cA_{\ell'}$ such that $E$ is the base change of $F$.

\item 
\label{Prop-D-bc-fields-fg}
If the field extension $k \subset \ell$ is finitely generated, then 
the induced t-structure on $\cD_{\ell}$ is tilted-noetherian. 
If further the t-structure on $\cD$ is noetherian, then so is the induced 
t-structure on $\cD_{\ell}$. 

\end{enumerate} 
\end{Prop}

\begin{proof}
Let $F \in \cD_{\ell}$. 
Since $F \in \Db(X_{\ell})$ it descends to an object $F_A \in \Db(X_{A})$ for 
some finitely generated $k$-subalgebra $A \subset \ell$.
By projecting into $\cD_{A} \subset \Db(X_{A})$, we may assume 
$F_A \in \cD_{A}$. 
But by Theorem~\ref{Thm-D-bc} the t-structure on $(\cD_{\qc})_A$ induces a 
bounded t-structure on $\cD_{A}$, and by its statement \eqref{DbT-f-flat} 
the pullback functor $(\cD_{\qc})_A \to (\cD_{\qc})_{\ell}$ is t-exact. 
It follows that the t-structure on $(\cD_{\qc})_{\ell}$ induces a bounded t-structure 
on $\cD_{\ell}$, and that it satisfies \eqref{Prop-D-bc-fields-exact-bc}. 

To prove \eqref{Prop-D-bc-fields-Efg}, we can let $\ell'$ be the fraction field of $A$; then
$F_{\ell'} \in \cA_{\ell'}$ as pullback to $\ell$ is t-exact and conservative.

If $k \subset \ell$ is finitely generated, then $\ell$ can be written as the fraction 
field of a finitely generated $k$-algebra. 
Hence $\Spec(\ell) \to \Spec(k)$ is an essentially perfect morphism, so~\eqref{Prop-D-bc-fields-fg} 
holds by Theorem~\ref{Thm-D-bc}.
\end{proof}

\section{Flat, torsion, and torsion free objects} 
\label{sec:flat-torsion-tf-obj}

In this section, we consider the following situation: 
\begin{itemize}
\item $g \colon X \to S$ is a flat morphism of 
{schemes which are quasi-compact with affine diagonal, where 
$X$ is noetherian of finite Krull dimension.} 
\item $\cD \subset \Db(X)$ is an $S$-linear strong semiorthogonal component whose projection functor
is of finite cohomological amplitude. 
\item $(\cD^{\leqslant 0 }, \cD^{\geqslant 0})$ is an $S$-local t-structure on $\cD$ with heart $\cA_S$. 
\end{itemize} 
The significance of $g \colon X \to S$ being flat is that any morphism $\phi \colon T \to S$ is faithful with respect to $g$. 
Following~\cite{AP:t-structures}, we adapt to our setup the relative notions of flat, torsion, and torsion free objects.
In Section~\ref{subsection-flat-objects} and Section~\ref{subsection-torsion-objects} 
we discuss these notions in a general setting, while in Section~\ref{subsection-dedekind-bases} we study in more detail the case where $S = C$ is a Dedekind scheme, which will be particularly important later in the paper. 

\subsection{Flat objects} 
\label{subsection-flat-objects}
For any (not necessarily closed) point $s \in S$, we write $\cD_s$, $(\cD_{\perf})_s$, and 
$(\cD_{\qc})_s$ for the base change categories along 
the canonical morphism $\phi \colon \Spec(\kappa(s)) \to S$.

\begin{Def} \label{def-flat-family-T}
	Let $\phi \colon T \to S$ be a morphism, and let $E \in \Dqc(X_T)$. Then $E$ is \emph{$T$-flat} if $E_t \in (\cA_{\qc})_t$ for every point $t \in T$, where $(\cA_{\qc})_t \subset (\cD_{\qc})_t$ is the heart of the t-structure given by Theorem~\ref{thm-Dqc-bc} applied to the composition $\Spec \kappa(t) \to T \to S$.
\end{Def}

\begin{Rem}
\label{Rem-t-flat-usual-flat}
If $\cD = \Db(X)$ with the standard t-structure, then Definition~\ref{def-flat-family-T} agrees with the usual notion of flatness for an object $F \in \Coh X$ under either of the following hypotheses: $T$ is noetherian or $X_T \to T$ is of finite presentation. 
Indeed, the statement reduces to the local affine case, where for $T$ noetherian it holds by the local criterion for flatness \citestacks{00MK}, and for $X_T \to T$ of finite presentation we can reduce to the case where $S$ is noetherian using \citestacks{00QX}. 
\end{Rem}

In Section~\ref{subsubsection-flat-objects} 
below we prove some results about flat objects in the 
case of a Dedekind base.

\subsection{Torsion and torsion free objects} 
\label{subsection-torsion-objects} 

In this subsection we assume $S$ is integral. Recall that, in accordance with Section~\ref{sec:setupnotation}, if $W\subset S$ is a closed subscheme, then $i_W\colon X_W\to X$ denotes the embedding of the fiber over $W$.

\begin{Def}\label{def:D-Stor}
An object $E \in \cD$ is called \emph{$S$-torsion} if it is the pushforward 
of an object in $\cD_Z$ for some proper closed subscheme $Z \subset S$. 
We denote by $\cD_\Stor$ the subcategory of $S$-torsion objects in $\cD$.
\index{DStor@$\cD_\Stor$, subcategory of $S$-torsion objects in $\cD$}
\end{Def}

\begin{Lem}\label{lem:D-Stor}
Let $K$ be the function field of $S$. 
For $E \in \cD$, we have $E\in \cD_\Stor$ if and only if $E_K = 0$, namely there is an exact sequence of triangulated categories
\begin{equation}\label{eqn:D-Stor}
\cD_\Stor \to \cD \to \cD_K.
\end{equation}
\end{Lem}
\begin{proof}
The statement follows directly from the corresponding statement for $\Db(X)$.
\end{proof}

\begin{Def}
\label{def-torsion}
An object $E \in \cA_S$ is called \emph{$S$-torsion free} if it contains no nonzero $S$-torsion subobject.
We denote by $\cA_{\Stor} \subset \cA_S$ the subcategory of $S$-torsion objects, and by $\cA_{\Stf} \subset \cA_S$ the subcategory of $S$-torsion free objects. 
\index{AStor@$\cA_{\Stor}$, subcategory of $S$-torsion objects in $\cA_S$}
\index{AStf@$\cA_{\Stf}$, subcategory of $S$-torsion free objects in $\cA_S$}
\end{Def}

\begin{Lem}
\label{Lem-AS-subs}
The following hold: 
\begin{enumerate}[{\rm (1)}]
\item \label{Astf-subs} The subcategory $\cA_{\Stf} \subset \cA_S$ is closed under subobjects and extensions. 
\item \label{Astor-subs} The subcategory $\cA_{\Stor} \subset \cA_S$ is closed under subobjects, quotients, and extensions. 
\item \label{Dstor-t-structure} The t-structure on $\cD$ induces one on $\cD_{\Stor}$, which is bounded 
if the given t-structure on $\cD$ is. 
\item \label{Astf-Astor-tensor} If $L$ is a line bundle on $S$, then tensoring by $g^*L$ preserves $\AA_{\Stor}$ and $\AA_{\Stf}$.
\end{enumerate}
\end{Lem} 

\begin{proof}
Part~\eqref{Astf-subs} follows immediately from the definitions. 
Let $K$ be the function field of $S$. 
Then $\Spec(K) \to S$ is flat, so by Theorem~\ref{Thm-D-bc}.\eqref{DbT-f-flat} the 
pullback functor $\cD \to \cD_K$ is t-exact. 
Using this and Lemma~\ref{lem:D-Stor}, parts~\eqref{Astor-subs} and~\eqref{Dstor-t-structure} follow easily. Part~\eqref{Astf-Astor-tensor} is immediate as $\blank \otimes g^*L$ preserves $\cA_S$, see Lemma~\ref{lem-tensor-vb}, and $\cD_{\Stor}$.
\end{proof}

\begin{Lem}
\label{lem:cohpushpull}
Let $T \subset S$ be the zero locus of a regular section of a vector bundle, and let $\cN$ be the normal bundle. 
Then for any $E \in \cA_T$, there are isomorphisms 
\begin{equation}\label{eqn:CohomologyOfPushPull}
\rH^i_{\cA_T}(i_T^*i_{T*}^{} E) \cong \bigwedge^{-i}g_T^*\cN^\vee\otimes E, 
\end{equation}
for all $i\in\Z$. In particular, we always have $\rH^0_{\cA_T}(i_T^*i_{T*}^{}E)\cong E$, and if $T$ is zero-dimensional, so that $\cN\cong\cO_T^{\oplus\dim S}$, then we also have $\rH^{-\dim S}_{\cA_T}(i_T^*i_{T*}^{}E)\cong E$.
Finally, $i_{T*}\colon \cA_T\to\cA_S$ is fully faithful. 
\end{Lem} 
\begin{proof}
The isomorphism on cohomology objects follows by the $t$-exactness of $i_{T*}$ and the Koszul complex as in \cite[Corollary~11.2]{Huybrechts:FM}.

To see the claim about the fully faithfulness of $i_{T*}$, we let $E,F\in\cA_T$, and observe that $i_T^*i_{T*}^{}E\in\cD_T^{\leqslant 0}$ by \eqref{eqn:CohomologyOfPushPull}. It then follows by adjunction and the $i=0$ case of \eqref{eqn:CohomologyOfPushPull} that
\begin{align*}\Hom_{\cA}(i_{T*}E,i_{T*}F)&\cong\Hom_{\cD}(i_{T*}E,i_{T*}F)\cong\Hom_{\cD_T}(i_T^*i_{T*}E,F)\\&\cong\Hom_{\cD_T}(\rH_{\cA_T}^0(i_T^*i_{T*}^{}E),F)\cong\Hom_{\cD_T}(E,F)\cong\Hom_{\cA_T}(E,F),\end{align*}
as required.
\end{proof}

The following is a version of Lemma~\ref{lem-extend-from-localisation}.\eqref{enum:lem-extend-morphism-in-heart} for $S$-torsion free objects.

\begin{Lem} \label{lem:tensortrick}
Assume that $g\colon X\to S$ is flat and $S$ is integral with function field $K$.
Let $L$ be an ample line bundle on $S$, and let $A, B \in \cA_S$ be objects with $A$ being $S$-torsion free. Given
any isomorphism $A_K \cong B_K$, there exists $k \in \Z$ and an injective map
$A \otimes g^* L^{-k} \into B$ inducing the given isomorphism over $K$.
\end{Lem}

\begin{proof}
 Since $A, B$ are bounded complexes, the isomorphism is defined over some open subset $U \subset S$, so by \cite[Lemma~2.1.8]{AP:t-structures}, there exists a map $A \otimes g^* L^{-k} \to B$ that induces the isomorphism over $K$.
Let $Q \in \langle \cA_S, \cA_S[1]\rangle$ be the cone of this map; note $Q_K = 0$. Hence $Q\in \cD_\Stor$ and all cohomology objects $\rH^i_{\cA_S}(Q)$ are $S$-torsion. Since $A$ is $S$-torsion free, the long exact
cohomology sequence shows $\rH^{-1}_{\cA_S}(Q) = 0$, proving the injectivity as claimed.
\end{proof}

\subsection{Dedekind bases} 
\label{subsection-dedekind-bases}

In this subsection we assume 
$S = C$ is a Dedekind scheme. 

\subsubsection{Torsion objects} 
For any non-trivial closed subscheme $W \subset C$, the ideal sheaf $I_W$ is a line bundle. 
We will abuse notation by writing $I_W \otimes (-)$ for the tensor product $g^* (I_W) \otimes (-)$, which preserves $\cA_C$.
Note that 
by Theorem~\ref{thm-t-structure-finite-map} there is an induced t-structure on $\cD_W$ whose heart we denote by $\cA_W$.

\begin{Lem} \label{lem:fibersinheart}
Let $W \subset C$ be a 0-dimensional subscheme with ideal sheaf $I_W$. 
Let $\cA_W$ be the heart of the induced t-structure on $\cD_W$ given by Theorem~\ref{thm-t-structure-finite-map}. 
Then for any $E \in \cA_C$ we have short exact sequences
\begin{equation}\label{eqn:fibersinheart1} 
0 \to \Ann(I_W; E) = i_{W*}^{} \rH^{-1}_{\cA_W} \left( E_W \right) \into I_W \otimes E \onto I_W \cdot E \to 0 \end{equation}
and
\begin{equation*}
0 \to I_W \cdot E \into E \onto E/I_W\cdot E = i_{W*}^{} \rH^0_{\cA_W} \left(E_W \right) \to 0,
\end{equation*} 
where $\Ann(I_W; E)$ and $I_W \cdot E$ denote the kernel and the image of the canonical map 
\begin{equation*}
I_W \otimes E \to E. 
\end{equation*}
Moreover, $\rH^i_{\cA_W}(E_W) = 0$ for $i \neq -1,0$.
\end{Lem}

\begin{proof}
We take cohomology of the exact triangle 
$I_W \otimes E \to E \to i_{W*}i_W^* E$ with respect to $\cA_S$, using that $i_{W*}$ is t-exact by Theorem~\ref{Thm-D-bc}.\eqref{DbT-f-finite-pushforward}.
\end{proof}

In the situation of Lemma~\ref{lem:fibersinheart}, we call the essential image of 
$i_{W*} \colon \cA_W \to \cA_C$ the subcategory of objects \emph{schematically supported over $W$}; 
it is equivalent to $\cA_W$ by the last claim of Lemma~\ref{lem:cohpushpull}.

\begin{Cor} \label{cor:schematicsupportsubsquotients}
Let $W \subset C$ be a 0-dimensional subscheme. 
Then the subcategory of $\cA_C$ of objects schematically supported over $W$ is closed under subobjects and quotients.
\end{Cor}
\begin{proof}
By Lemma~\ref{lem:fibersinheart}, an object $E \in \cA_C$ is schematically 
supported on $W$ if and only if the map $I_W \otimes E \to E$ vanishes. Given a subobject
$A \into E$, it follows that the composition
$I_W \otimes A \to A \into E$ vanishes (as it also factors via $I_W \otimes E$), and therefore $A$ is also schematically 
supported on $W$. The case of 
quotients follows similarly.
\end{proof}
We say that $E \in \cA_{\Ctor}$ is set-theoretically supported over a closed point $p \in C$ if it is scheme-theoretically supported over some infinitesimal neighborhood of $p$.
\begin{Lem} \label{lem:filtrationatp}
Let $p \in C$ be a closed point, $\pi$ be a local generator of $I_p$ and $E \in \cA_{\Ctor}$ be an object set-theoretically supported over $p$. Then $E$ admits a filtration 
\[ 0= \pi^{m+1}\cdot E \subset \pi^m \cdot E \subset \dots \subset \pi\cdot E \subset E
\]
where all filtration quotients $\pi^i\cdot E/\pi^{i+1} \cdot E$ are quotients of $E/\pi \cdot E$ in $\cA_p$.
\end{Lem}
\begin{proof}
By assumption, there exists $m$ such that $I_p^{m+1}\cdot E =0$. 
The local isomorphism between $\cO_C$ and $I_p$ induced by $\pi$ identifies $I_p^i \cdot E$ with $\pi^i \cdot E$, i.e., the image of the endomorphism of $E$ induced by $\pi^i$. Finally, by definition, $\pi^i$ induces a surjection $E/\pi \cdot E \onto \pi^i \cdot E/\pi^{i+1} \cdot E$.
\end{proof}

\subsubsection{Flat objects}
\label{subsubsection-flat-objects}

For any point $c \in C$, there is an induced heart $\cA_c$ on $\cD_c$. 
Indeed, if $c \in C$ is closed then this holds by Theorem~\ref{thm-t-structure-finite-map}, 
while if $c \in C$ is the generic point this holds by Theorem~\ref{Thm-D-bc}. 
In particular, suppose $E \in \Dqc(X)$ is an object such that $E_c \in \cD_c$ for every $c \in C$; 
this holds for instance if $E \in \cD$ (see Lemma~\ref{lem-relations-S-perfect}). 
Then $E$ is $C$-flat in the sense of Definition~\ref{def-flat-family-T} if and only if 
$E_c \in \cA_c$ for every point $c \in C$. 
This observation will be used without mention below. 

\begin{Lem} \label{lem:FlatIffTFreeCurve}
Let $E \in \cA_C$. 
Then $E$ is $C$-flat if and only if $E$ is $C$-torsion free. 
\end{Lem}

\begin{proof} 
If $c \in C$ is the generic point, then by Corollary~\ref{cor:base-change-tstructure-point}.\eqref{enum:base-change-generic} the pullback $\cD \to \cD_c$ is t-exact, i.e., $E_c \in \cA_c$ is automatic. 
Thus, $E$ is $C$-flat if and only if $E_c \in \cA_c$ for all closed points $c \in C$. 

Now it follows from Lemma~\ref{lem:fibersinheart} that $E$ is $C$-flat if and only if for every closed point $c \in C$ 
the map $I_c \otimes E \to E$ is injective in $\cA_C$. 
Moreover, Lemma~\ref{lem:fibersinheart} also shows that if this map is not injective for some $c$, then $E$ is not $C$-torsion free. 
Conversely, assume $E$ has a torsion subobject $A \into E$. 
We may assume that $A$ is set-theoretically supported over a closed point $c \in C$. 
Then there exists a positive integer $m$ so that 
the natural map $I_c^{\otimes m} \otimes A \to A$
vanishes. 
In particular this map $I_c^{\otimes m} \otimes E \to E$ is non-injective as a map in $\cA_C$, which implies the same for $I_c \otimes E \to E$. 
\end{proof}

\begin{Lem} \label{lem:flatinheart}
Let $E \in \cD$ be a $C$-flat object. Then $E \in \cA_C$.
\end{Lem}
\begin{proof}
Since $E_K \in \cA_K$ by assumption, and since pullback to $\cD_K$ is t-exact by Corollary~\ref{cor:base-change-tstructure-point}.\eqref{enum:base-change-generic}, we have $\rH^i_{\cA_C}(E) \in \cA_{\Ctor}$ for all $i \neq 0$.
Consider the maximal $i$ such that $0 \neq \rH^i_{\cA_C}(E) \in \cA_{\Ctor}$ and assume $i > 0$; by the previous lemmas, it is of the form
$i_{W*}(F)$ for some 0-dimensional closed subscheme $W \subset C$. Since the question is local on $C$, we may assume it is supported over a single closed point $p \in C$. Then as $\cO_W$ is an iterated extension of copies of $\cO_p$, it follows that $i_{W*} E_W=E\otimes i_{W*}\cO_W$ is a self-extension of a number of copies of $i_{p*} E_p=E\otimes i_{p*}\cO_p$. Therefore $i_{W*} E_W \in \cA_C$, and as $i_{W*}$ is t-exact and conservative, we have
$E_W \in \cA_W$.
Therefore, by adjunction
\[
\Hom\left(E, \rH^i_{\cA_C}(E)[-i]\right) = 
\Hom\left(E_W, F[-i]\right) = 0,
\]
a contradiction. Similarly, if $i< 0$ is minimal with $\rH^i_{\cA_C}(E) \neq 0$, and hence of the form $i_{W*}(F)$ for $F \in \cA_W$, we observe that since $X_W \subset X$ is the inclusion of a Cartier divisor with trivial restriction to itself, we have 
$i_W^! E = E_W[-1]$ (see e.g.~\citestacks{0AA4}). We therefore obtain an analogous contradiction from
\[
\Hom\left(\rH^i_{\cA_C}(E)[-i], E\right) =
\Hom(F[-i], i_W^!E) = \Hom(F, E_W[-1+i]) = 0.\qedhere
\]
\end{proof}

We say the heart $\cA_C$ satisfies \emph{openness of flatness} if for every $E \in \cD$, the 
set 
\begin{equation*}
\left\{c \in C \colon E_c \in (\cA_{\qc})_c \right\} 
\end{equation*} 
is open. 
Later in Section~\ref{subsec:fiberwise-collec-t-str} we will study this property in a more general setting (using Lemma~\ref{lem-relations-S-perfect}, our definition there --- Definition~\ref{def-open-flat-utau} --- 
is easily seen to be equivalent to the one above 
under the assumption that the base $C$ is Dedekind). 
As a consequence of Lemma~\ref{lem:flatinheart}, openness of flatness implies \emph{openness of the heart}: 

\begin{Cor} \label{cor:opennessflatness-opennessheart}
	Assume that $\cA_C$ satisfies openness of flatness. 
	Let $E \in \cD$ and assume $E_c \in \cA_c$ for some $c \in C$. Then there exists an open neighborhood $U \subset C$ of $c$ such that
	$E_U \in \cA_U$.
\end{Cor}

\subsubsection{Torsion theories}
The following property will play an important role later in the paper. 

\begin{Def}
\label{Def-C-tt} 
We say $\cA_C$ has a \emph{$C$-torsion theory} if the pair of subcategories $(\cA_{\Ctor}, \cA_{\Ctf})$ forms a torsion pair.
\end{Def}

\begin{Rem} \label{rem:noetherianandtorsiontheory}
It follows from Lemma~\ref{Lem-AS-subs}.\eqref{Astor-subs} that 
the heart $\cA_C$ has a $C$-torsion theory if and only if every object $E \in \cA_C$ contains a unique maximal $C$-torsion subobject $E_{\Ctor} \subset E$. 
In this case, we denote 
by $E_{\Ctf} = E/E_{\Ctor}$ the $C$-torsion free quotient in $\cA_C$. We also note that $\cA_C$ being noetherian implies the existence of a $C$-torsion theory.
\end{Rem} 

The following {lemma} can be helpful in proving the existence of a $C$-torsion theory.
\begin{Lem} \label{lem:maxtorsionses}
	Let $A \into E \onto B$ be a short exact sequence in a $C$-local heart,
	and assume that $A$ and $B$ admit maximal $C$-torsion subobjects. Then the same holds for $E$.
\end{Lem}
\begin{proof}
	We may assume that $B$ is $C$-torsion, as any $C$-torsion subobject of $E$ will factor via the preimage of the maximal $C$-torsion subobject of $B$; and evidently we may assume that $A$ is $C$-torsion free. Let $W$ be the schematic support of $B$, in the sense of Lemma~\ref{lem:fibersinheart}, and consider the short exact sequence
	$\Ann(I_W; E) \into I_W \otimes E \onto I_W \cdot E$; recall that
	$I_W \cdot E$ is the image of the natural map $I_W \otimes E \to E$. Since $I_W \cdot B = 0$, this map factors via $A \into E$; therefore,
	$I_W \cdot E \subset A$ is $C$-torsion free. Replacing $E$ by $E \otimes I_W^{-1}$ in this argument, we have found a torsion free quotient of $E$ by a $C$-torsion subobject.
\end{proof}

In fact, $E_{\Ctor}$ can be identified somewhat more explicitly as follows. 

\begin{Lem} 
\label{Lem-ECtor-H-1}
Assume $E \in \cA_C$ admits a maximal $C$-torsion subobject $E_{\Ctor} \subset E$. 
Let $W \subset C$ be the schematic support of $E_{\Ctor}$. 
Then $E_{\Ctor} = I_W^{-1} \otimes i_{W*}^{}\rH^{-1}_{\cA_W}(E_W)$. 
\end{Lem}

\begin{proof}
By the choice of $W$ we have $I_W \cdot E_{\Ctor} = 0$, so by Lemma~\ref{lem:fibersinheart} 
we find 
\begin{equation*}
E_{\Ctor} = I_W^{-1} \otimes i_{W*}^{}\rH^{-1}_{\cA_W}\left((E_{\Ctor})_W\right). 
\end{equation*} 
Moreover, upon restricting the short exact sequence $E_{\Ctor} \into E\onto E_{\Ctf}$ to $W$, we see that $\rH^{-1}_{\cA_W}\left((E_{\Ctor})_W\right)=\rH^{-1}_{\cA_W}\left(E_W\right)$, giving the claimed equality. Indeed, $E_{\Ctf}$ is $C$-flat by Lemma~\ref{lem:FlatIffTFreeCurve} so we also have $\rH^{i}_{\cA_W}\left((E_{\Ctf})_W\right)=0$ for $i\neq 0$ by the proof of Lemma~\ref{lem:flatinheart}.
\end{proof}

\begin{Prop} \label{prop:Ctorsiontheory-opennessflatness}
	Assume that $\cA_C$ has a $C$-torsion theory.
	\begin{enumerate}[{\rm (1)}]
		\item (Nakayama's Lemma) \label{enum:Nakayama}
		If $E \in \cA_C$ satisfies $E/I_c\cdot E = 0$ for some $c \in C$, then there exists an open neighborhood $c \in U \subset C$ such that $E_U = 0$. In particular, if $E$ is also $C$-torsion free, then $E = 0$.
		\item \label{enum:Ctorsiontheory-opennessflatness}
		The heart $\cA_C$ satisfies openness of flatness.
	\end{enumerate}
\end{Prop}
\begin{proof}
	By Lemmas~\ref{lem:fibersinheart} and~\ref{lem:FlatIffTFreeCurve}, there is a short exact sequence
	\[ \rH_{\cA_c}^0(\left(E_{\Ctor}\right)_c) \into \rH_{\cA_c}^0(E_c) \onto \rH_{\cA_c}^0(\left(E_{\Ctf}\right)_c) = \left(E_\Ctf\right)_c. \]
	Hence the assumption in \eqref{enum:Nakayama} implies $\left(E_\Ctf\right)_c = 0$, and so $E_\Ctf =0$. 
	So $E$ is $C$-torsion. 
	If its support contains $c$, then $c$ is a closed point, and we proceed as in the proof of Lemma~\ref{lem:FlatIffTFreeCurve} to show that the canonical map $I_c \otimes E \to E$ cannot be surjective, so that its cokernel $\rH^0_{\cA_c}(E_c)$ cannot vanish. 
	This contradiction completes the proof of \eqref{enum:Nakayama}.
	
	To prove \eqref{enum:Ctorsiontheory-opennessflatness}, assume that $E_c \in \cA_c$. Let $i$ be maximal with $F^i := \rH^i_{\cA_C}(E) \neq 0$, and assume $i > 0$. 
	Since $i_c^*$ is right t-exact by Theorem~\ref{Thm-D-bc}.\eqref{DbT-f-flat}, one deduces $F^i/I_c\cdot F^i = i_{c*}\rH^0_{\cA_c}(F^i_c) = 0$. 
	By \eqref{enum:Nakayama} we can replace $C$ by an open neighborhood $c \in U \subset C$ such that $F^i_U$ vanishes; repeating this process, we obtain $\rH^i_{\cA_U}(E_U) = 0$ for $i > 0$. 
	Similarly, if $i \leqslant 0 $ is minimal with $F^i \neq 0$, then we can conclude from Lemma~\ref{lem:fibersinheart} that $\rH^{-1}_{\cA_c}((F^i)_c)=0$ and thus that $F^i$ has no torsion supported at $c$. Indeed, we first observe that for any $j$ and $E\in\cD$, we can prove by induction and the final statement of Lemma~\ref{lem:fibersinheart} that $\left(\tau^{\geqslant j+1}(E)\right)_c\in\cD_c^{\geqslant j}$. Then the vanishing of $\rH^{-1}_{\cA_c}((F^i)_c)$ follows from the vanishing of both $\rH^{i-1}_{\cA_c}(E_c)$ and $\rH^{i-2}_{\cA_c}\left(\left(\tau^{\geqslant i+1}(E)\right)_c\right)$, and it follows from Lemma~\ref{lem:fibersinheart} that $F_i$ has no torsion supported at $c$. After restricting to a smaller open neighborhood of $c$, we may even assume that $F^i$ is $C$-torsion free.
	If $i < 0$, then $\left(\tau^{\geqslant i+1}(E)\right)_c\in\cD_c^{\geqslant i}$ implies that $F^i_c=\rH^i_{\cA_c}(E_c) =0 $, so that $F^i_U = 0$ by \eqref{enum:Nakayama}. Repeating this process, we obtain an open neighborhood $U$ such that $E_U\in\cA_U$ and is $U$-torsion free, and thus flat over $U$ by Lemma~\ref{lem:FlatIffTFreeCurve}.
\end{proof}

\begin{Rem}
\label{Rem-Nak}
Even if $\cA_C$ does not have a $C$-torsion theory, if $E \in \cA_C$ is any particular object that \emph{does} contain a maximal $C$-torsion subobject, then the conclusion of Proposition~\ref{prop:Ctorsiontheory-opennessflatness}.\eqref{enum:Nakayama} holds by the same argument. 
\end{Rem}
		
Note that in \cite[Section~3]{AP:t-structures} some of the previous results are proven in a more general setting when $S$ is smooth but not necessarily Dedekind, under the assumption that $\cA_S$ is noetherian.

\section{Inducing local t-structures on semiorthogonal components}
\label{sec:inducingtstructures}

In this section, we consider the following situation: 
\begin{itemize}
\item $g \colon X \to S$ is a morphism of 
{schemes which are quasi-compact with affine diagonal, where 
$X$ is noetherian of finite Krull dimension.}
\item $\cD \subset \Db(X)$ is an $S$-linear strong semiorthogonal component whose projection functor
is of finite cohomological amplitude. 
\end{itemize} 
Our goal is to generalize to the relative case the results of \cite[Section~4]{BLMS} on inducing t-structures on semiorthogonal components.

\begin{Def}\label{def:rel-spanning-class}
A \emph{relative spanning class} of $\cD$ is a set of objects $\cG$ of $\cD$ such that if $F \in\cD$ satisfies $\cHom_S(G,F)=0$ for all $G\in\cG$, then $F=0$.
\end{Def}

This property can be checked fiberwise: 

\begin{Lem}\label{lem:SpanningClassOnFibers}
Assume $g \colon X \to S$ is flat.
Let $\cG$ be a set of perfect objects in $\cD$. 
Then $\cG$ is a relative spanning class if and only if for all closed points $s\in S$ the restriction $\cG_s:=i_s^*\cG$ is a spanning class of $\cD_s$. 
\end{Lem}

\begin{proof}
For the forward direction, assume $\cG$ is a relative spanning class, and let $F \in \cD_s$ be an object which satisfies $\cHom_s(G_s, F) = 0$ for all $G \in \cG$.
Note that $\cHom_s(-,-)$ is simply $\RHom(-,-)$ regarded as a $\kappa(s)$-complex. 
We must show that $F = 0$. 
Consider the base change diagram 
\begin{equation*}
\vcenter{
\xymatrix{
 X_s \ar[r]^{i_s} \ar[d]_{g_s} & X \ar[d]^{g} \\ 
s \ar[r]^{j_s} & S 
}
}
\end{equation*} 
We have isomorphisms 
\begin{align*}
\cHom_S(G, i_{s*}F) & \simeq 
g_* \cHom(G, i_{s*}F) \\
& \simeq g_* i_{s*} \cHom(G_s, F) \\
& \simeq j_{s*} g_{s*} \cHom(G_s, F) \\ 
& \simeq j_{s*} \cHom_s(G_s, F), 
\end{align*} 
where the second holds by the local adjunction between $i_s^*$ and $i_{s*}$ 
and the others are evident. 
By assumption this vanishes for any $G \in \cG$, so $i_{s*}F = 0$ since $\cG$ is a relative spanning class, and thus $F = 0$. 

Conversely, assume $\cG_s$ is a spanning class of $\cD_s$ for all $s$, and let $F \in \cD$ be an object which satisfies $\cHom_S(G,F) = 0$ for all $G \in \cG$. 
We must show that $F = 0$. 
For any $s \in S$, we have $\cHom_S(G,F)_s \simeq \cHom_s(G_s, F_s)$ by Lemma~\ref{Lem-cHom-bc} and the flatness of $g$. 
By assumption this vanishes for any $G \in \cG$, so $F_s = 0$ since $\cG_s$ is a relative spanning class, 
and thus $F = 0$. 
\end{proof}

The following is an easy consequence of the definitions. 
\begin{Lem}\label{lem:spanning-class-orthogonal} 
Let $\cD = \langle \cD_1, \cD_2 \rangle$ be an $S$-linear semiorthogonal decomposition. 
Let $\cG$ be a spanning class of $\cD_2$. Then for an object $F \in \cD$, we have 
$F \in \cD_1$ if and only if $\cHom_S(G, F) = 0$ for all $G \in \cG$. 
\end{Lem} 

The following is a relative version of \cite[Lemma~4.3]{BLMS}. 

\begin{Lem}\label{lem:induce-relative-t}
Assume $g \colon X \to S$ is proper, and $S$ is noetherian and admits an ample line bundle. 
Let $\cA_S \subset \cD$ be the heart of a bounded $S$-local t-structure.
Let $\cD = \langle \cD_1, \cD_2 \rangle$ be an $S$-linear semiorthogonal decomposition. 
Let $\cG$ be a relative spanning class of $\cD_2$ 
such that $\cG \subset \cA_S \cap \cD_{2} \cap \Dperf(X)$ and 
\begin{equation}
\label{Hom-leq1}
\cHom_S(G, F) \in \Db(S)^{\leqslant 1} 
\end{equation} 
for all $G\in\cG$ and $F \in \cA_S$. 
Then 
\begin{equation*}
(\cA_{S})_1 =\cA_S \cap \cD_1 \subset \cD_1
\end{equation*}
is the heart of a bounded $S$-local t-structure on $\cD_1$, such 
that the inclusion $\cD_1 \to \cD$ is t-exact. 
\end{Lem}

\begin{proof}
To show $(\cA_{S})_1$ is the heart of a bounded t-structure on $\cD_1$, 
we check the conditions of Proposition~\ref{prop:BridgelandCriterion}. 
Condition~\eqref{eq:BridgelandCriterion1} is clearly satisfied.

For condition~\eqref{eq:BridgelandCriterion2} it suffices to show that for 
$F \in \cD_1$ we have $\rH^q_{\cA_S}(F) \in (\cA_S)_1$ for all $q \in \Z$. 
Note that this will also show $\rH^q_{(\cA_S)_1}(F) = \rH^q_{\cA_S}(F)$, so that $\cD_1 \to \cD$ is t-exact. 
Let $q$ be the smallest integer such that $\rH^q_{\cA_S}(F) \neq 0$, so that there is an exact 
triangle 
\begin{equation*} 
\rH^q_{\cA_S}(F)[-q] \to F \to \tau^{> q}(F). 
\end{equation*} 
It suffices to show $\rH^q_{\cA_S}(F) \in (\cA_S)_1$ for this particular $q$, because then by induction the statement 
follows for all $q$. 
By Lemma~\ref{lem:spanning-class-orthogonal} we have 
$\rH^q_{\cA_S}(F) \in (\cA_S)_1$ if and only if 
\begin{equation}
\label{HqASF-vanishing}
\cHom_S(G, \rH^q_{\cA_S}(F)) = 0
\end{equation}
for all $G \in \cG$. 
Since $G \in \cD_2$ and $F \in \cD_1$ we have $\cHom_S(G, F) = 0$ (see Remark~\ref{remark-relative-Hom-sod}), 
so by the 
above exact triangle we have an isomorphism 
\begin{equation*}
\cHom_S(G, \rH^q_{\cA_S}(F)) \simeq \cHom_S(G, \tau^{> q}(F))[q-1]. 
\end{equation*} 
By the assumption~\eqref{Hom-leq1} the left 
side lies in $\Db(S)^{\leqslant 1}$, while by Lemma~\ref{lem:HomS-heart} the right side lies in $\Db(S)^{>1}$. 
Hence both sides vanish, which proves~\eqref{HqASF-vanishing} for all $G \in \cG$. 

Finally, it follows from Theorem~\ref{thm:local-t-structure-tensor-ample} that 
the t-structure $(\cA_S)_1 \subset \cD_1$ is $S$-local. 
\end{proof}

To formulate a useful situation in which Lemma~\ref{lem:induce-relative-t} applies, 
we need the notion of a relative Serre functor. 
An $S$-linear functor $\rS_{\cD/S} \colon \cD \to \cD$ is called a relative Serre functor for 
$\cD$ over $S$ if there are functorial isomorphisms 
\begin{equation*}
\cHom_S(F, \rS_{\cD/S}(G)) \simeq \cHom_S(G, F)^{\vee} 
\end{equation*} 
for all $F, G \in \cD$. The case where $S$ is a point reduces to the usual notion of a Serre functor. 
For instance, if $g \colon X \to S$ is a 
smooth and proper morphism of noetherian schemes, then the functor 
\begin{equation*}
(- \otimes \omega_{g}[\dim(g)]) \colon \Db(X) \to \Db(X)
\end{equation*} 
is a relative Serre functor by Grothendieck duality. 
Moreover, if $\gamma \colon \cC \to \cD$ is a right admissible subcategory 
and $\rS_{\cD/S}$ is a relative Serre functor for $\cD$, then it follows easily 
from the definitions that $\cC$ has a relative Serre functor given by 
$\rS_{\cC/S} = \gamma^! \circ \rS_{\cD/S} \circ \gamma$. 
In particular, putting the previous two remarks together, we see that if $g \colon X \to S$ is a smooth and proper 
morphism of noetherian schemes, then any right admissible subcategory of $\Db( X)$ admits a relative Serre functor. 

Now we can give the relative version of \cite[Corollary~4.4]{BLMS}. 
\begin{Cor}\label{cor:induce-relative-t}
Assume $g \colon X \to S$ is smooth and proper, and $S$ is noetherian and admits an ample line bundle. 
Let $\cA_S \subset \cD$ be the heart of a bounded $S$-local t-structure. 
Let $\cD = \langle \cD_1, \cD_2 \rangle$ be an $S$-linear semiorthogonal decomposition. 
Let $\cG$ be a relative spanning class of $\cD_2$ such that $\cG \subset \cA_S \cap \cD_{2} \cap \Dperf(X)$ 
and every $G \in \cG$ satisfies $\rS_{\cD/S}(G) \in \cA_S[1]$, where 
$\rS_{\cD/S}$ denotes the relative Serre functor of $\cD$. 
Then 
\begin{equation*}
(\cA_S)_1 = \cA_S \cap \cD_1 \subset \cD_1
\end{equation*} 
is the heart of a bounded $S$-local t-structure on $\cD_1$, such 
that the inclusion $\cD_1 \to \cD$ is t-exact. 
\end{Cor}

\begin{proof}
Let $G \in \cG$ and $F \in \cA_S$. 
We have 
\begin{equation*}
\cHom_S(G, F) \simeq \cHom_S(F, \rS_{\cD/S}(G))^{\vee}. 
\end{equation*} 
Since $\rS_{\cD/S}(G) \in \cA_S[1]$, Lemma~\ref{lem:HomS-heart} gives 
\begin{equation*}
\cHom_S(F, \rS_{\cD/S}(G)) \in \Db(S)^{\geqslant -1}, 
\end{equation*} 
and hence 
\begin{equation*}
\cHom_S(F, \rS_{\cD/S}(G))^{\vee} \in \Db(S)^{\leqslant 1}. 
\end{equation*} 
Therefore the assumptions of Lemma~\ref{lem:induce-relative-t} are satisfied. 
\end{proof}

We also have the following fibral variant. 
\begin{Cor}\label{cor:induce-relative-t-fiberwise}
Assume $g \colon X \to S$ is smooth and proper, and $S$ is noetherian, regular, and admits an ample line bundle. 
Let $\cA_S \subset \cD$ be the heart of a bounded $S$-local t-structure. 
Let $\cD = \langle \cD_1, \cD_2 \rangle$ be an $S$-linear semiorthogonal decomposition. 
Let $\cG$ be a relative spanning class of $\cD_2$ such that $\cG \subset \cA_S$ 
and every $G \in \cG$ satisfies $\rS_{\cD_s}(G_s)\in \cA_s[1]$ for all closed points $s \in S$, 
where $\rS_{\cD_s}$ denotes the Serre functor of $\cD_s$ and $\cA_s \subset \cD_s$ is the 
induced heart given by Theorem~\ref{thm-t-structure-finite-map}. 
Then 
\begin{equation*}
(\cA_S)_1 = \cA_S \cap \cD_1 \subset \cD_1
\end{equation*} 
is the heart of a bounded $S$-local t-structure on $\cD_1$, such 
that the inclusion $\cD_1 \to \cD$ is t-exact. 
\end{Cor}

\begin{proof}
Note that our assumptions imply that $X$ is regular of finite Krull dimension, and thus $\Db(X) = \Dperf(X)$. 
Let $G \in \cG$ and $F \in \cA_S$. For every closed point $s \in S$, we have 
\begin{equation*}
\cHom_S(G,F)_s \simeq \cHom_s(G_s, F_s) \simeq \cHom_s(F_s, \rS_{\cD_s}(G_s))^{\vee} , 
\end{equation*} 
where the first isomorphism holds by Lemma~\ref{Lem-cHom-bc} and the flatness of $g$, and for the 
second we used that $F_s, G_s \in \cD_s$ since $F$ and $G$ are perfect. 
This object lies in $\Db(\kappa(s))^{\leqslant 1}$ since $F_s \in \cD_s^{\leqslant 0}$ and $\rS_{\cD_s}(G_s) \in \cA_s[1]$. 
Because this holds for all closed points $s$, it follows that $\cHom_S(G,F) \in \Db(S)^{\leqslant 1}$. Therefore the assumptions of Lemma~\ref{lem:induce-relative-t} are satisfied. 
\end{proof} 

\begin{Rem}
\label{rem-inducing-2-implies-1} 
The assumptions of Corollary~\ref{cor:induce-relative-t-fiberwise} imply those of 
Corollary~\ref{cor:induce-relative-t} if the heart $\cA_S$ is noetherian or if the base is Dedekind.
Indeed, it follows easily from the definitions that there is an isomorphism of functors 
$i_{s*} \circ i_s^* \circ \rS_{\cD/S} \simeq i_{s*} \circ \rS_{\cD_s} \circ i_s^*$. 
From this we see that $\rS_{\cD_s}(G_s)\in \cA_s[1]$ is equivalent to 
$i_s^* \rS_{\cD/S}(G) \in \cA_s[1]$. 
Hence \cite[Proposition~3.3.2]{AP:t-structures} (in the case $\cA_S$ noetherian) or Lemma~\ref{lem:flatinheart} (in case the base is Dedekind) implies $\rS_{\cD/S}(G) \in \cA_S[1]$ for all $G \in \cG$ 
under the assumptions of Corollary~\ref{cor:induce-relative-t-fiberwise}. 
\end{Rem}

\begin{Rem}
\label{rem-inducing-commute-fibers} 
In the situation of Corollary~\ref{cor:induce-relative-t-fiberwise}, the formation of the 
induced heart commutes with restriction to fibers. 
In symbols, if $s \in S$ is a closed point, then $((\cA_S)_1)_s = (\cA_s)_1$. 
This follows easily from the definitions. 
\end{Rem}

We will also need the following observation. 

\begin{Lem} \label{lem:induceCtorsion_Dedekind}
Assume in the setting of Lemma~\ref{lem:induce-relative-t} (or Corollary~\ref{cor:induce-relative-t} or \ref{cor:induce-relative-t-fiberwise}) that $S = C$ is a Dedekind scheme $C$ and $\cA_C$ has a $C$-torsion theory. 
Then $(\cA_C)_1$ also has a $C$-torsion theory.
\end{Lem}

\begin{proof}
It suffices to show that if $E \in (\cA_C)_1$ and $E_{\Ctor} \in \cA_C$ is 
its maximal $C$-torsion subobject in $\cA_C$, then $E_{\Ctor}$ lies in $(\cA_C)_1$. 
Let $W \subset C$ be the support of $E_{\Ctor}$. 
Then by Lemma~\ref{Lem-ECtor-H-1} we have $E_{\Ctor} = g^*(I_W^{-1}) \otimes i_{W*}\rH^{-1}_{\cA_W}(E_W) $. 
Note that 
\begin{equation*}
i_{W*}^{}\rH^{-1}_{\cA_W}(E_W) = \rH^{-1}_{\cA_S}(i_{W*}E_W) = \rH^{-1}_{(\cA_{S})_1}(i_{W*}E_W), 
\end{equation*} 
where the first equality holds by t-exactness of $i_{W*}$ and the second by t-exactness of the inclusion 
$\cD_1 \to \cD$ and the fact that $i_{W*}E_W \in \cD_1$. Since tensoring with the line bundle $g^*(I_W^{-1})$ 
is t-exact as an endofunctor of $\cD_1$, we conclude $E_{\Ctor} \in (\cA_C)_1$. 
\end{proof}

\newpage
\part{Moduli spaces}\label{part:ModuliSpaces}

\section{Moduli of complexes} \label{sec:modComplexes}
In \cite{Lieblich:mother-of-all}, Lieblich showed that for $g \colon X \to S$ a proper, flat, finitely presented morphism of schemes, 
there is an algebraic stack parametrizing ``families'' of objects in the bounded derived categories of the fibers of $g$ with vanishing negative self-$\Ext$s. 
In fact, \cite{Lieblich:mother-of-all} handles more generally the case where $X$ and $S$ are algebraic spaces, but we will not need this. 
In this section we review Lieblich's results; we often give references to \cite{stacks-project}, which contains a slightly different exposition that we find convenient.

\subsection{Relatively perfect objects}
\label{subsection-relatively-perfect} 
The meaning of a family of objects in the bounded derived categories of the 
fibers is made precise by the notion of a relatively perfect object. 
We use the definition from 
\citestacks{0DI0}, 
but it agrees with Lieblich's original definition, see 
\citestacks{0DI9}. 
Recall from~Section~\ref{sec:setupnotation} that a pseudo-coherent complex is one that is locally quasi-isomorphic to a bounded above complex of finitely generated locally free sheaves. 

\begin{Def} 
\label{def-S-perfect}
Let $g \colon X \to S$ be a morphism of schemes which is flat and locally of finite presentation. 
Then an object $E \in \rD(X)$ is \emph{$S$-perfect} if $E$ is pseudo-coherent and 
locally of finite $\Tor$-dimension over $g^{-1} \cO_S$.
\end{Def}

\begin{Rem}
\label{Rem-finite-tor}
If $X \to S$ is a morphism of schemes, then $E$ is locally of finite $\Tor$-dimension over $g^{-1} \cO_S$ 
if and only if for any affine open $U \subset X$ mapping into an affine open $V \subset S$, the 
complex $\rR\Gamma(U, E)$ is of finite $\Tor$-dimension over $\cO_S(V)$. 
\end{Rem}

The following summarizes the relations between the notions of $S$-perfect, bounded coherent, 
and perfect complexes. 

\begin{Lem}
\label{lem-relations-S-perfect}
Let $g \colon X \to S$ be a morphism of schemes which is 
flat and locally of finite presentation. 
\begin{enumerate}[{\rm (1)}] 

\item \label{S-perfect-db} 
Assume $X$ is quasi-compact. 
If $E \in \rD(X)$ is $S$-perfect, then $E \in \Db(X)$. 

\item \label{db-S-perfect}
Assume $S$ is regular of finite Krull dimension. 
If $E \in \Db(X)$, then $E$ is $S$-perfect. 

\item \label{perfect-S-perfect}
If $E \in \Dperf(X)$, then $E$ is $S$-perfect. 

\end{enumerate}
\end{Lem}

\begin{proof}
Recall from Section~\ref{sec:setupnotation} that $\Db(X)$ is defined as the category of pseudo-coherent 
complexes with bounded cohomology. 
Hence in the situation of~\eqref{S-perfect-db}, it suffices to show that $E$ has bounded cohomology. 
Since $X$ is quasi-compact this can be checked locally, where it holds by Remark~\ref{Rem-finite-tor}. 

In the situation of~\eqref{db-S-perfect}, 
to check $E$ is locally of finite $\Tor$-dimension over $g^{-1} \cO_S$ 
we may assume that $X = \Spec(A)$ and $S = \Spec(R)$. Then we must show 
$\textrm{R}\Gamma(X,E) \in \Db(A)$ has finite $\Tor$-dimension over $R$. 
But by regularity $R$ has finite global dimension \citestacks{00OE}, 
so we conclude by \citestacks{066P}. 

Finally, for \eqref{perfect-S-perfect} note that $E \in \rD(X)$ is 
perfect if and only if $E$ is pseudo-coherent and locally has finite $\Tor$-dimension 
\citestacks{08CQ}. 
Since $g \colon X \to S$ is flat, $E$ is then also 
locally of finite $\Tor$-dimension over $g^{-1} \cO_S$.
\end{proof} 

\begin{Lem}[{\citestacks{0DI5}}]
\label{lemma-S-perfect-bc} 
Let $g \colon X \to S$ be a morphism of schemes which is flat and locally of 
finite presentation. 
Let $\phi \colon T \to S$ be a morphism of schemes. 
If $E \in \rD(X)$ is $S$-perfect, then $\res{E}{T} \in \rD(X_T)$ is $T$-perfect. 
\end{Lem}

In particular, if $E \in \rD(X)$ 
is $S$-perfect, and we let $T = \Spec(\kappa(s))$ for a point $s\in S$, then Lemma~\ref{lem-relations-S-perfect} and Lemma~\ref{lemma-S-perfect-bc} show that the restriction $E_s$ of $E$ to the fiber of $g$ over $s$ is a bounded coherent complex. 
Therefore we may think of an $S$-perfect object as a family of bounded coherent complexes.
However, these restrictions $E_s$  are not necessarily perfect complexes, so one should not mistake ``$S$-perfect'' to mean a family of perfect complexes.

\subsection{Moduli of objects on a proper morphism} 
\label{subsection-lieblich-moduli} 
We will consider relatively perfect objects satisfying the following condition introduced in \cite[Definition~2.1.8, Proposition~2.1.9]{Lieblich:mother-of-all}.

\begin{Def}
Let $g \colon X \to S$ be a flat, proper, finitely presented morphism of schemes. 
An $S$-perfect object $E \in \rD(X)$ is \emph{universally gluable} if for every point $s \in S$ 
we have 
\index{Dpug@$\Dpug(X/S)$, category of universally gluable $S$-perfect objects}
\begin{equation*}
\Ext^i(E_s, E_s) = 0 \quad \text{ for } i < 0. 
\end{equation*}
We denote by $\Dpug(X/S) \subset \rD(X)$ the full subcategory 
of universally gluable $S$-perfect objects. 
\end{Def}

Now we can define the moduli stack of interest. 
In this paper, we will regard stacks as groupoid-valued (pseudo)functors (instead of fibered categories), but otherwise we follow the conventions of~\cite{stacks-project}. 
Let $(\Sch/S)$ denote the category of all $S$-schemes, and let $\Gpds$ denote the category of groupoids. 

\begin{Def}
Let $g \colon X \to S$ be a flat, proper, finitely presented morphism of schemes. 
We denote by 
\begin{equation*}
\cMpug(X/S) \colon (\Sch/S)^{\op} \to \Gpds 
\end{equation*} 
\index{Mpug@$\cMpug(X/S)$, functor of universally gluable $T$-perfect objects in $\Dpug(X_T/T)$}
the functor whose value on $T \in (\Sch/S)$ consists of 
all $E \in \Dpug(X_T/ T)$. 
On morphisms, $\cMpug(X/S)$ is given by pullback. 
\end{Def}

The main theorem of~\cite{Lieblich:mother-of-all} is as follows. 

\begin{Thm}[{\cite[Theorem~4.2.1]{Lieblich:mother-of-all}}]
\label{theorem-lieblich-moduli} 
Let $g \colon X \to S$ be a flat, proper, finitely presented morphism of schemes. 
Then $\cMpug(X/S)$ is an algebraic stack locally of finite presentation and locally quasi-separated over $S$, with separated diagonal. 
\end{Thm}
 
Later we will also need an auxiliary result on representability of $\Hom$ functors. 

\begin{Def}
Let $g \colon X \to S$ be a morphism of schemes. 
For $E, F \in \rD(X)$, we denote by 
\begin{equation*}
\uHom_S(E, F) \colon (\Sch/S)^{\op} \to \Sets 
\end{equation*}
the functor given by $T \mapsto \Hom_{\rD(X_T)}(\res{E}{T}, \res{F}{T})$. 
\end{Def}

The following criterion can be proved using the methods of~\cite{Lieblich:mother-of-all}. 

\begin{Lem}
\label{lemma-hom-representable}
Let $g \colon X \to S$ be a flat, proper, finitely presented morphism of schemes. 
Let $E, F \in \rD(X)$. 
Assume $E$ is pseudo-coherent, $F$ is $S$-perfect.
Then there exist a scheme $B$, which is affine and of finite presentation over~$S$, and $\xi\in \Hom_{\rD(X_B)}(\res{E}{B}, \res{F}{B})$ such that, for every geometric point $\bar{s}$ over $S$ and $\varphi \in \Hom_{\rD(X_{\bar{s}})}(\res{E}{\bar{s}}, \res{F}{\bar{s}})$, there exists a $\kappa(\bar{s})$-rational point $b$ of $B \times_{S} \Spec(\kappa(\bar{s}))$ such that $\xi_b=\varphi$.

Furthermore, if for every point $s$ of $S$ we have 
\begin{equation*}
\Ext^i(\res{E}{s}, \res{F}{s}) = 0 \quad \text{ for } i < 0,
\end{equation*}
then $\uHom_S(E, F)$ is representable by a scheme which is affine and of finite presentation over~$S$. 
\end{Lem}

\begin{proof}
The second part of the statement is \citestacks{0DLC}.
To prove the first statement we use a similar argument as in \emph{loc.~cit.}, which we now briefly recall.

We can reduce to the case where $S$ is affine. 
By \citestacks{0DKY}, there exists $L\in \Db(S)$ such that for every geometric point $\bar{s}$ over $S$ and for all $i$ we have that
\[
\Ext^i(\res{E}{\bar{s}}, \res{F}{\bar{s}})=\Ext^i(\res{L}{\bar{s}}, \kappa(\bar{s})).
\]
We can now consider the exact triangle given by the ``stupid'' truncation
\[
\sigma_{>0}L\to L\to\sigma_{\leqslant 0}L
\]
where, we can assume that $\sigma_{>0}L$ is a perfect complex sitting in strictly positive degrees.
In particular, $\Hom((\res{\sigma_{>0}L)}{\bar{s}}, \kappa(\bar{s}))=0$.
Hence we get a surjective morphism
\[
\Hom(\res{(\sigma_{\leqslant 0}L)}{\bar{s}}, \kappa(\bar{s}))\twoheadrightarrow\Hom(\res{L}{\bar{s}}, \kappa(\bar{s})).
\]
Now we can apply the argument in the proof of \citestacks{0DLC} to the complex $\sigma_{\leqslant 0}L$.
\end{proof}

\section{Moduli of objects in a subcategory}
\label{sec-moduli-objects-D}
In this section we explain how the results from Section~\ref{subsection-lieblich-moduli} extend to moduli of objects in an admissible subcategory, and prove some general results on boundedness of sub-moduli problems. 
We consider the following situation: 
\begin{itemize}
\item $g \colon X \to S$ is a flat, proper, finitely presented morphism of 
{schemes which are quasi-compact with affine diagonal, where 
$X$ is noetherian of finite Krull dimension.}
\item $\cD \subset \Db(X)$ is an $S$-linear strong semiorthogonal component whose projection functor is of finite cohomological amplitude.
\end{itemize} 

\subsection{Moduli of objects in \texorpdfstring{$\cD$}{D}} 
\label{subsection-moduli-admissible-subcat} 

\begin{Def}
\label{definition-cMpug(D/S)} 
We denote by 
\index{Mpug@$\cMpug(\cD/S)$, functor of universally gluable $T$-perfect objects in $\Dpug(X_T/T)$ such that $E_t\in\cD_t$ for all $t\in T$}
\begin{equation*}
\cMpug(\cD/S) \colon (\Sch/S)^{\op} \to \Gpds 
\end{equation*}
the functor whose value on $T \in (\Sch/S)$ consists of all $E \in \Dpug(X_T/ T)$ such that $E_t \in \cD_t$ for all $t \in T$. 
\end{Def}
We prove the following result in this section.
\begin{Prop}
\label{proposition-cMpug(D/S)-algebraic} 
The functor $\cMpug(\cD/S)$ is an algebraic stack locally of finite presentation over~$S$, and the canonical morphism $\cMpug(\cD/S) \to \cMpug(X/S)$ is an open immersion. 
\end{Prop} 

The key to showing that $\cMpug(\cD/S) \to \cMpug(X/S)$ is an open immersion is the following lemma.

\begin{Lem}
\label{lemma-D-Ds}
Let $T \to S$ be a morphism from a quasi-compact scheme with affine diagonal, and let $E \in \rD(X_T)$ be 
$T$-perfect. 
\begin{enumerate}[{\rm(1)}] 
\item \label{D-iff-Ds}
$E \in \cD_T$ if and only if $\res{E}{t} \in \cD_t$ for all 
$t \in T$. 

\item \label{Ds-locus-open} 
The set 
\begin{equation*}
\{ t \in T \sth \res{E}{t} \in \cD_t \} 
\end{equation*}
is open in $T$. 
\end{enumerate}
\end{Lem}

\begin{proof}
Note that since $X_T \to T$ is quasi-compact and $T$ is quasi-compact, the scheme $X_T$ is quasi-compact. 
Thus by Lemma~\ref{lem-relations-S-perfect}.\eqref{S-perfect-db} we have $E \in \Db(X_T)$. 
Decomposing $E$ with respect to the semiorthogonal decomposition 
\begin{equation*}
\Db(X_T) = \langle \cD_T, ({^\perp}\cD)_T \rangle 
\end{equation*}
gives an exact triangle 
\begin{equation*}
G \to E \to F
\end{equation*} 
where $F \in \cD_T$ and $G \in ({^\perp}\cD)_T$.
By Theorem~\ref{theorem-bc-sod}, for any $t \in T$ the triangle 
\begin{equation*}
 \res{G}{t} \to \res{E}{t} \to \res{F}{t}
\end{equation*}
obtained by restriction gives the decomposition of $\res{E}{t}$ with respect to the 
semiorthogonal decomposition 
\begin{equation*}
\Dqc(X_t) = \langle \cD_{t, \qc}, {^\perp}\cD_{t, \qc} \rangle. 
\end{equation*} 
In the above terms, we have $E \in \cD_T$ if and only if $G = 0$ and 
$\res{E}{t} \in \cD_t$ if and only if $\res{G}{t} = 0$. 
Hence part~\eqref{D-iff-Ds} holds since $G = 0$ if and only if 
$\res{G}{t} = 0$ for all $t \in T$, and part~\eqref{Ds-locus-open} holds 
since the set of $t \in T$ where $\res{G}{t} = 0$ is open. 
Indeed, both statements follow easily from Nakayama's Lemma and from the fact \citestacks{064U} that the maximal cohomology sheaf of a pseudo-coherent complex is a quasi-coherent sheaf of finite type.
\end{proof}

With this preparation, we can prove the proposition.

\begin{proof}[Proof of Proposition~\ref{proposition-cMpug(D/S)-algebraic}]
We show that the functor $\cMpug(\cD/S) \colon (\Sch/S)^{\op} \to \Gpds$ is a stack in the fppf topology and that the canonical morphism $\cMpug(\cD/S) \to \cMpug(X/S)$ is representable by 
open immersions.  The result then follows from the fact that $\cMpug(X/S)$ is an algebraic stack locally of finite presentation 
over $S$ by Theorem~\ref{theorem-lieblich-moduli}. 

As $\cMpug(\cD/S)$ is a subfunctor of the stack $\cMpug(X/S)$, it suffices to show that for $T \in (\Sch/S)$ and $E \in \Dpug(X_T/T)$, 
the condition $E \in \cMpug(\cD/S)(T)$ can be checked fppf locally on $T$.  To see that this is true, let $k \subset \ell$ be a field extension for field $k$ with a morphism $\Spec(k) \to S$.  
Denote by  $\cD_{k}$ and $\cD_{\ell}$ the base changes of $\cD$ along 
$\Spec(k) \to S$ and the induced map $\Spec(\ell) \to S$, respectively.    
Then for an $S$-perfect $E\in\rD(X)$, we claim that $E_k \in \cD_k$ if and only if $E_{\ell} \in \cD_{\ell}$. 
Indeed, by the same argument as in the proof of Lemma~\ref{lemma-D-Ds}, this boils 
down to the fact that the functor $\Db(X_k) \to \Db(X_{\ell})$ is conservative. It follows that the condition $E \in \cMpug(\cD/S)(T)$ can be checked fppf locally on $T$,
and therefore that $\cMpug(\cD/S)$ is stack.

It remains to show that $\cMpug(\cD/S) \to \cMpug(X/S)$ is representable by open immersions.
So let $T \in (\Sch/S)$ with a morphism $T \to \cMpug(X/S)$.
Since $\cMpug(\cD/S)$ is a stack, checking that second projection $\cMpug(\cD/S) \times_{\cMpug(X/S)} T\to T$ is an 
open immersion reduces to the case where $T$ is affine.
Let $E \in \Dpug(X_T/T)$ be the object corresponding to $T \to \cMpug(X/S)$. 
Then the fiber product $\cMpug(\cD/S) \times_{\cMpug(X/S)} T$ with its second projection is represented by the set 
\begin{equation*}
U = \{ t \in T \sth \res{E}{t} \in \cD_t \} 
\end{equation*}
with its inclusion into $T$.  This is an open immersion by Lemma~\ref{lemma-D-Ds}, so we are done. 
\end{proof}

\subsection{Boundedness} 
One of our ultimate goals in this paper is to construct well-behaved substacks of $\cMpug(\cD/S)$ in the presence of a ``stability condition on $\cD$ over $S$'', a notion that will be studied in Part~\ref{part:higher-dimensional-bases}. 
In particular, we will be interested in moduli functors which are bounded, in the following sense. 
 
\begin{Def}
\label{def-bounded}
A subfunctor $\cM \subset \cMpug(\cD/S)$ is \emph{bounded} if there exists a pair $(B, \cE)$ where $B$ is a scheme of finite type over $S$ and $\cE \in \cM(B)$ is an object such that for every geometric point $\bar{s}$ over $S$ and $E \in \cM(\kappa(\bar{s}))$, there exists a $\kappa(\bar{s})$-rational point $b$ of $B \times_{S} \Spec(\kappa(\bar{s}))$ such that $\cE_b \cong E$. 
\end{Def} 

If $\cM \subset \cMpug(\cD/S)$ is an open substack, then following \cite{Toda:K3}, boundedness can be phrased in terms of 
the intrinsic geometry of $\cM$ as follows. 

\begin{Lem}
\label{lem-open-bounded}
Let $\cM$ be an open and bounded substack of $\cMpug(\cD/S)$. 
Then $\cM$ is an algebraic stack of finite type over $S$. 
\end{Lem} 

\begin{proof}
Since $\cMpug(\cD/S)$ is an algebraic stack locally of finite presentation over $S$ by Proposition~\ref{proposition-cMpug(D/S)-algebraic}, the same holds for $\cM$. 

Now consider $B$ as in Definition~\ref{def-bounded}.
Then $B \to \cM$ is surjective by \citestacks{04ZR}, and so $\cM \to S$ is quasi-compact by \citestacks{050X}.
Hence it is of finite type. 
\end{proof} 

The following observation will be needed later. 

\begin{Lem}\label{lem:ExtensionOfBounded}
Assume $S$ is noetherian. 
Let $\cM \subset \cMpug(\cD/S)$ be an open substack. 
Assume there exist bounded subfunctors $\cM_1$ and $\cM_2$ of $\cMpug(\cD/S)$ 
such that for every geometric point $\bar{s}$ over $S$ and every $E \in \cM(\kappa(\bar{s}))$, 
there is an exact triangle 
\begin{equation*}
E_1 \to E \to E_2
\end{equation*}
with $E_1 \in \cM_1(\kappa(\bar{s}))$ and $E_2 \in \cM_2(\kappa(\bar{s}))$.
Then $\cM$ is bounded.
\end{Lem}

\begin{proof}
For $i=1,2$, let $(B_i, \cE_i)$ be a pair as in Definition~\ref{def-bounded} witnessing the boundedness of $\cM_i$. 
Let $B_{12} = B_1 \times_S B_2$, and let $(\cE_{i})_{B_{12}} \in \Dpug(X_{B_{12}}/B_{12})$ denote the pullback of $\cE_{i} \in \Dpug(X_{B_i}/B_{i})$. 
By Lemma~\ref{lemma-hom-representable}, the functor 
\begin{equation*}
\uHom_{B_{12}}((\cE_{2})_{B_{12}}, (\cE_{1})_{B_{12}}[1]) \colon 
(\Sch/B_{12})^{\op} \to \Sets 
\end{equation*}
is covered by a scheme $Z$ which is affine and of finite presentation over $B_{12}$. 
By composition with the morphism $B_{12} \to S$, we think of $Z$ as a scheme over $S$. 
Note that since each $B_i \to S$ is of finite type, so is $B_{12} \to S$, and hence 
so is $Z \to S$. 

By definition, there is a universal morphism 
\begin{equation*}
\alpha \in \Hom_{\rD(X_Z)}((\cE_2)_Z, (\cE_1)_Z[1]),
\end{equation*}
where $(\cE_i)_Z$ denotes the pullback of $\cE_i$ along the base change $Z \to B_{12} \to B_i$.
Define $\cE$ as the $[-1]$-shifted cone of $\alpha$, so that there is an exact 
triangle 
\begin{equation*}
(\cE_1)_Z \to \cE \to (\cE_2)_Z. 
\end{equation*}
Note that since $(\cE_1)_Z$ and $(\cE_2)_Z$ are $Z$-perfect objects of $\rD(X_Z)$, so is $\cE$ by \citestacks{0DI3}. 
Hence by \citestacks{0DLC} the locus in $Z$ where $\cE$ is universally gluable is an open subscheme $Z^{\circ} \subset Z$, i.e., there is an open subscheme $Z^{\circ} \subset Z$ characterized by the property that $z \in |Z^{\circ}|$ if and only if $\Ext^i(\cE_z, \cE_z) = 0$ for $i < 0$. 
Then $\cE_{Z^{\circ}} \in \Dpug(X_{Z^{\circ}}/Z^{\circ})$. 
Moreover, for all $z \in Z$ we have $\cE_z \in \cD_z$ because the analogous statement is true for $(\cE_1)_{Z}$ and $(\cE_2)_{Z}$ by Lemma~\ref{lemma-D-Ds}.
Thus $\cE_{Z^{\circ}}$ corresponds to a morphism 
\begin{equation*}
Z^{\circ} \to \cMpug(\cD/S)
\end{equation*} 
over $S$. 
Let $B \subset Z^{\circ}$ be the open subscheme given by the preimage of the open substack $\cM \subset \cMpug(\cD/S)$. 
Then the restriction $\cE_{B}$ lies in $\cM(B)$. 

We claim that the pair $(B, \cE_{B})$ witnesses the boundedness of $\cM$. 
Indeed, as observed above $Z \to S$ is of finite type. 
Since $S$ is noetherian, so is $Z$. Therefore, the open immersions $B \subset Z^{\circ} \subset Z$ are of finite type, and the same holds for the composition $B \to S$.
Further, by the assumption of the lemma and the construction of $B$, 
for every geometric point $\bar{s}$ over $S$ and $E \in \cM(\kappa(\bar{s}))$, 
there is a $\kappa(\bar{s})$-rational point $b$ of $B \times_S \Spec(\kappa(\bar{s}))$ such that $\cE_b \cong E$. 
\end{proof}

\subsection{Simple objects} 
Finally, we discuss substacks of $\cMpug(\cD/S)$ parameterizing objects which are simple in the following sense. 
\begin{Def} \label{def:simple}
An object $E \in \Dpug(X_T/T)$ is called \emph{simple} if $\Hom(E_{\overline{t}}, E_{\overline{t}}) = \kappa(\overline{t})$ holds for every 
geometric point $\overline{t}$ of $T$. 
\end{Def}

The above definition is equivalent to Lieblich's \cite[Definition~4.3.1]{Lieblich:mother-of-all}, which requires that the automorphism stack of $E$ is given by the multiplicative group over $T$. 
\index{sM@$\scM \subset \cM$, subfunctor parameterizing simple objects of $\cM$}
Given a subfunctor $\cM \subset \cMpug(\cD/S)$, we write $\scM \subset \cM$ for the subfunctor parameterizing simple objects of $\cM$. 

\begin{Lem}\label{Lem:SimpleAlgebraicSpace}
Let $\cM$ be an open substack of $\cMpug(\cD/S)$. 
\index{sM@$sM$, algebraic moduli space}
Then $\scM$ is an algebraic stack locally of finite presentation over $S$ and admits the structure of a $\mathbb{G}_m$-gerbe over an algebraic space $sM$ locally of finite presentation over $S$. If, moreover, $s\cM$ is bounded, then $\scM$ and $sM$ are of finite type over $S$. 
\end{Lem}

\begin{proof}
The simple objects $s\cMpug(\cD/S)$ form an open substack of $\cMpug(\cD/S)$ by \cite[Lemma~4.3.2]{Lieblich:mother-of-all} 
combined with Proposition~\ref{proposition-cMpug(D/S)-algebraic}.
So $\scM$, being the intersection of $\cM$ and $s\cMpug(\cD/S)$, is an open substack of $\cMpug(\cD/S)$ and 
therefore locally of finite presentation over $S$. 
Now as in \cite[Corollary~4.3.3]{Lieblich:mother-of-all}, we obtain $sM$ as the $\mathbb{G}_m$-rigidification 
of $\scM$. 
If $\scM$ is bounded, then it is of finite type over $S$ by Lemma~\ref{lem-open-bounded}, hence 
$sM$ is also of finite type. 
\end{proof}

\section{Fiberwise t-structures}

In this section, we consider the following situation: 
\begin{itemize}
\item $g \colon X \to S$ is a flat, finitely presented morphism of 
{schemes which are quasi-compact with affine diagonal, where 
$X$ is noetherian of finite Krull dimension.} 
\item $\cD \subset \Db(X)$ is an $S$-linear strong semiorthogonal component whose projection 
is of finite cohomological amplitude. 
\end{itemize} 

In Section~\ref{sec:Polishchuk} we studied the notion of an $S$-local 
t-structure on $\cD$, which is a compatible specification of 
t-structures over every quasi-compact open $U \subset S$. 
In this section, we consider the notion of a fiberwise collection of t-structures 
on $\cD$, which is just the specification of a t-structure over every point $s \in S$;
to get reasonable behavior, we focus on collections where ``openness of flatness'' holds. 
If $S$ is regular of finite Krull dimension, 
then by base change, i.e., Theorem~\ref{Thm-D-bc}, a bounded $S$-local t-structure on $\cD$ which 
is titled-noetherian locally on $S$ induces a fiberwise collection of t-structures on $\cD$. 
We study the latter, weaker notion since it is sufficient and well-suited for formulating moduli problems.

\subsection{Fiberwise collections of t-structures}

Let us start with the following.

\label{subsec:fiberwise-collec-t-str}
\begin{Def}
\label{definition-fiberwise-t-structures}
\index{tau@$\utau = (\tau_s)_{s \in S}$, fiberwise collection of t-structures on $\cD$ over $S$}
A \emph{fiberwise collection of t-structures on $\cD$ over $S$} is a collection $\utau = (\tau_s)_{s \in S}$ of t-structures on $\cD_s$ for every (closed or non-closed) point $s \in S$. 
\end{Def} 

In this setting, we have enough structure to define flat objects, 
analogously to Definition~\ref{def-flat-family-T}. 

\begin{Def}
\label{def-flat-family-utau}
Let $\utau$ be a fiberwise collection of t-structures on $\cD$ over $S$, and let $T \to S$ be a morphism. 
Let $t \in T$ and let $s \in S$ be its image. 
We denote by $\hat{\tau}_{t}$ the t-structure on $(\cD_{\qc})_t$ 
obtained via Theorem~\ref{thm-Dqc-bc} by base change of $\tau_s$ along $\Spec(\kappa(t)) \to \Spec(\kappa(s))$, 
and write $(\cA_{\qc})_t$ for the heart of $\hat{\tau}_t$. 
We say an object $E \in \Dqc(X_T)$ is $T$-\emph{flat with respect to $\utau$} if $E_t \in (\cA_{\qc})_t$ for 
every point $t \in T$. 
\end{Def} 

Note that the pullback of any $T$-flat object along a morphism $T' \to T$ is $T'$-flat, 
by t-exactness of the base change functor along field extensions (Theorem~\ref{thm-Dqc-bc}.\eqref{DT-f-flat}). 

\begin{Rem}
When they both apply, the notions of $T$-flatness from 
Definitions~\ref{def-flat-family-T} and~\ref{def-flat-family-utau} agree, and hence we may 
unambiguously use the term ``$T$-flat''.
\end{Rem}

Without any extra conditions, a fiberwise collection of t-structures is not well-behaved, 
because there is no compatibility imposed between the t-structures on different fibers. 
To remedy this, we consider the following additional conditions. 

\begin{Def}
\label{def-open-flat-utau}
Let $\utau$ be a fiberwise collection of t-structures on $\cD$ over $S$. 
Then $\utau$ \emph{satisfies openness of flatness} if for every $S$-perfect 
object $E \in \rD(X)$, the set 
\begin{equation*}
\set{ s \in S \sth \res{E}{s} \in (\cA_{\qc})_s } 
\end{equation*} 
is open. 
Similarly, $\utau$ \emph{universally satisfies openness of flatness} if for every $T \to S$ and 
every $T$-perfect object $E \in \rD(X_T)$, the set 
\begin{equation*}
\set{ t \in T \sth \res{E}{t} \in (\cA_{\qc})_t } 
\end{equation*} 
is open. 
\end{Def}

\begin{Rem}
In Definition~\ref{def-open-flat-utau} we only consider relatively perfect objects, since 
these are the only type of objects we shall need to consider in our discussion of moduli problems. 
\end{Rem}
In the next result, we show that in Definition~\ref{def-open-flat-utau} it suffices to check openness for affine schemes $T$ finitely presented over $S$.
\begin{Lem} \label{lem:universalopennessfromfiniteopenness}
Let $\utau$ be a fiberwise collection of t-structures on $\cD$ over $S$. 
Assume that for every morphism 
of finite presentation $T \to S$ from an affine scheme $T$
and every $T$-perfect object $E \in \rD(X_T)$, the set 
\begin{equation*}
\set{ t \in T \sth \res{E}{t} \in (\cA_{\qc})_t } 
\end{equation*} 
is open. 
Then $\utau$ universally satisfies openness of flatness. 
\end{Lem} 

\begin{proof}
For $T \to S$ an arbitrary morphism and $E \in \rD(X_T)$ a $T$-perfect object we must show 
the set $\set{ t \in T \sth \res{E}{t} \in (\cA_{\qc})_t }$ is open, 
assuming this holds for $T \to S$ of finite presentation with $T$ affine. 
We may immediately reduce to the case where $T$ is affine. 
In this case, by Lemma~\ref{lemma-limit-affine-fp} below we can write $T = \lim T_i$ 
as a cofiltered limit of affine schemes $T_i, i \in I$, which are of finite presentation over $S$. 
Assume $t \in T$ is a point such that $\res{E}{t} \in (\cA_{\qc})_t$. 
Then we must show there is an open $U \subset T$ containing $t$ such that 
for every $u \in U$ we have $\res{E}{u} \in (\cA_{\qc})_u$. 

By Lemma~\ref{lemma-D-Ds} we may assume $E \in \cD_T$. 
The object $E$ descends to a $T_0$-perfect object $E_0 \in \rD(X_{T_0})$ for 
some index $0 \in I$, see \citestacks{0DI8}. 
We have $E_0 \in \Db(X_{T_0})$ by Lemma~\ref{lem-relations-S-perfect}.\eqref{S-perfect-db}, hence we may assume $E_0 \in \cD_{T_0}$ by replacing $E_0$ with its projection into $\cD_{T_0}$. 
If $t_0$ denotes the image of $t$ under the projection $T \to T_0$, then 
the assumption $\res{E}{t} \in (\cA_{\qc})_t$ together with Theorem~\ref{thm-Dqc-bc}.\eqref{DT-Ui} shows that $(E_0)_{t_0} \in (\cA_{\qc})_{t_0}$. 
Since $T_0 \to S$ is a finitely presented morphism from an affine scheme, by assumption this means there is an open subset $U_0 \subset T_0$ containing $t_0$ such that for every $u_0 \in U_0$ we have $(E_0)_{{u_0}} \in \cA_{u_0}$. 
Then the preimage $U \subset T$ of $U_0$ under the projection $T \to T_0$ 
is the sought-for neighborhood of $t \in T$. 
Indeed, $\res{E}{u} \in (\cA_{\qc})_u$ holds for every $u \in U$ by Theorem~\ref{thm-Dqc-bc}.\eqref{DT-f-flat}. 
\end{proof}

The following lemma, invoked above, says there are ``enough'' affine schemes of finite presentation over 
a quasi-separated scheme. 
\begin{Lem}
\label{lemma-limit-affine-fp} 
Let $T \to S$ be a morphism of schemes with $T$ affine. 
Then $T \cong \lim T_i$ is a cofiltered limit of affine schemes $T_i$ 
which are of finite presentation over $S$. 
\end{Lem}

\begin{proof}
By \citestacks{09MV}, 
we can write $T = \lim T_i$ as a cofiltered limit of schemes $T_i$ of 
finite presentation over $S$ with affine transition morphisms.  We show that $T_i$ is affine for all $i$ large enough, which gives the result by omitting the first few terms.
The scheme $T_i$ is quasi-separated, being of finite presentation 
over the quasi-separated scheme $S$. 
Moreover, as $T$ is affine and thus quasi-compact, it follows from \citestacks{0CUF} and \citestacks{01YX} 
that after shrinking $T_i$ we can assume 
$T_i$ is quasi-compact for all $i$.  But then $T_i$ is affine for all $i$ large enough 
by \citestacks{01Z6}. 
\end{proof}

The following result will be useful later in our construction of Quot spaces. 
\begin{Lem}
\label{lem-openness-surjectivity}
Let $\utau$ be a fiberwise collection of t-structures on $\cD$ over $S$ which 
universally satisfies openness of flatness. 
Let $T \to S$ be a morphism from a quasi-compact scheme with affine diagonal. 
Let $E \to F$ be a morphism of $T$-perfect and $T$-flat objects in $\rD(X_T)$. 
Then the set 
\begin{equation*}
\set{ t \in T \sth \res{E}{t} \to \res{F}{t} \textup{ is surjective 
in } (\cA_{\qc})_t } 
\end{equation*}
is open. 
\end{Lem}

\begin{proof}
Let $G \in \rD(X_T)$ be the object determined by the exact triangle 
\begin{equation*}
G \to E \to F . 
\end{equation*} 
Note that $G$ is $T$-perfect because $E$ and $F$ are 
\citestacks{0DI3}, and moreover $G\in\cD_{T,\qc}$ by considering its decomposition with respect to the semiorthogonal decomposition $$\Dqc(X_T)=\langle \cD_{T,\qc}, ({^\perp}\cD)_{T,\qc} \rangle .$$ 
For $t \in T$ we consider the exact triangle 
\begin{equation*}
\res{G}{t} \to \res{E}{t} \to \res{F}{t}
\end{equation*} 
obtained by restriction. 
By $T$-flatness we have $\res{E}{t}, \res{F}{t} \in (\cA_{\qc})_t$, 
so the long exact cohomology sequence shows 
that $\res{G}{t} \in \cD_{t}^{[0,1]}$, and $\res{E}{t} \to \res{F}{t}$ is 
surjective if and only if $\res{G}{t} \in (\cA_{\qc})_t$. 
Thus the result follows from the assumption that $\utau$ universally 
satisfies openness of flatness. 
\end{proof}

\begin{Rem}\label{rmk:TiltedNoetherianInduced}
When $\tau_s$ is a tilted-noetherian t-structure for every $s \in S$, then the t-structure on $(\cD_{\qc})_t$ descends to one on $\cD_t$ by Proposition~\ref{Prop-D-bc-fields}; therefore, for a $T$-perfect object $E$ we can replace the condition $E_t \in (\cA_{\qc})_t$ by $E_t \in \cA_t$ everywhere.
\end{Rem}

%%%%%%%%%%%%%%%%%%%%%%

\subsection{Integrable collections of t-structures} 

In this subsection, we compare the notion of a fiberwise collection of t-structures 
to the notion of a local t-structure from Section~\ref{sec:Polishchuk}. 

\begin{Def}\label{def:integrateststructure20210217}
\index{Integrate@integrable t-structure over a scheme}
Let $\utau$ be a fiberwise collection of t-structures on $\cD$ over $S$, and let $T \to S$ be a morphism 
from a {scheme $T$ which is quasi-compact with affine diagonal.} 
We say $\utau$ is \emph{integrable over $T$} if there exists a t-structure 
$\tau_T$ on $\cD_T$ such 
that for every $t \in T$ the t-structure on $(\cD_{\qc})_t$ induced via Theorem~\ref{thm-Dqc-bc} by 
base change along the composition $\Spec \kappa(t) \to T$ agrees with the t-structure 
$\hat{\tau}_t$ from Definition~\ref{def-flat-family-utau}. 
In this situation, we say $\utau$ \emph{integrates over $T$} to the t-structure $\tau_T$ . 
\end{Def} 

The question of whether a given fiberwise collection of t-structures integrates over a 
scheme $T \to S$ is subtle. We will be particularly interested in cases 
where this holds for $T$ a Dedekind scheme. 
The following two results in this context will be needed later. 

\begin{Lem}
\label{lem-AC-C-flat}
Let $\utau$ be a fiberwise collection of t-structures on $\cD$ over $S$. 
Let $C \to S$ be a morphism from a Dedekind scheme $C$. 
Assume $\utau$ integrates over $C$ to a $C$-local t-structure on $\cD_C$ with 
heart $\cA_C$. 
Then for an object $E \in \rD(X_C)$, the following conditions are equivalent: 
\begin{enumerate}[{\rm (1)}] 
\item \label{E-heart-flat-C} $E \in \cA_C$ and $E$ is $C$-flat. 
\item \label{E-perf-flat-C} $E$ is $C$-perfect and $C$-flat. 
\end{enumerate}
\end{Lem}

\begin{proof}
If $E \in \cA_C$ then $E$ is $C$-perfect by Lemma~\ref{lem-relations-S-perfect}.\eqref{db-S-perfect}, 
so~\eqref{E-heart-flat-C} implies~\eqref{E-perf-flat-C}. 
Conversely, assume $E$ is $C$-perfect and $C$-flat. 
Then in particular $E_c \in \cD_c$ for every $c \in C$, 
so by Lemma~\ref{lemma-D-Ds}.\eqref{D-iff-Ds} we find that 
$E \in \cD_C$. 
(Strictly speaking in Section~\ref{sec-moduli-objects-D}, and thus tacitly in Lemma~\ref{lemma-D-Ds}, 
the morphism $X \to S$ is assumed to be proper, but this assumption 
is not used in the proof.) 
Thus $E \in \cA_C$ by Lemma~\ref{lem:flatinheart}. 
\end{proof}

\begin{Lem}
\label{lem-E-F-surjective}
Let $\utau$ be a fiberwise collection of t-structures on $\cD$ over $S$. 
Assume $\utau$ integrates over $C$ to a $C$-local t-structure on $\cD_C$ whose heart $\cA_C$ 
has a $C$-torsion theory. 
Then for a morphism $E \to F$ in $\cA_C$, the following conditions are equivalent: 
\begin{enumerate}
\item \label{E-F-surjective-AC}
$E \to F$ is surjective as a morphism in $\cA_C$. 
\item \label{E-F-surjective-fiberwise} 
For every point $c \in C$ the induced morphism 
 $\rH^0_{(\cA_{\qc})_c}(E_c) \to \rH^0_{(\cA_{\qc})_c}(F_c)$ is 
surjective in $(\cA_{\qc})_c$, where $(\cA_{\qc})_c \subset (\cD_{\qc})_c$ is the heart of the t-structure $\hat{\tau}_c$ 
from Definition~\ref{def-flat-family-utau}. 
\end{enumerate}
\end{Lem}

\begin{proof}
Let $G \in \cA_C$ be defined by the right exact sequence 
\begin{equation*}
E \to F \to G \to 0
\end{equation*} 
in $\cA_C$, so that $E \to F$ is surjective if and only if $G = 0$. 
But by Proposition~\ref{prop:Ctorsiontheory-opennessflatness}.\eqref{enum:Nakayama} we have $G = 0$ if and only if for every point $c \in C$ we have $\rH_{\cA_c}^0(G_c) = 0$, 
which by right t-exactness of the restriction functor 
$\cD \to \cD_{c}$ is equivalent to condition~\eqref{E-F-surjective-fiberwise}. 
\end{proof}

\section{Quot spaces} \label{sec:quotspaces}
In this section, we work in the following setup: 
\begin{itemize}
\item $g \colon X \to S$ is a flat, proper, finitely presented morphism of 
{schemes which are quasi-compact with affine diagonal, where 
$X$ is noetherian of finite Krull dimension.} 
\item $\cD \subset \Db(X)$ is an $S$-linear strong semiorthogonal component whose projection functor is of finite cohomological amplitude. 
\item $\utau$ is a fiberwise collection of t-structures on $\cD$ over $S$ which universally satisfies openness of flatness. 
\end{itemize} 
Our goal is to define Quot functors in this setting, and to show that under good conditions 
they are algebraic spaces that satisfy valuative criteria. 
We begin by introducing the moduli stack of flat objects, which is used in our proof of representability.

\subsection{Moduli of flat objects} 
We have seen in Section~\ref{subsection-moduli-admissible-subcat} that there 
is an algebraic stack $\cMpug(\cD/S)$ locally of finite presentation over $S$ which 
parametrizes relatively perfect, universally gluable objects of $\cD$. 
We can use $\utau$ to cut out a substack of flat objects. 

\begin{Def}
\label{definition-cDflat} 
We denote by 
\begin{equation*}
\cM_{\utau} \colon (\Sch/S)^{\op} \to \Gpds 
\end{equation*}
\index{Mtau@$\cM_{\utau}$, functor of of all objects $E \in \Dpug(X_T/T)$ which are $T$-flat with respect to $\utau$}
the functor whose value on $T \in (\Sch/S)$ consists of all objects $E \in \Dpug(X_T/T)$ which are $T$-flat with respect to $\utau$. 
\end{Def}

Note that above it would be equivalent to require $E \in \rD(X_T)$ is $T$-perfect and $T$-flat, 
since such an object is automatically universally gluable. 

\begin{Lem}
\label{lem-moduli-flat-objects}
The functor $\cM_{\utau}$ is an algebraic stack locally of finite presentation 
over $S$, and the canonical morphism $\cM_{\utau} \to \cMpug(\cD/S)$ is an open immersion. 
\end{Lem}

\begin{proof}
The morphism $\cM_{\utau} \to \cMpug(\cD/S)$ is representable by open immersions since $\utau$ universally 
satisfies openness of flatness. 
Hence the result follows from Proposition~\ref{proposition-cMpug(D/S)-algebraic}. 
\end{proof}

\subsection{Quot spaces} 
\label{subsection-quot-definition} 

Let us start with following definition.

\begin{Def}
\label{definition-quot} 
Let $E \in \cD$ be an $S$-perfect object. 
We denote by \index{QuotSE@$\Quot_S(E)$, Quot functor}
\begin{equation*}
\Quot_S(E) \colon (\Sch/S)^{\op} \to \Sets 
\end{equation*} 
the functor whose value on $T \in (\Sch/S)$ is the set of 
all morphisms $\res{E}{T} \to Q$ in $\rD(X)$, where: 
\begin{enumerate}[{\rm (1)}] 
\item \label{quot-flatness}
$Q \in \rD(X_T)$ is $T$-perfect and $T$-flat with respect to $\utau$. 
\item \label{quot-quotient}
The morphism $\rH^0_{(\cA_{\qc})_t}(\res{E}{t}) \to \rH^0_{(\cA_{\qc})_t}(\res{Q}{t}) = \res{Q}{t}$ in $(\cA_{\qc})_t$ is surjective for all $t \in T$. 
\end{enumerate}
Given $T' \to T$ in $(\Sch/S)$, the corresponding map 
\begin{equation*} 
\Quot_S(E)(T) \to \Quot_S(E)(T')
\end{equation*}
takes 
$\res{E}{T} \to Q$ to its pullback 
$\res{E}{T'} \cong \res{(\res{E}{T})}{T'} \to \res{Q}{T'}$ along $T' \to T$. 
\end{Def} 

\begin{Rem} 
It is straightforward to verify that the morphism 
\mbox{$\res{E}{T'} \to \res{Q}{T'}$} obtained by pullback in the above definition is indeed in $\Quot_S(E)(T')$. 
\end{Rem}

\begin{Rem}
Below we focus on the case where $E$ is $S$-flat. 
Then $\Quot_S(E)(T)$ has a slightly simpler description, 
since $\rH^0_{(\cA_{\qc})_t}(\res{E}{t}) = E_t$ for any $t \in T$. 
\end{Rem}

\begin{Prop} 
\label{proposition-quot-algebraic} 
Let $E \in \cD$ be an $S$-perfect and $S$-flat object. 
Then $\Quot_S(E)$ is an algebraic space 
locally of finite presentation over $S$. 
\end{Prop} 

The algebraicity of the usual Quot functor of a coherent sheaf can be shown using algebraicity of the stack of coherent sheaves, see \cite{lieblich:remarks-coherent-algebras}, \citestacks{09TQ}. 
We follow a similar strategy to prove Proposition~\ref{proposition-quot-algebraic}, where the role of the stack of coherent sheaves is replaced by the stack $\cM_{\utau}$ of flat objects. 
Towards this, we first prove the following. 

\begin{Lem} 
\label{lemma-quot-sheaf} 
Let $E \in \cD$ be an $S$-perfect and $S$-flat object. 
Then $\Quot_S(E)$ is a sheaf in the 
fppf topology on $(\Sch/S)$. 
\end{Lem}

\begin{proof}
Let $T_i \to T$ be an fppf cover in $(\Sch/S)$, 
and let $q_i \colon \res{E}{T_i} \to Q_i$ be an object of $\Quot_S(E)(T_i)$ such that 
for every $i,j$ the restrictions of $q_i$ and $q_j$ to $T_i \times_T T_j$ agree. 
By Theorem~\ref{theorem-lieblich-moduli} and Lemma~\ref{lemma-hom-representable}, these morphisms glue uniquely to a morphism $q \colon \res{E}{T} \to Q$ 
in $\Dpug(X_T/T)$. 
To finish we must show $q \colon \res{E}{T} \to Q$ is in $\Quot_S(E)(T)$, 
i.e., for every $t \in T$ we have $\res{Q}{t} \in (\cA_{\qc})_t$ and the morphism 
$\res{E}{t} \to \res{Q}{t}$ in $(\cA_{\qc})_t$ is surjective. 
But base change along an extension of fields is t-exact by Theorem~\ref{thm-Dqc-bc}.\eqref{DT-f-flat} and conservative, so these 
conditions can be checked fppf locally, see Remark~\ref{rem:conservativeandexact}.
\end{proof}

\begin{proof}[Proof of Proposition~\ref{proposition-quot-algebraic}] 
There is a canonical morphism 
\begin{equation*}
\Quot_S(E) \to \cM_{\utau}
\end{equation*} 
which for $T \in (\Sch/S)$ sends $(\res{E}{T} \to Q) \in \Quot_S(E)(T)$ to $Q \in \cM_{\utau}(T)$. 
It suffices to show $\Quot_S(E) \to \cM_{\utau}$ is representable by algebraic spaces and locally of finite presentation. 
Indeed, then since $\cM_{\utau}$ is algebraic stack locally of finite presentation over $S$ by Lemma~\ref{lem-moduli-flat-objects}, we conclude that $\Quot_S(E)$ is an algebraic space locally of finite presentation over $S$ by \citestacks{05UM} and \citestacks{04SZ}. 

So let $T \in (\Sch/S)$, let $T \to \cM_{\utau}$ be a morphism, 
and let 
\begin{equation*}
Z = \Quot_S(E) \times_{\cM_{\utau}} T \colon (\Sch/T)^\op \to \Sets 
\end{equation*} 
be the fiber product; we must show $Z \to T$ is a locally finitely presented morphism 
of algebraic spaces. 
Let $Q \in \cM_{\utau}(T)$ be the object corresponding to $T \to \cM_{\utau}$. 
For $T' \in (\Sch/T)$ we have
\begin{equation*}
Z(T') = 
\set{
\begin{array}{c}
\textup{morphisms } \res{E}{T'} \to \res{Q}{T'} \textup{ in } \rD(X_{T'}) \textup{ such that } \\ 
\res{E}{t'} \to \res{Q}{t'} 
\textup{ is surjective in } (\cA_{\qc})_{t'} \textup{ for all } t' \in T'
\end{array} . 
}
\end{equation*}
Hence $Z \to T$ factors through the forgetful morphism $Z \to \uHom_T(\res{E}{T}, Q)$. 
Note that $\uHom_T(\res{E}{T}, Q)$ is an algebraic space of 
finite presentation by Lemma~\ref{lemma-hom-representable}, whose 
hypotheses are satisfied since for every $t \in T$ we have $\res{E}{t}, \res{Q}{t} \in (\cA_{\qc})_t$. 
Hence the following claim will finish the proof: $Z \to \uHom_T(\res{E}{T}, Q)$ is an 
open immersion. 

To prove the claim, let $T' \in (\Sch/T)$, let $T' \to \uHom_T(\res{E}{T}, Q)$
be a morphism, and let 
\begin{equation*}
Y = Z \times_{\uHom_T(\res{E}{T}, Q)} T' \to T' 
\end{equation*}
be the fiber product; we must show $Y \to T'$ is an open immersion. 
Since $\Quot_S(E)$ is a sheaf in the fppf topology by Lemma~\ref{lemma-quot-sheaf},
it follows that $Y$ is too. 
Hence we may reduce to the case where $T'$ is affine. 
Let $f \colon \res{E}{T'} \to \res{Q}{T'}$ be the morphism corresponding to 
$T' \to \uHom_T(\res{E}{T}, Q)$. 
By Lemma~\ref{lem-openness-surjectivity} the subset 
\begin{equation*}
U = \set{ 
t' \in T' \sth 
f_{t'} \colon \res{E}{t'} \to \res{Q}{t'} \text{ is surjective}} 
\end{equation*}
is open in $T'$. The morphism $U \to T'$ represents $Y \to T'$, 
so we are done. 
\end{proof}

\subsection{Valuative criteria for Quot spaces} 

In this subsection, we investigate the valuative criterion for Quot spaces. 
First we formulate what we mean by valuative criteria; for later use in the paper, we consider the setting of algebraic stacks. 

\begin{Def}
\label{Def-vc}
Let $f \colon X \to Y$ be a morphism of algebraic stacks. 
Let $R$ be a valuation ring with field of fractions $K$, and let $\Spec(R) \to Y$ be a morphism.
Then we say $f$ \emph{satisfies the strong existence part of the valuative criterion 
with respect to $\Spec(R) \to Y$} if given any commutative solid diagram 
\begin{equation}
\label{vc-diagram} 
\vcenter{
\xymatrix{
\Spec(K) \ar[r] \ar[d] & X \ar[d] \\
\Spec(R) \ar[r] \ar@{-->}[ru] & Y 
}
}
\end{equation}
there exists a dotted arrow making the diagram commute. 
We say $f$ \emph{satisfies the uniqueness part of the valuative criterion 
with respect to $\Spec(R) \to Y$} if for any solid diagram~\eqref{vc-diagram}, the category 
of dotted arrows (see \citestacks{0CLA}) 
is either empty or a setoid with exactly one isomorphism class. 
\end{Def} 

\begin{Rem}
The commutativity of the diagram in Definition~\ref{Def-vc} 
must be understood in the 2-categorical sense, see 
\citestacks{0CL9}. 
In case $X$ and $Y$ are algebraic spaces, this subtlety disappears. 
Moreover, in this case the uniqueness part of the valuative criterion just says that there exists at 
most one dotted arrow in~\eqref{vc-diagram}. 
\end{Rem}

\begin{Rem}
We use the adjective ``strong'' because the standard ``existence part of the valuative criterion'' for algebraic stacks (or spaces) 
only requires the existence of a dotted arrow after passing to a field extension of $K$, 
see \citestacks{0CLK}.
\end{Rem}

We will show that under a suitable integrability hypothesis, 
Quot spaces satisfy the valuative criteria formulated above. 

\begin{Prop}
\label{proposition-quot-valuative-criteria} 
Let $E \in \cD$ be an $S$-perfect and $S$-flat object.
Let $R$ be a discrete valuation ring and let $\Spec(R) \to S$ be a morphism. 
Assume that $\utau$ integrates over $\Spec(R)$ to a bounded 
$\Spec(R)$-local t-structure on $\cD_R$ whose heart $\cA_R$ 
has a $\Spec(R)$-torsion theory. 
Then the morphism $\Quot_S(E) \to S$ satisfies 
the strong existence and the uniqueness 
parts of the valuative criterion with respect to $\Spec(R) \to S$. 
\end{Prop} 

Before proving the proposition, we give an alternate description 
of the Quot functor in the integrable case. 

\begin{Lem}
\label{lemma-quot-integrable}
Let $E \in \cD$ be an $S$-perfect and $S$-flat object.
Let $C \to S$ be a morphism from a Dedekind scheme. 
Assume that $\utau$ integrates over $C$ to a bounded $C$-local t-structure on $\cD_C$ whose heart $\cA_C$ has a $C$-torsion theory. 
Then there is an identification 
\begin{equation*}
\Quot_S(E)(C) = 
\set{
\begin{array}{c}
\textup{quotients } \res{E}{C} \to Q \textup{ in } \cA_C \\ 
\textup{such that } Q \textup{ is } C\textup{-flat}. 
\end{array} 
}
\end{equation*} 
\end{Lem} 

\begin{proof}
Note that $\res{E}{C} \in \cA_C$ by Lemma~\ref{lem-AC-C-flat}, 
since $\res{E}{C}$ is $C$-perfect and $C$-flat. 
By definition, a $C$-point of $\Quot_S(E)$ is a morphism 
$E_C \to Q$ in $\rD(X_C)$ such that $Q$ is $C$-perfect and $C$-flat 
and $E_c \to Q_c$ is surjective in $(\cA_{\qc})_c$ for all $c \in C$. 
Equivalently, by Lemmas~\ref{lem-AC-C-flat} and~\ref{lem-E-F-surjective}, 
$Q$ is a $C$-flat object of $\cA_C$ and $E_C \to Q$ is a quotient in $\cA_C$. 
\end{proof}

\begin{proof}[Proof of Proposition~\ref{proposition-quot-valuative-criteria}]
Let $K$ be the fraction field of $R$. 
For the strong existence part of the valuative criterion, 
by Lemma~\ref{lemma-quot-integrable} we must show that given 
a surjection $f \colon \res{E}{K} \to Q$ in $\cA_K$ with $Q$ a $K$-flat object, 
there is a surjection $\tilde{f} \colon \res{E}{R} \to \tilde{Q}$ in $\cA_R$ with 
$\tilde{Q}$ an $R$-flat object, such that $\tilde{f}$ restricts to $f$. 
To this end, we first use Lemma~\ref{lem-extend-from-localisation}.\eqref{enum:lem-extend-morphism-in-heart} to obtain $\tilde Q$ and a surjective map
$\tilde f' \colon E_R \to \tilde Q$ extending $f$.
After replacing $\tilde{Q}$ with $\tilde{Q}_{\Rtf}$, we may further assume $\tilde{Q}$ is 
$R$-torsion free and hence $R$-flat by Lemma~\ref{lem:FlatIffTFreeCurve}, proving the claim.

To show uniqueness, assume we are given two $R$-flat extensions $\tilde{Q}_1, \tilde{Q}_2$ of $Q$ and surjections
$\tilde{f}_i \colon E_R \to \tilde{Q}_i$. By Lemma~\ref{lem:tensortrick}, there exists an injective map
$i \colon \tilde{Q}_1 \into \tilde{Q}_2\otimes g^*\cO_R(\pi^k)$ for some $k\geqslant 0$ such that after pulling back to $K$, we have $i \circ f_1 = \pi^k f_2$. 
Hence the image of $i \circ f_1 - \pi^k f_2$ is $R$-torsion; since $\tilde{Q}_2$ is $R$-torsion free by Lemma~\ref{lem:FlatIffTFreeCurve}, the same holds true of $\tilde{Q}_2\otimes g^*\cO_R(\pi^k)$ so that this image must zero. It follows that the two morphisms are equal already over $R$. Since $\pi^k$ acts injectively on $R$-torsion free objects, we obtain a chain of isomorphisms 
$\tilde{Q}_1 \cong\im i = \im \pi^k\cong\tilde{Q}_2$ of quotients of $E_R$.
\end{proof}

\subsection{Nagata valuative criteria} 
\label{subsection-nagata}
Sometimes, we will only know (via Theorem~\ref{Thm-D-bc}) that $\utau$ integrates to a bounded $\Spec(R)$-local  t-structure when $\Spec(R) \to S$ is essentially of finite type;
therefore, we will only assume as much in Definition \ref{def:familyfiberstabilities}.
In this section, we prove that  when $S$ is Nagata, it is enough to consider such morphisms in the valuative criteria for universal closedness, see Lemma~\ref{lemma-valuative-criterion-uc}.
This will allow us for example to show universal closedness of moduli spaces of semistable objects in the proof of Theorem~\ref{thm:modulispacesArtinstacks}.
Otherwise, this section is completely independent of the rest of the paper.

For the definition of the Nagata property see 
\citestacks{032E} (for rings), 
\citestacks{033R} (for schemes), and 
\citestacks{0BAT} (for algebraic spaces). 
This property holds in essentially all naturally occurring examples. 
For instance, any quasi-excellent ring is Nagata \citestacks{07QV}. 

Recall that a domain $A$ with field of fractions $K$ is called Japanese if for any finite extension of fields $K \subset L$, the integral closure of $A$ in $L$ is finite over $A$. 
A ring $A$ is then called universally Japanese if for any finite type ring map $A \to B$ with $B$ a domain, the ring $B$ is Japanese. 
The key properties we need about Nagata rings are the following. 

\begin{Lem}[{\citestacks{0334}}]
\label{lem-nagata-iff-ujn} 
A ring is Nagata if and only if it is universally Japanese and noetherian. 
\end{Lem}

Recall that a ring map $A\to B$ is \emph{essentially of finite type} if $B$ is the localization of an $A$-algebra of finite type. We recall the following result:

\begin{Lem}[{\citestacks{0334} and \citestacks{032U}}]
\label{lem-nagata-eft} 
If $A$ is a Nagata ring and $A \to B$ is a ring map essentially of finite type, 
then $B$ is Nagata. 
\end{Lem}

\begin{Lem}
\label{lemma-nagata-dominate-dvr} 
Let $A$ be a Nagata local domain with fraction field $K$. 
Assume that $A$ is not a field. 
Let $K \subset L$ be a finitely generated field extension. 
Then there exists a DVR $R$ with fraction field $L$ which dominates $A$, such that $A \to R$ is essentially of finite type. 
\end{Lem}

\begin{proof}
The following argument is modeled on the proof of 
\citestacks{00PH}. 

If $L$ is not finite over $K$, choose a transcendence basis $x_1, \dots, x_r$ of $L$ over $K$. 
Let $\frm \subset A[x_1, \dots, x_r]$ be the maximal ideal generated by the maximal ideal of $A$ and $x_1, \dots, x_r$, 
and let $A' = A[x_1, \dots, x_r]_{\frm}$. 
Then $A'$ is a local domain which dominates $A$, the map $A \to A'$ 
is essentially of finite type, and $L$ is finite over the fraction field of $A'$. 
Moreover, $A'$ is Nagata by Lemma~\ref{lem-nagata-eft}. 
Replacing $A$ with $A'$ we may thus assume $L$ is finite over $K$. 

By \citestacks{00P8} we 
may find a noetherian local domain $A'$ of dimension $1$ with fraction field $K$ which dominates $A$ and such that $A \to A'$ is essentially of finite type. 
Again by Lemma~\ref{lem-nagata-eft}, $A'$ is Nagata. 
Replacing $A$ with $A'$ we may thus further assume $\dim(A) = 1$. 

Let $B \subset L$ be the integral closure of $A$ in $L$. 
By Lemma~\ref{lem-nagata-iff-ujn} the ring $A$ is Japanese, so $B$ is finite over $A$. 
Thus we may choose a prime $\fq \subset B$ lying over the maximal ideal of $A$. 
Set $R = B_{\fq}$. 
Then $R$ is a local normal domain of dimension $1$ which has fraction field $L$ and dominates $A$, and the map $A \to R$ is essentially of finite type. 
By Lemma~\ref{lem-nagata-iff-ujn} we see that $R$ is also noetherian, 
and hence a DVR. 
\end{proof}

The following notion will be used below. 
\begin{Def}\label{def:EssLocFT}
Let $Y$ be an algebraic space. 
For a scheme $Z$, we say a morphism $f\colon Z \to Y$ is \emph{essentially {(locally)} of finite type} if it is either (locally) of finite type or if $Z$ is affine and $f$ factors as 
\begin{equation*}
Z=\Spec(R[M^{-1}]) \to \Spec(R) \to Y
\end{equation*} 
where $\Spec(R) \to Y$ is a morphism {(locally)} of finite type and $M \subset R$ is a multiplicative system.
\end{Def}

\begin{Rem}
\label{rem-elft-eft}
If $Y$ is a locally noetherian scheme, then for any morphism $\Spec(R) \to Y$ locally of finite type, 
$\Spec(R)$ is noetherian and hence $\Spec(R) \to Y$ is of finite type. 
Thus if $Y$ is a locally noetherian scheme, a morphism $Z \to Y$ from an affine scheme $Z$ is essentially locally of finite type if and only if it is essentially of finite type. 
Since a Nagata scheme $Y$ is locally noetherian by \citestacks{033Z}, we will apply this observation in the results below whenever $Y$ is a Nagata scheme.
\end{Rem}

\begin{Rem}
We have chosen definitions that make our results as general as possible.  Note, however, that Definition~\ref{def:EssLocFT} is in a slightly 
different spirit than our definition of an essentially perfect morphism in Definition~\ref{def:EssPerfect}. One can show that for a noetherian scheme $Y$, any essentially locally of finite type morphism $Z \to Y$ that is also of finite Tor-dimension is essentially perfect, but the converse need not hold.  
\end{Rem}

\begin{Lem}
\label{lemma-specializations-lift}
Let $f \colon X \to Y$ be a locally of finite type morphism of schemes with $Y$ Nagata. 
Let $y \rightsquigarrow y'$, $y \neq y'$, be a specialization of points in $Y$, and let $x$ be a point of $X$ such that $f(x) = y$. 
Then there exists a commutative diagram 
\begin{equation*}
\xymatrix{
\Spec(K) \ar[r] \ar[d] & X \ar[d] \\
\Spec(R) \ar[r] & Y 
}
\end{equation*}
where $R$ is a DVR with field of fractions $K$ (which may be taken to be $\kappa(x)$), the bottom arrow is essentially of finite type and takes the generic point of $\Spec(R)$ to $y$ and the special point to $y'$, and the image point of the top arrow is $x$. 
\end{Lem}

\begin{proof}
We have a commutative diagram 
\begin{equation*}
\xymatrix{
\Spec(\kappa(x)) \ar[rr] \ar[d] & & X \ar[d] \\
\Spec(\kappa(y)) \ar[r] & \Spec(\cO_{Y,y'}) \ar[r]& Y 
}
\end{equation*} 
Let $A$ be image of the ring map $\cO_{Y,y'} \to \kappa(y)$. 
Then $A$ is a local domain with fraction field $\kappa(y)$, and $A$ is not a field since $y \neq y'$. 
By Lemma~\ref{lem-nagata-eft} the ring $A$ is Nagata because it is the quotient of the Nagata ring $\cO_{Y,y'}$. 
The field extension $\kappa(y) \subset \kappa(x)$ is finitely generated since 
$X \to Y$ is locally of finite type. 
Thus applying Lemma~\ref{lemma-nagata-dominate-dvr} we obtain a DVR $R$ with fraction 
field $\kappa(x)$ which dominates $A$, such that $A \to R$ is essentially of finite type. 
Thus we obtain a commutative diagram 
\begin{equation*}
\xymatrix{
\Spec(\kappa(x)) \ar[rrr] \ar[d] & & & X \ar[d] \\
\Spec(R) \ar[r] & \Spec(A) \ar[r] & \Spec(\cO_{Y,y'}) \ar[r]& Y 
}
\end{equation*} 
with all of the desired properties. 
\end{proof}

We will need the following lemma in our proof of the Nagata valuative 
criterion for universal closedness below. 

\begin{Lem}
\label{lem-uc-bc-fp}
Let $f \colon X \to Y$ be a quasi-compact morphism of algebraic spaces. 
Assume that for every locally of finite presentation morphism $Z \to Y$ with 
$Z$ affine, the base change $X_Z \to Z$ is closed. 
Then $f$ is universally closed. 
\end{Lem} 

\begin{proof} 
For any algebraic space $X$, we denote by $|X|$ the underlying topological space of $X$, 
see \citestacks{03BY}. 
A morphism of algebraic spaces is called closed if the induced map on 
topological spaces is closed. 

By \citestacks{0CM9}, to 
prove the lemma it suffices to show that for every locally of finite presentation morphism 
$Z \to Y$ of algebraic spaces, the base change $X_{Z} \to Z$ is closed. 
Choose a surjective \'{e}tale morphism $p \colon Z' \to Z$ where $Z'$ is a disjoint union of affine schemes $Z'_i$ 
\'{e}tale over $Z$. 
By assumption, the morphism $X_{Z'} \to Z'$ is closed after restriction to each affine $Z'_i$, 
and hence closed. 
Consider the diagram 
\begin{equation*}
\xymatrix{
|X_{Z'}| \ar[d]_{|f_{Z'}|} \ar[r]^{|p'|} & |X_Z| \ar[d]^{|f_Z|} \\ 
|Z'| \ar[r]^{|p|} & |Z|
}
\end{equation*} 
of continuous maps of topological spaces. 
We must show that for any closed subset $T \subset |X_Z|$, the 
set $|f_Z|(T) \subset |Z|$ is closed. 
Since $p \colon Z' \to Z$ is a surjective \'{e}tale morphism, 
the topology on $|Z|$ is the quotient topology for the surjection 
$|p| \colon |Z'| \to |Z|$, see \citestacks{03BX}. 
Hence we must show $|p|^{-1}|f_Z|(T) \subset |Z'|$ is closed. 
But since $|X_{Z'}| \to |Z'| \times_{|Z|} |X_{Z}|$ is surjective it follows 
that $|p|^{-1}|f_Z|(T) = |f_{Z'}||p'|^{-1}(T)$, which is closed by the observation above 
that $X_{Z'} \to Z'$ is closed. 
\end{proof}

\begin{Lem} 
\label{lemma-valuative-criterion-uc}
Let $f \colon X \to Y$ be a morphism 
of algebraic spaces such that: 
\begin{enumerate}[{\rm (1)}] 
\item \label{Y-Nagata-vc} 
$Y$ is Nagata. 
\item \label{f-ft-vc}
$f$ is of finite type. 
\item \label{f-epvc} 
$f$ satisfies the strong existence part of the valuative 
criterion with respect to any essentially locally of finite type morphism $\Spec(R) \to Y$ 
with $R$ a DVR. 
\end{enumerate} 
Then $f$ is universally closed. 
\end{Lem} 

\begin{proof}
By Lemma~\ref{lem-uc-bc-fp} it suffices to show that for every locally of finite presentation morphism 
$Z \to Y$ with $Z$ affine, the morphism $X_Z \to Z$ is closed. 
We claim that $X_Z \to Z$ inherits properties \eqref{Y-Nagata-vc}-\eqref{f-epvc}. 
Indeed, \eqref{Y-Nagata-vc} holds by \citestacks{035A}, 
\eqref{f-ft-vc} is clear, and \eqref{f-epvc} follows formally from the fact that $Z \to Y$ is locally 
of finite presentation. 

Thus to prove the lemma, it suffices to assume that $Y$ is affine and show 
$f \colon X \to Y$ is closed. 
In this situation, $X$ is a quasi-compact algebraic space since $f$ is of finite type and 
$Y$ is quasi-compact. 
Thus by \citestacks{03H6}
we may choose an affine scheme $U$ and a surjective \'{e}tale morphism $p \colon U \to X$. 
We must show that if $T \subset |X|$ is a closed subset of the underlying topological space of $X$, 
then $|f|(T)$ is closed in $Y$. 
The preimage $|p|^{-1}(T)$ is closed in $U$ and hence the set of points of an 
affine closed subscheme $V \subset U$. 
Hence $|f|(T)$ coincides with the image of the morphism of affine schemes $V \to Y$, 
which by \citestacks{00HY} is closed if and only if it is stable under specialization. 
In other words, we must show that specializations of points in $Y$ lift along 
$f$; but this follows from Lemma~\ref{lemma-specializations-lift} and property \eqref{f-epvc}. 
\end{proof}

\newpage
\part{Harder--Narasimhan structures over a curve}
\label{part:HNStrCurve}

In this part of the paper, we develop a theory of Harder--Narasimhan structures over a curve; 
this will be the basis of our definition of a stability condition over a general base scheme 
in Part~\ref{part:higher-dimensional-bases}. 
With the exception of Sections~\ref{sec:StabCond} and~\ref{sec:WeakStabCond} where we review the 
absolute setting, we work in the following setup. 
\begin{Setup} 
\label{setup-HN} 
Assume: 
\begin{itemize}
 \item $g \colon \cX \to C$ is a flat morphism as in the \hyperref[MainSetup]{Main Setup} and in addition it is projective, where $C$ is a Dedekind scheme;
 \item $\cD\subset\Db(\cX)$ is a $C$-linear strong semiorthogonal component of finite cohomological amplitude.
\end{itemize}
\end{Setup}

\section{Stability conditions and base change} 
\label{sec:StabCond}
We briefly recall the definition of stability conditions \cite{Bridgeland:Stab,Kontsevich-Soibelman:stability}, including the support property, and Bridgeland's deformation result. Then, in Section~\ref{subsec:basechangestability}, we prove a base change result for 
numerical stability conditions under arbitrary (not necessarily finitely generated) field extensions, extending the results of \cite{Pawel:basechange}.

\subsection{Definitions}
Let $\cD$ be a triangulated category.

\begin{Def} \label{def:slicing}
A \emph{slicing} $\cP$ of $\cD$ consists of full
additive subcategories $\cP(\phi) \subset \cD$ for each $\phi\in \mathbb{R}$, satisfying:
\index{P@$\cP$, $\cP(\phi)$ for each $\phi\in\R$, slicing}
\begin{enumerate}[{\rm (1)}] 
\item for all $\phi\in\mathbb{R}$, $\cP(\phi+1)=\cP(\phi)[1]$;
\item if $\phi_1>\phi_2$ and $E_j\in \cP(\phi_j)$, then $\Hom_{\cD}(E_1,E_2)=0$;
\item (\emph{HN filtrations}) for every nonzero $E \in \cD$ there exists a finite sequence of morphisms
\[ 0 = E_0 \xrightarrow{s_1} E_1 \to \cdots \xrightarrow{s_m} E_m = E \]
such that the cone of $s_i$ is in $\cP(\phi_i)$ for some sequence
$\phi_1 > \phi_2 > \dots > \phi_m$ of real numbers.
\end{enumerate}
\end{Def}

We write $\phi^+(E) := \phi_1$ and $\phi^-(E) := \phi_m$; moreover, for an interval $I \subset \R$, we write 
\begin{equation*}
 \cP(I) := \set{E \colon \phi^+(E), \phi^-(E) \in I} = \langle \cP(\phi) \rangle_{\phi \in I} \subset \cD.
\end{equation*}
\index{phi-+@$\phi^+(E)$ ($\phi^-(E)$), maximal (minimal) phase of an object}
We also write $\cP(\leqslant a):=\cP((-\infty,a])$ or $\cP(>b):=\cP((b,+\infty))$.
\index{P(I)@$\cP(I)$, slicing for an interval $I$}
\begin{Def} \label{def:stability_condition_p}
Let $\Lambda$ be a finite rank free abelian group with a group homomorphism $v \colon K(\cD)\to \Lambda$.
\begin{enumerate}[{\rm (i)}]
\item A \emph{pre-stability condition on $\cD$ with respect to $\Lambda$} is a pair $\sigma = (Z, \cP)$ where $\cP$ is a slicing of $\cD$ and $Z\colon \Lambda \to \C$ is a group homomorphism, that satisfy the following condition: 
\[
\text{for all }0\neq E \in\cP(\phi),\text{ we have }Z(v(E)) \in\R_{>0}\cdot e^{\mathfrak{i}\pi\phi}.
\]
We will often abuse notation and write $Z(E)$ for $Z(v(E))$. The nonzero objects of $\cP(\phi)$ are called $\sigma$-semistable of \emph{phase} $\phi$.
\item A pre-stability condition $\sigma = (Z, \cP)$ with respect to $\Lambda$ satisfies the \emph{support property} if there exists a quadratic form $Q$ on the vector space $\Lambda_\R$ such that
\index{sigma@$\sigma = (Z, \cP)$!(pre-)stability condition}
\begin{itemize}
\item the kernel of $Z$ is negative definite with respect to $Q$, and
\item for any $\sigma$-semistable object $E \in \cD$, we have
\[ Q(v(E)) \geqslant 0. \]
\end{itemize}
\item A \emph{stability condition} (with respect to $\Lambda$) is a pre-stability condition with respect to $\Lambda$ satisfying the support property.
\end{enumerate}
\end{Def}

We will briefly review the relevance of the support property for deformations of stability conditions in Section~\ref{subsec:spaceofostability}.
For now we note that an easy linear algebra statement implies the following.

\begin{Rem} \label{rem:supportfinitelength}
Assume that $Q$ is negative definite on the kernel of $Z$. Then for any $C > 0$, there can only be finitely many classes with
$\abs{Z(v)} < C$ and $Q(v) \geqslant 0$. 

It follows that if $(Z, \cP)$ satisfies the support property, then $\cP(\phi)$ is a finite length category
for every $\phi \in \R$, as the image of objects in $\cP(\phi)$ in $\R_{>0}\cdot e^{\ii\pi\phi}$ is a discrete set.
In other words, the support property guarantees the existence of a Jordan--H\"older filtration of a semistable object $E$, which has \emph{stable} filtration quotients; we call them the Jordan--H\"older factors of $E$.
\end{Rem} 

\begin{Def}\label{def:stability_function_p}
A \emph{stability function $Z$ on an abelian category $\cA$} is a morphism 
of abelian groups $Z\colon K(\cA) \to \C$ such that for all $0 \neq E \in\cA$, the complex number $Z(E)$ is in the semi-closed upper half plane $\H\sqcup \R_{<0} := \set{z\in\C \sth\Im z > 0, \text{ or } \Im z = 0\text{ and }\Re z < 0}$.
\index{Z@$Z\colon K(\cA) \to \C$,!stability function on $\cA$}

For $0\neq E \in \cA$ we define its \emph{phase} by $\phi(E) := \frac{1}{\pi} \arg Z(E) \in (0, 1]$. 
\index{phi@$\phi(E)$, phase of an object}
An object $E \in \cA$ is called \emph{$Z$-semistable} if for all subobjects $0 \neq A \hookrightarrow E$, we have $\phi(A) \leqslant \phi(E)$.
\end{Def}

\begin{Def}\label{def:satisfies_HN}
We say that a stability function $Z$ on an abelian category $\cA$ \emph{satisfies the HN property} if every object $E \in\cA$ admits a Harder--Narasimhan (HN) filtration: a sequence 
\[0 = E_0\hookrightarrow E_1 \hookrightarrow 
E_2\hookrightarrow \ldots \hookrightarrow E_m = E\]
such that $E_i /E_{i-1}$ is $Z$-semistable for $i = 1,\ldots, m$, with
\begin{equation*} 
\phi (E_1 /E_0 ) > \phi (E_2 /E_1 ) > \cdots > \phi (E_m /E_{m-1}).
\end{equation*}
\end{Def}

If $\cP$ is a slicing, then $\cP((\phi, \phi+1])$ is the heart of a bounded t-structure for all $\phi \in \R$; moreover, the slicing induces HN filtrations on $\cA := \cP((0, 1])$ with respect to $Z$. The converse also holds:
\begin{Lem}[{\cite[Proposition~5.3]{Bridgeland:Stab}}]
\label{lem:Bridgeland-stabviaheart} To give a pre-stability condition on $\cD$ is equivalent to giving a heart $\cA\subset \cD$ of a bounded t-structure, and a stability function $Z$ on $\cA$ with the HN property.
\index{sigma@$\sigma = (\cA, Z)$,!(pre-)stability condition}
\end{Lem}

The support property can be checked at the level of the abelian category, with the same definition.
Thus the above lemma can be rephrased in terms of stability conditions as well.

Let $X$ be a projective variety over a field $k$, and let $\cD \subset \Db(X)$ be a $k$-linear strong semiorthogonal component. Then the Euler characteristic 
\[
\chi_k(E, F) = \sum_i (-1)^i \dim_k\Hom(E, F[i])
\]
induces a pairing
\[
\chi_k \colon K(\cD_{\perf}) \times K(\cD) \to \Z.
\]
When the base field is clear from the context, we will always omit it from the notation.
We write $\Knum(\cD)$ for the quotient of $K(\cD)$ by the null space of $\chi$ on the right;
similarly we write, by some abuse of notation, $\Knum(\cD_{\perf})$ for the corresponding quotient on the left.
\index{Knum(D)@$\Knum(\cD)$, numerical $K$-group}
\index{Knum(D)@$\Knum(\cD_{\perf})$, numerical $K$-group}

\begin{Lem} \label{lem:Knumfinite}
Both $\Knum(\cD)$ and $\Knum(\cD_{\perf})$ are free and finitely generated abelian groups. 
\end{Lem}

\begin{proof}
By definition $\Knum(\cD)$ and $\Knum(\cD_{\perf})$ are torsion free; we need to show they are finitely generated. 
The pairing $\chi_k$ induces an inclusion 
$\Knum(\cD) \hookrightarrow \Hom(\Knum(\cD_{\perf}), \Z)$, 
so it is enough to prove the statement for $\Knum(\cD_{\perf})$. 
We may also reduce to the case where $\cD = \Db(X)$, as $\Knum(\cD_{\perf})$ is a
quotient of the image of $K(\cD_{\perf})$ in $\Knum(\Db(X))$.

Let $\pi \colon \widetilde{X} \to X$ be an alteration with $\widetilde{X}$ regular and projective (see \cite[Theorem~4.1]{deJong:alterations}). 
If $\CH^*_{\num}(\widetilde{X})$ denotes the Chow ring modulo numerical equivalence, then Riemann--Roch shows that for some integer $N$, $\Knum(\Dperf(\widetilde{X}))$ embeds into $\frac{1}{N}\CH^*_{\num}(\widetilde{X}) \subset \CH^*_{\num}(\widetilde{X}) \otimes \Q$; 
in particular, since $\CH^*_{\num}(\widetilde{X})$ is finitely generated (in finite characteristic, see e.g.~\cite[Example~19.1.4]{Fulton:IntersectionTheory}), so is $\Knum(\Dperf(\widetilde{X}))$.
On the other hand, using that \begin{equation}
\label{chiEF}
\chi_k(\pi^*(E), F) = \chi_k(E, \pi_*(F)) 
\quad \text{for} \quad E \in \Dperf(X) \text{ and } F \in \Db(\widetilde{X}), 
\end{equation}
we see that pullback induces a map 
$\Knum(\Dperf(X)) \to \Knum(\Dperf(\widetilde{X}))$. 
To finish the proof, we show this map is injective. 
Using again the relation~\eqref{chiEF}, 
we reduce to proving the following claim: if $\alpha \in K(\Db(X))$, then there exists a class $\widetilde{\alpha} \in K(\Db(\widetilde{X}))$ and an integer $d$ such that $\pi_*(\widetilde{\alpha}) = d\alpha$. 
This holds by d\'{e}vissage. 
Namely, by noetherian induction it suffices to show that if $F$ is a coherent sheaf on $X$ supported on a closed subset $Z \subset X$, 
then some multiple of $[F] \in K(X)$ lifts to $K(\widetilde{X})$, modulo classes of sheaves supported on proper subsets of $Z$. 
But we can find such a lift by considering the class $[\cO_{\widetilde{Z}}] \in K(\Db(\widetilde{X}))$ where $\widetilde{Z} \subset \widetilde{X}$ is a multisection of $\widetilde{X} \to X$ over $Z$. 
\end{proof}

\begin{Def}\label{def:NumericalStabilityCondition}
We say that a stability condition on $\cD$ is \emph{numerical} if the central charge factors via $Z \colon K(\cD) \to \Knum(\cD) \to \C$.
\end{Def}

\begin{Rem}\label{rmk:concr}
Concretely, a stability condition is numerical if $Z$ can be written as
\[
Z(\blank) = \sum_i z_i \chi(F_i,\blank),
\]
for some $z_i \in \C$ and $F_i \in \cD_{\perf}$. 
\end{Rem}

\subsection{The space of stability conditions} \label{subsec:spaceofostability}
We continue to fix a free abelian group $\Lambda$ of finite rank with a group homomorphism $v \colon K(\cD) \to \Lambda$, and write $\Stab_\Lambda(\cD)$ for the set of stability conditions satisfying the support property with respect to $\Lambda$.
We will review the deformation result of \cite{Bridgeland:Stab}, made somewhat more effective in \cite[Appendix~A]{BMS:abelian3folds}, which shows that $\Stab_\Lambda(\cD)$ is a complex manifold, locally homeomorphic to $\Hom(\Lambda, \C)$.

We first recall the metric topology on the set of slicings: 

\begin{PropDef}[{\cite[Section~6]{Bridgeland:Stab}}] \label{propdef:slicingsmetric}
Consider two slicings $\cP, \cQ$ of $\cD$, and write $\phi^{\pm}_\cP(E)$ and $\phi^{\pm}_\cQ(E)$ for the maximal and minimal phase occurring in the HN filtration of the object $E$ with respect to $\cP$ and $\cQ$, respectively. Then
\begin{align*}
d(\cP, \cQ) & := \sup_{E \in \cD} \left\{\left| \phi^+_\cQ(E) - \phi^+_\cP(E) \right|, \left| \phi^-_\cQ(E) - \phi^-_\cP(E) \right|\right\} \in \R_{\geqslant 0} \cup \{+\infty\} \\
 & = \sup_{\phi \in \R, E \in \cP(\phi)} \left\{ \left| \phi^+_\cQ(E) - \phi \right|, \left| \phi^-_\cQ(E) - \phi \right| \right\} \\
 & = \inf_{\phi\in\R} \left\{\epsilon\in\R_{\geqslant 0} \sth \cP(\leqslant \phi) \subset \cQ(\leqslant \phi + \epsilon) \text{ and } \cP(>\phi) \subset \cQ(>\phi - \epsilon) \right\},
\end{align*}
and this quantity defines a generalized metric on the set of slicings, i.e., it satisfies the triangle inequality and $d(\cP, \cQ) = 0$ if and only if $\cP = \cQ$.
\end{PropDef}
This induces a topology on $\Stab_\Lambda(\cD)$ as the coarsest topology such that the two forgetful maps to $\Hom(\Lambda, \C)$ and to the set of slicings on $\cD$ are continuous. 
\index{StabLambda(D)@$\Stab_\Lambda(\cD)$, space of stability conditions on $\cD$}

\begin{Thm}[{\cite{Bridgeland:Stab}, \cite[Proposition~A.5]{BMS:abelian3folds}}] \label{thm:deformstability}
The space $\Stab_\Lambda(\cD)$ of stability conditions on $\cD$ is a complex manifold, and the canonical
map 
\[ \cZ \colon \Stab_\Lambda(\cD) \to \Hom(\Lambda, \C),
\quad (Z, \cP) \mapsto Z \]
is a local isomorphism. 

More precisely, assume that $\sigma = (Z, \cP)$ satisfies the support property with respect to the quadratic form $Q$. Consider the open subset of $\Hom(\Lambda, \C)$ consisting of central charges whose kernel is negative definite with respect to $Q$, and write $P_Z$ for the connected component of this open subset that contains $Z$. If $U\subset \Stab_\Lambda(\cD)$ is the connected component of $\cZ^{-1}(P_Z)$ containing $\sigma$, then $\cZ|_U \colon U \to P_Z$ is a covering map.
\end{Thm}
In other words, a path in $P_Z$ starting at $Z$ lifts uniquely to a path in $\Stab_\Lambda(\cD)$ starting at $\sigma$.

We also note the following standard facts about the hearts of stability conditions, see the example at the end of \cite[Section~1]{Polishchuk:families-of-t-structures} and \cite[Proposition~5.0.1]{AP:t-structures}:

\begin{Lem} \label{lem:stabheartnoetherian}
Let $\sigma = (\cA, Z)$ be a stability condition with respect to $\Lambda$.
\begin{enumerate}[{\rm (1)}] 
\item The heart $\cA$ is tilted-noetherian.
\item If $Z$ is defined over $\Q[\ii]$, i.e., if its image is contained in $\Q \oplus i\Q$, then
$\cA$ is noetherian.
\end{enumerate}
\end{Lem}

The support property combined with some simple linear algebra arguments imply local finiteness of wall-crossing:
\begin{Lem}[{\cite[Section~9]{Bridgeland:K3}, \cite[Proposition~2.8 and Lemma~2.9]{Toda:K3}, \cite[Proposition~3.3]{BM:localP2}}] \label{lem:wallcrossinglocallyfinite}
Fix a class $\vv \in \Lambda$. 
\index{vv@$\vv\in \Lambda$, Mukai vector}
Then there exists a locally finite set of walls, i.e., real codimension one submanifolds, $\cW_\vv \subset \Stab_\Lambda(\cD)$ with the following properties. 
\begin{enumerate}[{\rm (1)}]
\item \label{enum:finitewalls} Let $W$ be an intersection of finitely many walls in $\cW_\vv$ (the empty intersection, i.e., $W = \Stab_\Lambda(\cD)$, being allowed).
Let $\cC \subset W$ be a \emph{chamber}, namely a connected component of the complement of the union of intersections of $W$ with any other wall. 
Then the set of $\sigma$-semistable objects of class $\vv$ is constant for $\sigma \in \cC$; the same holds for the set of $\sigma$-stable objects. 
\item \label{enum:Qi-enough} Up to the normalisation\footnote{This can be achieved via the action of $\wGL2$, the universal cover of $\GL_2^+(\R)$; this group has an action on $\Stab_\Lambda(\C)$ covering the action of $\GL_2^+(\R)$ on $\Hom(\Lambda, \C) \cong \Hom(\Lambda, \R^2)$.} $Z(\vv) \in \R_{<0}$, every such wall in $\cW_\vv$ is locally described by a linear equation of the form $\Im Z(\uu)= 0$ for some $\uu \neq  0$, defined over $\Q$.
Every chamber $\cC$ contains stability conditions with $Z$ defined over $\Q[i]$.
\end{enumerate}
\end{Lem}
\begin{proof}
The case of $\sigma$-semistable objects is treated in \cite{Toda:K3}. From the proof, one can see that an object that is $\sigma$-semistable in a chamber $\cC$ is either $\sigma$-stable for all $\sigma \in \cC$ or strictly $\sigma$-semistable for all $\sigma \in \cC$.
\end{proof}

\subsection{Stability conditions under base change} \label{subsec:basechangestability}
Now consider an algebraic variety $X$ defined over a base field $k$ %, an admissible subcategory $\cD %= \cD_k \subset \Db(X)$, 
and a field extension $\ell$ of $k$.
The pullback $\cD \to \cD_\ell$, $E\mapsto E_\ell$ induces maps
\[
K(\cD) \to K(\cD_\ell)\quad\text{and}\quad
K(\cD_{\mathrm{perf}}) \to K(\cD_{\ell, \perf})
\]
that preserve the Euler characteristic pairing.

\begin{Lem} \label{lem:Knumpullback}
Pullback also induces a map $\eta_{\ell/k} \colon \Knum(\cD_{\perf}) \to \Knum(\cD_{\ell, \perf})$ on numerical $K$-groups.
\end{Lem}
\begin{proof}
We have to show that if $E \in \cD_{\perf}$ satisfies $\chi_k(E, F) = 0$ for all $F \in \cD$, then we also have
$\chi_\ell(E_\ell, G) = 0$ for all $G \in \cD_\ell$. Since $G$ is defined over some intermediate field extension $k\subset \ell'\subset\ell$ with $\ell'/k$ finitely generated (as in the proof of Proposition~\ref{Prop-D-bc-fields}) and since pullback along $\ell/\ell'$ preserves the Euler characteristic, it is sufficient to consider the case where $\ell/k$ is itself finitely generated.
When $\ell/k$ is finite, the claim follows from adjunction via
\[ \chi_\ell(E_\ell, G) = \chi_\ell(E, G^{(k)}) = \dim_k(\ell) \chi_k(E, G^{(k)}) = 0, \]
where $G^{(k)} \in \cD$ denotes the pushforward of $G \in \cD_\ell$.

By induction, it suffices to prove the claim for $\ell = k(x)$. Let $R = k[x]$. 
In that case, we can lift $G \in \cD_{k(x)}$ to an object $\widetilde G \in \cD_R$.
Then base change implies
\[
\chi_\ell(E_\ell, G) = \rk \lHom_R(E_R, \widetilde G) = \chi_k(E, i_0^*\widetilde G) = 0
\]
where $i_0$ is the inclusion of the fiber over $0 \in \bA^1_k$. 
\end{proof}

Dualizing induces a pushforward map\index{eta_lk@$\eta_{\ell/k}^\vee \colon \Knum(\cD_\ell)\to \Knum(\cD) \otimes \Q$, pushforward map}
\begin{equation} \label{eq:defetaellkvee}
\eta_{\ell/k}^\vee \colon \Knum(\cD_\ell) \into \Hom(\Knum(\cD_{\ell, \perf}), \Z) \to \Hom(\Knum(\cD_{\perf}), \Z) \to \Knum(\cD) \otimes \Q.
\end{equation}

\begin{PropDef} \label{propdef:KnumDell}
\index{Knum(D)l@$\Knum(\cD)_\ell$, field extension $k\subset \ell$,!base changed numerical $K$-group}
The image $\Knum(\cD)_\ell := \im \eta_{\ell/k}^\vee$
contains $\Knum(\cD)$ as a subgroup of finite index; conversely, it is contained in
$\Knum(\cD)_{\overline{k}}$.
\end{PropDef}
\begin{proof}
Since $\eta_{\ell/k}$ preserves the Euler characteristic pairing, we have $\eta_{\ell/k}^\vee([F_\ell]) = [F]$ for $F \in \cD$, so $\Knum(\cD)_\ell$ contains $\Knum(\cD)$. That it is a finite index subgroup immediately follows from the definitions and Lemma~\ref{lem:Knumfinite}. For the second claim, it is sufficient to show
that if $k$ is algebraically closed, and $\ell/k$ is finitely generated, then 
$\im \eta_{\ell/k}^\vee$ is contained in $\Knum(\cD)$. Let $R/k$ be a finitely generated ring with fraction field $\ell$; then the same argument as in the end of the proof of Lemma~\ref{lem:Knumpullback}, with
$0 \in \bA^1_k$ replaced by any $k$-rational point of $\Spec R$, proves this containment.
\end{proof}

If $Z$ is the central charge of a numerical stability condition on $\cD$, we write $Z_\ell$ for the composition of $Z \circ \eta_{\ell/k}^\vee$.
Concretely, if $Z(\blank) = \sum_i z_i \chi(F_i,\blank)$ is a central charge as in Remark~\ref{rmk:concr}, then
\[
Z_\ell(\blank) = \sum_i z_i \chi((F_i)_\ell,\blank).
\]
Now fix a Mukai homomorphism $v \colon \Knum(\cD) \to \Lambda$ as in Definition~\ref{def:stability_condition_p}.
We let \index{Lambdal@$\Lambda_\ell$, field extension $k\subset \ell$,!base changed lattice}
\[\Lambda_\ell \subset \Lambda \otimes \Q\]
be the subgroup generated by $\Lambda$ and the image of $v$ extended to $\Knum(\cD)_\ell$; again, it contains $\Lambda$ as a subgroup of finite index. 
We write \[v_\ell \colon \Knum(\cD_\ell) \to \Lambda_\ell\] for the induced map.
\index{vl@$v_\ell$, field extension $k\subset \ell$,!base changed Mukai homomorphism}
If $Z$ factors via a group homomorphism $\Knum(\cD) \to \Lambda$,
then $Z_\ell$ factors via the composition $\Knum(\cD_\ell) \to \Knum(\cD)_\ell \to \Lambda_\ell$:
\[
\xymatrix{K(\cD_\ell)\ar[d]\ar@/^3.5pc/[rrrd]^{Z_\ell}
\ar[rrd]|!{[r];[rd]}\hole_(.4){v_\ell}
&K(\cD)\ar[l]_(.3){(\blank)_{\ell}}\ar@/^1pc/[rrd]^(.35){Z}\ar[d]\ar[rd]^{v} \\
\Knum(\cD_\ell)\ar[r]^{\eta_{\ell/k}^\vee} & \Knum(\cD)_\ell %\Knum(\cD
\ar[r] & \Lambda_\ell\ar[r]&\C
}
\]
For an object $E \in \cD$ we have the equality 
\[v[E] = v_\ell[E_\ell] \in \Lambda, \text{ and so }Z_\ell(E_\ell)=Z(E);\] on the other hand, if $k \subset \ell$ is finite, $F \in \cD_\ell$, and $F^{(k)} \in \cD$ denotes its pushforward, then 
\[v[F^{(k)}] = \dim_k(\ell) \cdot v_\ell[F].\]

\begin{Rem}
As $\Lambda \otimes \Q = \Lambda_\ell \otimes \Q$, the difference between the two lattices is irrelevant when considering central charges. However, we will also use $\Lambda$ to keep track of classes of semistable objects; then the integral structure will become relevant.
\end{Rem}

Base change for central charges as above, combined with base change for t-structures as in Proposition~\ref{Prop-D-bc-fields}, induces base change for stability conditions for arbitrary field extensions:
\begin{Thm} \label{thm:base-change-stability-condition}
Let $\sigma:= (\cA, Z)$ be a numerical stability condition with respect to $\Lambda$ on the strong semiorthogonal component $\cD \subset \Db(X)$, where $X$ is a variety defined over a field $k$.
Let $k \subset \ell$ be a (not necessarily finitely generated) field extension. Then:
\begin{enumerate}[{\rm (1)}]
\item \label{enum:mainbasechangeclaim}
The pair $\sigma_\ell := (\cA_\ell,Z_\ell)$ defines a numerical stability condition on $\cD_\ell$.
\index{sigmal@$\sigma_\ell = (\cA_\ell,Z_\ell)$, field extension $k\subset \ell$,!base changed stability condition}
\item \label{enum:semistableclasses}
If there exists a $\sigma_\ell$-semistable object in $\cD_\ell$ of class $v \in \Lambda$, then there exists a $\sigma$-semistable object in $\cD$ of class $n \cdot v$ for some 
$n \in \Z_{>0}$. 
(In particular, $\sigma_\ell$ satisfies the support property with respect to the same quadratic form as $\sigma$.)
\item \label{enum:pullbackremainsstable}
An object $E \in \cD$ is $\sigma$-semistable if and only if the pullback $E_\ell \in \cD_\ell$ is $\sigma_\ell$-semistable.
\item \label{enum:geomstablepreserved}
Let $E$ be an object in $\cD$.
If $k \subset \overline{k}$ and $\ell \subset \overline{\ell}$ denote the algebraic closures, then $E_{\overline{k}}$ is $\sigma_{\overline{k}}$-stable if and only if
$E_{\overline{\ell}}$ is $\sigma_{\overline{\ell}}$-stable.
\end{enumerate}
\end{Thm}
The resulting map $\Stab(\cD) \to \Stab(\cD_\ell)$ is a homeomorphism onto a union of connected components
of $\Stab(\cD_\ell)$; however, we omit the proof of this fact. Similar results were proved in a slightly different setup for finite separable or Galois field extensions in \cite{Pawel:basechange}.

\begin{proof}
We start with some general observations. For $\phi \in \R$, we write $\cA^\phi := \cP(\phi - 1, \phi] \subset \cD$ (with $\cA=\cA^1$). By base change via Proposition~\ref{Prop-D-bc-fields} (which applies by Lemma~\ref{lem:stabheartnoetherian}), this induces a heart $\cA^\phi_\ell \subset \cD_\ell^{}$ (and we consider $\cA_\ell:=\cA_\ell^1$). We use these hearts to define
\index{Pl@$\cP_\ell^{}(\phi)$, field extension $k\subset \ell$,!base changed slicing}
\begin{equation} \label{Pellequation}
\cP_\ell^{}(\phi) := \bigcap_{\phi'-1 < \phi \leqslant \phi'} \cA^{\phi'}_\ell.
\end{equation}
To prove that $\cP_\ell$ is a slicing, we only need to establish existence of HN filtrations, as the first two properties in Definition~\ref{def:slicing} are evident.

It follows from the corresponding statement in Proposition~\ref{Prop-D-bc-fields} that pullback $f^* \colon \cD \to \cD_\ell$ sends $\cP(\phi)$ to $\cP_\ell(\phi)$; in particular, statement \eqref{enum:pullbackremainsstable} will be automatic.

\begin{step}{1}\label{step:finte}
Claims \eqref{enum:mainbasechangeclaim} and \eqref{enum:semistableclasses} hold for finite field extensions $k \subset \ell$.
\end{step}
In this case, the hearts $\cA^\phi_\ell$ are already constructed in Theorem~\ref{Thm-D-bc}, and by Theorem~\ref{Thm-D-bc}.\eqref{DbT-f-finite-pushforward} the pushforward $f_* \colon \cD_\ell \to \cD$ satisfies
\begin{equation} \label{eq:pushforwardslicing}
f_* \left( \cA^\phi_\ell \right) \subset \cA^\phi \quad \text{and therefore} \quad
f_* \left( \cP_\ell(\phi) \right) \in \bigcap_{\phi'-1 < \phi \leqslant \phi'} \cA^{\phi'} = \cP(\phi).
\end{equation}
Since $Z_\ell(E) = \dim_k(\ell)\cdot Z(f_* E) $ this immediately implies the compatibility of $Z_\ell$ both with $\cA_\ell$, and with $\cP_\ell$. It also implies part \eqref{enum:semistableclasses}, with $n = \dim_k(\ell)$.

It remains to prove the existence of HN filtrations for $E \in \cA_\ell$. For each $0 < \phi \leqslant 1$, it follows from
\cite[Lemma~1.1.2]{Polishchuk:families-of-t-structures} that $\cA_\ell^\phi$ is obtained by tilting $\cA_\ell$ at a torsion pair 
$\left(\cT_\ell^{> \phi}, \cF_\ell^{\leqslant \phi}\right)$, analogous to the torsion pair $\left(\cT^{>\phi}, \cF^{\leqslant \phi}\right)$ in $\cA$. These torsion pairs induce an a priori infinite filtration
\begin{equation} \label{eq:infinitefiltration}
0 = T_\ell^{>1} \subset \dots \subset T_\ell^{> \phi} \subset \dots \subset T_\ell^{>0} = E
\end{equation}
of $E$ in $\cA_\ell$. Since 
\[ f_* \cT_\ell^{>\phi} = f_* \left(\cA_\ell^{} \cap \cA_\ell^\phi \right) \subset \cT^{>\phi} \quad \text{and} \quad 
f_* \cF_\ell^{\leqslant \phi} \subset \cF^{\leqslant \phi},
\]
the pushforward of this filtration is induced by the HN filtration of $f_* E \in \cA$, and thus finite; since $f_*$ is conservative, the original filtration in $\cA_\ell$ also has to be finite. The filtration quotients have to lie in $\cP_\ell(\phi) = \cF_\ell^{\leqslant \phi} \cap \bigcap_{\phi' < \phi} \cT_\ell^{>\phi'}$ for appropriate $\phi$, and thus this is the desired HN filtration.

\begin{step}{2}\label{step:k(x)andrational}
Claims \eqref{enum:mainbasechangeclaim} and \eqref{enum:semistableclasses} hold for $\ell = k(x)$ and $Z$ defined over $\Q[\ii]$.
\end{step}
Let $R = k[x]$; we write $\cD_R^{}$ and $\cA_R^{}, \cA_R^\phi$ for the triangulated category and hearts obtained from $\cD$ and $\cA$ via Theorem~\ref{theorem-bc-sod} and Theorem~\ref{Thm-D-bc}, respectively.
Any object in $\cA_\ell$ is of the form $E_\ell$ for some $E \in \cA_R$ (see Lemma~\ref{lem-open-restriction-es}). 
By Lemma~\ref{lem:stabheartnoetherian} and Theorem~\ref{Thm-D-bc}.\eqref{D-bc-noetherian}, $\cA_R$ is noetherian;
by Remark~\ref{rem:noetherianandtorsiontheory} we may assume that $E$ is $R$-torsion free, and thus $R$-flat by Lemma~\ref{lem:FlatIffTFreeCurve}. 
It follows that $E_p \in \cA_p$ for all closed points $p \in \Spec R$; 
since $[i_{p*}E_p]$ and $[E_\ell]$ define the same class in $\Lambda$ (via $\cD_R\to\cD_\ell$ and $v_\ell$), %(after pushforward to $\cD$), 
this shows that $Z_\ell$ is compatible with $\cA_\ell$. 
Applying the same argument to $\cA^\phi$ instead of $\cA$, we obtain by the compatibility of $Z_\ell$ with all $\cA^\phi_\ell$, and thus with $\cP_\ell$.

Now consider again the filtration in \eqref{eq:infinitefiltration}. Since $Z_\ell$ is compatible with $\cA_\ell$ and $\cA^\phi_\ell$, and since for all $0<\phi'<\phi<1$, $T_\ell^{>\phi'}/T_\ell^{>\phi} \in \cF_\ell^{\leqslant \phi}$, we have $\Im Z_\ell(T_\ell^{>\phi'}/T_\ell^{>\phi}) > 0$ whenever $T_\ell^{>\phi'}/T_\ell^{>\phi} \neq 0$. 
Since $\Im Z_\ell$ is discrete, the filtration has to be finite.

Finally, if $E_\ell \in \cP_\ell(\phi)$, we may choose a representative $E \in \cA^\phi_R$ that is $R$-torsion free as an object in $\cA^\phi_R$. Then $E_p \in \cA^\phi$ satisfies $Z(E_p) \in \R_{>0}\cdot e^{\ii\pi\phi}$, and thus it is $\sigma$-semistable.
The equality $v_\ell[E_\ell] = v_\ell[i_{p*}E_p]$ in $\Lambda$ proves claim \eqref{enum:semistableclasses}.

\begin{step}{3}\label{step:rational}
Claims \eqref{enum:mainbasechangeclaim} and \eqref{enum:semistableclasses} hold when $Z$ is defined over $\Q[\ii]$, and $\ell$ arbitrary.
\end{step}
By Steps~\ref{step:finte} and \ref{step:k(x)andrational}, the claim holds when $\ell$ is finitely generated over $k$; the general case follows with the same type of arguments as those used in Proposition~\ref{Prop-D-bc-fields}.
To prove the existence of HN filtrations for $E \in \cD_\ell$, we may assume by Proposition~\ref{Prop-D-bc-fields}.\eqref{Prop-D-bc-fields-Efg} that $E$ is obtained by base change from $E_{k'}$ for some finitely generated $k \subset k' \subset \ell$; by construction of $\cP_\ell(\phi)$ and $\cP_{k'}(\phi)$, the pullback of the HN filtration in $\cD_{k'}$ induces one in $\cD_\ell$.
The compatibility of the central charge $Z_\ell$ with $\cP_\ell$, and claim \eqref{enum:semistableclasses} of the Theorem follow similarly from their analogues for $\cD_{k'}$.

\begin{step}{4}\label{step:mainbasechange} Claims \eqref{enum:mainbasechangeclaim} and \eqref{enum:semistableclasses} hold in full generality.
\end{step} 
Consider a stability condition $\sigma = (Z, \cP)$ on $\cD$ that satisfies the support property with respect
to a given quadratic form $Q$.
By Theorem~\ref{thm:deformstability}, for any $\epsilon > 0$ there exists a neighborhood $Z \in U \subset \Hom(\Lambda, \C)$,
depending only on $Q$, such that $U$ embeds into $\Stab(\cD)$ with image containing $\sigma$, and such that any two stability conditions $\sigma_{1} = (Z_{1}, \cP_{1})$ and $\sigma_{2} = (Z_{2}, \cP_{2})$ in the image satisfy $d(\cP_1, \cP_2) < \epsilon$.
Let $U_\Q \subset U$ be the dense set of central charges defined over $\Q[\ii]$.
For any $Z' \in U_\Q$ and the corresponding stability condition $\sigma' = (Z', \cP')$ on $\cD$, we have
an induced a stability condition $\sigma'_\ell$ on $\cD_\ell$ via base change by the previous step.

We first claim that the resulting map $U_\Q \to \Stab(\cD_\ell)$ is continuous.
Indeed, if $\sigma' = (Z', \cP')$ and $\sigma'' = (Z'', \cP'')$ are two stability conditions with central charges in $U_\Q$, and with $d(\cP', \cP'') < \epsilon$, we recall from Proposition/Definition~\ref{propdef:slicingsmetric} that this is equivalent to $\cP'(\leqslant \phi) \subset \cP''(\leqslant \phi+\epsilon)$ and $\cP'(> \phi) \subset \cP''(> \phi-\epsilon)$.
Since the construction of the base change t-structure evidently preserves inclusions, this implies $\cP_\ell'(\leqslant \phi) \subset \cP_\ell''(\leqslant \phi+\epsilon)$ and $\cP_\ell'(> \phi) \subset \cP_\ell''(> \phi-\epsilon)$, 
and thus $d(\cP'_\ell, \cP''_\ell) < \epsilon$.

Since the stability conditions in the image of $U_\Q$ satisfy the support property with respect to $Q$, we can use Theorem~\ref{thm:deformstability} to extend this map to $U \to \Stab(\cD_\ell)$.
It remains to show that for $\sigma \in U \setminus U_\Q$, the stability condition $\sigma_\ell = (Z_\ell, \overline{\cP_\ell})$ obtained via this map satisfies the description in the Theorem;
as observed above, it is enough to show that $\overline{\cP_\ell} = \cP_\ell$, where the latter is defined by equation \eqref{Pellequation}.
Let $\sigma' = (Z', \cP')$ be a stability condition with $Z' \in U_\Q$ and $d(\cP, \cP') < \epsilon$.
Then 
\[ \cP_\ell(\leqslant \phi) \subset \cP'_\ell(\leqslant\phi + \epsilon) \subset \overline{\cP_\ell}(\leqslant\phi + 2\epsilon)
\]
for all $\phi \in \R$ by the same argument as above; since this holds for all $\epsilon > 0$, and since $\overline{\cP_\ell}$ is a slicing, this shows $\cP_\ell(\leqslant \phi) \subset \overline{\cP_\ell}(\leqslant\phi)$.
The dual argument shows $\cP_\ell(>\phi) \subset \overline{\cP_\ell}(> \phi)$.
Since both $\left(\cP_\ell(>\phi), \cP_\ell(\leqslant \phi)\right)$ and $\left(\overline{\cP_\ell}(>\phi), \overline{\cP_\ell}(\leqslant \phi)\right)$ define t-structures, this gives $\cP_\ell = \overline{\cP_\ell}$ as desired.

Finally, to extend claim~\eqref{enum:semistableclasses} from $U_\Q$ to $U$, let $\sigma \in U \setminus U_\Q$, and let $\sigma' \in \Stab(\cD_\ell)$ be its image.
Then $\sigma'$ is contained in a chamber $\cC$ in the sense of Lemma~\ref{lem:wallcrossinglocallyfinite}.\eqref{enum:Qi-enough}.
Let $E$ be a $\sigma'$-semistable object of class $v$. 
Then there is a dense set of stability conditions $\sigma'' \in \cC$ in the image of $U_\Q$; in particular $E$ is also $\sigma''$-semistable, and the proof of \eqref{enum:semistableclasses} in the previous cases gives an object $F \in \cD_k$ of class $n \cdot \vv$ that is semistable at the corresponding point of $U_\Q$.
Since the set where an object is semistable is closed,  $F$ is also semistable with respect to $\sigma$, which proves the claim.

\begin{step}{5} \label{step:geomstablepreserved}
Claim \eqref{enum:geomstablepreserved} holds.\end{step}
First we observe that one direction is clear.
Indeed, a non-trivial Jordan--H\"{o}lder filtration of $E_{\overline{k}}$ would pullback to a non-trivial Jordan--H\"{o}lder filtration of $E_{\overline{\ell}}$.
For the reverse direction, suppose that $E \in \cD$ is such that $E_{\overline{k}}$ is $\sigma_{\overline{k}}$-stable (and hence $E_{k'}$ is $\sigma_{k'}$ for any finite field extension $k \subset k'$.
Then, following the same overall logic as in steps~\ref{step:finte}--\ref{step:mainbasechange}, we need to show the following: if $Z$ is defined over $\Q[\ii]$, and if $k \subset k'$ and $k'(x) \subset \ell$ are finite field extensions, then $E_{\ell}$ is $\sigma_{\ell}$-stable.
This is shown similarly to Step~\ref{step:k(x)andrational}.
Under the assumptions, we can use the $\wGL2$-action on $\sigma$ to achieve that $E \in \cP_{k'}(1)$ with $Z$ still defined over $\Q[\ii]$; in particular, $\cA_{k'}$ is noetherian by Lemma~\ref{lem:stabheartnoetherian}.
We first choose a smooth curve $C$ defined over $k'$ with fraction field $\ell$ and infinitely many closed points.
If $E_{\ell}$ is strictly $\sigma_{\ell}$-semistable, we use Lemma~\ref{lem-extend-from-localisation}.\eqref{enum:lem-extend-filtration-from-localisation} to lift its Jordan--H\"older filtration to a filtration of $E_C$ in $\cA_C$.
Let $U \subset C$ be the open subset where every factor of this filtration is torsion free (which exists as $\cA_{k'}$ is noetherian, so that $\cA_C$ is noetherian by Theorem~\ref{Thm-D-bc}.\eqref{D-bc-noetherian}, and thus has a $C$-torsion theory by Remark~\ref{rem:noetherianandtorsiontheory}).
Then this induces a non-trivial filtration of $E_c \in \cA_c$ for every $c \in U$; since $\Im Z_c(E_c) = 0$, the same is true for each filtration factor.
Thus this is in fact a non-trivial filtration in $\cP_c(1)$. Since $k(c)/k$ is a finite field extension, pullback via an embedding $k(c) \into \overline{k}$ contradicts the stability of $E_{\overline{k}}$.
\end{proof}

We note that the property of being $\sigma$-\emph{stable} is not necessarily preserved by pullback, e.g.~for a stable object whose endorphism ring is given by a field extension of $k$.
Consequently, while base change preserves HN filtrations, it does not preserve Jordan--H\"older filtrations.
Therefore, we make the following definition:

\begin{Def} \label{def:geometricallystable}
In the setting of Theorem~\ref{thm:base-change-stability-condition}, we say that an object $E$ is \emph{geometrically $\sigma$-stable} if it is stable after base change to the algebraic closure $\overline{k}$ of $k$.
\end{Def}
By Theorem~\ref{thm:base-change-stability-condition}.\eqref{enum:geomstablepreserved}, this property is preserved by field extensions.

\section{Harder--Narasimhan structures over a curve}
\label{sec:defnHNoveraCurve}

The aim of this section is to introduce a notion of Harder--Narasimhan (HN) structures over a one-dimensional base.
It will include stability conditions on the fibers of $\cD$, but additionally provide HN filtrations for every object $E \in \cD$;
this strengthens the classical notion of relative HN filtrations.
As in the case of stability conditions, we will often need the auxiliary notion of \emph{weak Harder--Narasimhan structure}, which will be done in Section~\ref{sec:defnweakHNstructure}.
In this section, we omit any proofs, as they are essentially the same as for the case of weak HN structures;
instead, we introduce the definitions and state their basic properties along with some discussion.

\subsection{Definitions}
\label{subsec:mainDefcurve}

We work in Setup~\ref{setup-HN}.
We remind the reader of the notation $K$, $p \in C$, $W \subset C$, $\cD_K$, $\cD_p$, and $\cD_W$ introduced in Section~\ref{sec:setupnotation}.
Given a heart $\cA_C \subset \cD$ of a bounded t-structure local over $C$, we similarly write $\cA_K$, $\cA_p$, and $\cA_W$ for the corresponding hearts, given by Theorem~\ref{thm-t-structure-finite-map} and Corollary~\ref{cor:base-change-tstructure-point}, respectively.
We also refer to Section~\ref{subsection-torsion-objects} for the categories of $C$-torsion objects $\cD_{\Ctor}\subset \cD$ and 
$\cA_{\Ctor} = \cA \cap \cD_{\Ctor}$.
We also introduced there the notion of $C$-torsion free objects $\cA_{\Ctf} \subset \cA$, which are objects with no $C$-torsion subobjects.

The following gives a notion of central charge that is ``constant in families''.
\begin{Def} \label{def:familycentralcharges}
A \emph{central charge on $\cD$ over $C$} is a pair
$(\ZK,\Zc)$ where 
\index{ZK,ZCtor@$(\ZK,\Zc)$, central charge on $\cD$ over $C$}
\[ Z_K \colon K(\cD_K) \to \C \quad \text{and} \quad
\Zc \colon K(\cD_{\Ctor}) \to \C \]
are group homomorphisms with the following property: for all $E \in \cD$, and all proper closed subschemes $W \subset C$, we have
\begin{equation} \label{eq:ZKZP}
Z_K(E_K) = \frac 1{\len W}\Zc\left( i_{W*}^{}E_W^{}\right).
\end{equation}
\end{Def}
Since $\cD \to \cD_K$ is essentially surjective, $Z_K$ is determined by $\Zc$, and thus \eqref{eq:ZKZP} becomes a consistency condition for $\Zc$: the right-hand-side should be independent of $W$.
We think of this as requiring that $\Zc$ is constant in families of objects over $C$.

\begin{Ex} \label{ex:ZCtorviachiF}
Equation \eqref{eq:ZKZP} is satisfied whenever $\Zc$ can be written as a linear combination of functions of the form $\chi_F^{}$, for $F \in \Dperf(\cX)$, defined by
\[
\chi_F^{}(E) := \len_{\cO_C} g_* \left(E \otimes F\right)
= \sum_{i \in \Z} (-1)^i\, \len_{\cO_C} \rH^i\left(g_* (E \otimes F)\right).
\]
Note that this sum is well-defined: for $E \in \cD_{\Ctor}$, each cohomology sheaf of $g_* (E \otimes F)$ is a sheaf
with zero-dimensional support.
Moreover, since $E$ is bounded, $F$ is perfect, and $g$ is projective, we have $g_* (E \otimes F) \in \Db(C)$.
Then, the right-hand-side in \eqref{eq:ZKZP} is independent of $W$: indeed, by base change it is equal to the rank of the complex $g_* (E \otimes F)$ over $C$.
\end{Ex}

\begin{Def} \label{def:HNstructure_C}
A \emph{Harder--Narasimhan structure on $\cD$ over $C$} consists of a triple $\sigma_C = (\ZK,\Zc,\cP)$ where 
\index{sigmaC@$\sigma_C = (\ZK,\Zc,\cP)$, $C$ a Dedekind scheme!Harder--Narasimhan structure on $\cD$ over $C$!}
\begin{itemize}
\item $\cP$ is a slicing of $\cD$, and
\item $(\ZK,\Zc)$ is a central charge on $\cD$ over $C$,
\end{itemize}
satisfying the following two properties:
\begin{description}
\item[$C$-linearity] The slicing $\cP$ is \emph{local over $C$}, 
i.e., for every open $U \subset C$ there exists a slicing
$\cP_U$ of $\cD_U$ such that the pullback sends $\cP(\phi)$ to $\cP_U(\phi)$.
\item[Compatibility] For all $\phi \in \R$ and all $0 \neq E \in \cP(\phi)$, we have either
\begin{eqnarray*} E_K \neq 0 &\text{and} & \quad Z_K(E_K) \in \R_{> 0} \cdot e^{\ii\pi\phi}, \quad \text{or} \\
E \in \cD_{\Ctor} & \text{and} &\Zc(E) \in \R_{> 0} \cdot e^{\ii\pi\phi}.
\end{eqnarray*}
\end{description}
The nonzero objects of $\cP(\phi)$ are said to be \emph{$\sigma_C$-semistable} of \emph{phase} $\phi$, and the
simple objects of $\cP(\phi)$ are said to be \emph{$\sigma_C$-stable}.
\end{Def}

\begin{Rem} \label{rem:local-slicing-tensor-ample}
Just as in Theorem~\ref{thm:local-t-structure-tensor-ample} for the case of t-structures, $\cP$ is local over $C$ if and only if $\cP(\phi)$ is invariant under tensoring with $g^*L$ for all $\phi$ and all line bundles $L$ on $C$.
\end{Rem}

Before giving a concrete example of a HN structure over $C$, 
we show how to construct one from an appropriate t-structure; in other
words, we prove an analogue of \cite[Proposition~5.3]{Bridgeland:Stab} (see also Lemma~\ref{lem:Bridgeland-stabviaheart}).

\begin{Def}
Let $\cA_C \subset \cD$ be the heart of a bounded $C$-local t-structure.
A \emph{stability function on $\cA_C$ over $C$} is a central charge $(\ZK,\Zc)$ on $\cD$ over $C$ such that $Z_K$ is a stability function on $\cA_K$, and $\Zc$ is a stability function on $\cA_{\Ctor}$.
\end{Def}

\begin{Ex} \label{ex:slopestabilitycurves}
Let $\cX \to C$ be a family of curves, and let $\cO_\cX(1)$ be a relative polarization.
For $E \in (\Coh \cX)_{\Ctor}$, let $p_1^{}(E)$ and $p_0^{}(E)$ be the coefficients of the linear and constant terms, respectively, of the Hilbert polynomial of $E$, defined as in Example~\ref{ex:ZCtorviachiF} via the length of $g_* E(m)$ as $\cO_C$-modules.
Then
\[\Zc(E) := \mathfrak{i}p_1^{}(E) - p_0^{}(E) 
\]
defines a stability function on $\Coh \cX$ over $C$.
\end{Ex}

\begin{Def} \label{def:ZC}
Given a stability function on $\cA_C$ over $C$, we define the following \emph{central charge for objects $E \in \cA_C$}:
\index{ZC@$Z_C$, central charge for objects in $\cA_C$}
\[ Z_C(E) := \begin{cases} Z_K(E_K) & \text{if $Z_K(E_K) \neq 0$,} \\
\Zc(E) & \text{otherwise}.
\end{cases} \]
We then assign to $E \neq 0$ the \emph{slope} $\mu_C(E) \in \R \cup \{+\infty\}$ by
\index{muC@$\mu_C$, slope for objects in $\cA_C$}
\[ \mu_C(E) := 
\begin{cases}
+\infty & \text{if $\Im Z_C(E) = 0$} , \\
-\frac{\Re Z_C(E)}{\Im Z_C(E)} & \text{otherwise.}
\end{cases}
\]
Often we will use the phase $\phi(E) := \frac{1}{\pi} \arg Z_C(E) \in (0, 1]$ instead of the slope $\mu_C(E)$.
\index{phi@$\phi(E)$, phase of an object}
\end{Def}

This slope function satisfies the weak see-saw property (see Lemma~\ref{lem:seesawweak}).
We can easily construct examples where the strong see-saw property is not satisfied.

\begin{Ex}
\label{Ex:seesawstrong}
Let $E\in \cA_C$ with $E_K \neq 0$ and $p\in C$.
Assume that there exists a quotient $i_{p*}E_p\twoheadrightarrow Q\in \cA_\Ctor$ such that $\mu_C(i_{p*}E_p)\neq\mu_C(Q)$.
Then the composition $f\colon E \twoheadrightarrow Q$ satisfies $\mu_C(\ker f)=\mu_C(E)\neq\mu_C(Q)$.
\end{Ex}

\begin{Def}\label{def:semistable_C}
An object $E \in \cA_C$ is called \emph{$Z_C$-semistable} if for all proper subobjects $0 \neq A \hookrightarrow E$, we have $\phi(A) \leqslant \phi(E/A)$
(or equivalently, $\mu_C(A)\leqslant \mu_C(E/A)$).
\end{Def}

\begin{Def} \label{def:satisfies_HN_C}
We say that a stability function $(\ZK,\Zc)$ on $\cA_C$ over $C$ \emph{satisfies the HN property} if every object $E \in\cA_C$ admits a Harder--Narasimhan (HN) filtration: a sequence 
\[0 = E_0\hookrightarrow E_1 \hookrightarrow 
E_2\hookrightarrow \ldots \hookrightarrow E_m = E\]
such that $E_i /E_{i-1}$ is $Z_C$-semistable for $i = 1,\ldots, m$, with
\begin{equation*}
\phi (E_1 /E_0 ) > \phi (E_2 /E_1 ) > \cdots > \phi (E_m /E_{m-1}).
\end{equation*}
The objects $E_i/E_{i-1}$ are called the \emph{HN factors} of $E$.
\end{Def}

Since $i_{p*} \colon \cA_p \to \cA_C$ is exact, fully faithful, and since the image is closed under subobjects and quotients by Corollary~\ref{cor:schematicsupportsubsquotients}, $\Zc$ automatically
induces a compatible notion of semistability on $\cA_p$, and if $(\ZK,\Zc)$ satisfies the HN property in $\cA_C$, then so does $Z_p :=\Zc \circ i_{p*}$ on $\cA_p$.
Combined with Lemma~\ref{lem:Bridgeland-stabviaheart}, this yields:
\begin{Lem} \label{lem:inducedfiberprestability}
A HN structure $\sigma_C$ on $\cD$ over $C$ induces a pre-stability condition $\sigma_p = (\cA_p, Z_p)$ on $\cD_p$ for every $p \in C$.
\index{sigmac@$\sigma_c = (\cA_c, Z_c)$, $c\in C$ Dedekind scheme,!when $c=p$ is a closed point}
\end{Lem}

By definition, a HN structure also gives a pre-stability condition $\sigma_K = (\cA_K, Z_K)$ on $\cD_K$; therefore, a HN structure $\sigma_C$ induces a pre-stability condition $\sigma_c$ on $\cD_c$ for every point $c \in C$.

The usual notion of relative slope-stability for a family of torsion free sheaves asks that the \emph{generic} fiber is slope-stable;
in contrast, $Z_C$-stability requires stability for \emph{all} fibers (see also Remark~\ref{rem:allfibersnotdestquot_plus}):

\begin{Lem} \label{lem:allfibersstable}
Let $E \in \cA_C$ be a $C$-torsion free object.
Then $E$ is $Z_C$-semistable if and only if $E_K$ is $\ZK$-semistable and $E_p \in \cA_p$ is $Z_p$-semistable for all closed points $p\in C$.
\end{Lem}

\begin{proof} See Lemma~\ref{lem:allfibersnotdestquot}.
\end{proof}

\begin{Ex}
Consider slope-stability on a family of curves $g \colon \cX \to C$ as in Example~\ref{ex:slopestabilitycurves}.
Let $x \in \cX$ be a closed point, $p := g(x)$ and let $F \subset \cX$ be the fiber containing $x$.
By Lemma~\ref{lem:allfibersstable} the ideal sheaf $I_x$ is not semistable:
indeed, $i_{p*}(I_x)_p = \cO_F(-x) \oplus \cO_x$ is not semistable.
The HN filtration of $I_x$ is given by $I_F \into I_x$, since $I_F$ is semistable by Lemma~\ref{lem:allfibersstable}, and the quotient $I_x/I_F = \cO_F(-x)$ is the push-forward of a semistable sheaf from the fiber and therefore easily shown to be semistable with respect to $Z_{\Ctor}$ in $\cA_{\Ctor}$.

Note that $I_F \into I_x$ is also the simplest possible example of a semistable reduction of a flat family of sheaves;
in particular, the notion of HN filtrations with respect to $\mu_C$ will require semistable reduction as a necessary ingredient, see Proposition~\ref{prop:HNviaHN}.
\end{Ex}

\begin{Prop} \label{prop:stabviaheart}
To give a HN structure on $\cD$ over $C$ is equivalent to giving a heart $\cA_C \subset \cD$ of a bounded $C$-local t-structure, together with a stability function $(\ZK,\Zc)$ on $\cA_C$ over $C$ satisfying the HN property.
\index{sigmaC@$\sigma_C = (\cA_C,\ZK,\Zc)$, $C$ a Dedekind scheme!Harder--Narasimhan structure on $\cD$ over $C$}
\end{Prop}
\begin{proof} See the proof of Proposition~\ref{prop:stabviaheartweak}.

We only stress that for $E \in \cA_C$, either all of its HN factors $E_i$ with respect to $\cP$ are in $\cD_{\Ctor}$, in which case $E \in \cD_{\Ctor}$ and $\Zc(E) = \sum_i\Zc(E_i)$ is in the semiclosed upper half plane $\H\sqcup \R_{< 0}$;
or, otherwise, some of its HN factors have $Z_K(E_i) \neq 0$, and we have $Z_K(E_K) \in \H\sqcup \R_{< 0}$.
Thus, $(\ZK,\Zc)$ is a stability function on $\cA_C$ over $C$.
\end{proof}

\begin{Rem} \label{rem:wGL2action}
Let $\wGL2$ be the universal cover of the group $\GL_2^+(\R)$ of real $2 \times 2$-matrices with positive determinant.
From Definition~\ref{def:HNstructure_C}, it is evident that $\wGL2$ acts on the set of pre-stability conditions on $\cD$ over $C$ in the same manner as it acts on the set of stability conditions on a triangulated category.
Indeed, let $\R \to S^1$ be the universal cover given by $\phi \mapsto e^{\ii\pi\phi}$;
then $\wGL2$ can be described as the set of pairs $(G, g)$ where $G \in \GL_2^+(\R)$, and $g \colon \R \to \R$ is a choice of a lift of the induced action of $G$ on $S^1$.
Then
\[ (G, g) (\ZK,\Zc,\cP) = (g^*\cP, G \circ Z_K, G \circ\Zc)
\]
where $g^* \cP(\phi) = \cP(g^{-1} \phi)$.
\end{Rem}

\subsection{Existence of HN filtrations}
\label{subsec:existenceHN}
Finally, we discuss the existence of HN filtrations.

\begin{Prop} \label{prop:HNviaHN}
Let $(\ZK,\Zc)$ be a stability function on $\cA_C$ over $C$.
If $(\ZK,\Zc)$ satisfies the HN property, then all of the following three conditions are satisfied:
\begin{enumerate}[{\rm (1)}] 
\item \label{enum:HNAK} 
The pair $(\cA_K, Z_K)$ satisfies the HN property.
\item \label{enum:HNAtor}
The pair $(\cA_{\Ctor},\Zc)$ satisfies the HN property.
\item \label{enum:ssred} (\emph{Semistable reduction}) 
For any $C$-torsion free object $E \in \cA_{\Ctf}$ such that $E_K \in \cA_K$ is $Z_K$-semistable, there is a $Z_C$-semistable subobject $F \subset E$ with $E/F \in \cA_{\Ctor}$.
\end{enumerate}
Moreover, if $\cA_C$ has a $C$-torsion theory $(\cA_{\Ctor}, \cA_{\Ctf})$, then the converse also holds true.
\end{Prop}

We show the necessity of the three conditions in the following remark, and refer to Proposition~\ref{prop:HNviaHNweak} for the proof of the converse.

\begin{Rem} \label{rem:HNnecessaryconditions}
We first comment on each of the conditions given in Proposition~\ref{prop:HNviaHN}.
\begin{enumerate}[{\rm (1)}]
\item If $E \in \cA_C$ is $Z_C$-semistable, then by Lemma~\ref{lem:allfibersstable} the object $E_K \in \cA_K$ is $Z_K$-semistable.
Moreover, given a HN filtration $0 = E_0 \into E_1 \into \dots \into E_m=E$, then for any 
$i < j$ we have $\mu_K\left((E_i/E_{i-1})_K\right) > \mu_K\left((E_j/E_{j-1})_K\right)$ if both quotients are non-zero.
Hence, the HN filtration of $E$ in $\cA_C$ induces a HN filtration of $E_K \in \cA_K$, i.e., condition \eqref{enum:HNAK} is necessary.
On the other hand, given the HN filtration of $E_K \in \cA_K$, we can attempt to construct a HN filtration of $E \in \cA_C$ as a refinement of a lift of the HN filtration of $E_K$ to $\cA_C$.
\item Since $\cA_{\Ctor}$, as a subcategory of $\cA_C$, is closed under subobjects and quotients, the inclusion $\cA_{\Ctor} \subset \cA_C$ identifies $\Zc$-semistable objects with objects in $\cA_{\Ctor}$ that are $Z_C$-semistable as objects in $\cA_C$.
In particular, condition \eqref{enum:HNAtor} is clearly necessary.
\item Consider the first step of the HN filtration of such an $E$.
It has to be a $Z_C$-semistable subobject $F \subset E$ with $\mu_C(F) = \mu_C(E)$.
\end{enumerate}

Also recall that when $\cA_C$ is noetherian, then the assumption that $(\cA_{\Ctor}, \cA_\Ctf)$ is a torsion pair is automatic, see Remark~\ref{rem:noetherianandtorsiontheory}.
However, it is easy to construct examples of hearts local over $C$ that do not satisfy this assumption:
for example, consider the heart $\cB_C$ obtained by tilting at the torsion pair $(\cA_{\Ctor}, \cA_{\Ctf})$.
 \end{Rem}

\section{Weak stability conditions, tilting, and base change} 
\label{sec:WeakStabCond}

In order to construct stability conditions on surfaces or higher-dimensional varieties, one may use the auxiliary notion of \emph{weak stability conditions}.\footnote{In \cite{BMS:abelian3folds,PT15:bridgeland_moduli_properties}, this notion is called a very weak stability condition.}
The procedure will be analogous for Harder--Narasimhan structures over a curve.
In this section, we recall the definition of weak stability conditions following \cite{BMS:abelian3folds,PT15:bridgeland_moduli_properties}.
Then we study analogues of operations on stability conditions that become more subtle for weak stability conditions, 
namely tilting (Section~\ref{subsec:tiltingweakstability}) and base change (Section~\ref{subsec:bcweakstability}).

\subsection{Definitions} 
\label{subsec:WeakStabCond}
We begin by recalling the definition of weak stability conditions by following \cite{BMS:abelian3folds,PT15:bridgeland_moduli_properties}, and review some of the analogues of basic properties of stability conditions that become more subtle for weak stability conditions.

\begin{Def} \label{def:weak stability condition}
Let $\cD$ a triangulated category.
A \emph{weak pre-stability condition on $\cD$} is a pair $\sigma = (Z, \cP)$ where $\cP$ is a slicing of $\cD$, and
$Z \colon K(\cD)\to \C$ is a group homomorphism, that satisfy the following condition:
\index{sigma@$\sigma = (Z, \cP)$!weak (pre-)stability condition!}
\[\text{For all }0\neq E \in\cP(\phi),\text{ we have }Z(E) \in\begin{cases}
\R_{>0} \cdot e^{\ii \pi \phi} & \text{if $\phi \notin \Z$} \\
\R_{\geqslant 0} \cdot e^{\ii \pi \phi} & \text{if $\phi \in \Z$.}
\end{cases}\]
\end{Def}
As in Definition~\ref{def:stability_condition_p}, we say that $\sigma$ is a \emph{weak pre-stability condition on $\cD$ with respect to $\Lambda$}, if $Z$ factors through a group homomorphism $v:K(\cD)\to \Lambda$.

\begin{Def} A \emph{weak stability function $Z$ on an abelian category $\cA$} is a morphism 
of abelian groups $Z \colon K(\cA) \to \C$ 
such that for all $0 \neq E \in\cA$, the complex number $Z(E)$ is in $\H\sqcup \R_{\leqslant 0} := \set{z\in\C \sth\Im z > 0, \text{ or } \Im z = 0\text{ and }\Re z \leqslant 0}$.
\index{Z@$Z\colon K(\cA) \to \C$,!weak stability function on $\cA$}

The function $Z$ allows one to define a \emph{slope} for any $0 \neq E \in \cA$ by setting 
\index{muZ@$\mu_Z$, slope function}
\[ \mu_{Z}(E) := \begin{cases} - \frac{\Re Z(E)}{\Im Z(E)} & \text{if }\Im Z(E) > 0 \\
+\infty & \text{otherwise}
\end{cases}
\]
and a notion of stability: An object $0 \not= E\in\AA$ is $Z$-\emph{semistable} if for every proper subobject $F$, we have $\mu_{Z}(F) \leqslant \mu_{Z}(E/F)$.
\end{Def}

\begin{Def} \label{def:A0}
Given a weak stability function $Z$ on an abelian category $\cA$, we define $\cA^0 \subset \cA$ as the subcategory of objects $E \in \cA$ with $Z(E) = 0$.
\index{A0@$\cA^0 \subset \cA$ subcategory of $\cA$ where $Z$ vanishes}
\end{Def}

HN filtrations and the HN property for weak stability functions are defined exactly as in Definition~\ref{def:satisfies_HN}.

\begin{Lem}[{\cite[Section~2.1]{PT15:bridgeland_moduli_properties}}]
To give a weak pre-stability condition on $\cD$ is equivalent to
giving a heart $\cA\subset \cD$ of a bounded t-structure, and a weak stability function $Z$ on $\cA$ satisfying the HN property.
\index{sigma@$\sigma = (\cA, Z)$,!weak (pre-)stability condition}
\end{Lem}

\begin{Ex} \label{ex:slopestabilityasweakstability}
Let $X$ be a polarized variety over a field, and let $p_n, p_{n-1}$ be the leading coefficients of the Hilbert polynomial.
Then $(\ii p_n - p_{n-1},\Coh X)$ defines a weak stability condition, and $(\Coh X)^0$ is the category of sheaves supported in codimension at least two.
\end{Ex}

The analogue of Lemma~\ref{lem:stabheartnoetherian} is more involved.
\begin{Def} \label{def:noetheriantorsionsubcat}
Let $\cB \subset \cA$ be an abelian subcategory of an abelian category $\cA$.
We say that $\cB$ is a \emph{noetherian torsion subcategory} of $\cA$ if $\cB$ is a noetherian abelian category, and if there exists a torsion pair $(\cB, \cB^\perp)$ in $\cA$.
\end{Def}

\begin{Rem}\label{rem:alternatenoetheriantorsion}
An extension-closed abelian subcategory $\cB \subset \cA$ is a noetherian torsion subcategory if and only if for every object $E \in \cA$ any
increasing sequence $B_1 \subset B_2 \subset \dots \subset E$ of subobjects of $E$ with $B_i \in \cB$ terminates.
\end{Rem}

\begin{Lem}[{\cite[Lemma~2.17]{PT15:bridgeland_moduli_properties}}] \label{lem:discreteimpliesNoetherian}
Let $(\cA, Z)$ be a weak pre-stability condition, such that
$\cA^0 \subset \cA$ is a noetherian torsion subcategory, and $Z$ is defined over $\Q[\ii]$.
Then $\cA$ is noetherian.
\end{Lem}

\begin{proof}
Assume otherwise, and consider a sequence of non-trivial surjections $E_1 \onto E_2 \onto \cdots$.
Then $\Im Z(E_i)$ is discrete monotone decreasing non-negative function; hence we may assume it to be constant.

Let $M_i \into E_i$ be the maximal subobject with $\Im Z(M_i) = 0$, which exists by the existence of HN filtrations.
If $C_i$ is the kernel of $E_i \onto E_{i+1}$, then $\Im Z(C_i) = 0$, and so $C_i \subset M_i$.
Similarly, it follows that the composition
$M_i \into E_i \onto E_{i+1}$ factors via a surjection $M_i \onto M_{i+1}$ with kernel $C_i$.
Hence we get a sequence of non-trivial surjections
$M_1 \onto M_2 \onto \cdots$.

Arguing again by discreteness of the central charge, we may assume that $Z(M_i)$ is constant. Then the kernels
$D_i$ of the composition $M_1 \onto M_{i}$ form a strictly increasing sequence of subobjects of $M_1$ in $\cA^0$.
Since $\cA^0$ is assumed to be a noetherian torsion subcategory, this is a contradiction to Remark~\ref{rem:alternatenoetheriantorsion}.
\end{proof}

Finally, we define the support property for weak pre-stability conditions exactly as in Definition~\ref{def:stability_condition_p}, and call $\sigma$ a \emph{weak stability condition} if it satisfies the support property with respect to some $v, \Lambda$ and $Q$.
\index{sigma@$\sigma = (Z, \cP)$!weak (pre-)stability condition}
\begin{Rem} \label{rem:vA0=0}
	If $\sigma$ is a weak stability condition, then $v(\cA^0) = 0$.
	Indeed, every $E \in \cA^0$ is automatically semistable, and thus $v(E)$ is a vector in $\Ker Z$ with $Q(v(E)) \geqslant 0$.
\end{Rem}

\begin{Rem} \label{rem:JHfiltrationweakstability}
If $\sigma = (Z, \cP)$ is a weak stability condition, and $E$ is semistable of phase $\phi \in \R \setminus \Z$, then $E$ admits a Jordan--H\"older filtration.
\end{Rem}

\begin{Lem} \label{lem:boundedslopenoetherian}
Assume that $\sigma = (\cA, Z)$ is a weak stability condition, such that $\cA^0$ is a noetherian torsion subcategory.
Let $\mu \in \R \cup \set{+\infty}$, and let $E_1 \into E_2 \into E_3 \into \dots \subset E$ be an increasing sequence of subobjects of a fixed object with $\mu(E_i) \geqslant \mu$ for all $i$.
Then this sequence terminates.
\end{Lem}
\begin{proof}
We have $\mu^+(E_i) \leqslant \mu^+(E)$ and $\mu(E_i) \geqslant \mu$.
From this one can deduce that the central charges of all HN factors of all $E_i$ are contained in a compact region.
If $\mu^+(E)<+\infty$ (and thus $\mu<+\infty$), this region is the parallelogram with two horizontal edges and two edges corresponding to slope $\mu^+(E)$, and with opposite vertices given by $0$ and the complex number $z$ determined by $\Im z = \Im Z(E)$ and $\mu(z) = \mu$.
In the case $\mu^+(E)=+\infty$ let $F \subset E$ be the first step of the HN filtration; then we can replace $0$ by $Z(F) \in \R_{>0}$ and $\mu^+(E)$ by $\mu^+(E/F)$ in the previous construction of the parallelogram.

It follows by Remark~\ref{rem:supportfinitelength} that there are only finitely many classes  $a_1, \dots, a_m \in \Lambda$ that can occur as the classes of HN factors of $E_i$.
Since $Z(E_i)$ is contained in the same compact region described in the previous paragraph, and since every $Z(E_i)$ is an integral non-negative linear combination of the same finite set of complex numbers $Z(a_k) \in \H\sqcup \R_{<0}, k = 1, \dots, m$, this leaves only finitely  possibilities for $Z(E_i)$. Therefore, $Z(E_i)$ has to become constant for $i \geqslant i_0$.
But then $E_i/E_{i_0}$ is a subobject of $E/E_{i_0} \in \cA^0$ for all $i \geqslant i_0$, and thus the sequence terminates by the assumption that $\cA^0$ is noetherian.
\end{proof}

\subsection{Tilting weak stability conditions} \label{subsec:tiltingweakstability}
Due to the special role played by objects with central charge zero, there is in general no analogue for weak stability conditions of the $\wGL2$-action on stability conditions explained in Remark~\ref{rem:wGL2action}.
In this subsection, we explore conditions under which a weak stability condition $(\cA, Z)$ can nevertheless be tilted.

Given $\beta\in\R$, we can define a pair $(\cT^\beta,\cF^\beta)$ of subcategories of $\cA$ given by 
\index{T,Fbeta@$(\cT^\beta,\cF^\beta)$, torsion pair at $\beta\in \R$}
\begin{equation}\label{eqn:torsion pair}
\begin{split}
\cT^\beta:=\langle E\in\cA\text{ $\mu$-semistable with }\mu(E)>\beta\rangle = \stv{E}{\mu^-(E) > \beta},\\
\cF^\beta:=\langle E\in\cA\text{ $\mu$-semistable with } \mu(E)\leqslant \beta\rangle = \stv{E}{\mu^+(E) \le \beta}.
\end{split}
\end{equation}
Existence of HN filtrations combined with the weak see-saw property ensure that $(\cT^\beta, \cF^\beta)$ is a torsion pair (see Definition~\ref{def:TorsionPair}); we write $\cA^{\sharp\beta} = \langle \cF^\beta[1], \cT^\beta \rangle$ for the corresponding tilted heart.
\index{Asharpbeta@$\cA^{\sharp\beta} = \langle \cF^\beta[1], \cT^\beta \rangle$, tilted heart at $\beta$}
\index{Cohbeta(X)@$\Coh^\beta X_s$, tilted heart of $\Coh X_s$ at slope $\beta$}

By \cite[Lemma~2.16]{BLMS} and the proof of \cite[Proposition~2.15]{BLMS}, the following property guarantees that a weak stability condition can be tilted, see Proposition~\ref{prop:RotateWeakStability} below; we will sketch a proof for completeness.

\begin{Def} \label{def:tiltingproperty}
A weak stability condition $\sigma = (\cA, Z)$ has the \emph{tilting property} if
\index{sigma@$\sigma = (Z, \cP)$!weak (pre-)stability condition!having the tilting property}
\begin{enumerate}[{\rm (1)}] 
 \item $\cA^0 \subset \cA$ is a noetherian torsion subcategory, and
 \item \label{enum:tiltingpropertyExt1} for every $F \in \cA$ with $\mu^+(F) < +\infty$, there exists a short exact sequence $F \into \tF \onto F^0$ with $F^0 \in \cA^0$ and
 $\Hom(\cA^0, \tF[1]) = 0$.
\end{enumerate}
\end{Def}

\begin{Ex} \label{ex:slopestabilityhastiltingproperty}
Slope stability of sheaves as in Example~\ref{ex:slopestabilityasweakstability} has the tilting property, with $F \into \tilde F$ given by the embedding of a torsion free sheaf into its double dual.
\end{Ex}

\begin{Rem} \label{rem:tiltingpropertyextensions}
The tilting property implies that for $F$ as in part~\eqref{enum:tiltingpropertyExt1} of Definition~\ref{def:tiltingproperty}, any sequence of inclusions $F = F_0 \into F_1 \into F_2 \dots$
with $F_i/F \in \cA^0$ terminates: since $\Ext^1(F_i/F, \widetilde F) = 0$, the inclusion $F \into \tF$ factors via $F_i$, and thus $F_i/F$ is an increasing sequence of subobjects of $\tF/F \in \cA^0$.
\end{Rem}

\begin{Rem} \label{rem:tiltingpropertytilt}
We will see in the proof of Proposition~\ref{prop:RotateWeakStability} that part \eqref{enum:tiltingpropertyExt1} is equivalent to the condition that 
$\cA^0$ is a torsion subcategory of the heart obtained by tilting $\cA$ at the torsion pair $(\cT^{\beta}, \cF^{\beta})$, for every $\beta \in \R$.
\end{Rem}

\begin{Prop} \label{prop:RotateWeakStability}
Let $\sigma = (\cA, Z)$ be a weak stability condition with the tilting property.
Then $\sigma^\beta := \left(\cA^{\sharp\beta}, \frac{Z}{\ii-\beta}\right)$ is again a weak stability condition, and $(\cA^{\sharp\beta})^0 \subset \cA^{\sharp\beta}$ is a noetherian torsion subcategory.
\index{sigmabeta@$\sigma^\beta = (\cA^{\sharp\beta}, Z^{\sharp\beta})$,!tilted weak stability condition at $\beta\in \R$}
\end{Prop}

We first notice that $Z^{\sharp\beta}:=\frac{Z}{\ii-\beta}$ is a weak stability function on $\cA^{\sharp\beta}$.
Semistable objects can then be easily classified: we omit the proof (see, for example, \cite[Lemma~2.19]{PT15:bridgeland_moduli_properties}).

\begin{Lem}\label{lem:RotateWeakStabilitySemistableObjects}
Let $E\in\cA^{\sharp\beta}$ with $E \notin (\cA^{\sharp\beta})^0$.
Then $E$ is $Z^{\sharp\beta}$-semistable if and only if
\begin{enumerate}[{\rm (1)}] 
 \item\label{enum:RotatingWeakStability1} either $\Im Z(E) \geqslant 0$ and $E\in\cA$ is a $Z$-semistable object with $\Hom(\cA^0, E) = 0$,
 \item\label{enum:RotatingWeakStability2} or $\Im Z(E) < 0$ and $E$ is an extension
 \[
 U[1] \to E \to V
 \]
 where $U\in\cA$ is a $Z$-semistable object and $V\in\cA_0$; moreover, if either $\Im Z^{\sharp\beta}(E)>0$ or $E$ is $Z^{\sharp\beta}$-stable, then $\Hom(V',E)=0$, for all $V'\in\cA^0$.
\end{enumerate}
\end{Lem}

\begin{proof}[Proof of Proposition~\ref{prop:RotateWeakStability}.]
We only need to check that $(\cA^{\sharp\beta})^0$ is a noetherian torsion subcategory:
then, using Lemma~\ref{lem:RotateWeakStabilitySemistableObjects}, one can use the HN filtration of $\rH^0_{\cA}(E)$ and modify the HN filtration of $\rH^{-1}_{\cA}(E)$ via subobjects and quotients in $(\cA^{\sharp\beta})^0$ to construct the HN filtration of $E \in \cA^{\sharp\beta}$. Moreover,
the Lemma, combined with Remark~\ref{rem:vA0=0}, also shows that the set of classes of semistable objects in $\Lambda$ is unchanged, so that $\sigma^{\beta}$ satisfies the support property.
(We will spell out a similar argument in more detail in the proof of Proposition~\ref{prop:tiltingweakHNstructures}.)

By construction, $(\cA^{\sharp\beta})^0=\cA^0$; in particular, it is noetherian by assumption.
We need to show it is a torsion subcategory.
Let $F\in\cA^{\sharp\beta}$.
We can write it as an extension
\[
0 \to M[1]\to F\to N \to 0 
\]
in $\cA^{\sharp\beta}$ with $M\in\cF^\beta$ and $N\in\cT^\beta$.
By assumption, there exists a short exact sequence in $\cA^{\sharp\beta}$
\[
0\to M^0 \to M[1] \to \tM[1] \to 0
\]
with $M^0\in\cA^0$ and $\Hom((\cA^{\sharp\beta})^0,\tM[1])=0$.

By replacing $F$ with $F/M^0$, we can assume $M=\tM$.
Hence, any injective morphisms $F^0\into F$ in $\cA^{\sharp\beta}$, with $F^0\in(\cA^{\sharp\beta})^0$, induces an injective morphism $F^0\into N$.
Since $\cA^0\subset \cA$ is a noetherian torsion subcategory, there is a maximal such subobject $F^0\into F$, which proves what we wanted.
\end{proof}

\begin{Rem}
When $Z$ is defined over $\Q[\ii]$ and $\beta$ is also rational, then Lemma~\ref{lem:discreteimpliesNoetherian}
implies additionally that $\cA^\beta$ is noetherian.
\end{Rem}

\subsection{Base change for weak stability conditions}
\label{subsec:bcweakstability}
We now explain how to extend our base change result, Theorem~\ref{thm:base-change-stability-condition}, to the case of weak stability conditions.
Throughout this subsection we assume that $\cD \subset \Db(X)$ is a strong semiorthogonal component of the derived category of a variety defined over a field $k$, and let $\sigma = (\cA, Z)$ be a weak stability condition such that $\cA^0 \subset \cA$ is a noetherian torsion subcategory, and such that $Z$ is defined over $\Q[\ii]$.

We first recall the construction of the slicing from the proof of Theorem~\ref{thm:base-change-stability-condition}.
\begin{Def} \label{def:slicingbasechangeweak}
Consider a field extension $k \subset \ell$.
Let $\cA^\phi := \cP(\phi-1, \phi]$, and write $\cA_\ell^\phi$ for the heart in $\cD_\ell$ obtained from $\cA^\phi$ via base change as in Proposition~\ref{Prop-D-bc-fields}.
(Since $\cA$ is noetherian by Lemma~\ref{lem:discreteimpliesNoetherian}, $\cA^\phi$ is tilted-noetherian.)
We define
\index{Pl@$\cP_\ell^{}(\phi)$, field extension $k\subset \ell$,!base changed slicing}
\begin{equation} \label{Pellequation-weak}
\cP_\ell^{}(\phi) := \bigcap_{\phi'-1 < \phi \leqslant \phi'} \cA^{\phi'}_\ell.
\end{equation}
\end{Def}
\begin{Prop} \label{prop:basechangeweakstabilityviaopenness}
Let $\cD \subset \Db(X)$ be a strong semiorthogonal component of the derived category of a variety defined over a field $k$, and let $\ell/k$ be a field extension.
Let $\sigma = (\cA, Z)$ be a weak numerical stability condition on $\cD$ and assume the following:
\begin{enumerate}[{\rm (1)}] 
 \item 
 $\cA^0 \subset \cA$ is a noetherian torsion subcategory, and $Z$ is defined over $\Q[\ii]$.
 \label{enum:basechangeweakassumption1}
 \item \label{enum:semistablegenericopen} Either $\ell/k$ is algebraic, or
 if $C$ is an affine Dedekind scheme essentially of finite type over $k$,
 and if $E \in \cD_C$ satisfies $E_{K(C)} \in \cP_{K(C)}(\phi)$, then there exists an open subset $U \subset C$ such that $E_c \in \cP_c(\phi)$ for all $c \in U$.
\end{enumerate}
Then $\sigma_\ell = (\cA_\ell, Z_\ell)$ defines a weak stability condition on $\cD_\ell$, with slicing given by $\cP_\ell$ as in Definition~\ref{def:slicingbasechangeweak}.
The pullback 
$\cD_k \to \cD_\ell$ preserves the properties of being semistable, or geometrically stable, respectively.
\index{sigmal@$\sigma_\ell = (\cA_\ell,Z_\ell)$, field extension $k\subset \ell$,!base changed weak stability condition}
\end{Prop}
We will later consider a variant of the openness condition in \eqref{enum:semistablegenericopen}, see Definition~\ref{def:GenericOpennessSemiStab}.

\begin{proof}
We follow the proof of Theorem~\ref{thm:base-change-stability-condition}. 
Note that by Lemma~\ref{lem:discreteimpliesNoetherian}, assumption \eqref{enum:basechangeweakassumption1} is equivalent to the condition that $\cA$ is noetherian and $Z$ is defined over $\Q[\ii]$, and thus by Proposition~\ref{Prop-D-bc-fields}\eqref{Prop-D-bc-fields-fg}, it is preserved by base change along finitely generated field extensions. 
Also by our definition of essentially of finite type for affine schemes in Definition \ref{def:EssLocFT}, assumption \eqref{enum:semistablegenericopen} is automatically preserved under finitely generated field extensions. 

Step~\ref{step:finte} carries over without any change.
Now consider Step~\ref{step:k(x)andrational}, the case of $\ell = k(x)$, and again set $R = k[x]$.
Since $\cA$ and thus $\cA_R$ is noetherian, any object $E_{k(x)} \in \cA_{k(x)}$ lifts to an $R$-torsion free object $E_R \in \cA_R$, see Remark~\ref{rem:noetherianandtorsiontheory}.
By Lemma~\ref{lem:FlatIffTFreeCurve}, we have $E_c \in \cA_c$ for all closed points $c \in \A^1_k$; since $Z_{k(x)}(E_{k(x)}) = Z_c(E_c)$ this proves the compatibility of $Z_{k(x)}$ with $\cA_{k(x)}$.

Our next observation is that if $E \in \cP_{k(x)}(1)$, then the same argument combined with condition \eqref{enum:semistablegenericopen} shows $Z_{k(x)}(E) = 0$.
Conversely, if $E_R$ is a torsion free lift of an object $E \in \cA_{k(x)}$ with $\Im Z_{k(x)}(E) = 0$, then $E_c \in \cP_c(1)$ for all closed points $c \in \A^1_k$ by Step~\ref{step:finte}, i.e., $E_c \in \bigcap_{0 \leqslant \phi < 1} \cA_c^\phi$.
By Lemma~\ref{lem:flatinheart}, this implies
$E_R \in \bigcap_{0 \leqslant \phi < 1} \cA_R^\phi $; since $\cD_R \to \cD_{k(x)}$ is t-exact by Theorem~\ref{Thm-D-bc}.\eqref{DbT-f-flat}, this implies 
$E \in \bigcap_{0 \leqslant \phi < 1}\cA_{k(x)}^\phi = \cP_{k(x)}(1)$.

The existence of the HN filtration now follows exactly as in the proof of Theorem~\ref{thm:base-change-stability-condition}, and the compatibility of $Z_{k(x)}$ with $\cP_{k(x)}$ follows again from assumption \eqref{enum:semistablegenericopen}; this concludes Step~\ref{step:k(x)andrational}.

Step~\ref{step:rational} carries over without change.
Finally, Step~\ref{step:geomstablepreserved} is trivial when $\ell/k$ is algebraic, and otherwise becomes easier under our assumption \eqref{enum:semistablegenericopen}: given a Dedekind domain $C$, a non-trivial Jordan--H\"older filtration over $K(C)$ induces a non-trivial filtration in $\cP_c(\phi)$ for the open subset where every Jordan--H\"older factor restricts to a semistable object in $\cP_c(\phi)$.
\end{proof}

\section{Weak Harder--Narasimhan structures over a curve}
\label{sec:defnweakHNstructure}

In this section, we introduce the notion of a \emph{weak Harder--Narasimhan structure over a curve}.
From this section until the end of Part~\ref{part:HNStrCurve}, we work in Setup~\ref{setup-HN}.

\subsection{Definitions}\label{subsec:mainWeakDefn}
The following is the weak version of Definition~\ref{def:HNstructure_C}.
\begin{Def} \label{def:wHNstructure_C}
A \emph{weak Harder--Narasimhan structure on $\cD$ over $C$} consists of a triple $\sigma_C = (\ZK,\Zc,\cP)$, where
\index{sigmaC@$\sigma_C = (\ZK,\Zc,\cP)$, $C$ a Dedekind scheme!weak HN structure on $\cD$ over $C$!}
\begin{itemize}
\item $\cP$ is a slicing of $\cD$, and
\item $(\ZK,\Zc)$ is a family of central charges over $C$
\end{itemize}
satisfying the following properties:
\begin{description}
\item[$C$-linearity] The slicing $\cP$ is local over $C$.
\item[Compatibility] For all $\phi \in \R$ and all 
$0 \neq E \in \cP(\phi)$, we have either
\begin{eqnarray*}
E_K \neq 0 & \text{and} & 
Z_K(E_K) \in \begin{cases}
\R_{>0} \cdot e^{\ii \pi \phi} & \text{if $\phi \notin \Z$} \\
\R_{\geqslant 0} \cdot e^{\ii \pi \phi} & \text{if $\phi \in \Z$}
\end{cases}, \quad \text{or} \\
E \in \cD_{\Ctor} & \text{and} &\Zc(E) \in 
\begin{cases}
\R_{>0} \cdot e^{\ii \pi \phi} & \text{if $\phi \notin \Z$} \\
\R_{\geqslant 0} \cdot e^{\ii \pi \phi} & \text{if $\phi \in \Z$}.
\end{cases}
\end{eqnarray*}
\end{description}
\end{Def}

We have to keep in mind that for $E \in \cA_C$, we no longer have $E_K = 0$ if and only if $Z_K(E_K) = 0$.
There is also the following notion of a weak stability function on a local heart; 
we will see in Proposition~\ref{prop:stabviaheartweak} that weak HN structures can be constructed via a local heart with an appropriate weak stability function.

\begin{Def}\label{def:weakstabfunction_C}
Let $\cA_C \subset \cD$ be the heart of a bounded $C$-local t-structure.
A \emph{weak stability function for $\cA_C$ over $C$} is a pair of central charges $(\ZK,\Zc)$ on $\cD$ over $C$ such that $Z_K$ is a weak stability function on $\cA_K$, and $\Zc$ is a weak stability function on $\cA_{\Ctor}$.
\end{Def}

\begin{Ex} \label{ex:slopestability}
Assume that $g \colon \cX \to C$ has relative dimension $n$, and let $\cO_\cX(1)$ be a relative polarization.
For $E \in \cA_{\Ctor}$, let $p_n, p_{n-1}$ be the two leading coefficients of the Hilbert polynomial of $E$, defined as in Example~\ref{ex:slopestabilitycurves}.
Then
\[\Zc(E) := \mathfrak{i} p_{n}(E) - p_{n-1}(E)
\]
defines a weak stability function for $\Coh \cX$ over $C$.
\end{Ex}

\begin{Def} \label{def:ZCweak}
Given a weak stability function for $\cA_C$ over $C$, we define
\index{ZC@$Z_C$, central charge for objects in $\cA_C$}
\[
Z_C(E) := \begin{cases}
Z_K(E_K) & \text{if $E_K \neq 0$} \\
\Zc(E) & \text{otherwise.}
\end{cases}
\]
Then the slope $\mu_C(E) \in \R \cup \{+\infty\}$ for $E \neq 0$ is defined as before in Definition~\ref{def:ZC}.
\index{muC@$\mu_C$, slope for objects in $\cA_C$}
\end{Def}

This slope function satisfies a weak version of the see-saw property:
\begin{Lem} \label{lem:seesawweak}
Given a non-trivial short exact sequence $A \into E \onto B$ in $\cA_C$, we have
\begin{equation}
\mu_C(A) \leqslant \mu_C(E) \leqslant \mu_C(B)\ \text{or}\ \mu_C(A) \geqslant \mu_C(E) \geqslant \mu_C(B).
\end{equation}
If $E \in \cA_{\Ctor}$ and $Z_C(A) \neq 0 \neq Z_C(B)$, we moreover have
$\mu_C(A) < \mu_C(B) \Leftrightarrow \mu_C(E) < \mu_C(B)$.
\end{Lem}
\begin{proof}
By Corollary~\ref{cor:base-change-tstructure-point}.\eqref{enum:base-change-generic}, pullback to $\Spec(K)$ is t-exact.
Hence, if $E_K\neq0$, the weak seesaw property follows from the corresponding property on $(\cA_K,Z_K)$.
Similarly, if $E_K=0$, then $E\in \cA_\Ctor$, and 
the weak seesaw property follows from the corresponding property on $(\cA_\Ctor,\Zc)$.
\end{proof}

The weak see-saw property still allows us to define $Z_C$-semistability for objects in $\cA_C$ exactly as in Definition~\ref{def:semistable_C}.
HN filtrations and the HN property for weak stability functions for $\cA_C$ are defined as in Definition~\ref{def:satisfies_HN_C}.
As in Lemma~\ref{lem:inducedfiberprestability}, $\Zc$ induces a weak stability function on $\cA_p$ for every $p \in C$, compatible with semistability and HN filtrations, and thus:
\begin{Lem} \label{lem:inducedfiberprestabilityweak}
A weak HN structure $\sigma_C $ on $\cD$ over $C$ induces a weak pre-stability condition $\sigma_p$ on $\cD_p$ for every closed point $p \in C$. The pushforward $i_{p*}$ preserves semistability, phases and HN filtrations. 
\index{sigmac@$\sigma_c = (\cA_c, Z_c)$, $c\in C$ Dedekind scheme,!when $c=p$ is a closed point}
\end{Lem}

For a stability function on $\cA_C$ over $C$, $Z_C$-semistability of a $C$-torsion free 
object amounts to semistability of $E_p$ on \emph{all} fibers $p$, see Lemma~\ref{lem:allfibersstable}.
In the weak case, however, we can only detect the existence of destabilizing quotients $E_p \onto Q$ with $\mu(E_p) > \mu(Q)$; non-existence of such a quotient is not equivalent to $E_p$ being $Z_p$-semistable, in case $E_p$ has a subobject in $\cA_p^0$.

\begin{Lem} \label{lem:allfibersnotdestquot}
Let $E \in \cA_C$ be a $C$-torsion free object.
Then $E$ is $Z_C$-semistable if and only if $E_K$ is $\ZK$-semistable and $i_{W*}E_W \in \cA_{\Ctor}$ does not have a destabilizing quotient for all zero-dimensional subschemes $W \subset C$ or, equivalently, if and only if $E_K$ is $\ZK$-semistable and 
$E_p \in \cA_p$ does not have a destabilizing quotient for all closed points $p\in C$.
\end{Lem}

Let $W \subset C$ be a zero-dimensional subscheme, and consider Lemma~
\ref{lem:fibersinheart} for $E$ above: since $E$ is torsion free, it shows that $i_{W*}E_W = E/I_W \cdot E \in \cA_\Ctor$.

\begin{proof}
Assume first that $E$ is $Z_C$-semistable.
If $E_K$ is not $Z_K$-semistable,
then we can apply Lemma~\ref{lem-extend-from-localisation}.\eqref{enum:lem-extend-morphism-in-heart} to a destabilizing subobject $A_K\into E_K$ and lift it to a destabilizing subobject $ A \into E$.
We conclude that $E_K$ is $Z_K$-semistable.
Moreover, for any quotient $i_{W*}E_W \onto Q$, the composite map
$E \onto i_{W*}E_W \onto Q$
gives $\mu_C(Q)\geqslant\mu_C(E)=\mu_C(i_{W*}E_W)$.
Hence $i_{W*}E_W$ does not have a destabilizing quotient.

Conversely, suppose that we have a destabilizing sequence
\[
0 \to A \to E \to B \to 0,
\]
with $\mu_C(A)>\mu_C(B)$.
Since $E$ is $C$-torsion free, $A$ is also $C$-torsion free.
If $B_K\neq 0$, then $\mu_K(A_K)=\mu_C(A)> \mu_C(B)=\mu_K(B_K)$ shows that $E_K$ is not $\ZK$-semistable.
If $B\in \cA_\Ctor$, then let $W\subset C$ be such that $B$ is supported on $W$.
By Lemma~\ref{lem:fibersinheart}, we have
$i_{W*}E_W = E/I_W \cdot E$ since $E$ is torsion free, and the surjection
$E \onto B$ factors via $E \onto i_{W*}E_W = E/I_W \cdot E \onto B$.
Since $\mu_C(i_{W*}E_W) = \mu_C(E) > \mu_C(B)$, this gives a destabilizing quotient of $i_{W*}E_W$.

Finally, assume that $E_p \in \cA_p$, and therefore, $i_{p*}E_p \in \cA_{\Ctor}$ does not have a destabilizing quotient for any $p \in C$, but $B \in \cA_{\Ctor}$ is a destabilizing quotient
of $E$ in $\cA_C$.
We may assume that the subscheme $W$ supporting $B$ contains a single closed point $p \in W$; so we have
$I_W = I_p^m$ for some $m > 0$.
Then $B/I_p\cdot B$ is a quotient of $i_{p*}E_p$; by the see-saw property and the assumptions, it follows that $\mu_C(I_p\cdot B) < \mu_C(E)$.
However, $I_p \cdot B$ is naturally a quotient of $I_p \otimes E/\left(I_p^{m-1} \cdot I_p \otimes E\right) \cong E/I_p^{m-1}E$ supported on a subscheme of smaller length;
proceeding by induction, we get a contradiction.
\end{proof}

\begin{Rem}\label{rem:allfibersnotdestquot_plus}
In the context of Lemma~\ref{lem:allfibersnotdestquot}, assume further that $\cA_C$ has a $C$-torsion theory.
If $i_{W*}E_W \in \cA_{\Ctor}$ does not have a destabilizing quotient for all zero-dimensional subschemes $W \subset C$, then
$E_K$ is automatically $\ZK$-semistable.

Otherwise, we can lift the destabilizing surjection $E_K \onto B_K$ to a surjection $E \onto B$.
Let $W \subset C$ be any non-trivial closed subset.
Since $i_W^*$ is right t-exact, we have a sequence of surjections 
\[ E \onto i_{W*}E_W \onto i_{W*} \rH^0_{\cA_W}(B_W) \onto
i_{W*} \rH^0_{\cA_W}\left((B_\Ctf)_W\right) = i_{W*}\left((B_\Ctf)_W\right)
\]
where the last equality follows from Lemma~\ref{lem:FlatIffTFreeCurve}.
The composition is destabilizing since $Z\left((B_\Ctf)_W\right) = Z_K(B_K)$.
\end{Rem}

We now state and prove the equivalent of Proposition~\ref{prop:stabviaheart}.

\begin{Prop} \label{prop:stabviaheartweak}
To give a weak HN structure on $\cD$ over $C$ is equivalent to giving a heart
$\cA_C$ of a bounded $C$-local t-structure, together with a weak stability function $(\ZK,\Zc)$ on $\cA_C$ over $C$ satisfying the HN property.
\index{sigmaC@$\sigma_C = (\cA_C,\ZK,\Zc)$, $C$ a Dedekind scheme!weak HN structure on $\cD$ over $C$}
\end{Prop}
\begin{proof}
Suppose we are given a weak HN structure $\sigma_C = (\ZK,\Zc,\cP)$.
Since the slicing $\cP$ is $C$-local, the same follows for the heart $\cA_C := \cP(0,1]$.
Moreover, for $E \in \cA_C$, $\Zc(E)$ and $Z_K(E_K)$ are clearly in the semiclosed upper half plane $\H\sqcup \R_{\leqslant 0}$; thus, $(\ZK,\Zc)$ is a weak stability function on $\cA_C$ over $C$.

To show that $(\ZK,\Zc)$ satisfies the HN property, it suffices to show that $E\in\cP(\phi)$ is $Z_C$-semistable; then the HN filtration with respect to $\cP$ gives the HN filtration with respect to $(\ZK,\Zc)$.
So given $E \in \cP(\phi)$, we first show that $E$ has no subobject with $\mu_C(A) > \mu_C(E)$ using the HN filtration of $A$ with respect to $\cP$: indeed, if $A_1$ is the first step of that filtration, then $A_1 \into A \into E$ would be a chain of inclusions (in particular non-zero), and
thus $\phi(A_1) \leqslant \phi$.
This leads to a contradiction both when $E_K \neq 0$ and when $E \in \cA_{\Ctor}$.
Then one can easily show that $E$ has no quotients with $\mu_C(Q)<\mu_C(E)$ by an analogous argument; thus $E$ is $Z_C$-semistable, as required.

Conversely, for all $\phi \in (0, 1]$ we define $\cP(\phi)$ to be the category of $Z_C$-semistable objects $E \in \cA_C$ with $Z_C(E) \in \R_{>0} \cdot e^{\ii\pi\phi}$ for $\phi\not\in\Z$ and $Z_C(E) \in \R_{\geqslant 0} \cdot e^{\ii\pi\phi}$ for $\phi\in\Z$.
One verifies that $\cP$ is a slicing with the exact same arguments as in the case of weak stability conditions (or stability conditions as in \cite[Lemma~5.3]{Bridgeland:Stab}).
\end{proof}

\subsection{Existence of HN filtrations}\label{subsec:HNweak}

Our main tool in proving that a given weak stability function on $\cA_C$ over $C$ satisfies the HN property is the following result.

\begin{Prop} \label{prop:HNviaHNweak} Let
$(\ZK,\Zc)$ be a weak stability function on $\cA_C$ over $C$.
If it satisfies the HN property, then all of the following three conditions are satisfied:
\begin{enumerate}[{\rm (1)}] 
\item \label{enum:HNAKweak} The pair $(\cA_K, Z_K)$ satisfies the HN property.
\item \label{enum:HNAtorweak}
The pair $(\cA_{\Ctor},\Zc)$ satisfies the HN property.
\item \label{enum:ssredweak} (\emph{Semistable reduction}) For any $C$-torsion free object $E \in \cA_{\Ctf}$ such that $E_K \in \cA_K$ is $Z_K$-semistable, there is a $Z_C$-semistable subobject $F \subset E$ with $E/F \in \cA_{\Ctor}$.
\end{enumerate}
Moreover, if $\cA_C$ has a $C$-torsion theory $(\cA_{\Ctor}, \cA_{\Ctf})$, then the converse also holds true.
\end{Prop}

The explanations in Remark~\ref{rem:HNnecessaryconditions} show that these conditions are necessary.
In particular, $(\cA_K, \ZK)$ and
$(\cA_\Ctor,\Zc)$ are weak pre-stability conditions on $\cD_K$ and $\cD_\Ctor$, respectively.
We start the proof of the converse with the following observation:

\begin{Lem} \label{lem:semistsub}
Let $(\ZK,\Zc)$ be a weak stability function on $\cA_C$ over $C$, such that $\cA_C$ has a $C$-torsion theory $(\cA_{\Ctor}, \cA_{\Ctf})$ and 
condition \eqref{enum:ssred} of Proposition~\ref{prop:HNviaHNweak} holds.
Then the same condition holds for all objects $E$, not necessarily $C$-torsion free, with $E_K \in \cA_K$ being $Z_K$-semistable.
\end{Lem}

\begin{proof}
Let $F \into E_{\Ctf}$ be a semistable subobject as given by assumption \eqref{enum:ssred}.
We then apply Lemma~\ref{lem:tensortrick} to the isomorphism $(E_{\Ctf})_K \cong E_K$ and obtain an inclusion
\[ F \otimes g^* L^{-k} \into E_{\Ctf} \otimes g^* L^{-k} \into E.\qedhere\]
\end{proof}

Let $(\ZK,\Zc)$ be a (weak) stability function on $\cA_C$ over $C$.
We begin with more definitions:
\begin{Def}
We write
\index{muC+,muc-@$\mu_C^+(E)$ ($\mu_C^-(E)$), maximal (minimal) slope for objects in $\cA_C$}
\begin{eqnarray*}
\mu_C^+(E) &:=& \sup \set{\mu_C(F) \sth 0 \neq F \subseteq E}, \\
\mu_C^-(E) &:=& \inf \set{\mu_C(Q) \sth F \onto Q \neq 0}.
\end{eqnarray*}
\end{Def}

\begin{Def} \label{Def:MDS} Assume $E \in \cA_C$ is not $Z_C$-semistable.
A \emph{maximal destabilizing subobject} (mds) of $E \in \cA_C$ is a $Z_C$-semistable subobject $M \into E$
such that for all $F \into E$ we have $\mu_C(F) \leqslant \mu_C(M) $, and $\mu_C\bigl(F/(F \cap M)\bigr) < \mu_C(M)$ whenever $F$ is not a subobject of $M$.
\end{Def}
This is equivalent to $M$ being $Z_C$-semistable with $\mu_C(M) > \mu_C(F')$ for all $F' \into E/M$; in other words, $M$ is the first step of the HN filtration if it exists.

\begin{proof}[Proof of Proposition~\ref{prop:HNviaHNweak}] 
Assume conditions \eqref{enum:HNAKweak}-\eqref{enum:ssredweak} hold and $\cA_C$ has a $C$-torsion theory. 
We want to construct a mds $M \into E$. 
By assumption \eqref{enum:HNAtor}, we may assume that $E_K \in \cA_K$ is non-zero. By assumption \eqref{enum:HNAK}, it has a mds with respect to $Z_K$. Arguing as in the proof of Lemma~\ref{lem:allfibersnotdestquot},
we may assume it is of the form $N_K \subset E_K$ for some 
subobject $N \subset E$ in $\cA_C$.

Since $\cA_C$ has a $C$-torsion theory, we may assume that $N \into E$ is saturated, i.e., that $E/N$ is $C$-torsion free; in particular \[E_{\Ctor} = N_{\Ctor}.\] Write $\mu:= \mu_C(N) = \mu_K^+(E_K)$. 
 We distinguish two cases:
\begin{enumerate}[{\rm (A)}]
\item $E_{\Ctor} \neq 0$, and $\mu^+_{C}(E_{\Ctor}) > \mu$.

In this case, we apply assumption \eqref{enum:HNAtor} and let $M \subset E_{\Ctor}$ be the mds in
$\cA_{\Ctor}$; in particular $\mu_C(M)=\mu^+_{C}(E_{\Ctor})$.
We claim that $M$ is an mds of $E$ in $\cA_C$. 
Indeed, given any subobject $F \into E$, either $\mu_C(F) \leqslant \mu < \mu_C(M)$ or $\mu_C(F) > \mu$. 
In the second case, since $\mu = \mu_K^+(E_K)$ it follows that $F$ is $C$-torsion and factors through $E_{\Ctor} \subset E$, and hence $\mu_C(F) \leqslant \mu_C(M)$ by the choice of $M$. 
Thus $\mu_C(F) \leqslant \mu_C(M)$ for any subobject $F \into E$. 
Now assume $F$ is not a subobject of $M$. 
If $F$ is $C$-torsion, then $F \subset E_{\Ctor}$, so $\mu_C(F/(F \cap M)) < \mu_C(M)$ follows by the choice of $M$. 
If $F$ is not $C$-torsion, then $F_K = (F/(F \cap M))_K$ and again we find $\mu_C(F/(F \cap M)) = \mu_K(F_K) \leqslant \mu < \mu_C(M)$.

\item $E_{\Ctor}= 0$, or $E_{\Ctor} \neq 0$ with $\mu^+_{C}(E_{\Ctor}) \leqslant \mu$.

Let $F \subset N$ be the $Z_C$-semistable subobject given by assumption \eqref{enum:ssred} or Lemma~\ref{lem:semistsub}, respectively.
By the existence of HN filtrations in $\cA_{\Ctor}$, and the resulting torsion pair as in \eqref{eqn:torsion pair}, there exists a subobject $F \subseteq M \subseteq N$ such that $\mu^-_C(M/F) \geqslant \mu > \mu^+_C(N/M)$. We claim that $M$ is an mds for $E$.

Note that $M_K = N_K$ is $Z_K$-semistable. Thus, if $M$ is not semistable, and if
$A \subset M$ is a subobject with $\mu_C(A) > \mu_C(M/A)$, then either $A$ is $C$-torsion with
$\mu_C(A) > \mu$, or $M/A$ is $C$-torsion with $\mu > \mu_C(M/A)$. The former case is
impossible, as $A$ would be a subobject of $N_{\Ctor} = E_{\Ctor}$ and contradict the assumption
on $\mu_C^+(E_{\Ctor})$. In the latter case, consider the short exact sequence
\[ F/(F \cap A) \into M/A \onto M/(F + A). \]
We have $\mu_C(F/(F\cap A)) \geqslant \mu$ by semistability of $F$, and $\mu_C(M/(F+A)) \geqslant \mu_C^-(M/F) \geqslant \mu$ by construction; thus we get a contradiction to the see-saw property.

Since $E/N$ is $C$-torsion free by assumption, any subobject $F \into E/N$ satisfies 
$\mu_C(F) = \mu_K(F_K) \leqslant \mu_K^+\left((E/N)_K\right) < \mu$; so $\mu_C^+(E/N) < \mu$. Combined with the inequality 
$\mu^+_C(N/M) < \mu$, this shows $\mu_C^+(E/M) < \mu$. 
\end{enumerate}
We have thus produced a mds for $E$. To conclude, we need to show that if we replace
$E$ by $E/M$ and repeat the above procedure, the process terminates.
Indeed, in case (A) we preserve $E_K$ and reduce the length of the HN filtration of $E_{\Ctor}$,
and in case (B) we reduce the length of the HN filtration of $E_K$.
\end{proof}

We will often need $\cA_{\Ctor}^0$ and $\cA_K^0$ to be noetherian torsion subcategories, see Definition~\ref{def:noetheriantorsionsubcat}. 
For example, the following observation extends Lemma~\ref{lem:stabheartnoetherian}. 

\begin{Prop} \label{prop:discreteimpliesNoetherian}
Assume that the heart $\cA_C$ has a $C$-torsion theory, and that
$\cA_{\Ctor}^0 \subset \cA_{\Ctor}$ and $\cA_K^0 \subset \cA_K$ are noetherian torsion subcategories.
If there exists a weak HN structure
$\sigma_C = (\cA_C,\ZK,\Zc)$ over $C$ with heart
$\cA_C$ such that $\Zc$ has discrete image, then $\cA_C$ is noetherian.
\end{Prop}

\begin{proof}
Let $E_1 \onto E_2 \cdots $ be an infinite sequence
of surjections in $\cA_C$. Applying Lemma~\ref{lem:discreteimpliesNoetherian} to $(\cA_K, Z_K)$, we see that the induced sequence of surjections $(E_1)_K \onto (E_2)_K \onto \dots$ stabilizes, in other words we may assume that the kernel 
of every surjection $E_i \onto E_{i+1}$ is $C$-torsion. However, by assumption
and Lemma~\ref{lem:discreteimpliesNoetherian} applied to $(\cA_\Ctor,\Zc)$ we know that $\cA_\Ctor \subset \cA_C$ is a noetherian torsion subcategory; therefore, see Remark~\ref{rem:alternatenoetheriantorsion}, this sequence terminates.
\end{proof}

\section{Semistable reduction}\label{sec:SemistableReduction}

The aim of this section is to study in detail condition \eqref{enum:ssred} in Proposition~\ref{prop:HNviaHN} or \ref{prop:HNviaHNweak}.
We first show that Langton and Maruyama's semistable reduction also works in our context.
Then condition \eqref{enum:ssred} follows if semistability satisfies generic openness (see Definition~\ref{def:GenericOpennessSemiStab}).
This gives us the first example of an HN structure, namely coherent sheaves.

\subsection{The Langton-Maruyama Theorem}\label{subsec:Langton}

We now state the main result of this section.
It is the analogue of \cite[Theorem~2.B.1]{HL:Moduli} (see \cite[Lemma~3.4]{Yoshioka:EllipticSurfaces} or \cite[Lemma~2.5]{HMS:Orientation}); the proof follows the same lines.
The main difference is that we use the existence of a locally finite type Quot space to circumvent the use of completed local rings; 
indeed, our results in 
Part~\ref{part:sod-t-structure-families} do not prove the existence of a heart 
on base changes to such rings.

\begin{Thm}[Langton-Maruyama]\label{thm:Langton}
Let $(\ZK,\Zc)$ be a (weak) stability function on $\cA_C$ over $C$ such that: 
\begin{enumerate}[{\rm (1)}] \setcounter{enumi}{-1}
\item \label{enum:finitecoverZconstant}
Given a Dedekind scheme $D$, a dominant map $f \colon D \to C$ essentially of finite type, a closed point $q \in D$ and
$F \in \cD_D$, we have
\[
\frac 1{\dim_{\kappa(f(q))}\kappa(q)} Z_{f(q)}(f_* F_q) = \frac 1{\dim_{K}K(D)} Z_K\left(f_* F_{K(D)}\right).
\]
\item\label{enum:Langtonopennessflatness} $\cA_C$ universally satisfies openness of flatness. 
\item\label{enum:LangtonHNsupport} $(\cA_\Ctor,\Zc)$ defines a (weak) stability condition on $\cD_\Ctor$. 
\item\label{enum:Langtonweak} If $(\ZK,\Zc)$ is only a weak stability function, we also assume that
$(\cA_{\Ctor},\Zc)$ has the tilting property.
\end{enumerate}
Let $E\in\cA_C$ be a $C$-flat object whose restriction $E_K\in\cA_K$ to the generic fiber is $Z_K$-semistable. 
If there is a closed point $p\in C$ and a quotient $i_{p*}E_p\onto Q$ with 
$\mu_C(Q)<\mu_C(i_{p*}E_p)$, then there exists a subobject $E'\subset E$ such that $E/E' \in \cA_{\Ctor}$ is supported over $p$,
$ \mu_C(E')>\mu_C(E/E')$, 
and $\mu_C(Q')\geqslant\mu_C(i_{p*}E'_p)$ for all quotients $i_{p*}E'_p\onto Q'$.
\end{Thm}

Assumption \eqref{enum:finitecoverZconstant} is a weak version of \emph{universal local constancy of central charges}, see Definition~\ref{definition-uogs}.\eqref{ulccc}, which we study later in the context of families of stability conditions. 
Openness of flatness in assumption \eqref{enum:Langtonopennessflatness} refers to Definition~\ref{def-open-flat-utau} for the collection of fiberwise t-structures induced by $\cA_C$. 
We remind the reader that under this assumption, Quot spaces are locally of finite type over $C$ by Proposition~\ref{proposition-quot-algebraic}.
We will explore the support property for $(\cA_\Ctor,\Zc)$, part of assumption \eqref{enum:LangtonHNsupport}, in more detail in Section~\ref{subsec:support-for-HNstr}.

We denote the local ring at $p$ by $R$, and write $\pi$ for a generator of its maximal ideal.

\begin{proof}
If the claim of the theorem were not true, we could define a descending filtration
\begin{equation}\label{eqn:LangtonFiltration}
\dots\subset E_{n+1} \subset E_n \subset \dots \subset E_0:=E,
\end{equation}
with $E_n/E_{n+1}$ a $C$-torsion object supported over $p$ and $\mu_C(E_{n+1})>\mu_C(E_n/E_{n+1})$, as follows.
Let $Q^n$ be the maximal destabilizing quotient for $i_{p*}(E_n)_p$, which exists by condition \eqref{enum:LangtonHNsupport};
this means that we can write the last step of the HN filtration as $F^n \into i_{p*}(E_n)_p$ where
$Q^n = i_{p*}(E_n)_p/F$ is the associated semistable quotient, and $\mu_C^-(F^n) > \mu_C(Q^n)$.
Then $E_{n+1}$ is defined as the kernel of the composition $E_{n}\onto i_{p*}(E_n)_p \onto Q^n$; namely, $E_n/E_{n+1}=Q^n$.

Applying $i_{p*}i_p^*$ to $E_{n+1} \into E_n \onto Q^n$, taking cohomology, and using Lemma~\ref{lem:cohpushpull} for $Q^n$ gives an exact sequence
\begin{equation}\label{eqn:Langton}
0 \to Q^n \to i_{p*}(E_{n+1})_p \to F^n \to 0.
\end{equation}

Consider the sequence of maps $\phi^n\colon Q^n\into i_{p*}(E_{n+1})_p \onto Q^{n+1}$, which fit into a sequence of commutative diagrams with exact rows and columns
\[
\xymatrix{
0\ar[r]&Q^n\ar[r]\ar[d]\ar^{\phi^n}[dr]& i_{p*}(E_{n+1})_p\ar[r]\ar[d]& F^n\ar[r]\ar[d]&0\\
0\ar[r]&\im(\phi^n)\ar[r]\ar[d]& Q^{n+1}\ar[r]\ar[d]& \cok(\phi^n)\ar[r]\ar[d]&0.\\
&0&0&0
}
\] 
Suppose that $\phi^n=0$. 
Then $F^n\twoheadrightarrow Q^{n+1}$, 
but by definition $\mu_C^-(F^n)>\mu_C(Q^{n+1})$, so as a consequence of the see-saw property, $\Hom(F^n,Q^{n+1})=0$. 
Hence $\phi^n\neq 0$. 

Then as $Q^n$ and $Q^{n+1}$ are both $Z_C$-semistable, we must have 
\[\mu_C(Q^n)\leqslant\mu_C(\im (\phi^n))\leqslant\mu_C(Q^{n+1}).\]
By construction, we have $\mu_C(Q^n)<\mu_C(E_{n+1})=\mu_C(E)$, so $$\mu_C(Q^1)\leqslant\mu_C(Q^n)<\mu_C(E)$$ for all $n\geqslant 1$. 
Similarly, we have $0<\Im Z_C(Q^n)\leqslant \Im Z_C(E)$ for all $n\geqslant 1$, so the central charges $\{Z_C(Q^n)\}$ lie in a bounded subset of $\C$. 
As the $Q^n$ are all $Z_C$-semistable, it follows from Condition \eqref{enum:LangtonHNsupport} (support property on $\cD_\Ctor$) and Remark~\ref{rem:supportfinitelength} that there are only finitely many values for $Z_C(Q^n)$. 
In particular, the increasing chain $\mu_C(Q^n)$ must stabilize, so we may assume that $\mu_C(Q^n)=\mu_C(Q^{n+1})$ from the outset. 

We would like to show that the values $Z_C(Q^n)$ stabilize as well. 
To that end, assume that $\Im\Zc(\im(\phi^n))<\Im\Zc(Q^{n+1})$, so that $\Im\Zc(\cok(\phi^n))>0$. 
Then as $$\mu_C(Q^n)=\mu_C(\im(\phi^n))=\mu_C(Q^{n+1})$$ we must have $\mu_C(\cok(\phi^n))$ equal to the same number by the see-saw property in $\cA_\Ctor$, see Lemma~\ref{lem:seesawweak}.
Again, by definition $\mu_C^-(F^n)>\mu_C(Q^{n+1})=\mu_C(\cok(\phi^n))$, so as a consequence of the see-saw property, $\Hom(F^n,\cok(\phi^n))=0$. 
This contradicts the fact that we have a surjection $F^n \twoheadrightarrow \cok(\phi^n)$ by construction, where $\cok(\phi^n) \neq 0$ by our assumption $\Im\Zc(\im(\phi^n))<\Im\Zc(Q^{n+1})$. 
Thus we must have \[\Im\Zc(Q^n)\geqslant \Im\Zc(\im(\phi^n))=\Im\Zc(Q^{n+1})>0.\]
But as the $Z_C(Q^n)$ take only finitely many values, we see that for $n\gg0$, $Z_C(Q^n)$ must indeed stabilize. 

It follows that $\Ker(\phi^n)$ must satisfy $Z_C(\Ker(\phi^n))=0$ for $n \gg 0$. But then $\Ker(\phi^n)=0$ for $n\gg0$ by the $Z_C$-semistability of $Q^n$.
So we may assume from the beginning that we have an ascending chain of $Z_C$-semistable $C$-torsion objects
\[Q^1\subset Q^2\subset \cdots\subset Q^n\subset Q^{n+1}\subset\cdots\] 
with $Z_C(Q^n)=Z_C(Q^{n+1})$ for all $n\geqslant 1$. 

Now we have $Q^{n+1}/Q^n\in\cA_\Ctor^0$ for all $n\geqslant 1$. Since $(\cA_\Ctor,\Zc)$ has the tilting property by Condition \eqref{enum:Langtonweak}, this sequence stabilizes, see Remark~\ref{rem:tiltingpropertyextensions}.

We can conclude now that $\phi^n$ is an isomorphism for all $n \gg 0$, and so the exact sequence \eqref{eqn:Langton} splits, i.e., $i_{p*}(E_{n+1})_p\cong F^n \oplus Q^n$.
In particular, the objects $F^n$ are also constant.
Let us set $F:=F^n$ and $Q:=Q^n$, for all $n\gg0$.
Up to replacing $E$ with $E_n$, for $n\gg0$, we can assume that our filtration \eqref{eqn:LangtonFiltration} has the property that $i_{p*}(E_n)_p\cong F \oplus Q$ for all $n$.

We now restrict to the local ring $R = \cO_{C, p}$. 
Note that by Theorem~\ref{Thm-D-bc} there is an induced heart on the base change $\cD_{R}$, such that the restriction functor $\cD \to \cD_{R}$ is t-exact. 
We abuse notation by still using $E$ and $Q$ to denote the restrictions of these objects to $\cD_{R}$. We also set $G^n:=E/E_n$.
As $E_{n-1}/E_{n}=Q$, there exists an exact sequence
\begin{equation} \label{eq:GnGn-1}
 0 \to Q \to G^{n}\to G^{n-1} \to 0.
\end{equation}
Also, by construction, we have $\pi E_{n-1} \subset E_n$, so by induction it follows that $\pi^n E\subset E_n$. Therefore, the quotient $E\onto G^n$ factors through $E \onto E/\pi^n E\onto G^n$ so that $\pi^n$ acts as zero on 
$G^n$. By Lemma~\ref{lem:fibersinheart} it follows that $G^n$ is the pushforward of an object in $\cA_{R/(\pi^{n})}$, which we also denote by $G^n$. 

Moreover, it is not difficult to see from the construction that $E_{n-1}$ is the span of $E_n$ and $\pi E_{n-2}$ (as subobjects of $E$).
By induction, this shows that $E_{n-1}$ is the span of $E_n$ and $\pi^{n-1}E$; hence $\pi^{i}G^n/\pi^{i+1}G^n = E_i/E_{i+1} = Q$, with isomorphisms of the filtration steps induced by multiplication by $\pi$. 
By Lemma~\ref{lem:flatonthickening} below, this shows that $G^n$ is a flat object in $\cA_{R/(\pi^n)}$. 

By Condition~\eqref{enum:Langtonopennessflatness} and Proposition~\ref{proposition-quot-algebraic}, 
the Quot space $\Quot_{\Spec(R)}(E)$ is an algebraic space locally of finite type over $\Spec(R)$. 
The quotients $E/\pi^nE \onto G^n$ give a compatible system of $R/(\pi^n)$ points of $\Quot_{\Spec(R)}(E)$. 
It follows that the Quot space admits a point finite over $K$ that has the point $E_p \onto Q$ as a specialization.
Therefore, there exists the spectrum $D$ of a DVR over $R$, an element $E_D \onto F$ of $\Quot_{\Spec R}(E)(D)$ that on the special fiber $q \in D$ is given by a base change of $E_p \onto Q$, and such that the composition $D \to \Spec R \to C$ is dominant and essentially of finite type. As $E_q \onto F_q$ is destabilizing, the same holds by condition~\eqref{enum:finitecoverZconstant} for $E_K^{\oplus \dim_K K(D)} \onto f_* F_K$. This contradicts the $Z_K$-semistability of $E_K$.
\end{proof}

In the proof, we needed the following standard flatness criterion for modules in our context:
\begin{Lem} \label{lem:flatonthickening}
Let $p \in C$, let $R$ be its local ring, and write $R_n = R/(\pi^n)$ for the ring defining the $(n-1)$-th infinitesimal neighborhood of $p$. Then an object $B \in \cA_{R_n}$ is flat if and only if multiplication by $\pi$ induces isomorphisms
\[ B/\pi B \cong \pi B/\pi^2 B \cong \dots \cong \pi^{n-1}B. \]
\end{Lem}
\begin{proof}
Let $\cX_{R_n}$ be the base change of $\cX$ to $\Spec(R_n)$, and write $j$ for the inclusion $\cX_p \into \cX_{R_n}$.
Since $j_*$ is exact, it is enough to test whether $j_* j^*B = B \otimes R_n/(\pi) \in \cA_{R_n}$.
Using the (2-periodic) minimal resolution $R_n/\pi = \dots R_n \xrightarrow{\cdot \pi^{n-1}} R_n \xrightarrow{\cdot \pi} R_n$ and its naive truncation at $\leqslant -2$, we obtain an exact triangle
\[
j_*j^* B[1] \to \bigl( B\xrightarrow{\pi} B \bigr) \to j_* j^* B.
\]
Since $j^*$ is right-exact, it follows that $\rH^{-1}_{\cA_{R_n}}(j_* j^* B) = 0 $ if and only if the map $B/\pi \xrightarrow{\cdot \pi^{n-1}} \Ker \pi$ is surjective, and $\rH^{-2}_{\cA_{R_n}}(j_* j^* B) = 0 $ if and only if it is injective. This is easily seen to be equivalent to the assumption of the lemma.
\end{proof}

In our examples, condition \eqref{enum:Langtonweak} in Theorem~\ref{thm:Langton} will follow from a compatibility of the weak stability condition with the duality functor.
Then the existence of HN filtrations will be implied by the following generic openness property. 

\begin{Def}\label{def:GenericOpennessSemiStab}
A (weak) stability function $(\ZK,\Zc)$ on $\cA_C$ over $C$ satisfies {\em generic openness of semistability} if the following condition holds:
given $E\in\cA_C$ a $C$-torsion free object such that $E_K$ is $Z_K$-semistable, there exists a nonempty open subset $U\subset C$ such that for all $p\in U$ and for all quotients $i_{p*}E_p\onto Q$, we have $\mu_C(Q)\geqslant \mu_C(E)$.
\end{Def}

\begin{Rem}
\label{remark-HN-generic-openness} 
If $(\ZK,\Zc)$ is a (weak) stability function on $\cA_C$ over $C$ which satisfies the HN property, then generic openness of semistability holds. 
Indeed, if $E \in \cA_C$ is $C$-torsion free such that $E_K$ is $Z_K$-semistable, then by Proposition~\ref{prop:HNviaHNweak} there is a $Z_C$-semistable subobject $F \subset E$ with $F_K = E_K$. Then there exists a nonempty open subset $U \subset C$ such that $F_U = E_U$, and the claim follows from Lemma~\ref{lem:allfibersnotdestquot} applied to $F$. 
\end{Rem}

\begin{Cor}\label{cor:HNgeneral}
With the same assumptions as in Theorem~\ref{thm:Langton}, assume further that $\cA_C$ has a $C$-torsion theory, that $(\cA_K, Z_K)$ has the HN property, and that generic openness of semistability holds. Then the (weak) stability function $(Z_K,Z_{C\text{-}\mathrm{tor}})$ on $\cA_C$ over $C$ satisfies the HN property.
\end{Cor}

Theorem~\ref{thm:Ctorsiontheoryautomatic} below shows the existence of a $C$-torsion theory under the above hypotheses.

\begin{proof}
We use Proposition~\ref{prop:HNviaHNweak}; we only need to verify assumption \eqref{enum:ssredweak}, semistable reduction.

Let $E\in\cA_C$ be a $C$-torsion free object such that $E_K$ is $Z_K$-semistable.
By generic openness of semistability, there exists a non-empty open subset $U\subset C$ such that $i_{p*}E_p$ does not have a destabilizing quotient for all closed points $p\in U$.
Hence, we are left with finitely many points in $C$, to which we apply Theorem~\ref{thm:Langton}.
We thus obtain a subobject $F\into E$ such that $E/F\in\cA_{C\text{-}\mathrm{tor}}$ and $i_{p*}F_p$ does not have a destabilizing quotient for all closed points $p\in C$.
By Lemma~\ref{lem:allfibersnotdestquot}, $F$ is $Z_C$-semistable, as we wanted.
\end{proof}

\subsection{An example: coherent sheaves}\label{subsec:CoherentSheaves}

Let us consider the stability function $(Z_K,Z_{C\text{-}\mathrm{tor}})$ on $\Coh \cX$ given by slope-stability on the fibers as in Example~\ref{ex:slopestability}. That is, given a flat morphism $g \colon \cX \to C$ of relative dimension $n$ and 
$\cO_\cX(1)$ a relative polarization, we define 
\[\Zc(E) := \mathfrak{i} p_{n}(E) - p_{n-1}(E),
\]
where $p_n$ and $p_{n-1}$ are the two leading coefficients of the Hilbert polynomial of $E$, defined as in Example~\ref{ex:slopestabilitycurves}.
We define $\ZK$ similarly over the generic fiber. As $\cO_\cX(1)$ is globally defined, $(\ZK,\Zc)$ defines a weak stability function for $\Coh \cX$ over $C$.
Here the subcategory $(\Coh \cX)_\Ctor$ consists of sheaves supported on fibers of $g$ and $(\Coh \cX)_\Ctf$ consists of flat families of sheaves on the fibers of $g$.
Moreover, it is clear that a $Z_C$-semistable sheaf $E\in(\Coh \cX)_\Ctor$ is just a sheaf supported on fibers of $g$ that is $\mu$-semistable in the classical sense \cite[Definition~1.2.12]{HL:Moduli}.
It follows that $E\in(\Coh \cX)_\Ctf$ is $Z_C$-semistable if and only if $E$ is a family of $\mu$-semistable sheaves over $C$ in the classical sense. 

\begin{Prop}\label{prop:thingsworkforsheaves}
If the fibers of $g\colon \cX\to C$ are normal, integral, noetherian schemes, then the weak stability function $(\ZK,\Zc)$ on $\Coh \cX$ satisfies the HN property.
\end{Prop}

\begin{proof}
We use Corollary~\ref{cor:HNgeneral}, for which we must first verify the assumptions of Theorem~\ref{thm:Langton}. 
Assumptions~\eqref{enum:finitecoverZconstant} and \eqref{enum:Langtonopennessflatness} are standard. 
As $\Coh \cX$ is noetherian, it admits a $C$-torsion theory by Remark~\ref{rem:noetherianandtorsiontheory}, and of course the subcategory $(\Coh \cX)_\Ctor^0$, which consists of sheaves supported in codimension at least 2 in fibers, is also noetherian.
The existence of HN filtrations for objects in $(\Coh \cX)_\Ctor$ or $\Coh \cX_K$ follows exactly as in the construction of HN filtrations for classical slope-stability, see \cite[Theorem~1.6.7]{HL:Moduli}.
For the support property, we may simply choose $Q=0$ and $\Lambda$ as in \cite[Section~3.2]{PT15:bridgeland_moduli_properties}, which verifies assumption \eqref{enum:LangtonHNsupport}. 

Next we verify assumption \eqref{enum:Langtonweak}.
Take $E\in(\Coh \cX)_\Ctor$ with $\mu_C^+(E)<+\infty$, and let us assume without loss of generality that $E=i_{p*}E'$ for $E'\in\Coh \cX_p$.
The condition $\mu_C^+(E)<\infty$ means that $E$ is a torsion free sheaf on $\cX_p$, so $E\into E^{\vee\vee}$ by the integrality of $\cX_p$, where $E^{\vee\vee}:=\lHom(\lHom(E,\cO_{\cX_p}),\cO_{\cX_p})$.
By the normality of $\cX_p$, we get that $E^{\vee\vee}$ is reflexive and the quotient $E^{\vee\vee}/E$ is supported on $\cX_p$ in codimension at least 2.
Thus $E^{\vee\vee}/E\in(\Coh \cX)_{\Ctor}^0$. The reflexivity of $E^{\vee\vee}$ gives the vanishing of $\Hom((\Coh \cX)_\Ctor^0,E^{\vee\vee}[1])$, finishing the verification of assumption \eqref{enum:Langtonweak}.

Finally, as classical slope-stability satisfies generic openness of semistability by \cite[Proposition~2.3.1]{HL:Moduli}, the proposition follows from Corollary~\ref{cor:HNgeneral}.
\end{proof}

\section{Torsion theories and Harder--Narasimhan structures}\label{sec:HNstructuresTorsiontheory}

\subsection{A \texorpdfstring{$C$}{C}-torsion theory via semistable reduction} 
The goal of this section is to show that the existence of a $C$-torsion theory is automatic in our setting; our proof is similar to that of semistable reduction, Theorem~\ref{thm:Langton}.
This gives a partial converse to Proposition~\ref{prop:Ctorsiontheory-opennessflatness}, which shows that the existence of a $C$-torsion theory implies openness of flatness. 

\begin{Thm} \label{thm:Ctorsiontheoryautomatic}
Let $\cA_C$ be heart of a $C$-local t-structure on 
$\cD$, with the following assumptions:
\begin{enumerate}[{\rm (1)}] 
 \item $\cA_C$ universally satisfies openness of flatness.
 \item For every closed point $p \in C$, there exists
 a (weak) stability condition $\sigma_p = (\cA_p, Z_p)$; in the weak case, we also assume that $\cA_p^0 \subset \cA_p$ is a noetherian torsion subcategory.
\end{enumerate}
Then $\cA_C$ admits a $C$-torsion theory.
\end{Thm}

The crucial ingredient will be the following claim:

\begin{Lem} \label{lem:Ipntorsionfree}
With the same assumptions as in Theorem~\ref{thm:Ctorsiontheoryautomatic}, 
if $p \in C$ and $E \in \cA_C$, then there exists $n >0$ such that $I_p^n \cdot E$ has no torsion subobject supported over $p$.
\end{Lem}

\begin{proof}[Proof of Theorem~\ref{thm:Ctorsiontheoryautomatic}]
Consider $E \in \cA_C$.
Since $\cA_C$ satisfies openness of flatness, and since $E_K \in \cA_K$, there exists a finite set of closed points $p_1, \dots, p_m$ such that $E$ is flat on the complement
$C \setminus \{p_1, \dots, p_m\}$. For each $i$, let 
$n_i$ be such that $I_{p_i}^{n_i} \cdot E$ has no torsion supported over $p_i$ as in Lemma~\ref{lem:Ipntorsionfree}. Since $I_{p_i}^{n_i} \cdot E$ and $E$ are isomorphic on the complement of $p_i$, one sees easily by induction that $\left(I_{p_1}^{n_1} \cdot \dots \cdot I_{p_m}^{n_m} \cdot E\right)\otimes I_{p_1}^{-n_1}\otimes\cdots\otimes I_{p_m}^{-n_m}$ is a $C$-torsion free quotient of $E$ whose kernel is $C$-torsion.
\end{proof}

We now turn to the proof of Lemma~\ref{lem:Ipntorsionfree}.
Let $\pi$ be a local generator of $I_p$ around $p$; it acts on any torsion object supported over $p$.
Consider the exact sequence 
\begin{equation*} 
I_p^{i+1} \cdot E \into I_p^i\cdot E \onto I_p^i \cdot E/I_p^{i+1} \cdot E
\end{equation*}
for $i \geqslant 0$. Applying $i_{p*}i_p^*$, taking cohomology with respect to $\cA_C$, and using Lemmas~\ref{lem:cohpushpull} and \ref{lem:fibersinheart} gives a four term exact sequence
\begin{equation} \label{eq:4termIpEsequence}
0 \to \Ann_\pi \left(I_p^{i+1} \cdot E\right) \xrightarrow{\alpha_i}
\Ann_\pi(I_p^i \cdot E) \xrightarrow{\beta_i} I_p^i \cdot E/I_p^{i+1} \cdot E \xrightarrow{\gamma_i}
I_p^{i+1}\cdot E/I_p^{i+2} \cdot E \to 0
\end{equation}
where we used the abbreviation $\Ann_\pi(F) = \Ann(I_p ; F)$.
In particular, there is a sequence of surjections
\begin{equation} \label{eq:Emodpsurjections}
 E/I_p \cdot E \xrightarrowdbl{\gamma_0} I_p\cdot E/I_p^2 \cdot E \xrightarrowdbl{\gamma_1} I_p^2 \cdot E/I_p^3 \cdot E \xrightarrowdbl{\gamma_2} \cdots
\end{equation}
Our key claim is the following:

\begin{Lem}
The sequence \eqref{eq:Emodpsurjections} terminates.
\end{Lem}

\begin{proof}
For $i \geqslant 0$ we define 
\begin{align*}
F_i & = \im{(\Ann_\pi(I_p^i \cdot E) \xrightarrow{\ \beta_i \ } I_p^i \cdot E/I_p^{i+1} \cdot E)} , \\ 
Q_i & = \cok{(\Ann_\pi(I_p^{i+1} \cdot E) \xrightarrow{\ \alpha_0 \circ \dots \circ \alpha_i \ } \Ann_\pi(E))} , \\ 
K_i & = \Ker{(E/I_p \cdot E \xrightarrow{\gamma_i \circ \dots \circ \gamma_0} I_p^{i+1} \cdot E/I_p^{i+2} \cdot E)}. 
\end{align*}
Note that by construction, all of these objects are scheme-theoretically supported over $p$, and hence may be regarded as objects of $\cA_p$. 
Since $\Ann_\pi(E) = i_{p*} \rH_{\cA_p}^{-1}(E_p)$ by Lemma~\ref{lem:fibersinheart}, 
the slope of $Q_i$ is bounded below by 
$\mu_p(Q_i) \geqslant \mu_p^{-}(\rH_{\cA_p}^{-1}(E_p))$. On the other hand, we have $Q_0 = F_0 = K_0$ 
and short exact sequences
\[ K_{i-1} \into K_i \onto F_i \quad \text{and} \quad
F_i \into Q_i \onto Q_{i-1}.
\]
By induction it follows that $K_i$ and $Q_i$ have the same class $K(\cD_p)$, and therefore the slope $\mu_p(K_i) = \mu_p(Q_i)$ satisfies the same bound.
Thus our claim follows from Lemma~\ref{lem:boundedslopenoetherian}.
\end{proof}

\begin{proof}[Proof of Lemma~\ref{lem:Ipntorsionfree}]
Let $R_n = \cO_C/I_p^n$ be the $n$-th infinitesimal neighborhood of $p$.
By the previous lemma, after replacing $E$ if necessary, we may assume that the sequence \eqref{eq:Emodpsurjections} is a sequence of isomorphisms; therefore, we have
$\Ann_\pi(I_p^i\cdot E) = \Ann_\pi(E)$ for all $i$. By Lemma~\ref{lem:flatonthickening} this means that
$F_n := \Ann_{\pi^n}(E) \in \cA_{R_n}$ is a flat object over $R_n$, which we assume to be non-zero for contradiction.

Therefore, we get a sequence of compatible morphisms
$\Spec(R_n) \to \cM_{\utau}$, where $\cM_{\utau}$ is the functor of flat objects with respect to the fiberwise collection of t-structures induced by $\cA_C$.
By Lemma~\ref{lem-moduli-flat-objects}, $\cM_{\utau}$ is an algebraic stack which is locally of finite type over $C$.
By Artin approximation, there exists a Dedekind domain $h \colon D \to \cM_{\utau}$ of finite type over $C$, together with a point $q \in D$, such that $D \to C$ is dominant, maps $q$ to $p$ and is \'etale at $q$, and such that $h$ is induced by $F_n$ in the $n$-th infinitesimal neighborhood of $q$, which we can identify with $\Spec R_n$.
Let $F$ be the object corresponding to $h$. 
We consider the object $\cHom_D(F, E_D) \in \Db(D)$ given by the relative derived sheaf Hom. 
Then we have 
\[ \rH^{-1} \left( K(D) \otimes \cHom_D(F, E_D) \right) 
= \Hom(F_{K(D)}, E_{K(D)}[-1]) = 0. \]
On the other hand, we have 
\begin{align*} 
\rH^{-1} \left(R_n \otimes \cHom_D(F, E_D)\right) & = 
\Hom(F_n, E_{R_n}[-1]) \\ 
& = \Hom(F_n, \rH^{-1}_{\cA_{R_n}}(E_{R_n})) \\ 
&= \Hom(F_n, F_n)
\end{align*}
where the second two equalities follow from Lemma~\ref{lem:fibersinheart}; 
in particular, this cohomology sheaf contains $R_n$ as a subsheaf for all $n > 0$. 
This contradicts the above vanishing.
\end{proof}

\subsection{Harder--Narasimhan structures and torsion theories}

In this section, we show that the existence of a $C$-torsion theory on the heart descends to the slices of a Harder--Narasimhan structure, and vice versa.

\begin{Def}
	We say that a (weak) HN structure $\sigma_C = (\ZK,\Zc,\cP)$ has a $C$-torsion theory if for every $\phi \in \R$ the category $\cP(\phi)$
	admits a torsion pair $\left(\cP(\phi)_\Ctor, \cP(\phi)_\Ctf\right)$ where
	$\cP(\phi)_\Ctor = \cP(\phi) \cap \cD_\Ctor$ and 
	$\cP(\phi)_\Ctf = \cP(\phi)_\Ctor^\perp$.
\end{Def}
Note that while $\cP(\phi)$ is only a quasi-abelian category, the notion of a torsion pair makes sense: we ask that every $E \in \cP(\phi)$ fits into a \emph{strict} short exact sequence $E_{\Ctor} \into E \onto E_{\Ctf}$, which is just an exact triangle in $\cD$ with $E_{\Ctor} \in \cP(\phi)_\Ctor$ and
 $E_{\Ctf} \in \cP(\phi)_\Ctf$.

\begin{Lem} \label{lem:torsionincPphi}
 Let $\sigma_C = (\ZK,\Zc,\cP)$ be a (weak) HN structure with associated heart $\cA_C$. 
	Assume that $E \in \cP(\phi)$ for $\phi \in (0,1]$, and let $W \subset C$ be a closed subset. Then the objects 
	$\Ann(I_W; E)$ and $I_W \cdot E$ in $\cA_C$ 
	are also semistable of phase $\phi$. 
\end{Lem}
\begin{proof}
	The claim is automatic for $I_W \cdot E$, as it is both a quotient of $I_W \otimes E$ and a subobject of $E$, which are semistable objects of the same phase. 
	
	Now consider $A:= \Ann(I_W; E) = i_{W*}\rH^{-1}_{\cA_W}(E_W)$; by semistability of $I_W \otimes E$ we know $\phi(A) \leqslant \phi$. Semistability of $E$ implies that $\phi(E/I_W \cdot E) \geqslant \phi$. On the other hand, we have
\[\Zc(E/I_W \cdot E) -\Zc(A) =\Zc(E_W) = \len(W) \cdot Z_C(E) \in \R_{>0}\cdot e^{\ii\pi\phi}.
\]
This is only possible if both inequalities are equalities.
So $A \subset I_W \otimes E$ is an inclusion of objects of the same phase, with the latter being semistable; therefore, $A$ is also semistable.
\end{proof}

\begin{Prop} \label{prop:Ctorsiontheoryviaheart}
 Let $\sigma_C = (\ZK,\Zc,\cP)$ be a (weak) HN structure with associated $C$-local heart $\cA_C$. 
	Then $\sigma_C $ has a $C$-torsion theory if and only if $\cA_C$ has a $C$-torsion theory.
\end{Prop}
\begin{proof}
	Assume that $\sigma_C$ has a $C$-torsion theory. We first claim that any object $E \in \cP(\phi)_\Ctf$ is also torsion free as an object in $\cA_C$.
	Indeed, if $W$ is the schematic support of a torsion subobject of $E$, then $\Ann(I_W; E)$ would be a subobject	in $\cP(\phi)_\Ctor$ by Lemma~\ref{lem:torsionincPphi}.
	Combined with Lemma~\ref{lem:maxtorsionses} and the existence of HN filtrations, this shows that every object in $\cA_C$ has a maximal $C$-torsion subobject.
	
	Conversely, assume that $\cA_C$ has a $C$-torsion theory. 
	If $E \in \cP(\phi)$ for $\phi \in (0,1]$, then $E_{\Ctor} = \Ann(I_W; E)$ for $W$ the schematic support of $E_{\Ctor}$ by Lemma~\ref{Lem-ECtor-H-1}.
	Thus $E_{\Ctor}$ and $E_{\Ctf}$ are also objects of $\cP(\phi)$ by Lemma~\ref{lem:torsionincPphi}. 
	This verifies the condition defining the existence of a $C$-torsion theory for $\sigma_C$ for $\phi \in (0,1]$, which clearly implies the condition holds for all $\phi \in \R$. 
\end{proof}

\section{Harder--Narasimhan structures via stability conditions on fibers}\label{sec:HNstructuresfibers}

\subsection{Support property for HN structures} \label{subsec:support-for-HNstr}
Condition \eqref{enum:LangtonHNsupport} of Theorem~\ref{thm:Langton} included in particular the assumption that $(\cA_{\Ctor},\Zc)$ satisfies the support property.
In this subsection, we briefly explore the appropriate lattices adapted to the support property on $\cD_{\Ctor}$. 

\begin{Def} 
A \emph{Mukai homomorphism} on $\cD$ over $C$ with respect to $\Lambda$ is a pair $(v_K,v_\Ctor)$ where 
\index{vK,vCtor@$(v_K,v_\Ctor)$,!Mukai homomorphism on $\cD$ over $C$!}
\[
v_K\colon K(\cD_K)\to \Lambda \quad \text{and} \quad
v_{\Ctor} \colon K(\cD_{\Ctor}) \to \Lambda
\]
are group homomorphisms with the following property: for all $E \in \cD$, and all proper closed subschemes $W \subset C$, we have
\begin{equation} \label{eq:vKvCtor}
v_K(E_K) = \frac 1{\len W} v_{\Ctor}\left( 
i_{W*}E_W\right).
\end{equation}
\end{Def}

Notice that \eqref{eq:vKvCtor} can be equivalently stated as
$v_K(E_K)=v_{\Ctor}\left(i_{p*}^{}E_p^{}\right)$, for all closed point $p\in C$.

\begin{Rem}
\label{remark-mukai-hom-Z} 
Given a Mukai homomorphism on $\cD$ over $C$ with respect to $\Lambda$ and a group homomorphism $Z \colon \Lambda \to \C$, we obtain a central charge 
on $\cD$ over $C$ by setting $Z_K = Z \circ v_K$ and $Z_{\Ctor} = Z \circ v_{\Ctor}$. 
\end{Rem}

We denote by $v_p^{}:=v_{\Ctor}\circ {i_{p*}\colon K(\cD_p)}\to \Lambda$; thus we have defined $v_c \colon K(\cD_c) \to \Lambda$ for every (closed or non-closed) point $c \in C$.
The following observation is immediate from the fact that $K(\cD_{\Ctor})$ is the direct sum of $i_{p*}(K(\cD_p))$ over all closed points $p \in C$.

\begin{Lem}\label{lem:MukaiCtorviafibers}
To give a Mukai homomorphism is equivalent to giving a collection of homomorphisms $v_c \colon K(\cD_c) \to \Lambda$ such that for $E \in \cD$, the vector $v_c(E_c)$ is independent of $c \in C$.
\end{Lem}

\begin{Def}
We say that a Mukai homomorphism is \emph{numerical on fibers} if $v_c$ factors via $K(\cD_c) \to \Knum(\cD_c)$ for all $c \in C$.
\index{vK,vCtor@$(v_K,v_\Ctor)$,!Mukai homomorphism on $\cD$ over $C$!numerical on fibers}
\end{Def}

We will discuss Mukai homomorphisms more systematically in Section~\ref{subsec:SupportProperty}.
 
\begin{Def} Fix a Mukai homomorphism $(v_K, v_\Ctor)$ with lattice $\Lambda$. 
A (weak) HN structure satisfies the support property with respect to a quadratic form $Q$ on $\Lambda_\R$ if $(\cA_\Ctor,\Zc)$ satisfies the support property with respect to $v_\Ctor$ and $Q$. 
\index{sigmaC@$\sigma_C = (\ZK,\Zc,\cP)$, $C$ a Dedekind scheme!Harder--Narasimhan structure on $\cD$ over $C$!satisfying the support property}
\index{sigmaC@$\sigma_C = (\ZK,\Zc,\cP)$, $C$ a Dedekind scheme!weak HN structure on $\cD$ over $C$!satisfying the support property}
\end{Def}

Recall Lemmas~\ref{lem:inducedfiberprestability} and \ref{lem:inducedfiberprestabilityweak}, which associate to a (weak) HN structure $\sigma_C $ a (weak) pre-stability condition $\sigma_c$ on $\cD_c$ for all $c \in C$.
The support property for $\sigma_C $ is equivalent to the collection $(\sigma_c)_{c \in C}$ satisfying the support property with respect to a uniform quadratic form: 

\begin{Lem} \label{lem:supportpropertyviafibers}
	Let $\sigma_C$ be a (weak) HN structure with a $C$-torsion theory.
	Then $\sigma_C$ satisfies the support property with respect to a quadratic form $Q$ if and only if for every $c \in C$, or equivalently for every closed point $c \in C$, the induced (weak) pre-stability condition $\sigma_c$ satisfies the support property with respect to $Q$ and $v_c$.
\end{Lem}
\begin{proof}
	Assume $\sigma_C$ satisfies the support property. 
	The claim is automatic for closed points $p$ from the construction, as the inclusion	$i_{p*} \colon \cA_p \to \cA_{\Ctor}$ preserves semistable objects. 
	As for the generic point $K$, by the existence of a $C$-torsion theory, every semistable object $E_K \in \cA_K$ lifts to a torsion free object $E \in \cA_C$. 
	By Proposition~\ref{prop:HNviaHNweak}, we may assume $E$ is $Z_C$-semistable. 
	If $\sigma_C$ is a HN structure, then $i_{p*}E_p$ is semistable by Lemma~\ref{lem:allfibersstable}, and thus	$Q(v_K(E_K)) = Q(v_p(E_p)) \geqslant 0$.
	In the weak case, we deduce from Lemma~\ref{lem:allfibersnotdestquot} that the HN filtration of $E_p$ is of the form $E_p^0 \into E_p \onto E_p/E_p^0$ with $E_p^0 \in \cA_p^0$. 
	By Remark~\ref{rem:vA0=0}, it follows that $Q(v_K(E_K)) = Q(v_p(E_p)) = Q(v_p(E_p/E_p^0)) \geqslant 0$.
	
	 Conversely, assume that the $\sigma_p$ uniformly satisfy the support property for all closed points $p \in C$.
	 Every indecomposable semistable object $E \in \cA_{\Ctor}$ is set-theoretically supported over a point $p \in C$; we consider the filtration $0 = \pi^{n+1} \cdot E \subset \pi^n \cdot E \subset \dots \subset \pi \cdot E \subset E$ given by Lemma~\ref{lem:filtrationatp}. Since all filtration quotients $\pi^i\cdot E/\pi^{i+1}\cdot E \in \cA_p$ are also quotients of $E/\pi\cdot E$ of the same the phase, they do not have destabilizing quotients.
	 Hence $Q(v_{\Ctor}(E)) \geqslant 0$ follows from the simple linear algebra argument in \cite[Lemma~A.6]{BMS:abelian3folds}.
	 The same linear algebra argument gives the result for the general case of a decomposable $E$.
\end{proof}

\subsection{HN structures via stability conditions on fibers}

To conclude our investigation of HN structures, we show that under appropriate assumptions, they can be constructed via stability conditions on fibers.
\begin{Thm} \label{thm:HNstructureviafibers}
Let $\cA_C$ be a $C$-local heart in $\cD$ that universally satisfies openness of flatness. Fix a group homomorphism $Z\colon\Lambda\to\C$, a quadratic form $Q$ on $\Lambda_\R$, and a Mukai homomorphism $(v_K, v_\Ctor)$. 
Assume that the induced pair $(Z_K, Z_{\Ctor})$ satisfies condition~\eqref{enum:finitecoverZconstant} of Theorem~\ref{thm:Langton}. 
Then the triple $\sigma_C = (\cA_C,\ZK,\Zc)=(\cA_C,Z\circ v_K,Z\circ v_\Ctor)$ is a HN structure with a $C$-torsion theory and satisfying the support property with respect to $Q$ if and only if
	\begin{enumerate}[{\rm (1)}] 
		\item \label{enum:stabfibers}
		$\sigma_c = (\cA_c, Z \circ v_c)$ gives a pre-stability condition on $\cD_c$ for all (closed or non-closed) $c \in C$; 
		\index{sigmac@$\sigma_c = (\cA_c, Z_c)$, $c\in C$ Dedekind scheme,!induced stability condition on $\cD_c$}
		\item \label{enum:uniformQ} all the $\sigma_c$ satisfy the support property with respect to $Q$; and
		\item \label{enum:genericopenness} generic openness of semistability holds (Definition~\ref{def:GenericOpennessSemiStab}).
	\end{enumerate}
	Similarly, if $\sigma_C$ is a weak HN structure, then \eqref{enum:stabfibers}-\eqref{enum:genericopenness} hold.
	Conversely, if, in addition to \eqref{enum:stabfibers}-\eqref{enum:genericopenness}, we also have 
	\begin{enumerate}[{\rm (1)}] \setcounter{enumi}{3}
	\item \label{enum:fiberstiltingproperty} $\sigma_c$ has the tilting property (Definition~\ref{def:tiltingproperty}), 
	\end{enumerate} 
	then $\sigma_C$ is a weak HN structure with a $C$-torsion theory and satisfying the support property with respect to $Q$. Moreover, $\sigma_C$ induces the stability conditions $\sigma_c$.
\end{Thm}
\begin{proof}
The proof is divided in three steps.

	\begin{step}{1} (Weak) HN structure with support property and a $C$-torsion theory $\Rightarrow$ (weak) stability conditions on the fibers satisfying \eqref{enum:stabfibers}-\eqref{enum:genericopenness}.
	\end{step}
 We have already seen, in Lemmas~\ref{lem:inducedfiberprestability} (resp. \ref{lem:inducedfiberprestabilityweak}) and \ref{lem:supportpropertyviafibers} that a HN structure (resp. a weak HN structure) $\sigma_C$ induces on the fibers stability conditions (resp. weak stability conditions) $\sigma_c$ that satisfy the support property with respect to $Q$.
 This gives conditions \eqref{enum:stabfibers} and \eqref{enum:uniformQ}. Condition \eqref{enum:genericopenness}, generic openness of semistability, holds by Remark~\ref{remark-HN-generic-openness}.
 
 \begin{step}{2}
 Stability conditions on the fibers satisfying \eqref{enum:stabfibers}-\eqref{enum:genericopenness} $\Rightarrow$ HN structure with a $C$-torsion theory and satisfying the support property 
 \end{step}
 Now assume that the $\sigma_c=(\cA_c,Z\circ v_c)$ are stability conditions on the fibers, satisfying the support property with respect to a uniform quadratic form $Q$, such that generic openness of semistability holds.
 By Theorem~\ref{thm:Ctorsiontheoryautomatic}, the heart $\cA_C$ has a $C$-torsion theory, and we will use Corollary~\ref{cor:HNgeneral} to prove that $\sigma_C$ is a HN structure.
 Then it follows that $\sigma_C$ has a $C$-torsion theory by Proposition~\ref{prop:Ctorsiontheoryviaheart} and that $\sigma_C$ satisfies the support property with respect to $Q$ by Lemma~\ref{lem:supportpropertyviafibers}, as required.
 
 We first observe that $(\ZK,\Zc)$ is a stability function on $\cA_C$ over $C$.
 Indeed, $\ZK$ is a stability function on $\cA_K$ by assumption, and by decomposing $E\in\cA_{\Ctor}$ according to its set-theoretic support and then using the filtration of Lemma~\ref{lem:filtrationatp}, it follows that $\Zc$ is a stability function on $\cA_{\Ctor}$ from the fact that $Z_p$ is a stability function on $\cA_p$.
 Moreover, by assumption $\ZK$ has the HN property on $\cA_K$. 
 
 Our next claim is that $(\cA_{\Ctor},\Zc)$ has the HN property so that assumption \eqref{enum:LangtonHNsupport} of Theorem~\ref{thm:Langton} is satisfied by Lemma~\ref{lem:Bridgeland-stabviaheart}.
 Given an object $E \in \cA_{\Ctor}$, it is not difficult to see that by considering the decomposition of $E$ according to its set-theoretic support, which is a direct sum of objects set-theoretically supported over distinct closed points, it suffices to show that HN filtrations exist for objects set-theoretically supported over a single closed point $p \in C$.
 Given such an $E$, we proceed by induction on the length of its schematic support in $C$.
 Let $\pi$ be a local generator of $I_p$. Since both $(\pi \cdot E)\otimes I_p^{-1}$ and $E/\pi \cdot E$ have a HN filtration by the induction assumption, and since both are quotients of $E$, one can construct a $\Zc$-semistable quotient $E \onto Q^0$ of phase $\phi(Q^0) = \min\{\phi^-\left(E/\pi\cdot E\right), \phi^-\left(\pi\cdot E\right)\}$.
 It is not difficult to show from the see-saw property, Lemma~\ref{lem:seesawweak}, that $Q^0$ is of minimal phase, i.e., every other quotient $E \onto Q'$ satisfies $\phi(Q') \geqslant \phi(Q^0)$.
 Repeating this argument for the kernel $E^1$, we obtain a $\Zc$-semistable quotient $E^1\onto Q^1$ of minimal phase with kernel $E^2$. As $E/E^2$ is itself a quotient of $E$, so that $\phi(E/E_2)\geqslant\phi(Q^0)$, the see-saw property gives $\phi(E^1/E^2)\geqslant\phi(E/E^1)$.
 Continuing in this way gives a decreasing filtration $E= E^0 \supset E^1 \supset E^2 \supset \dots$ such that the filtration quotients are $\Zc$-semistable with $\phi(E^0/E^1) \leqslant \phi(E^1/E^2) \leqslant \dots$.
 
 We need to show that this process terminates.
 Every central charge $\Zc(E^i/E^{i+1})$ of any of the semistable filtration quotients is contained in the parallelogram with adjacent edges of angle $\pi \phi^-(E)$ and $\pi$, and with $0$ and $\Zc(E)$ as opposite vertices.
 Since each $E^i/E^{i+1}$ satisfies the support property inequality of $\sigma_p$ (by the second part of the proof of Lemma~\ref{lem:supportpropertyviafibers}), it follows, see Remark~\ref{rem:supportfinitelength}, that there are only finitely many possible classes $v_{\Ctor}(E^i/E^{i+1})$.
 Therefore there is an $i$ with $\Im\Zc (E^i/E^{i+1}) = 0$ for all $i$, which is only possible if $E^i$ itself is semistable with $\Im\Zc (E^i) = 0$, so the process terminates.
 Grouping the consecutive $i$ with filtration quotients of equal phase gives the HN filtration of $E$. Therefore, $(\cA_\Ctor,\Zc)$ has the HN property. 
 
 The remaining assumptions of Theorem~\ref{thm:Langton} and Corollary~\ref{cor:HNgeneral} are part of our hypotheses, so $\sigma_C$ is indeed a HN structure as claimed.
 
 \begin{step}{3}
 	Weak stability conditions on the fibers satisfying \eqref{enum:stabfibers}-\eqref{enum:fiberstiltingproperty} $\Rightarrow$ weak HN structure with a $C$-torsion theory and satisfying the support property
 \end{step}
The structure of the arguments carries over exactly, with assumption \eqref{enum:fiberstiltingproperty} needed to ensure that all assumptions of Theorems~\ref{thm:Langton} and~\ref{thm:Ctorsiontheoryautomatic} are satisfied. 
The only additional argument needed is to show that if $\sigma_p = (\cA_p, Z_p)$ has the tilting property for all closed points $p \in C$, then the same holds for $(\cA_\Ctor,\Zc)$. 

For both conditions in Definition~\ref{def:tiltingproperty}, it suffices to consider objects set-theoretically supported over a single closed point $p \in C$; we let $\pi$ be a local generator of its maximal ideal. 
We first show noetherianity of $\cA_{\Ctor}^0$. 
Given $E \in \cA_{\Ctor}^0$ set-theoretically supported over $p$, consider the filtration induced by $\pi$, as in Lemma~\ref{lem:filtrationatp}. Then any surjection $E \onto Q$ induces surjections $\pi^i\cdot E /\pi^{i+1} \cdot E \onto \pi^i \cdot Q /\pi^{i+1} \cdot Q$.
Since these are surjections in the noetherian category $\cA_p^0$, we easily conclude that any sequence of surjections $E \onto E_1 \onto E_2 \cdots $ terminates, i.e., that $\cA_{\Ctor}^0$ is noetherian.
A similar inductive argument, using the same filtration, shows that $\cA_\Ctor^0$ is a torsion subcategory. 

Next we verify condition~\eqref{enum:tiltingpropertyExt1} of Definition~\ref{def:tiltingproperty} for $(\cA_\Ctor,\Zc)$ for $F \in \cA_\Ctor$, $\mu^+(F) < +\infty$, set-theoretically supported over $p$. 
As explained in Remark~\ref{rem:tiltingpropertytilt}, we must show $F[1]$ has a maximal subobject in $\cA_{\Ctor}^0$ with respect to the heart $\cA_{\Ctor}^{\sharp\beta}$ for $\beta \geqslant \mu^+(F)$. 
It follows that using the short exact sequence $\pi \cdot F \into F \onto F/\pi \cdot F$ and induction on the length of the support of $F$, we can reduce to the case where $F$ is scheme-theoretically supported over $p$.
In this case, $F = i_{p*}E$ and by assumption there exists a short exact sequence $E \into \widetilde{E} \onto E^0$ in $\cA_p$ with $E^0 \in \cA_p^0$ and $\Hom(\cA_p^0, \widetilde{E}[1]) = 0$.
We claim that the pushforward by $i_{p*}$ of this sequence gives the desired exact sequence for $F$.
It suffices to show $\Hom(\cA_\Ctor^0, i_{p*}\widetilde{E}[1]) = 0$. 
This reduces to showing $\Hom(i_{p*}T, i_{p*}\widetilde{E}[1]) = 0$ for every $T \in \cA^0_p$. 
Using adjunction and Lemma~\ref{lem:cohpushpull}, this reduces to showing $\Hom(T, \widetilde{E}) = 0 =\Hom(T, \widetilde{E}[1]) = 0$. 
The first equality holds since $\mu^+(\widetilde{E}) < +\infty$, and the second holds by our choice of $\widetilde{E}$.
\end{proof}

\section{Tilting weak Harder--Narasimhan structures over a curve}
\label{sec:tilt1}

Recall from Remark~\ref{rem:wGL2action} that the universal cover $\wGL2$ of $\GL_2^+(\R)$ acts
on the set of Harder--Narasimhan structures on $\cD$ over $C$. In terms of hearts and a family
of stability functions as in Proposition~\ref{prop:stabviaheart}, this corresponds to tilting
the corresponding heart $\cA_C$.

As we already saw in the case of weak stability conditions in Section~\ref{subsec:tiltingweakstability}, this procedure is much more subtle for weak HN structures, and may not exist in general, due to the special role of objects with central charge 0, which are arbitrarily defined to have phase 1 if they are in the heart.
In particular, HN filtrations are \emph{not} preserved under tilting. 
In this section, we give a criterion to ensure that a weak HN structure can be tilted, extending Proposition~\ref{prop:RotateWeakStability}.

Consider a weak HN structure $\sigma_C=(\cA_C,\ZK,\Zc)$ on $\cD$ over $C$.
We write $(\cT_C^\beta,\cF_C^\beta)$ for the torsion pair defined as in \eqref{eqn:torsion pair}, with $\mu$ replaced by $\mu_C$.
The tilted heart $\cA_C^{\sharp\beta}:=\langle\cF_C^\beta[1],\cT_C^\beta\rangle$ is $C$-local, since $Z_C$-semistability is invariant under tensoring with the pullback of a line bundle from $C$.

We first need a relative analogue of the tilting property defined in Definition~\ref{def:tiltingproperty}:
\begin{Def} \label{def:tiltingpropertyoverC}
Given a weak HN structure $\sigma_C = (\cA_C,\ZK,\Zc)$, we write $\cA_C^0 \subset \cA_C$ for the subcategory objects $E$ with $Z_C(F) = 0$ for every subquotient $F$ of $E$. We say that $\sigma_C$ has the \emph{tilting property} if 
\index{sigmaC@$\sigma_C = (\ZK,\Zc,\cP)$, $C$ a Dedekind scheme!weak HN structure on $\cD$ over $C$!having the tilting property}
\begin{enumerate}[{\rm (1)}] 
 \item $\cA_C^0 \subset \cA_C$ is a noetherian torsion subcategory, and
 \item \label{enum:tiltingpropcurvesExt1}
 for every $F \in \cA_C$ with $\mu_C^+(F) < +\infty$, there exists a short exact sequence $F \into \widetilde F \onto F^0$ with $F^0 \in \cA_C^0$ and
 $\Hom(\cA_C^0, \widetilde F[1]) = 0$.
\end{enumerate}
\end{Def}

\begin{Rem}
Let $E$ be an object with $E_K \neq 0$ but $Z_K(E_K) = 0$, and $F \in \cA_{p}$ for some closed point $p \in C$. Then 
$Z_C(E \oplus i_{p*}F) = 0$, but
$E \oplus i_{p*}F \in \cA_C^0$ only holds if $Z_p(F) = 0$.
\end{Rem}

\begin{Rem} \label{rem:tiltingpropertycurvesses}
Assume that $\cA^0_C$ is a noetherian torsion subcategory, and consider a short exact sequence $0 \to A \to E \to B \to 0$ of objects in $\cA_C$ with no morphisms from $\cA^0_C$.
Then if $A, B$ satisfy the condition in \eqref{enum:tiltingpropcurvesExt1}, the same holds for $E$.
This is seen most easily by considering the tilt of $\cA_C$ at the torsion pair $(\cA_C^0, (\cA_C^0)^\perp)$:
the titled heart contains the objects $A[1], E[1], B[1]$, and we have to show that each of them has a maximal subobject in $\cA_C^0$.
\end{Rem}

\begin{Ex} \label{ex:slopestabilitycurveshastiltingproperty}
Consider relative slope stability $\sigma_C$ for a smooth family $\cX \to C$ of higher-dimensional varieties as a weak HN structure as in Example~\ref{ex:slopestability}.
Then $\cA_C^0$ consists of sheaves whose support has codimension $\geqslant 2$ in every fiber, and $\sigma_C$ has the tilting property.
Indeed, by Remark~\ref{rem:tiltingpropertycurvesses}, it is enough to consider the case where $E \in \Coh \cX$ is either torsion free, or the pushforward $E = i_{p*} F$ of a torsionfree sheaf $F \in \Coh \cX_p$.
In the latter case, we use the double dual of $F$ as in Example~\ref{ex:slopestabilityhastiltingproperty}. In the former case, let $E^{\vee \vee}$ be the double dual of $E$ in $\Coh \cX$, and let $G \subset E^{\vee \vee}/E$ be the maximal subsheaf whose support has codimension $\geqslant 2$ in every fiber.
The preimage $\tilde E \subset E^{\vee \vee}$ of $G$ has the desired property.
\end{Ex}

\begin{Prop} \label{prop:tiltingweakHNstructures}
Let $\cA_C$ be the heart of a $C$-local t-structure on $\cD$ that satisfies universal openness of flatness.
Let $\sigma_C = (\cA_C,\ZK,\Zc)$ be a weak HN structure on $\cD$ over $C$. Assume that:
\begin{enumerate}[{\rm (1)}]
 \item $\sigma_C$ has the tilting property, and 
 \item the induced weak stability condition $\sigma_c$ has the tilting property for every $c \in C$.
\end{enumerate}
Then $(\cA^{\sharp\beta}_C)^0 \subset \cA^{\sharp\beta}_C$ is a noetherian torsion subcategory, and 
$\sigma_C^{\sharp\beta} = (\cA_C^{\sharp\beta}, \frac{\ZK}{\ii - \beta}, \frac{\Zc}{\ii-\beta})$ is a weak HN structure.
 \index{sigmaCsharpb@$\sigma_C^{\sharp\beta}$, tilted weak HN structure at slope $\beta$}
\end{Prop}

\begin{proof}
The first claim follows just as in the proof of Proposition~\ref{prop:RotateWeakStability}.

We need to prove that the rotated central charge, which we denote by $Z_C^{\sharp\beta}$, satisfies the HN property on $\cA_C^{\sharp\beta}$.
We first observe, in analogy with Lemma~\ref{lem:RotateWeakStabilitySemistableObjects}, that $Z_C^{\sharp\beta}$-semistable objects in $\cA_C^{\sharp\beta}$ are either
\begin{enumerate}[{\rm (1)}] 
 \item $Z_C$-semistable objects of $\cT_C^{\sharp\beta}$, or
 \item objects $E$ such that $\rH^{-1}_{\cA_C}(E)$ is a $Z_C$-semistable object of $\cF_C^{\sharp\beta}$ and $\rH^0_{\cA_C}(E) \in \cA_C^0$; in addition, we require $\Im Z_C(E) = 0 $ or 
 $\Hom(\cA_C^0, E) = 0$.
\end{enumerate}

The HN filtrations of objects in $\cT^{\sharp\beta}$ remain unchanged, and thus it remains to consider objects in $\cF_C^{\sharp\beta}[1]$. To be able to proceed by induction, we consider slightly more generally an object $E \in \cA_C^{\sharp\beta}$ with $\rH^{0}_{\cA_C}(E) \in \cA_C^0$. Let $F = \rH^{-1}_{\cA_C}(E)$, and let
$G \into F \onto Q$ be the last step of the HN filtration of $F$ in $\cA_C$, with $Q$ being $Z_C$-semistable. 
Then $E'=E/G[1]$ is an object where $\rH^{-1}_{\cA_C}(E')$ is $Z_C$-semistable and $\rH^{0}_{\cA_C}(E')=\rH^{0}_{\cA_C}(E) \in \cA_C^0$.
Replacing $E'$ with the quotient $Q'$ by its maximal subobject in $\cA_C^0$, we obtain a semistable quotient of $E$ in $\cA_C^{\sharp\beta}$ of the same slope as $Q[1]$.
Moreover, the kernel $G'$ of $E \onto Q'$ will again be an object with $\rH^{0}_{\cA_C}(G') \in \cA_C^0$, and the length of the HN filtration of $\rH^{-1}_{\cA_C}(G')$ is smaller than that of $F$.
Therefore, this procedure terminates. 
\end{proof}

\begin{Rem}
We can now explain why the definition of $Z_C$-semistability for a $C$-torsion free object $E \in \cA_C$ requires \emph{all} fibers $E_p$ to be semistable, rather than just the general fiber, as e.g.~in the definition of \emph{relative slope stability}.
Let $E \in \cA_C$ be a torsion free object with $\mu_C(E) = \beta$, such that $E_K$ is semistable. Let $p \in C$ be a closed point such that $i_{p*}E_p$ admits a destabilizing short exact sequence $A \into i_{p*}E_p \onto Q$, with $A, Q$ semistable and $\mu_C(A) > \mu_C(i_{p*}E_p) = \mu_C(E) > \mu_C(Q)$.
If $E$ was defined to be $Z_C$-semistable, then $E[1] \in \cA_C^{\sharp\beta}$. On the other hand, $A \in \cT_C^\beta \subset \cA_C^{\sharp\beta}$;
if $F$ denotes the kernel of the surjection $E\otimes I_p^{-1} \onto i_{p*}E_p \onto Q$ in $\cA_C$, then we would obtain a short exact sequence $A \into E[1] \onto F[1]$ in $\cA_C^{\sharp\beta}$.
In other words, despite $E$ being semistable and torsion free in $\cA_C$, the shift $E[1]$ would not be torsion free in $\cA_C^{\sharp\beta}$.
Moreover, the existence of a maximal torsion subsheaf of $E[1]$ is essentially equivalent to semistable reduction for $E$;
thus we have instead built semistable reduction into our basic setup of $Z_C$-stability.
\end{Rem}

\newpage
\part{Stability conditions over a higher-dimensional base}
\label{part:higher-dimensional-bases}

In this part of the paper, we introduce a notion of stability conditions for a category $\cD$ over a higher-dimensional base $S$;
its key property will be that it comes equipped with relative moduli spaces. 
Throughout, we work in the following setup. 

\begin{Setup}
\label{setup-rel-stability}
Assume: 
\begin{itemize}
 \item $g \colon \cX \to S$ is a flat morphism as in the \hyperref[MainSetup]{Main Setup} and in addition it is projective;
 \item $\cD\subset\Db(\cX)$ is an $S$-linear strong semiorthogonal component of finite cohomological amplitude.
\end{itemize}
\end{Setup}

\section{Flat families of fiberwise (weak) stability conditions}
\label{sec:FlatFamilyFiberwiseStability}

In this section we introduce the notion of a flat family of fiberwise (weak) stability conditions, and prove some basic results in this context. 
In Section~\ref{sec:StabCondOverS} we will define the notion of a (weak) stability condition over a base by further imposing a suitable support property. 

\subsection{Definitions} 

Our goal is to provide a notion of a stability condition over $S$ 
that is strong enough to yield relative moduli spaces of stable objects, and flexible enough to allow for deformation results, and constructions via Bogomolov--Gieseker type inequalities.
To do so, we will consider various compatibility conditions on fiberwise collections of stability conditions. 

\begin{Def}
\label{definition-uogs}
Let $\us = \left(\sigma_s = (Z_s, \cP_s)\right)_{s \in S}$ be a collection of (weak) numerical stability conditions on $\cD_s$ 
for every (closed or non-closed) point $s \in S$. 
In the weak case, we assume that $\sigma_s$ satisfies the assumptions of Proposition~\ref{prop:basechangeweakstabilityviaopenness} for 
all $s \in S$. 
\begin{enumerate}[{\rm (1)}]
\item \label{ulccc}
$\us$ \emph{universally has locally constant central charges} if for every morphism $T \to S$ and every $T$-perfect object $E \in \rD(\cX_T)$ such that $E_t \in \cD_t$ for all $t \in T$, the function $T \to \C$ given by $t \mapsto Z_t(E_t)$ is locally constant.
In this situation, if $T$ is connected we often write $Z(E)$, $\phi(E)$, and $\mu(E)$ for the constant values $Z_t(E_t)$, $\phi_t(E_t)$, and $\mu_t(E_t)$. 
\item $\us$ \emph{universally satisfies openness of geometric stability} if for every morphism $T \to S$ and every $T$-perfect object 
$E \in \rD(\cX_T)$, the set 
\begin{equation*}
\set{ t \in T \sth E_t \in \cD_t \ \text{and is geometrically} \ \sigma_t\text{-stable} }
\end{equation*} 
is open. 
\item $\us$ \emph{universally satisfies openness of lying in $\cP(I)$} for an interval $I \subset \R$ 
if for every morphism $T \to S$ and every $T$-perfect object 
$E \in \rD(\cX_T)$, the set 
\begin{equation*}
\set{ t \in T \sth E_t \in \cP_{t}(I) }
\end{equation*} 
is open. 
\end{enumerate}
\end{Def} 

The (weak) stability conditions $\sigma_t$ appearing in Definition~\ref{definition-uogs} 
are given by the base change results Theorem~\ref{thm:base-change-stability-condition} and Proposition~\ref{prop:basechangeweakstabilityviaopenness}, 
and geometric stability is meant in the sense of Definition~\ref{def:geometricallystable}. 
Further, we note that if $T$ is quasi-compact with affine diagonal, then by Lemma~\ref{lemma-D-Ds} the condition on $E \in \rD(X_T)$ in~\eqref{ulccc} is equivalent to $E \in \cD_T$.

\begin{Rem}
\label{remark-uolp-uoflatness} 
In the situation of Definition~\ref{definition-uogs}, for any $\phi \in \R$ we obtain a 
fiberwise collection of t-structures given by 
\begin{equation*}
\cD_s^{\leqslant 0} = \cP_s(>\phi), \quad \cD_s^{\geqslant 0} = \cP_s(\leqslant \phi+1).
\end{equation*} 
For $\phi = 0$, we call this the \emph{fiberwise collection of t-structures underlying $\us$}. 
Note that for $I = (0, 1]$, universal openness of lying in $\cP(I)$ is precisely the 
condition of universal openness of flatness (in the sense of Definition~\ref{def-open-flat-utau}) 
of this fiberwise collection of t-structures. 
\end{Rem}

\begin{Lem} 
\label{Lem-uogs-ftbc}
Let $\us = \left(\sigma_s = (Z_s, \cP_s)\right)_{s \in S}$ be a collection of (weak) numerical stability conditions as in Definition~\ref{definition-uogs}. 
Then the properties of universal local constancy of central charges, universal openness of geometric stability, and universal openness of lying in $\cP(I)$ for an interval $I \subset \R$ 
can be checked on morphisms $T \to S$ of finite type from an affine scheme. 
\end{Lem}

\begin{proof}
Using parts \eqref{enum:pullbackremainsstable} and \eqref{enum:geomstablepreserved} of Theorem~\ref{thm:base-change-stability-condition} and the analogous statements in Proposition~\ref{prop:basechangeweakstabilityviaopenness}, this follows by the argument in the proof of Lemma~\ref{lem:universalopennessfromfiniteopenness}. 
\end{proof}

We next show that for collections of stability conditions, universal openness of geometric stability implies universal generic openness of semistability in a suitable sense:

\begin{Lem}\label{lem:SemiStUnivGenOpenn}
Let $\us = \left(\sigma_s = (Z_s, \cP_s)\right)_{s \in S}$ be a collection of numerical stability conditions which universally has locally constant central charges and universally satisfies openness of geometric stability. 
If $T \to S$ is a morphism from an irreducible, quasi-compact scheme with affine diagonal, and if $E \in \cD_T$ is a $T$-perfect object whose generic fiber 
$E_{K(T)}$ is $\sigma_{K(T)}$-semistable, then there exists a nonempty open subset $U \subset T$ such that $E_t$ is $\sigma_t$-semistable for all $t \in U$. 
\end{Lem}

\begin{proof}
Let $K(T) \subset \overline{K}$ be the algebraic closure, and consider the Jordan--H\"older filtration for $E_{\overline{K}}$.
By Proposition~\ref{Prop-D-bc-fields}.\eqref{Prop-D-bc-fields-Efg}, this filtration is defined over a finitely generated field extension $K(T) \subset L$; let $E^i_L$ be the filtration factors, which are by definition geometrically $\sigma_L$-stable.
Let $f \colon T' \to T$ be a morphism of finite type such that $L = K(T')$ as field extensions of $K(T)$.
By Lemma~\ref{lem-extend-from-localisation}.\eqref{enum:lem-extend-filtration-from-localisation}, we can lift the JH filtration of $E_L$ to a sequence of morphisms
\[ 0 = F^0_{T'} \xrightarrow{f_1} F^1_{T'} \xrightarrow{f_2} \dots \to F^m_{T'} = E_{T'}
\] 
in $\cD_{T'}$ such that $E^i_{T'} := \cone(f_i)$ is a lift of $E^i_L$ to $\cD_{T'}$. Moreover, in light of \cite[Proposition~2.2.1]{Lieblich:mother-of-all}, 
by shrinking $T'$ we may assume that all $F^i_{T'}$ are 
$T'$-perfect. 
By our assumption on universal openness of geometric stability, by further shrinking $T'$ we may assume $E^i_{t'}$ is geometrically $\sigma_{t'}$-stable for all ${t'} \in T'$ and for all $i$.
Since the central charges are universally locally constant, the objects $E^i_{t'}$ also have the same phase.
It follows that $E_{t'}$ is $\sigma_{t'}$-semistable, and hence $E_t$ is $\sigma_t$-semistable for all $t \in f(T')$ by Theorem~\ref{thm:base-change-stability-condition}.\eqref{enum:pullbackremainsstable}.
Since $f(T')$ contains an open set, this concludes the proof of the lemma.
\end{proof}

The following, combined with Definition~\ref{def:fiberwisesupport} later on, is the main definition of this paper.

\begin{Def} \label{def:familyfiberstabilities}
A \emph{flat family of fiberwise stability conditions on $\cD$ over $S$} is a collection of numerical stability conditions
$\us = \left(\sigma_s = (Z_s, \cP_s)\right)_{s \in S}$ on $\cD_s$ for every (closed or non-closed) point $s \in S$ 
such that: 
\index{sigmau@$\us = (\sigma_s = (Z_s, \cP_s))_{s \in S}$,!flat family of fiberwise (weak) stability conditions on $\cD$}
\begin{enumerate}[{\rm (1)}]
\item \label{enum:Zsindependentbasechange} 
$\us$ universally has locally constant central charges. 
\item \label{enum:stabilityopenbasechange} $\us$ universally satisfies openness of geometric stability. 
\item \label{enum:curveassumption} 
$\us$ integrates to a HN structure over any Dedekind scheme $C \to S$ essentially of finite type (see Definition~\ref{def:EssLocFT}) over $S$, 
i.e., the stability conditions $\sigma_c$ for $c \in C$ are induced, in the sense of Lemma~\ref{lem:inducedfiberprestability}, by a HN structure $\sigma_C$ on $\cD_C$ over $C$.
\end{enumerate}
We define a \emph{flat family of fiberwise weak stability conditions} analogously via weak stability conditions on the fibers, but with some additional assumptions: 
\begin{enumerate}[{\rm (1)}] \setcounter{enumi}{-1}

\item \label{enum:A0noetheriantorsionfibers}
For each $s \in S$, the central charge $Z_s$ is defined over $\Q[\ii]$, and $\cA_s^0 \subset \cA_s$ (see Definition~\ref{def:A0}) is a noetherian torsion subcategory (see Definition~\ref{def:noetheriantorsionsubcat}).

\item[(\mylabel{enum:genOpSstweak}{2'})] 
In addition to universal openness of geometric stability, $\us$ satisfies the following property: 
For any morphism $T \to S$ essentially of finite type with $T$ irreducible and any $T$-perfect $E \in \cD_T$ whose generic fiber $E_{K(T)}$ is $\sigma_{K(T)}$-semistable, there exists a nonempty open subset $U \subset T$ such that $E_t$ is $\sigma_t$-semistable for all $t \in U$.

\item[(\mylabel{enum:curveassumptionweak}{3'})] 
For any Dedekind scheme $C \to S$ essentially of finite type over $S$, 
the weak stability conditions $\sigma_c$ for $c \in C$ are induced, in the sense of Lemma~\ref{lem:inducedfiberprestabilityweak}, by a weak HN structure $\sigma_C$ on $\cD_C$ over $C$ with the additional property that $\cA_C^0$ (see Definition~\ref{def:tiltingpropertyoverC}) is a noetherian torsion subcategory. 
\end{enumerate}
\end{Def}

\begin{Rem}\label{rem:mainDef}
Let us elaborate on the conditions appearing in Definition~\ref{def:familyfiberstabilities}:
\begin{enumerate}[{\rm (1)}] 
\item Universal openness of geometric stability in \eqref{enum:stabilityopenbasechange} is the main consistency condition 
relating the slicings for different fibers and it is necessary for the existence of moduli spaces of semistable objects, see Theorem~\ref{thm:modulispacesArtinstacks}.
In Proposition~\ref{prop:yeswehaveopennessofflatness} below, we show that a flat family of fiberwise (weak) stability conditions 
automatically satisfies the other compatibility condition introduced above --- universal openness of lying in $\cP(I)$ for any $I \subset \R$. 
In the case of weak stability conditions, we initially interpret condition \eqref{enum:genOpSstweak} in the sense of Definition~\ref{def:slicingbasechangeweak}; this in combination with condition \eqref{enum:A0noetheriantorsionfibers} means that the assumptions of Proposition~\ref{prop:basechangeweakstabilityviaopenness} are satisfied, 
and thus the notions of universal openness of geometric stability and lying in $\cP(I)$ are indeed well-defined. 

\item By Lemma~\ref{lem:SemiStUnivGenOpenn} condition~\eqref{enum:genOpSstweak} automatically holds for a flat family of fiberwise 
stability conditions, but in the weak case we need to include generic openness of semistability as an extra assumption. 

\item \label{enum:Jacobson}
If $S$ is of finite type over a field, or more generally a \emph{Jacobson scheme}, since the closed points of $S$ are dense in every closed subset of $S$, $\us$ is determined by $\sigma_s$ for all closed points $s$:
indeed this follows from universal openness of geometric stability and universal generic openness, once we invoke \cite[Proposition~2.2.1]{Lieblich:mother-of-all} to lift objects to relatively perfect ones.\footnote{We have learned a related statement for complete DVRs from Fabian Haiden.}
In particular, from \eqref{enum:Zsindependentbasechange} it follows that the central charge at a point $s$ is determined by the central charge at any specialization of $s$.
The existence of central charges on non-closed points is a consistency condition for central charges on closed points, generalizing Definition~\ref{def:familycentralcharges}.
The existence of slicings for non-closed points axiomatizes the classical notion of generic stability and generic HN filtrations.
This notion is significantly strengthened by condition \eqref{enum:curveassumption} when the base is a Dedekind scheme.
\item Clearly, a flat family of fiberwise stability conditions over $S$ induces one over $T$ for a base change $T \to S$ essentially of finite type.

\item In case $S=\Spec(k)$ is a point, a ``flat family of fiberwise stability conditions on $\cD$'' (and, similarly, ``a stability condition on $\cD$ over $k$'' in Definition~\ref{def:fiberwisesupport} below) is stronger than the notion of a stability condition on $\cD$ from Section~\ref{sec:StabCond}.
Indeed, Definition~\ref{def:familyfiberstabilities} includes openness of stability as a requirement (while Definition~\ref{def:fiberwisesupport} will include boundedness).
\end{enumerate}
\end{Rem}

We single out the following remark for emphasis. 

\begin{Rem} \label{rem:cond3viahearts} 
Consider a Dedekind scheme $C \to S$. The key requirement in condition \eqref{enum:curveassumption} of Definition~\ref{def:familyfiberstabilities} is the existence of a t-structure on $\cD_C$ integrating the t-structures on the fibers $\cD_c$ induced by the stability conditions $\sigma_c$ and universally satisfying openness of flatness. 
Indeed, we will see that, given $\us$, the heart $\cA_C$ universally satisfies openness of flatness (Proposition~\ref{prop:yeswehaveopennessofflatness}) and has a $C$-torsion theory (Corollary~\ref{cor:Ctorsionautomatic}).
Conversely, if a heart $\cA_C$ exists and universally satisfies openness of flatness, then by Theorem~\ref{thm:Ctorsiontheoryautomatic} it does admit a $C$-torsion theory; once we impose a support property which is uniform across all fibers, see Definition~\ref{def:fiberwisesupport}, the existence of a HN structure then follows from Theorem~\ref{thm:HNstructureviafibers} and Lemma~\ref{lem:SemiStUnivGenOpenn}. 

In the case of weak stability conditions, the same logic holds if in addition, every weak stability condition $\sigma_s$ has the tilting property, this is preserved under base change, and $\cA_C^0$ is a noetherian torsion subcategory. 

Conversely, a (weak) HN structure on $S = C$ satisfying support property on fibers induces a flat family of fiberwise (weak) stability conditions only if we additionally require various base change compatibilities;
for example, universal openness of geometric stability (and universal generic openness, in the weak stability case) rather than just generic openness for objects that are semistable over $K(C)$ (see Definition~\ref{def:GenericOpennessSemiStab}).
\end{Rem}

\subsection{Universal openness of flatness}

In order to take advantage of the results of Part~\ref{part:ModuliSpaces}, in particular the existence of Quot spaces shown in Section~\ref{sec:quotspaces}, 
we have to show that the fiberwise collection of t-structures underlying a flat family of (weak) stability conditions universally satisfies openness of flatness. 
More generally (see Remark~\ref{remark-uolp-uoflatness}), we show the following. 

\begin{Prop} \label{prop:yeswehaveopennessofflatness}
Let $\us$ be a flat family of fiberwise (weak) stability conditions on $\cD$ over $S$. 
Then $\us$ universally satisfies openness of lying in $\cP(I)$ for any interval $I \subset \R$. 
\end{Prop}

Given a base change $T \to S$ of finite type and a $T$-perfect object $E \in \cD_T$, we obtain two functions
\index{phi+E@$\phi^+_E$ ($\phi^-_E$), maximal (minimal) phase functions}
\begin{equation}
\label{phipmE}
\phi^+_E \colon T \to \R \cup \{-\infty\} \quad \text{and} \quad \phi^-_E \colon T \to \R \cup \{+\infty\}
\end{equation}
that assign to $t \in T$ the maximal and minimal phase $\phi^{\pm}(E_t)$ of the HN filtration of $E_t$ with respect to the slicing $\cP_t$; here we set $\phi^\pm$ of the zero object to be $\mp \infty$ for convenience.
The key observation underlying Proposition~\ref{prop:yeswehaveopennessofflatness} is the semicontinuity of these functions.
\begin{Lem} \label{lem:phipmsemicontinuous}
The functions $\phi^+_E$ and $\phi^-_E$ are, respectively, upper and lower semicontinuous constructible functions on $T$.
\end{Lem}

\begin{proof}[Proof of Proposition~\ref{prop:yeswehaveopennessofflatness}]
By Lemma~\ref{Lem-uogs-ftbc}, it is enough to prove openness of lying in $\cP(I)$ for a finite type base change 
$T \to S$ and $E \in \cD_T$ as above. 
Suppose for concreteness that the interval $I$ is of the form $I = (a,b]$; the argument is the same for other types of intervals. 
Then the set 
\begin{equation*}
\set{ t \in T \sth E_t \in \cP_{t}(I) } = \left(\phi^-_E\right)^{-1}(a,+\infty) \cap \left(\phi^+_E\right)^{-1}(-\infty, b]
\end{equation*} 
is open by Lemma~\ref{lem:phipmsemicontinuous}.
\end{proof}

\begin{proof}[Proof of Lemma~\ref{lem:phipmsemicontinuous}]
We first note that $\phi^\pm_E$ are preserved by arbitrary base change, as the pullback under base change for field extensions in Theorem~\ref{thm:base-change-stability-condition} or Proposition~\ref{prop:basechangeweakstabilityviaopenness} preserves the HN filtration.

\begin{step}{1}
$\phi^{+}_E$ and $\phi^{-}_E$ are constructible functions.
\end{step}
Due to the compatibility with base change, and since $T$ is noetherian, it is enough to prove that if $T$ is irreducible, then there exists an open set 
$U \subset T$ on which $\phi^\pm_E$ is constant. 
To prove this claim, let $0 \to E_1 \to E_2 \to \dots \to E_m = E$ be an arbitrary lift to $\cD_T$ of the HN filtration of $E_\eta$, where $\eta \in T$ is the generic point, given by Lemma~\ref{lem-extend-from-localisation}.\eqref{enum:lem-extend-morphism-from-localisation}.

Let $F_i$ be the cone of the map $E_{i-1} \to E_i$; by construction, the objects $(F_i)_\eta$ are the HN factors of $E_\eta$, and thus semistable.
By the generic openness of semistability (see Lemma~\ref{lem:SemiStUnivGenOpenn}), there exists a nonempty open set $U \subset T$ such that every $(F_i)_t$ is semistable for all points $t \in U$. It follows that the $(E_i)_t$ induce the HN filtration of $E_t$ for all $t \in U$, and in particular that $\phi^\pm_E$ are constant on $U$.

\begin{step}{2}
$\phi^+_E$ and $\phi^{-}_E$ are, respectively, monotone increasing and decreasing under specialization.
\end{step}
As $\phi^\pm_E$ are preserved by base change, and since we assume $S$ (and thus $T$) to be Nagata, it suffices by Lemma~\ref{lemma-specializations-lift} to consider a DVR $R$ essentially of finite type over $T$:
if $k, K$ are the special and generic point of $\Spec R$, respectively, we have to show $\phi^+(E_k) \geqslant \phi^+(E_K)$ and $\phi^-(E_k) \leqslant \phi^-(E_K)$. We may assume $E_K \neq 0$, otherwise the claim is trivial. 

By Definition~\ref{def:familyfiberstabilities} assumption \eqref{enum:curveassumption}, the (weak) stability conditions $\sigma_k$ and $\sigma_K$ are induced by a (weak) Harder--Narasimhan structure $\sigma_R$ over $\Spec R$.
Let 
\begin{equation*}
 0 = E_0 \to E_1 \to \dots \to E_m = E_R
\end{equation*}
be the HN filtration of $E_R$ with respect to $\sigma_R$.
Its base change to the fraction field induces the HN filtration of $E_K$ (in the sense that some filtration quotients might be $R$-torsion, inducing isomorphisms $E_i \to E_{i+1}$ after base change to $K$); 
therefore $\phi^+(E_K) \leqslant \phi(E_1) = \phi^+(E_R)$, and likewise $\phi^-(E_K) \geqslant \phi^-(E_R)$. 

It remains to prove $\phi^-(E_R) \geqslant \phi^-(E_k)$ and $\phi^+(E_R) \leqslant \phi^+(E_k)$. Since $i_{k*}$ is t-exact with 
respect to the t-structures obtained from $\cP(> \phi)$ for
any $\phi \in \R$, see Lemma~\ref{lem:inducedfiberprestabilityweak}, it follows that $i_k^*$ is right t-exact, and
$i_k^! = i_k^*[-1]$ is left t-exact; in other words, 
\begin{equation} \label{eq:immediatebound}
E_k \in \cP_k[\phi^-(E_R), \phi^+(E_R) + 1].
\end{equation} 

This immediately implies the desired claim for $\phi^+$. Indeed, consider the triangle $E_1 \to E_R \to G$ obtained from the HN filtration, and let 
$A$ be the first HN factor of $(E_1)_k$. Then
\[ \phi(A) \geqslant \phi(E_1)= \phi^+(E_R) > \phi^+(G) \geqslant \phi^+(G_k[-1]).
\]
Hence the composition $A \to (E_1)_k \to E_k$ is non-zero, and so $\phi^+(E_k) \geqslant \phi(A) \geqslant \phi^+(E_R)$.

An analogous argument reduces the case of $\phi^-$ to the following claim: if $E \in \cA_R$ is $\sigma_R$-semistable, then $\phi(E) = \phi^-(E_k)$.
To prove this claim, first suppose $Z_K(E_K) \neq 0$. Then
$Z_k(E_k) = Z_K(E_K) \in \R_{>0}\cdot e^{\ii\pi\phi}$. Together with $E_k \in \cP_k[\phi, \phi+1]$, this is only possible if $\phi^-(E_k) = \phi$ as claimed. 

Now suppose $Z_K(E_K) = 0$. 
Then $\phi = 1$, so $E/\pi E\in \cP_k(1)$ by our analysis above; to conclude, it suffices to prove $E/\pi E\neq 0$.
Assume otherwise.
It follows immediately that
$E_k = \rH^{-1}_{\cA_k}(E_k)[1]$ with $Z_k\left(\rH^{-1}_{\cA_k}(E_k)\right) = 0$; the same holds for all quotients of $E$ in $\cA_R$.

Recall from Definition~\ref{def:familyfiberstabilities}.\eqref{enum:curveassumptionweak} that $\cA_R^0$ is noetherian, so if $E$ were in $\cA_R^0$, then it would have a maximal $R$-torsion subobject which would contradict $E/\pi E= 0$ by Remark~\ref{Rem-Nak}.
Thus $E\notin\cA_R^0$, so there exist quotients $E \onto Q_1 \onto Q_2$ with a short exact sequence
$A \into Q_1 \onto Q_2$ and $Z_R(A) \neq 0$.
Since $Z_K((Q_i)_K) = 0$, this is only possible if $A$ is $R$-torsion. On the other hand, the surjection $\rH^{-1}_{\cA_k}((Q_2)_k) \onto A/\pi A$ in $\cA_k$ shows $Z_k(A/\pi A) = 0$; via Lemma~\ref{lem:filtrationatp} this gives a contradiction.
\end{proof}

Finally, as consequence of Proposition~\ref{prop:yeswehaveopennessofflatness} and Theorem~\ref{thm:Ctorsiontheoryautomatic}, we have:
\begin{Cor} \label{cor:Ctorsionautomatic}
Let $\us$ be a flat family of fiberwise (weak) stability conditions on $\cD$ over $S$.
Let $C \to S$ be essentially of finite type, with $C$ Dedekind, and let $\cA_C$ be the heart of the (weak) HN structure on $\cD_C$.
Then $\cA_C$ has a $C$-torsion theory.
\end{Cor}

\subsection{Product stability conditions}
In this subsection, we assume that $X_0$ and $S$ are of finite type over a field $k$ with $X_0$ projective, and let $X = X_0 \times_{\Spec k} S$ be the product.
Further, assume that $\cD_{0} \subset \Db(X_0)$ is a $k$-linear strong semiorthogonal component of finite cohomological amplitude, and $\cD = (\cD_{0})_S \subset \Db(X)$ is its base change to $S$. 

By Theorem~\ref{thm:base-change-stability-condition}, any stability condition $\sigma_0$ on $\cD_0$ induces a collection $\us$ of fiberwise stability conditions on $\cD$ over $S$; 
we call such a $\us$ a \emph{product stability condition}. 
If $\sigma_0$ universally satisfies openness of geometric stability and the heart of $\sigma_0$ universally satisfies openness of flatness, then $\us$ is a flat family of fiberwise stability conditions by Theorems~\ref{thm:HNstructureviafibers} and \ref{Thm-D-bc}. As a partial converse, we have the following result, strengthening \cite[Proposition~2.6]{LiPertusiZhao:TwistedCubics}.

\begin{Prop}\label{prop:ProductStabilityCondition}
If $S$ as above is connected and has a $k$-rational point, then the only flat families of fiberwise stability conditions on $\cD$ over $S$ are product stability conditions.
\end{Prop}
\begin{proof}
Let $\us=(\sigma_s=(Z_s,\cP_s))_{s\in S}$ be a flat family of fiberwise stability conditions on $\cD_S$ over $S$. 
By Remark~\ref{rem:mainDef}.\eqref{enum:Jacobson}, $\us$ is determined by $\sigma_s$ for all closed points $s\in S$.

\begin{step}{1} \label{step:2kratPoints}
Assume that $S$ is irreducible, and that $s, t \in S$ are two $k$-rational points. Then $\sigma_s = \sigma_t$.
\end{step}
Let $F\in\cP_s(\phi)\subset \cD_s$.
We claim that $F\in \cP_t(\phi)$ and $Z_s(F)=Z_t(F)$.

Indeed, if we consider $E:=p^*F$, where $p\colon X\to X_0$ is the natural projection, then $E$ is $S$-perfect and $E_t=E_s= F$, so by Definition~\ref{def:familyfiberstabilities}.\eqref{enum:Zsindependentbasechange} $Z_s(F)=Z_t(F)$.
Now, suppose that $E_t=F$ is not $\sigma_t$-semistable. 
Then, if we consider the mds $F_1$ of $F$ with respect to $\sigma_t$, we have $\phi(F_1)>\phi(F)$.
By Definition~\ref{def:familyfiberstabilities}.\eqref{enum:stabilityopenbasechange}, there exist open sets $U\ni s$ and $V\ni t$, such that $F$ is semistable for all points in $U$ and $F_1$ is semistable for all points in $V$.
Since $U\cap V\neq \emptyset$, the inequality $\phi(F_1)>\phi(F)$ contradicts the existence of a non-trivial morphism $F_1\to F$.

\begin{step}{2} \label{step:prod2}
Assume that $S$ is irreducible, and that it contains a point $s \in S$ such that $\sigma_s$ is obtained by base change from a stability condition $\sigma_0$ on $\cD_0$. Then $\us$ is a product stability condition.
\end{step}
Let $t \in S$; we need to show that $\sigma_t$ is equal to the stability condition $(\sigma_0)_{k(t)}$ obtained by base change from $\sigma_0$ via the field extension $k \subset k(t)$. Let $\ell$ be a field extension of $k$ containing both $k(t)$ and $k(s)$. Consider the base change $S_\ell := S \times_{\Spec k} \Spec \ell \to S$. Then every irreducible component of $S_\ell$ contains $\ell$-rational points $s_\ell, t_\ell$ mapping to $s, t$, respectively. By Step~\ref{step:2kratPoints}, we have
\[ 
(\sigma_t)_\ell = (\sigma_s)_\ell = (\sigma_0)_\ell = \left((\sigma_0)_{k(t)}\right)_\ell.
\]
By Theorem~\ref{thm:base-change-stability-condition}, two stability conditions that become equal after base change are equal. This proves the claim.

\begin{step}{3}
The general case.
\end{step}

If $S$ is any connected scheme, then by induction on the number of irreducible components and the previous steps, every irreducible component contains a point $t$ such that $\sigma_t$ is obtained by base change from a stability condition $\sigma_0$ on $\cD_0$. By Step~\ref{step:prod2}, the result follows.
\end{proof}

\section{Stability conditions over \texorpdfstring{$S$}{S}}\label{sec:StabCondOverS}

In this section we give the definition of a (weak) stability condition over a given base $S$. To this end we need to define a suitable \emph{support property} for a flat family of fiberwise (weak) stability conditions. This consists of two properties: the existence of a uniform quadratic form controlling the central charges, and boundedness of geometrically stable objects. 
We show that a stability condition over $S$ has well-behaved moduli spaces of semistable objects. 

\subsection{Support property}\label{subsec:SupportProperty}
So far, we have assumed the support property on each fiber; to make this notion useful, we will
have to assume the existence of a uniform quadratic form controlling the central charges and classes of semistable objects
on all fibers. 

Recall the definition of the numerical $K$-groups $\Knum(\cD_s)$ underlying Definition~\ref{def:NumericalStabilityCondition}, the existence of a pushforward map
$\eta_{t/s}^\vee \colon \Knum(\cD_t) \onto \Knum(\cD_s)_t \subset \Knum(\cD_s) \otimes \Q$ for every point
$t$ over $s$, see \eqref{eq:defetaellkvee}, and the property $\Knum(\cD_s)_t \subset
\Knum(\cD_s)_{\overline{s}}$ shown in Proposition and Definition~\ref{propdef:KnumDell}.
Condition \eqref{enum:Zsindependentbasechange} on universal local constancy of central charges 
in Definition~\ref{def:familyfiberstabilities} naturally leads to the following. 

\begin{Def} \label{def:gluefiberwiseKnum}
We define the \emph{relative numerical Grothendieck group} $\Knum(\cD/S)$ as the quotient of
$\bigoplus_{s \in S} \Knum(\cD_s)_{\overline{s}}$
by the saturation of the subgroup generated by elements of the form
\index{Knum(D/S)@$\Knum(\cD/S)$, relative numerical $K$-group}
\begin{equation} \label{eq:KnumDSrelation}
\eta_{t_1/f(t_1)}^\vee [ E_{t_1}] - \eta_{t_2/f(t_2)}^\vee [ E_{t_2}]
\end{equation}
for all tuples $(f, E, t_1, t_2)$ where $f \colon T \to S$ is a morphism from a connected scheme $T$, $E \in \rD(X_T)$ is a $T$-perfect object such that 
$E_t \in \cD_t$ for all $t \in T$, and
$t_1, t_2 \in T$.
\end{Def}

\begin{Rem}
Analogous to Lemma~\ref{Lem-uogs-ftbc}, we would obtain the same group $\Knum(\cD/S)$ if we only considered morphisms $f \colon T \to S$ of finite type from a connected affine scheme in the definition. 
\end{Rem}

Given $f \colon T \to S$ and $E$ as in Definition~\ref{def:gluefiberwiseKnum}, we write $[E]$ for the element of $\Knum(\cD/S)$ given by the image of $[E_t] \in \Knum(\cD_t)$ under the composition
\begin{equation*}
\Knum(\cD_t) \to \Knum(\cD_{f(t)})_t \into \Knum(\cD_{f(t)})_{\overline{f(t)}} \to \Knum(\cD/S)
\end{equation*}
for any $t \in T$, which is independent of $t \in T$ by the definition of $\Knum(\cD/S)$. 

By the evident universal property of $\Knum(\cD/S)$, for any flat family of fiberwise (weak) stability conditions there
exists a central charge $Z \colon \Knum(\cD/S) \to \C$ such that for all $s \in S$, the central charge $Z_s$ factors as
$Z_s \colon \Knum(\cD_s) \to \Knum(\cD/S) \xrightarrow{Z} \C$. 

\begin{Def}
\label{def-relative-mukai}
Let $\Lambda$ be a finite rank free abelian group. 
A \emph{relative Mukai homomorphism} for $\cD$ over $S$ with respect to $\Lambda$ 
is a group homomorphism 
\index{v@$v \colon \Knum(\cD/S) \to \Lambda$, relative Mukai homomorphism}
$v \colon \Knum(\cD/S) \to \Lambda$. 
\end{Def}

\begin{Rem}
\label{rem-mukai-ulc}
By definition, a relative Mukai homomorphism $v \colon \Knum(\cD/S) \to \Lambda$ is \emph{universally locally constant} 
in the sense that for every morphism $T \to S$ and every $T$-perfect object $E \in \rD(X_T)$ such that $E_t \in \cD_t$ for all $t \in T$, 
the function $T \to \Lambda$ given by $t \mapsto v([E_t])$ is locally constant. 
\end{Rem}

We will always restrict our attention to flat families of fiberwise (weak) stability conditions 
where $Z \colon \Knum(\cD/S) \to \C$ factors via a relative Mukai homomorphism $v \colon \Knum(\cD/S) \to \Lambda$. 
We will frequently use the following choice.

\begin{PropDef}\label{propdef:lattice}
There is an Euler characteristic pairing
\[
\chi\colon K(\cD_{\perf}) \times \Knum(\cD/S) \to \Z
\quad \text{given by} \quad 
\chi([F], [E]) = \chi(F_t, E)
\]
for some point $t$ over $S$ and objects $E \in \cD_t$ and $F \in \cD_{\perf}$. 
We write $\cN(\cD/S)$ for the quotient 
\[\cN(\cD/S) := \Knum(\cD/S)/\Ker \chi,\] 
and call it the \emph{uniformly numerical relative Grothendieck group} of $\cD$ over $S$.
\index{N(D/S)@$\cN(\cD/S)$, uniformly numerical relative Grothendieck group}
If $S$ is quasi-projective over a field, then $\cN(\cD/S)$ is a free abelian group of finite rank.
\end{PropDef}

\begin{proof}
We need to show that $\chi$ satisfies the relation given in \eqref{eq:KnumDSrelation}. This follows since $\chi(F_{t_1}, E_{t_1})$ and $\chi(F_{t_2}, E_{t_2})$ both compute the rank of the object $\cHom_T(F_T, E)$ in $\Db(T)$. 

To prove the claim on being a free abelian group of finite rank, we first observe that $\Knum(\cD/S)$ is generated by objects defined over closed points.
Therefore, $\cN(\cD/S)$ is a subgroup of the numerical Grothendieck group for compactly supported objects in $\Db(\cX)$ considered in \cite[Section~5.1]{BCZ:nef}.
%Combining the argument of \cite[Lemma~5.1.1]{BCZ:nef} with that of Lemma~\ref{lem:Knumfinite} shows this latter group is free of finite rank.
Since $\cX$ is quasi-projective over a field, we can compactify it and then use directly the argument in the proof of \cite[Lemma~5.1.1]{BCZ:nef}, by using Lemma~\ref{lem:Knumfinite} to avoid both the normality and the characteristic zero assumptions in \cite[Lemma~5.1.1]{BCZ:nef}.
\end{proof}

\begin{Ex}\label{ex:KnumFamilySurfaces}
Let $g \colon \cX \to S$ be a smooth family of polarized surfaces with $S$ connected and whose geometric generic fiber has Picard rank one, and let $\cO_{\cX}(1)$ denote a relatively ample line bundle.
Then the coefficients of the relative Hilbert polynomial with respect to $\cO_{\cX}(1)$ as in Example~\ref{ex:slopestability} (or, equivalently, the degrees of the Chern character on the fibers with respect to $\cO_{\cX}(1)$) induce an isomorphism $\cN(\Db(\cX)/S) \cong \Z^3$. 
\end{Ex}

\begin{Ex}[Yoshioka's trick] \label{ex:modification} Consider the previous example in the case where the base $S = C$ is a curve.
We can modify this construction at a single closed point $c \in C$ whose fiber $\cX_c$ has higher Picard rank as follows. 
Let $L \subset \Knum(\Db(\cX_c))$ be the saturated subgroup generated by (the restrictions of) $\cO_\cX$, $\cO_\cX(1)$ and the class of a point. 
Then the quotient map to $\Knum(\Db(\cX_c))/L$ extends to 
a map $\Knum(\Db(\cX)/S) \to \Knum(\Db(\cX_c))/L $ by setting it identically to zero for objects supported over any point $c' \in C, c \neq c'$. The choice of
\[
\Knum(\Db(\cX)/S) \to \cN(\Db(\cX)/S) \oplus \Knum(\Db(\cX_c))/L
\]
will allow us to deform the central charge specifically for objects supported over $c$. 

This can be generalized to an arbitrary family over $C$ by letting
$L$ be the subgroup of $\Knum(\cD_c)$ generated by 
all classes of the form $\eta_{t/c}^\vee [E_t]$ for objects $E$ defined over a scheme $f \colon T \to C$ with $f$ dominant, $t \in T$ and $f(t) = c$. 
\end{Ex}

\begin{Ex}\label{ex:LambdaFamilyChernClasses}
Generalizing Example~\ref{ex:KnumFamilySurfaces}, let $g\colon \cX\to S$ be a flat projective morphism of relative dimension $n$, with normal and integral fibers.
The coefficients $p_n, \dots, p_0$ of the Hilbert polynomial with respect to $\cO_X(1)$ define, up to normalization, a morphism $\cN(\Db(\cX)/S) \to \Z^{n+1}$.

The assignment $\us=\big(\sigma_s=(\mathfrak{i}p_n - p_{n-1},\Coh \cX_s)\big)_{s\in S}$ is a flat family of weak fiberwise stability conditions on $\Db(\cX)$: condition \eqref{enum:stabilityopenbasechange} follows from \cite[Proposition~2.3.1]{HL:Moduli}, while condition \eqref{enum:curveassumptionweak} follows from Section~\ref{subsec:CoherentSheaves}.
 \index{sigmaucurve@$\us:=(\sigma_s=(\mathfrak{i}\ch_{X/S,0}-\ch_{X/S,1},\Coh X_s))$!(weak) stability condition on $\Db(\cX)$ over $S$ given by slope-stability on each fiber}
Moreover, for each $s$ the subcategory $\cA^0_s$ is the category of torsion sheaves supported in codimension $\geqslant 2$, which satisfies condition \eqref{enum:A0noetheriantorsionfibers} because $\Coh \cX_s$ is noetherian.
Similarly, to verify condition \eqref{enum:curveassumptionweak}, we note that $\cA^0_C$ consists of sheaves whose support intersects every fiber in codimension $\geqslant 2$. 

This can be generalized to modules $\Coh(\cX,\cB)$ over a sheaf of coherent algebras $\cB$ on $\cX$ (or, more generally, a ``sheaf of differential operators'' in the sense of \cite[Section~2]{simpson:mod1}); the Hilbert polynomial being the one of the underlying sheaf.
In the example of cubic fourfolds (see Part~\ref{part:CubicFourfolds}), we will use the case when $\cB$ is a sheaf of the even parts of a Clifford algebra.
\end{Ex}

From now on, we fix $\Lambda$ and a relative Mukai homomorphism $v \colon \Knum(\cD/S) \to \Lambda$.

\begin{Def}\label{def:Lambda0}
Given a flat family of fiberwise (weak) stability conditions $\us$, for which $Z$ factors via $v$, we let $\Lambda_0$ be the saturated subgroup of $\Lambda$ generated by $v([E_t])$ for all $E_t\in \cA_t^0$ (see Definition~\ref{def:A0}) and all points $t$ over $S$.
Let $\oLambda$ be the free abelian group $\Lambda/\Lambda_0$, and $\ov$ the composition of $v$ with the quotient map.
\index{Lambda0@$\Lambda_0$, saturated subgroup of $\Lambda$ generated by Mukai vectors of elements in $\cA_s^0$}
\index{Lambdabar@$\oLambda$, free abelian group $\Lambda/\Lambda_0$}
\end{Def}

\begin{Rem}
If each $\sigma_s$ is a stability condition, then $\Lambda_0=\set{0}$ and $\oLambda=\Lambda$.
\end{Rem} 

In our definition of the support property for a flat family of fiberwise (weak) stability conditions, we will need to impose a boundedness assumption on moduli spaces. 
The following definition summarizes the relevant moduli spaces we can 
consider in our context. 

\begin{Def}\label{def:Mus(v)andMusst(v)} 
Let $\us$ be a flat family of fiberwise (weak) stability conditions on $\cD$ over $S$. 
Fix a vector $\vv \in \Lambda$ and $\phi\in \R$ such that $Z(\vv)\in \R_{>0} e^{\mathfrak{i}\pi \phi}$. 
\index{vv@$\vv\in \Lambda$, Mukai vector}
\begin{enumerate}[{\rm (1)}]
\item 
We denote by 
\index{Msigmastv@$\fM^{\st}_{\us}(\vv)$, stack parameterizing geometrically $\us$-stable objects of class $\vv$ and phase $\phi$}
\begin{equation*}
\cM^{\st}_{\us}(\vv) \colon (\Sch/S)^{\op} \to \Gpds 
\end{equation*}
the functor whose value on $T \in (\Sch/S)$ consists of all 
$T$-perfect objects $E \in \rD(\cX_T)$ such that for all $t \in T$ we have $E_t \in \cD_t$, $E_t$ is geometrically $\sigma_t$-stable of phase $\phi$, and $v(\left[ E_t \right]) = \vv$. 
An object $E \in \cM^{\st}_{\us}(\vv)(T)$ is called a \emph{family of geometrically $\us$-stable objects of class $\vv$} over $T$. 

\item We denote by 
\index{Msigmav@$\fM_{\us}(\vv)$, stack parameterizing $\us$-semistable objects of class $\vv$ and phase $\phi$}
\begin{equation*}
\fM_{\us}(\vv) \colon (\Sch/S)^{\op} \to \Gpds 
\end{equation*}
the functor whose value on $T \in (\Sch/S)$ consists of all 
$T$-perfect objects $E \in \rD(\cX_T)$ such that for all $t \in T$ we have $E_t \in \cD_t$, $E_t$ is $\sigma_t$-semistable of phase $\phi$, and $v(\left[ E_t \right]) = \vv$. 
An object $E \in \fM_{\us}(\vv)(T)$ is called a \emph{family of $\us$-semistable objects of class $\vv$} over $T$. 

\item For an interval $I \subset \R$, we denote by 
\index{P(I;vv)@$\cP_{\us}(I; \vv)$, stack parameterizing objects in $\cP(I)$ of class $\vv$}
\begin{equation*}
\cP_{\us}(I; \vv) \colon (\Sch/S)^{\op} \to \Gpds 
\end{equation*}
the functor whose value on $T \in (\Sch/S)$ consists of all 
$T$-perfect objects $E \in \rD(\cX_T)$ such that for all $t \in T$ we have $E_t \in \cP_t(I) \subset \cD_t$ and $v(\left[ E_t \right]) = \vv$. 
\end{enumerate}
\end{Def}

We will always omit the phase $\phi$ from the notation. 
Note that by Theorem~\ref{thm:base-change-stability-condition} and Proposition~\ref{prop:basechangeweakstabilityviaopenness} the above prescriptions indeed define functors, i.e., the conditions considered on $E \in \rD(X_T)$ are stable under base change. 
Also note that $\cM_{\us}(\vv) = \cP_{\us}([\phi, \phi];\vv)$.

Recall from Section~\ref{sec-moduli-objects-D} the moduli stack 
$\cMpug(\cD/S)$ of objects in $\cD$. 
\begin{Lem}
\label{lem-cM-cP-substacks}
Let $\us$ be a flat family of fiberwise (weak) stability conditions on $\cD$ over $S$, and fix a vector $\vv \in \Lambda$ and an interval $I \subset \R$. 
Then $\cM^{\st}_{\us}(\vv)$, $\cM_{\us}(\vv)$, and $\cP_{\us}(I; \vv)$ are open substacks of $\cMpug(\cD/S)$. 
\end{Lem}

\begin{proof}
Since the relative Mukai homomorphism $v$ is universally locally constant (by Remark~\ref{rem-mukai-ulc}) and $\us$ 
universally satisfies openness of geometric stability, 
$\cM^{\st}_{\us}(\vv)$ is an open substack of $\cMpug(\cD/S)$. 
Similarly, since $\us$ also universally satisfies openness of lying in $\cP(I)$ by Proposition~\ref{prop:yeswehaveopennessofflatness}, 
$\cP_{\us}(I; \vv)$ (and thus $\cM_{\us}(\vv)$ as a special case) is an open substack of $\cMpug(\cD/S)$. 
\end{proof}

\begin{Def}\label{def:boundedness}
Let $\us$ be a flat family of fiberwise (weak) stability conditions on $\cD$ over $S$. 
We say that $\us$ \emph{satisfies boundedness} if $\fM^{\st}_{\us}(\vv)$ is bounded in the sense of Definition~\ref{def-bounded} for every $\vv \in \Lambda$.
\end{Def}

\begin{Rem}
In Section~\ref{subsec:GrothendieckLemma} we show that in the case of stability conditions, we also obtain boundedness of $\cP_{\us}(I, \vv)$ for $I$ of length less than $1$, and of certain Quot spaces. 
\end{Rem}

Now we can finally define the support property, completing our main definition. 
We use the notation of Definition~\ref{def:Lambda0}, and write the central charge as the composition of $Z \colon \Lambda \to \C$ with the fixed relative Mukai homomorphism $v$.

\begin{Def} \label{def:fiberwisesupport}
We say that a flat family of fiberwise (weak) stability conditions $\us$ satisfies the \emph{support property with respect to $\Lambda$} if: 
\begin{enumerate}[{\rm (1)}] \setcounter{enumi}{3}
\item\label{enum:fiberwisesupport} There exists a quadratic form 
$Q$ on $\oLambda_\R = \oLambda \otimes \R$ such that
\begin{enumerate}[(a)]
\item \label{enum:fiberwisesupporta} 
the kernel $(\ker Z)/\Lambda_0\subset\oLambda$ is negative definite with respect to $Q$, and 
\item \label{enum:fiberwisesupportb} for every $s \in S$ and for every $\sigma_s$-semistable object
$E \in \cD_s$, we have $Q(v(E)) \geqslant 0$. 
\end{enumerate}
\item \label{enum:bounded} $\us$ satisfies boundedness.
\end{enumerate}
In this case, we call $\us$ a \emph{(weak) stability condition on $\cD$ over $S$ with respect to $\Lambda$}.
\index{sigmau@$\us = (\sigma_s = (Z_s, \cP_s))_{s \in S}$,!(weak) stability condition on $\cD$ over $S$ with respect to $\Lambda$}
\end{Def}

\begin{Rem} \label{rem:suppC}
If the flat family $\us$ of fiberwise (weak) stability conditions satisfies the support property, then every $\sigma_s$ satisfies the support property as a stability condition on $\cD_s$ with respect to $Q$ and the composition
$\Knum(\cD_s) \to \Knum(\cD/S) \to \Lambda$.
Similarly, consider a
 Dedekind scheme $C$ and a morphism $C \to S$ essentially of finite type. Our relative Mukai homomorphism induces a Mukai homomorphism $v \colon \Knum(\cD_{\Ctor}) \to \Lambda$ by Lemma~\ref{lem:MukaiCtorviafibers}, and the induced (weak) stability condition on $\cD_{\Ctor}$ satisfies the support property by Lemma~\ref{lem:supportpropertyviafibers}.%, where $\sigma_C$ has a $C$-torsion theory by Theorem~\ref{thm:Ctorsiontheoryautomatic} and Corollary~\ref{cor:Ctorsionautomatic}.
 \end{Rem}

\begin{Rem}
Let us explain the role of assumption \eqref{enum:bounded} in the support property.
In the absolute case, the support property is the key behind the deformation result for stability conditions in Theorem~\ref{thm:deformstability}. 
It implies that under a small deformation of a central charge, there are only finitely many classes of objects that could destabilize a given object $E$. 
Now consider an object $E \in \cD_C$ defined over a curve $C$.
To ensure that openness of stability is preserved, we need to show that unless the generic fiber of $E$ gets destabilized, each such class can only destabilize $E_c$ for finitely many closed points $c \in C$.
This can only be ensured by showing that the set of potentially destabilizing quotients is bounded; see Section~\ref{sec:deformfamiliystability} for the full proof.
\end{Rem}

\begin{Ex}\label{ex:SlopeStabilityAsProperFamilyofWeakStability}
In the setting of Example~\ref{ex:LambdaFamilyChernClasses}, $\us=\big(\sigma_s=(Z_s:=\mathfrak{i}p_n - p_{n-1},\Coh \cX_s)\big)_{s\in S}$ gives a weak stability condition on $\Db(\cX)$ over $S$.
 \index{sigmaucurve@$\us:=(\sigma_s=(\mathfrak{i}\ch_{\cX/S,0}-\ch_{\cX/S,1},\Coh X_s))$!(weak) stability condition on $\Db(\cX)$ over $S$ given by slope-stability on each fiber}
Indeed, $\Lambda_0=\ker Z$ is the saturated subgroup generated by $(p_{n-2},\ldots,p_{0})$ and so $\oLambda\cong \Z^2$.
We can therefore choose any non-negative quadratic form $Q$ on $\oLambda_\R$ to satisfy \eqref{enum:fiberwisesupporta} and \eqref{enum:fiberwisesupportb}.
Boundedness is \cite[Theorem~4.2]{Langer:positive}.

As before, this generalizes to the case $\Coh(\cX,\cB)$; as remarked in \cite[Section~4]{Langer:positive}, boundedness can be proved over an arbitrary base $S$ with the same argument as in \cite[Proposition~3.5]{simpson:mod1}.
\end{Ex}

\begin{Rem} \label{rem:fiberwiseGL2action}
The action of $\wGL2$ on stability conditions on the fibers, and on HN structures over curves (see Remark~\ref{rem:wGL2action}), preserves all properties in Definitions~\ref{def:familyfiberstabilities} and \ref{def:fiberwisesupport}, and thus acts on the set of stability conditions over $S$.
\end{Rem}

\subsection{Boundedness results}
\label{subsec:GrothendieckLemma}

Now we show that the boundedness of a flat family of fiberwise stability conditions leads to boundedness of other moduli functors. 

\begin{Lem} \label{lem:CPIvbounded}
Let $\us$ be a stability condition on $\cD$ over $S$. 
Let $I \subset \R$ be an interval of length less than $1$, and $\vv \in \Lambda$. 
Then $\cP_{\us}(I; \vv)$ is a bounded (in the sense of Definition~\ref{def-bounded}) algebraic stack 
of finite type over $S$; in particular, so is $\fM_{\us}(\vv)$. 
\end{Lem}
\begin{proof}
Let $\phi_0, \phi_1$ be the endpoints of (the closure of) $I$, and consider an object $E \in \cP_t(I)$ of class $\vv$ for a point $t$ over $S$; let $\overline{t}$ denote the algebraic closure of $t$.
Then every \emph{stable factor} (i.e., a Jordan--H\"older factor of one of its HN factors) of $E_{\overline{t}}$ with respect to $\sigma_{\overline{t}}$ has central charge in the parallelogram with angles $\pi \phi_0$ and $\pi \phi_1$ and with $0$ and $Z(\vv)$ as opposite vertices.
By Remark~\ref{rem:supportfinitelength}, this means that as $t$ ranges over all points over $S$, there are only finitely many possible classes occurring as stable factors of $E_{\overline{t}}$. 
Combining this with the boundedness of $\us$ for stable objects---condition~\eqref{enum:bounded} in Definition~\ref{def:fiberwisesupport}---, and the fact that $\cP_{\us}(I; \vv) \subset \cMpug(\cD/S)$ is open by Lemma~\ref{lem-cM-cP-substacks}, we conclude by Lemma~\ref{lem:ExtensionOfBounded} that $\cP_{\us}(I; \vv)$ is bounded. 
Thus by Lemma~\ref{lem-open-bounded}, $\cP_{\us}(I; \vv)$ is an algebraic stack of finite type over $S$. 
\end{proof}

Recall from Section~\ref{sec:quotspaces} that given a fiberwise collection of t-structures $\utau$
universally satisfying openness of flatness, we have defined a moduli stack $\cM_{\utau}$ of flat objects in $\cD$, as well as a Quot space $\Quot_S(E)$ for any object $E \in \cM_{\utau}(S)$. 
Note that if $\us$ is a flat family of fiberwise (weak) stability conditions on $\cD$ over $S$, then by Proposition~\ref{prop:yeswehaveopennessofflatness} the fiberwise collection of t-structures underlying $\us$ universally satisfies openness of flatness. 

\begin{Lem}\label{lem:GrothendieckLemma} 
Let $\us$ be a stability condition on $\cD$ over $S$. 
Let $E \in \cM_{\utau}(S)$ where $\utau$ is the fiberwise collection of t-structures underlying $\us$. 
For $\phi \in (0,1)$, let $\Quot_S^{\leqslant \phi}(E)$ be the subfunctor of $\Quot_S(E)$ which assigns to $T \in (\Sch/S)$ the set of $(E_T \to Q) \in \Quot_S(E)(T)$ satisfying $\phi(Q_t) \leqslant \phi$ for all $t \in T$.
\index{QuotSEphi@$\Quot_S^{\leqslant \phi}(E)$, subfunctor parametrizing quotients with phase $\leqslant \phi$}
Then $\Quot_S^{\leqslant \phi}(E)$ is an algebraic space of finite type over $S$, and the morphism $\Quot_S^{\leqslant \phi}(E) \to S$ is universally closed.
\end{Lem}

\begin{proof}
The canonical morphism $\Quot_S^{\leqslant \phi}(E) \to \Quot_S(E)$ is representable by open immersions because $\us$ universally has locally constant central charges. Since $\Quot_S(E)$ is an algebraic space locally of finite type over $S$ by Proposition~\ref{proposition-quot-algebraic}, it follows that 
$\Quot_S^{\leqslant \phi}(E)$ is too. 

Next we prove that $\Quot_S^{\leqslant \phi}(E)$ is in fact of finite type over $S$. 
Recall the function $\phi^-_{E}$ defined in~\eqref{phipmE}. 
By Lemma~\ref{lem:phipmsemicontinuous}, $\phi^-_E$ is a constructible function on the noetherian topological space $S$, and hence has a minimum $\phi_0 > 0$. In particular, given $(E_T \to Q) \in \Quot_S^{\leqslant \phi}(E)(T)$ and a point $t \in T$, every stable factor of the base change $Q_{\overline{t}}$ to the algebraic closure has phase $\geqslant \phi_0$. 
Since we also have $\Im Z(Q) \leqslant \Im Z(E)$ and $\phi(Q) \leqslant \phi$, it follows that every stable factor of $Q_{\overline{t}}$ has central charge in the parallelogram with angles $\pi \phi_0$ and $\pi$ and with $0$ and $z$ as opposite vertices, where 
$z \in \C$ is the complex number of phase $\phi$ such that $\Im z = \Im Z(E)$. 
By Remark~\ref{rem:supportfinitelength} it follows that there is a finite set $\Gamma \subset \Lambda$ of classes occurring as stable factors of $Q_{\overline{t}}$; 
moreover, this set $\Gamma$ depends only on $E$, $\phi$, and $Z$, and hence works uniformly for any $(E_T \to Q) \in \Quot^{\leqslant \phi}_S(E)(T)$ and $t \in T$. As the central charge of $Q$ lies in the triangle with vertices $0$, $z$, and $w$, where $w$ is complex number of phase $\phi_0$ such that $\Im w=\Im Z(E)$, the set of sums of classes in $\Gamma$ whose central charge lie in this triangle is another uniform finite set $\Gamma'$, which contains all of the possible classes of $Q$. Therefore, the morphism $\Quot^{\leqslant \phi}_S(E) \to \cM_{\utau}$ sending $E_T \to Q$ to $Q$ factors through the canonical morphism 
$\coprod_{\vv \in \Gamma'} \cP_{\us}([\phi_0, 1]; \vv) \to \cM_{\utau}$. 
The proof of Proposition~\ref{proposition-quot-algebraic} shows that the morphism 
$\Quot^{\leqslant \phi}_S(E) \to \coprod_{\vv \in \Gamma'} \cP_{\us}([\phi_0, 1]; \vv)$ is of finite type. 
The target of this morphism is of finite type over $S$ by Lemma~\ref{lem:CPIvbounded}, hence so is $\Quot^{\leqslant \phi}_S(E)$. 

By combining condition~\eqref{enum:curveassumption} of Definition~\ref{def:familyfiberstabilities}, Corollary~\ref{cor:Ctorsionautomatic}, and Proposition~\ref{proposition-quot-valuative-criteria}, we find that the morphism $\Quot_S(E) \to S$ satisfies the strong existence part of the valuative criterion with respect to any essentially of finite type morphism $\Spec(R) \to S$ with $R$ a DVR. 
By universal local constancy of the central charges of $\us$, the same holds for $\Quot_S^{\leqslant \phi}(E) \to S$. 
Thus Lemma~\ref{lemma-valuative-criterion-uc} shows the morphism $\Quot_S^{\leqslant \phi}(E) \to S$ is universally closed. 
(Note that $\Spec(R) \to S$ is essentially of finite type if and only if it is essentially locally of finite type, see Remark~\ref{rem-elft-eft}.)
\end{proof}

\subsection{Relative moduli spaces}
\label{subsec:PropernessRelativeModuli}

In this section, we show that in our setting, relative moduli spaces of semistable objects are well-behaved. 
First we prove that if $\us$ is a stability condition on $\cD$ over $S$, then the moduli stack $\fM_{\us}(\vv)$ is \emph{quasi-proper} over $S$, that is, it satisfies the strong existence part of the valuative criterion. The proof is essentially the same as \cite[Theorem~4.1.1]{AP:t-structures}, and is the reason we require the existence of HN structures after base change to Dedekind schemes.
In characteristic zero, we further use \cite{AHLH:good_moduli} to show that it admits a good moduli space $M_{\us}(\vv)$ (in the sense of Alper) which is proper over $S$. 

\begin{Lem}\label{lem:ModuliSpaceValuativeCriteria}
Let $\us$ be a stability condition on $\cD$ over $S$. Then for every $\vv \in \Lambda$, the morphism $\fM_{\us}(\vv) \to S$ satisfies the strong existence part of the valuative criterion for any DVR essentially of finite type over $S$.
\end{Lem}

\begin{proof}
Let $\Spec(R) \to S$ be a morphism from a DVR that is essentially of finite type, let $K$ be its field of fractions, $k$ its residue field, and
assume we are given a lift $\Spec(K) \to \fM_{\us}(\vv)$, corresponding to a $\sigma_K$-semistable object in $\cD_K$ of class
$\vv$. We may assume that it is the base change $E_K$ of $E_R \in \cD_R$.

By assumption, we have a HN structure $\sigma_R$ on $\cD_R$ over $\Spec R$ satisfying the support property. Since the HN filtration of $E_R$ induces the one of $E_K$, we may assume that $E_R$ is $\sigma_R$-semistable. By Corollary~\ref{cor:Ctorsionautomatic} and Proposition~\ref{prop:Ctorsiontheoryviaheart} we can assume that it is $R$-torsion free. By Lemma~\ref{lem:allfibersstable}, its special fiber $E_k$ is $\sigma_k$-semistable.
\end{proof}

We recall the notion of a good moduli space from \cite{Alper:GoodModuli}.

\begin{Def}\label{def:AlperGoodModuliSpace}
Let $\cY$ be an algebraic stack over $S$.
We say that $\cY$ admits a \emph{good moduli space} if there exist an algebraic space $Y$ over $S$ and an $S$-morphism $\pi\colon\cY\to Y$ such that:
\begin{enumerate}[{\rm (1)}]
 \item $\pi$ is quasi-compact and the functor $\pi_*\colon \QCoh \cY \to \QCoh Y$ is exact; and 
 \item the natural map $\cO_Y \to \pi_*\cO_{\cY}$ is an isomorphism.
\end{enumerate}
\end{Def}

\begin{Thm}\label{thm:modulispacesArtinstacks}
Let $\us$ be a (weak) stability condition on $\cD$ over $S$, and let $\vv \in \Lambda$.
\begin{enumerate}[{\rm (1)}]
 \item \label{enum:MstAlgStackFT} $\fM_{\us}^{\st}(\vv)$ is an algebraic stack of finite type over $S$. 
 \item \label{enum:MAlgStackFT} If $\us$ is a stability condition, then $\fM_{\us}(\vv)$ is an algebraic stack of finite type over $S$. Moreover, if $\fM_{\us}(\vv)=\fM_{\us}^{\st}(\vv)$, then it is a $\G_m$-gerbe over its coarse moduli space $M_{\us}(\vv)$, which is an algebraic space proper over $S$.
 \item \label{enum:GoodModuliChar0} Suppose further that $S$ has characteristic $0$. If $\us$ is a stability condition, then $\fM_{\us}(\vv)$ admits a good moduli space $M_{\us}(\vv)$ which is an algebraic space proper over $S$.
\end{enumerate}
\index{Msigmav@$\fM_{\us}^{}(\vv)$ ($\fM^{\st}_{\us}(\vv))$,!substacks of $\cM_{\utau}$, parameterizing $\us$-semistable (geometrically $\us$-stable) objects of class $\vv$ and phase $\phi$}
\index{Msigmav@$M_{\us}(\vv)$, coarse moduli space of $\fM_{\us}(\vv)$}
\end{Thm}

\begin{proof}
Claim \eqref{enum:MstAlgStackFT} follows immediately from Lemma~\ref{lem-open-bounded} and the definitions. Now let us assume that $\us$ is a stability condition.
We have already seen in Lemma~\ref{lem:CPIvbounded} and Proposition~\ref{prop:yeswehaveopennessofflatness} that the functor $\fM_{\us}(\vv)$ is bounded and an open substack of $\cMpug(\cD/S)$, and hence an algebraic stack of finite type over $S$ again by Lemma~\ref{lem-open-bounded}.

If $\fM_{\us}(\vv)=\fM_{\us}^{\st}(\vv)$, then every object $E \in \fM_{\us}(\vv)$ is simple (see Definition~\ref{def:simple}), as every stable object over an algebraically closed field has only trivial endomorphisms. 
The $\G_m$-structure over its coarse moduli space follows from
Lemma~\ref{Lem:SimpleAlgebraicSpace}. 

Lemma~\ref{lemma-valuative-criterion-uc} and Lemma~\ref{lem:ModuliSpaceValuativeCriteria} show that $M_{\us}(\vv)$ is universally closed over $S$.\footnote{Recall that we assume in the \hyperref[MainSetup]{Main Setup} that $S$ is Nagata, so checking the valuative criterion on DVRs essentially of finite type over $S$ is sufficient.}

Now consider an \'etale covering $\tM \to M_{\us}(\vv)$ from a scheme $\tM$ that admits a universal family. The pullback of the diagonal in $M_{\us}(\vv) \times_S M_{\us}(\vv)$ to $\tM \times_S \tM$ is the locus where the two pullbacks $\cE_1, \cE_2$ of the universal family have isomorphic fibers. Since the fibers are stable of the same phase, this is also the locus where there exists a non-trivial morphism between the fibers of $\cE_1, \cE_2$, which is represented by a closed immersion. Since the property of being a closed immersion is \'etale local on the base, this shows that $M_{\us}(\vv)$ is separated over $S$.

Finally, to prove \eqref{enum:GoodModuliChar0} we use the recent groundbreaking result \cite{AHLH:good_moduli}. 
By \cite[Proposition~7.9]{Alper:GoodModuli}, we can reduce to the case where $S$ is affine.
In this case, we can argue exactly as in the proof of \cite[Theorem~7.25]{AHLH:good_moduli} to obtain that $\fM_{\us}(\vv)$ admits a separated good moduli space $M_{\us}(\vv)$.
To show properness, we can directly use \cite[Proposition~3.43(3)]{AHLH:good_moduli} together with Lemma~\ref{lem:ModuliSpaceValuativeCriteria}.
\end{proof}

We finish this section by extending the \emph{Positivity Lemma}, \cite[Theorem~4.1]{BM:proj}, which gives a numerically positive divisor class on every fiber of $M_{\us}(\vv)$, in the case where semistability and stability coincide. We show that it is induced
by a divisor class on $M_{\us}(\vv)$ when the central charge factors via the uniformly numerical relative Grothendieck group $\cN(\cD/S)$ of $\cD$ over $S$ given by Proposition and Definition~\ref{propdef:lattice}, and that it descends to good moduli spaces.

The group $N^1(\cM_{\us}(\vv)/S)$ of relative real numerical Cartier divisors on $\cM_{\us}(\vv)$ over $S$ is the quotient of
the group of real Cartier divisors on $\cM_{\us}(\vv)$ modulo those that have degree zero on every curve $C \to \cM_{\us}(\vv)$ contracted
to a point in $S$. A class in $N^1(\cM_{\us}(\vv)/S)$ has a well-defined degree on every such curve; hence we can talk about
\emph{relatively nef} or \emph{relatively strictly nef} classes that pair non-negatively, or positively, with every such
contracted curve, respectively.

\begin{Thm}\label{thm:PositivityLemmaFamily}
In the setting of Theorem~\ref{thm:modulispacesArtinstacks}, assume that the Mukai morphism to $\Lambda$ factors via $\cN(\cD/S)$. 
Let $M_{\us}(\vv)$ be either the coarse moduli space (when semistability and stability coincide) or the good moduli space (in characteristic 0). Then there is a relative real numerical Cartier divisor class $\ell_{\us} \in N^1(\cM_{\us(\vv)}/S)$, naturally associated to $\us$, that is relatively nef. Moreover, we have
$\ell_{\us}.C = 0$ if and only if $C$ is a curve of S-equivalent objects (i.e., if $C$ is contracted in the morphism 
$\cM_{\us}(\vv) \to M_{\us}(\vv)$). It descends to a relative numerical Cartier divisor class $l_{\us} \in N^1(M_{\us}(\vv)/S)$
that is relatively strictly nef.
\end{Thm}
\begin{proof}
In order to simplify notation, we first apply the action of $\C \subset \widetilde{\GL_2^+(\R)}$, the universal cover
of $\C^* \subset \GL_2^+(\R)$, so that we may assume that $Z(\vv) = -1$. We assume for simplicity that $Z$ is defined over $\Q[\ii]$; once we prove Theorem~\ref{thm:deformfamiliystability}, the general case can be reduced to that one as every wall of $\Stab(\cD/S)$ is defined over $\Q$. From the definition of
$\cN(\cD/S)$ and our assumption on the Mukai morphism, there exists $F \in \cD_{\perf}$ and $a \in \Q$ such that $\Im Z(\blank) = a\chi([F], \blank)$; by linearity it is enough to consider the case $a = 1$.

Now recall that the numerical Cartier divisor class $\ell_{\sigma_s}$ is determined by 
\[
\ell_{\sigma_s}.C = \Im Z_s(v_s(p_{\cX_s*} \cE_C)),
\]
where $C \to \cM_{\sigma_s}(\vv)$ is a curve in the moduli stack,
$\cE_C \in \cD_C$ is the associated family of $\sigma_s$-semistable objects, and $p_{\cX_s*} \colon \cD_C \to \cD_s$ is the pushforward.
In the notation above, we have $\ell_{\sigma_s}.C=\chi(F_s,p_{\cX_s*} \cE_C)$.

Noting that $\chi(F_s, E) = \Im Z_s(E) = 0$ for any object $E \in \cD_s$ with $v_s(E) = \vv$, it follows that the (dual of the) determinant line bundle construction can be applied to $F$: namely, if $\widetilde F$ denotes the pullback of $F$ to 
$\cM_{\us}(\vv) \times_S \cX$, and $\cE$ is the universal family, then
$\lHom_{\cM_{\us}(\vv)}(\widetilde{F}, \cE)$ is a complex of rank zero, and
its determinant
\[ \cL_F := \det(\lHom_{\cM_{\us}(\vv)}(\widetilde{F}, \cE))
\]
is a line bundle whose degree on a curve $C$ as above agrees with
$\ell_{\sigma_s}.C$ by adjunction (see also \cite[Proposition~4.4]{BM:proj} for an analogous argument).
We can thus define $\ell_{\us}$ by \[\ell_{\us}:=[\cL_F]\in N^1(\cM_{\us}(\vv)/S).\] As $\ell_{\us}|_{\cM_{\sigma_s}}=\ell_{\sigma_s}$ by construction, its claimed positivity properties are purely a statement on fibers, given by \cite[Positivity Lemma~3.3]{BM:proj}.

It remains to show that it descends to a class on the coarse or good moduli space; we will prove the latter case. By \cite[Theorem~10.3]{Alper:GoodModuli}, it is enough to show that stabilizer groups of geometric points of $\cM_{\us}(\vv)$ with closed image act trivially on $\cL_F$. Such a point corresponds to a polystable object
$E = \oplus_i E_i^{\oplus m_i}$ where $E_i \in \cD_{\overline{s}}$ are distinct stable objects of the same phase, for some
geometric point $\overline{s}$ over $S$; in particular, $\Im Z_{\overline s}(E_i) = 0$. Its stabilizer group is $ \prod_i \GL(m_i, \kappa(\overline{s}))$; each factor acts on $\cL_F|_{E}$ via 
$\det^{\chi(F_{\overline{s}}, E_i)} = \det^{\Im Z_{\overline s}(E_i)} = 1$.
\end{proof}

\section{Deforming stability conditions over \texorpdfstring{$S$}{S}} \label{sec:deformfamiliystability}

The goal of this section is an analogue of Bridgeland's deformation result for stability conditions, Theorem~\ref{thm:deformstability}, for the case of stability conditions over $S$. 

We continue to fix a group homomorphism $v \colon \Knum(\cD/S) \to \Lambda$, and let $\Stab_\Lambda(\cD/S)$ denote the set of stability conditions on $\cD$ over $S$ with respect to $\Lambda$.

\begin{Def}
We define the topology on $\Stab_\Lambda(\cD/S)$ as the coarsest topology such that the canonical map
\[
\Stab_\Lambda(\cD/S) \to \Stab_\Lambda(\cD_s), \quad \us \mapsto \sigma_s
\]
is continuous for every $s \in S$.
\index{StabLambda(D/S)@$\Stab_\Lambda(\cD/S)$, space of stability conditions on $\cD$ over $S$}
\end{Def}

\begin{Thm}\label{thm:deformfamiliystability} 
The space $\Stab_\Lambda(\cD/S)$ of stability conditions on $\cD$ over $S$ is a complex manifold, and the forgetful map 
\[
\cZ \colon \Stab_\Lambda(\cD/S) \to \Hom(\Lambda, \C),
\]
is a local isomorphism. 

More precisely, assume that $\us$ satisfies the support property with respect to the quadratic form $Q$, and write $P_Z \subset \Hom(\Lambda, \C)$ for the connected component containing $Z$ of the set of central charges whose kernel is negative definite with respect to $Q$.
Then there is an open neighborhood $\us \in U \subset \Stab_\Lambda(\cD)$ such that $\cZ|_U \colon U \to P_Z$ is a covering.
\end{Thm}

We follow the proof strategy in \cite{Arend:shortproof}.
As pointed out in Remark~\ref{rem:fiberwiseGL2action}, $\wGL2$ acts on $\Stab_\Lambda(\cD/S)$, lifting the action of $\GL_2(\R)$ on $\Hom(\Lambda, \C) \cong \Hom(\Lambda, \R^2)$.
Therefore, we can use the same simplification as in \cite{Arend:shortproof} and only treat the case of a purely \emph{real} variation of the central charge.
This implies that the hearts $\cA_s$ on the fibers, as well as the local hearts $\cA_C$ for Dedekind schemes $C \to S$ in condition \eqref{enum:curveassumption}, remain unchanged.

More precisely, we can assume by \cite[Section~7]{Arend:shortproof} that $Q$ has signature $(2, \rk \Lambda -2)$, that $Z$ and $Q$ satisfy the normalization of \cite[Lemma~4.2]{Arend:shortproof}, and consider a deformation of the form $W = Z + u \circ p$ where
$p$ is the orthogonal projection $\Lambda_\R \to \Ker Z$ (with respect to $Q$), and $u \colon \Ker Z \to \R$ is a linear map with 
operator norm (with respect to the standard norm on $\R$, and the norm $\norm{\cdot}$ induced by $-Q$ on $\Ker Z$) satisfying $\norm{u} < 1$; since it is sufficient to prove the theorem for small deformations of $Z$, we may later choose a smaller bound
$\norm{u} < \delta$. We also note that by the normalization in \cite[Lemma~4.2]{Arend:shortproof} $\abs{W(v) - Z(v)} \leqslant \norm{u} \abs{Z(v)}$ for all $v$ with $Q(v) \geqslant 0$.

Let $\uvs$ denote the collection of stability conditions on the fibers with central charge $W$, where each $\varsigma_s = (W, \cQ_s)$ is induced
from $\sigma_s = (Z, \cP_s)$ via Theorem~\ref{thm:deformstability}.
To show that $\uvs$ is a stability condition on $\cD$, we only need to show that properties \eqref{enum:stabilityopenbasechange}--\eqref{enum:curveassumption} of Definition~\ref{def:familyfiberstabilities} and
\eqref{enum:fiberwisesupport}--\eqref{enum:bounded}  Definition~\ref{def:fiberwisesupport} are also satisfied for $\uvs$.

We begin with a standard argument comparing the slicings at each point.
For $\phi \in (0, 1)$,
let $\epsilon \in [0, \frac 12]$ be such that $\frac{\sin \pi \epsilon}{\sin \pi \phi} = \norm{u}$; then $0 < \phi - \epsilon < \phi + \epsilon < 1$.
\begin{Lem} \label{lem:boundQsPs}
For all $s \in S$ we have
$\cQ_s(\phi) \subset \cP_s[\phi - \epsilon, \phi + \epsilon]$. 
\end{Lem}
\begin{proof}
Let $E \in \cQ_s(\phi)$ and let $A \subset E$ be the first step of its HN filtration with respect to $Z$. Then as $Q(A)\geqslant 0$,
\[
\phi \geqslant \phi(W(A)) \geqslant \phi\bigl( Z(A) + \norm{u} \abs{Z(A)} \bigr).
\]
The last term depends only on $\phi(Z(A))$, and the law of sines in the triangle
$0, Z(A)$, $Z(A) + \norm{u} \abs{Z(A)} $ shows that it equals $\phi$ for $\phi(Z(A)) = \phi + \epsilon$. Therefore,
$\phi^+(E) = \phi(Z(A)) \leqslant \phi + \epsilon$. An analogous argument with the maximal destabilizing quotient of $E$ shows
$\phi^-(E) \geqslant \phi - \epsilon$.
\end{proof}

\begin{proof}[Proof of Theorem~\ref{thm:deformfamiliystability}]
Since each $\varsigma_s$ satisfies the support property with respect to the same quadratic form as $\sigma_s$, property \eqref{enum:fiberwisesupport} will be automatic. Also, \eqref{enum:Zsindependentbasechange} holds by construction of $\Knum(\cD/S)$.

By Proposition~\ref{prop:yeswehaveopennessofflatness}, $\cP[\phi-\epsilon, \phi+\epsilon]$ is an open substack of $\cMpug(\cD/S)$; 
to prove that $\uvs$ satisfies \eqref{enum:stabilityopenbasechange}, it thus remains to show that $\fM_{\uvs}^\st(\vv) \subset \cP[\phi - \epsilon, \phi + \epsilon]$ is an open substack. 
So %(after possibly replacing $T$ by an open subset)
consider an object $E \in \cD_T$ with
$E_t \in \cP_t[\phi - \epsilon, \phi + \epsilon]$ for all $t \in T$.
Applying Lemma~\ref{lem:boundQsPs} to a quotient $E_{\overline{t}} \onto Q$ after base change to the algebraic closure $\cA_{\overline{t}}$, we see that any such quotient that is $W$-semistable with $\phi(W(Q)) \leqslant \phi$ satisfies $\phi(Z(Q)) \leqslant \phi + \epsilon$. 
In particular, any such quotient occurs in the Quot scheme $\Quot_T^{\leqslant \phi + \epsilon}(E)$ featured in our Grothendieck Lemma~\ref{lem:GrothendieckLemma}.

Since the class of objects in $\Knum(\cD/S)$ is locally constant in families, the condition $\phi(W(Q)) \leqslant \phi$ picks out a union of finitely many connected components of $\Quot_T^{\leqslant \phi + \epsilon}(E)$. 
The union of their images in $T$ is the locus where $E$ is not geometrically $W$-stable. 
The image of each component is closed by universal closedness of the Quot scheme in Lemma~\ref{lem:GrothendieckLemma}.
Therefore, openness of geometric stability holds for $\uvs$, verifying \eqref{enum:stabilityopenbasechange}.

Now fix $\vv \in \Lambda$. 
By Lemma~\ref{lem:boundQsPs}, we have $\fM_{\uvs}(\vv) \subset \cP([\phi - \epsilon, \phi + \epsilon] ; \vv)$. 
The latter is bounded by Lemma~\ref{lem:CPIvbounded}, hence also the former, establishing condition \eqref{enum:bounded}.

Finally, given a Dedekind scheme $C \to S$ essentially of finite type over $S$, the fiberwise collection of t-structures on $\cD_c$ for $c \in C$ induced by $\us$ integrates to a $C$-local heart $\cA_C$, in the sense of Definition~\ref{def:integrateststructure20210217}; it universally satisfies openness of flatness by Proposition~\ref{prop:yeswehaveopennessofflatness}. Since the hearts $\cA_c$ on the fibers are the same for $\sigma_c$ and $\varsigma_c$, the same holds for $\uvs$. 
We conclude with Theorem~\ref{thm:HNstructureviafibers}.
\end{proof}

\section{Inducing stability conditions over \texorpdfstring{$S$}{S}}\label{sec:inducing}

Let $\cD \subset \Db(\cX)$ be 
an $S$-linear strong semiorthogonal component of finite cohomological amplitude, with a relative exceptional collection $E_1, \dots, E_{m}$ (see Definition~\ref{def:relativeexceptional}). We write
$\cD = \langle \cD_1, \cD_2 \rangle$ for the $S$-linear semiorthogonal decomposition given by Lemma~\ref{lem-E1Em-sod} with $\cD_2 = \langle \alpha_{E_1}(\Db(S)), \dots, \alpha_{E_m}(\Db(S)) \rangle$.
The goal of this section is a criterion to induce a stability condition on $\cD_1$ over $S$ from a weak stability condition on $\cD$ over $S$;
this generalizes the absolute case treated in \cite[Proposition~5.1]{BLMS}. Recall that we already studied how to induce local t-structures on $\cD_1$ in Section~\ref{sec:inducingtstructures}.

We consider the saturated subgroup generated by the image of $v$
\begin{equation}\label{eqn:DefLambda1}
\Lambda_1:=\langle v(K_{\num}(\cD_1/S)) \rangle \subset \Lambda.
\end{equation}
Also recall the subgroup $\Lambda_0 \subset \Lambda$ introduced in Definition~\ref{def:Lambda0}, generated by classes of semistable objects with vanishing central charge.

\begin{Thm}\label{thm:InducingStabilityOverBase}
In Setup~\ref{setup-rel-stability} 
we further assume that $g$ is smooth.
Let $\us = \left(\sigma_s = (Z_s, \cA_s)\right)_{s \in S}$ be a weak stability condition on $\cD$ over $S$.
Assume the following conditions hold: 
\begin{enumerate}[{\rm (1)}] 
 \item $(E_i)_s \in \cA_s$ for all $i$ and $s \in S$. 
 \item $\rS_{\cD_s} ((E_i)_s) \in \cA_s[1]$ for all $i$ and $s \in S$, where $\rS_{\cD_s}$ denote the Serre functor of $\cD_s$. 
 \item \label{enum:ExcnotinA0} $v(E_i) \notin \Lambda_0$ for all $i$. 
 \item $\Lambda_0 \cap \Lambda_1 = 0$. 
 \item \label{enum:boundednessforinducing} For all $\vv \in \Lambda$, the set
 \[
 \left\{F \in M_{\us}(\vv') \colon \vv' \in \vv + \Lambda_0, \chi(E_i, F) \geqslant 0\, \text{ for all $i = 1, \dots, m$} \right\}
 \]
 is bounded.
\end{enumerate}
For each $s \in S$, let $\cA_{s,1}$ be the heart in $(\cD_1)_s$ given by Corollary~\ref{cor:induce-relative-t-fiberwise}, and let $Z_{s,1}$ be the central charge given by the restriction of $Z_s$ along $K((\cD_1)_s) \to K(\cD_s)$. 
Then the collection 
\[\us_1 = \left((\sigma_{s})_1 = (Z_{s, 1}, \cA_{s, 1})\right)_{s \in S}\] 
is a stability condition on $\cD_1$ over $S$ with respect to $\Lambda_1$.
\end{Thm}

\begin{proof}
By \cite[Proposition~5.1]{BLMS},
$\us_1 = \left((\sigma_{s})_1 = (Z_{s, 1}, \cA_{s, 1})\right)_{s \in S}$ is a collection of numerical stability conditions.
We need to check that $\us_1$ satisfies the conditions \eqref{enum:Zsindependentbasechange}-\eqref{enum:bounded} 
in Definitions~\ref{def:familyfiberstabilities} and \ref{def:fiberwisesupport}. 
To simplify notation, we assume throughout the proof that there is only one exceptional object $E$; the general case holds by similar arguments. 

\begin{step}{1}
$\us_1$ universally has locally constant central charges, i.e., condition \eqref{enum:Zsindependentbasechange} holds. 
\end{step}
This is automatic because the central charges of $\us_1$ are the restrictions of those of $\us$. 

\begin{step}{2}
\label{step-2-inducing}
The fiberwise collection of t-structures $\ut_1$ underlying $\us_1$ universally satisfies openness of flatness. 
Moreover, for any Dedekind scheme $C \to S$ essentially of finite type over $S$, 
$\ut_1$ integrates over $C$ to a bounded $C$-local t-structure 
whose heart has a $C$-torsion theory. 
\end{step}

By Proposition~\ref{prop:yeswehaveopennessofflatness} the fiberwise collection of t-structures underlying $\us$ 
universally satisfies openness of flatness, so by construction the same follows for $\us_1$. 
Let $\cA_C$ be the heart of the HN structure on $\cD_C$ that integrates $\us$ over $C$. 
Since by assumption $E_c \in \cA_c$ for all $c \in C$, Lemma~\ref{lem:flatinheart} shows $E_C \in \cA_C$. 
Now Corollary~\ref{cor:induce-relative-t-fiberwise} applies to show $(\cA_C)_1$ is the heart of a bounded $C$-local t-structure on $\cD_1$, which 
by Remark~\ref{rem-inducing-commute-fibers} integrates $\ut_1$ over $C$. 
Finally, $(\cA_C)_1$ has a $C$-torsion theory by Theorem~\ref{thm:Ctorsiontheoryautomatic}. 

\begin{step}{3} \label{step-inducing-qform}
$\us_1$ satisfies condition~\eqref{enum:fiberwisesupport} in the support property. 
\end{step}

The existence of the uniform quadratic form $Q$ on $(\Lambda_1)_\R$ with properties \eqref{enum:fiberwisesupporta} and \eqref{enum:fiberwisesupportb} follows directly from the proof of \cite[Proposition~5.1]{BLMS}, as the construction of the quadratic form on each fiber given there depends only on the slopes of the exceptional objects.

\begin{step}{4} \label{step-inducing-boundedness}
$\us_1$ satisfies the following boundedness property: given $\vv \in \Lambda_1$ and $0 < \phi_0 < \phi_1 \leqslant 1$, there exists a scheme $B$ of finite type over $S$ and a $\ut_1$-flat object $\widetilde{F} \in (\cD_{1})_B$, such that every geometric point of $\cP_{\us_1}([\phi_0, \phi_1], \vv)$ is of the form $\widetilde{F}_{\bar b}$ for some geometric point $\bar b$ of $B$.
\end{step} 

Indeed, under the simplifying assumption that we have only one exceptional object $E$, let 
\[
\phi^- := \min\left(\phi_0, \phi^-(E)\right),
\]
let $\vv \in \Lambda_1$, let $\bar{s}$ be a geometric point of $S$, and let $F\in \cP_{\us_1}([\phi_0, \phi_1],\vv)(\bar{s})$. 
By \cite[Remark~5.12]{BLMS}, we have $F \in \cP_{\us}([\phi^-, 1])$.
Now consider any short exact sequence $A \into F \onto A'$ in $\cA$ corresponding to a Harder-Narasimhan filtration step of $F$ with respect to $\sigma_{\bar s}$.
The class of $A'$ in $\Lambda/\Lambda_0$ is bounded by the support property. Moreover, as in the proof of Corollary~\ref{cor:induce-relative-t-fiberwise} we see that
$\Ext^k(E, A') = 0$ for $k \geqslant 2$, and similarly
$\Ext^1(E, A') = \Ext^2(E, A) = 0$. Therefore, $\chi(E,A') \geqslant 0$, and the set of such $A'$ is bounded by assumption~\eqref{enum:boundednessforinducing}. This in turn implies that the set of classes $v(A) \in \Lambda$ of $A$, and hence also the set of $A$ itself is bounded. Using Lemma~\ref{lem:ExtensionOfBounded} one concludes the proof of Step~\ref{step-inducing-boundedness}.

\begin{step}{5}
\label{step-5-inducing} 
$\us_1$ universally satisfies openness of geometric stability and satisfies boundedness, i.e., conditions \eqref{enum:stabilityopenbasechange} and \eqref{enum:bounded} hold. 
\end{step}

By Lemma~\ref{Lem-uogs-ftbc}, to show \eqref{enum:stabilityopenbasechange}
it suffices to show that if $T \to S$ is finite type morphism 
from a connected affine scheme and $F \in \rD(\cX_T)$ is a $T$-perfect object, then 
the set 
\begin{equation*}
U = \set{ t \in T \sth F_t \in (\cD_1)_t \text{ and is geometrically} \ \sigma_{1,t} \text{-stable} }
\end{equation*} 
is open. 
By Step~\ref{step-2-inducing} the fiberwise collection of t-structures $\utau_{1}$ 
underlying $\us_1$ universally satisfies openness of flatness, so we 
may assume $F_t \in \cA_{t,1}$ for all $t \in T$, i.e., $F \in \cM_{\ut_1}(T)$. 
By \cite[Remark~5.12]{BLMS}, if $F_t$ is geometrically $\sigma_{1,t}$-stable, then we have 
\begin{equation*}
 \phi_{\sigma_t}^-(F_t) \geqslant \phi_0 := \min(\phi_{\sigma_{1,t}}(F_t), \phi^-(E_T)) , 
\end{equation*} 
and thus 
$F_t \in \cP_{\sigma_t}(\phi_0, 1]$. 
By Proposition~\ref{prop:yeswehaveopennessofflatness}, $\us$ universally satisfies openness of lying in $\cP(\phi_0, 1]$, so 
we may therefore assume $F_t \in \cP_{\sigma_t}(\phi_0, 1]$ for all $t \in T$. 

Let $\Quot_{T, \ut_1}(F) \to T$ denote the Quot space of $F$ with respect to the 
fiberwise collection of t-structures $\ut_1$. 
As in Lemma~\ref{lem:GrothendieckLemma}, 
for $\phi \in (0,1)$ we let $\Quot_{T, \ut_1}^{\le\phi}(F)$ be the subfunctor 
of $\Quot_{T, \ut_1}(F)$ which assigns to 
$T' \in (\Sch/T)$ the set of $(F_{T'} \to Q) \in \Quot_{T, \ut_1}(F)(T')$ satisfying 
$\phi(Q_{t'}) \leqslant \phi$ for all $t' \in T'$. 
We claim that $\Quot_{T, \ut_1}^{\le\phi}(F)$ is an algebraic space of finite type over $T$, 
and the morphism $\Quot_{T, \ut_1}^{\le\phi}(F) \to T$ is universally closed. 

Indeed, using the same arguments as in the proof of Proposition~\ref{proposition-quot-algebraic} it is enough to know that the quotients occurring in $\Quot_{T, \ut_1}^{\le\phi}(F)$ belong to a bounded $\ut_1$-flat family of objects; since each such quotient lies in $\cP_{\us_1}(\phi_0, 1](\vv)$ for some $\vv \in \Lambda_1$, and since the set of possible $\vv$ is finite by the same arguments as in Lemma~\ref{lem:GrothendieckLemma}, this follows from the previous steps.

The non-geometrically stable locus of $F$ is the union of the images of all connected components of $\Quot^{\le\phi}(F)$ except those where the quotient $Q$ satisfies $v(Q) = v(F)$; hence it is closed by Proposition~\ref{proposition-quot-valuative-criteria}, and thus being geometrically stable is open in $T$.

To conclude Step~\ref{step-5-inducing}, we note that universal openness of geometric stability, combined with the weaker boundedness statement in Step~\ref{step-inducing-boundedness}, immediately
implies boundedness of $\us_1$.

\begin{step}{6}
$\us_1$ integrates to a HN structure along any Dedekind scheme $C \to S$ essentially of finite type over $S$, i.e., condition~\eqref{enum:curveassumption} holds. 
\end{step}

This follows from the previous steps, as explained in Remark~\ref{rem:cond3viahearts}.
\end{proof}

\newpage
\part{Construction}
\label{part:Tilting}

\section{Main construction statements}\label{sec:MainConstr}

The main goal of Part~\ref{part:Tilting} is to construct stability conditions for families of surfaces, or of threefolds that individually admit stability conditions.
The first step is to show that we can \emph{tilt} the weak stability condition given by slope-stability on the fibers in order to obtain a relative version of \emph{tilt-stability} (see Section~\ref{sec:tiltingslope}); for threefolds, we have to tilt a second time in order to produce a stability condition over the base (see Section~\ref{sect:tiltingtilt}).

In more detail, we formulate slope stability as a weak stability condition with heart $\Coh X$, and central charge depending on $\ch_0, \ch_1$ only.
We use this to obtain a new heart $\Coh^\beta X$ via tilting.
Using the classical Bogomolov--Gieseker inequality\footnote{The classical Bogomolov--Gieseker inequality \cite{Reid:Bog,Bogomolov:Ineq,Gieseker:Bog} holds in characteristic zero.
In positive characteristic it holds e.g.~for varieties that can be lifted to characteristic zero, see \cite[Theorem~1]{Langer:Higgs}; the weaker version of \cite{Langer:positive} is not enough to define tilt-stability.}, one then constructs a weak stability condition, called \emph{tilt-stability}, with heart $\Coh^\beta X$ and with central charge depending on $\ch_0, \ch_1, \ch_2$.
\index{Cohbeta(X)@$\Coh^\beta X_s$, tilted heart of $\Coh X_s$ at slope $\beta$}
Tilting again produces a heart $\cA^{\alpha,\beta}$. 
Assuming a conjectural Bogomolov--Gieseker type inequality for tilt-stable objects in $\Coh^\beta X$, proposed in \cite{BMT:3folds-BG}, one can then produce an actual stability condition with heart
$\cA^{\alpha,\beta}$.

Generalizing each of these steps to families of threefolds will lead to the following result; the existence of moduli spaces generalizes \cite[Theorem~1.5]{PT15:bridgeland_moduli_properties} to the relative setting.

\begin{Thm} \label{mainthm:construction}
Let $g \colon \cX \to S$ be a polarized flat family of smooth projective varieties.
\begin{enumerate}[{\rm (1)}]
\item\label{enum:construction1} If the fibers of $g$ are either one-dimensional, or two-dimensional satisfying the classical Bogomolov--Gieseker inequality, then the standard construction of stability conditions on curves or surfaces produces a stability condition $\us$ on $\Db(\cX)$ over $S$.
\item\label{enum:construction2} If the fibers of $g$ are three-dimensional and additionally satisfy the conjectural Bogomolov--Gieseker inequality of \cite{BMT:3folds-BG, BMS:abelian3folds}, then the construction of stability conditions
proposed in [ibid.] produces a stability condition $\us$ on
$\Db(\cX)$ over $S$.
\end{enumerate} 
In each of these situations, given a vector $\vv\in\Lambda$ for the corresponding choice of $\Lambda$, the relative moduli space $\fM_{\us}(\vv)$ exists as an algebraic stack of finite type over $S$, and the map $\fM_{\us}(\vv) \to S$ satisfies the strong valuative criterion of universal closedness. 
If $\fM_{\us}(\vv)= \fM_{\us}^\st(\vv)$, then $\fM_{\us}(\vv)$ has a coarse moduli space $M_{\us}(\vv)$, proper over $S$.
Finally, in characteristic $0$, $\fM_{\us}(\vv)$ always admits a good moduli space $M_{\us}(\vv) \to S$ which is proper over $S$.
\index{Msigmav@$\fM_{\us}^{}(\vv)$ ($\fM^{\st}_{\us}(\vv))$,!substacks of $\cM_{\utau}$, parameterizing $\us$-semistable (geometrically $\us$-stable) objects of class $\vv$ and phase $\phi$}
\index{Msigmav@$M_{\us}(\vv)$, coarse moduli space of $\fM_{\us}(\vv)=\fM_{\us}^{\st}(\vv)$}
\end{Thm}

We in fact prove a slightly stronger version of \eqref{enum:construction1} in Proposition~\ref{prop:tiltstabfamily}, and a weak version of \eqref{enum:construction2} but for arbitrary dimension, and without assuming the conjectural inequality, in Proposition~\ref{prop:RotatingTiltStability}. The latter will be crucial for us in the case of cubic fourfolds in Part~\ref{part:CubicFourfolds}.

\begin{Rem}\label{rem:BomegaNot||}
The construction of stability conditions via tilting depends on two globally defined $\Q$-divisors $\omega$ and $B$ on $\cX$, with $\omega$ relatively ample.
As in \cite{BMS:abelian3folds}, we assume that $\omega$ and $B$ are parallel, namely 
\[
\omega=\alpha H\qquad \text{and}\qquad B=\beta H,
\]
where $H=c_1(\cO_{\cX}(1))$ is the polarization.
Using \cite[Section~3]{PT15:bridgeland_moduli_properties} one can extend our arguments to $B$ and $\omega$ not necessarily parallel.
\end{Rem}

\section{Tilting slope-stability in families}\label{sec:tiltingslope}

In this section we show that the notion of tilt-stability from \cite{BMT:3folds-BG} works in families.
We will divide the construction into two steps.
We first rotate slope-stability and show that this gives a family of weak stability conditions as well.
Then we use deformation of families of weak stability conditions to prove the analogue statement for tilt-stability.
When specialized to families of curves or surfaces, this will prove part \eqref{enum:construction1} of Theorem~\ref{mainthm:construction}.

We continue to work in Setup~\ref{setup-rel-stability}, 
with the additional assumptions that $g\colon \cX \to S$ is a smooth projective morphism of relative dimension $n\leqslant 3$.
We fix a relatively ample divisor $\cO_{\cX}(1)$.

Since $g\colon \cX \to S$ is smooth, instead of working with the relative Hilbert polynomial as in Example~\ref{ex:LambdaFamilyChernClasses} we can work with Chern characters. 
To make this precise, observe that the Chern characters on the fibers, when paired with the appropriate power of the relative ample class $H = c_1\left(\cO_{\cX}(1)\right)$, yield maps, for all $s\in S$,
\index{chXs@$\ch_{\cX_s}\colon \Knum(\Db(\cX_s))\to\Q^{n+1}$, Chern characters on the fibers}
\[
\ch_{\cX_s}\colon \Knum(\Db(\cX_s))\to\Q^{n+1}.
\]
These maps are locally constant in families, and thus they factor through a map
\index{chXS@$\ch_{\cX/S}\colon \Knum(\Db(\cX)/S)\to\Q^{n+1}$}
\[
\ch_{\cX/S}=\oplus_{i=0}^{n} \ch_{\cX/S,i} \colon \Knum(\Db(\cX)/S)\to\Q^{n+1}.
\]
By the Hirzebruch-Riemann-Roch theorem, the image of $\ch_{\cX/S}$ coincides with the image of the Hilbert polynomial; we denote it by $\Lambda$.

Consider the weak stability condition 
\index{sigmauSlope@$\us:=(\sigma_s=(\mathfrak{i}\ch_{\cX/S,0}-\ch_{\cX/S,1},\Coh \cX_s))$!(weak) stability condition on $\Db(\cX)$ over $S$ given by slope-stability on each fiber}
\[
\us:=\left(\sigma_s=\left(\mathfrak{i}\ch_{\cX/S,0}-\ch_{\cX/S,1},\Coh \cX_s\right)\right)
\]
on $\Db(\cX)$ over $S$ given by slope-stability on each fiber; its properties can be verified as in Example~\ref{ex:SlopeStabilityAsProperFamilyofWeakStability}.
The weak stability conditions $\sigma_s$ have the tilting property, and, for every Dedekind domain $C$, the HN structure $\sigma_C$ has the tilting property as well, see Examples~\ref{ex:slopestabilityhastiltingproperty} and \ref{ex:slopestabilitycurveshastiltingproperty}.

If $n\geqslant2$, we let $\Lambda_0\subset\Lambda$ be the subgroup generated by the image of $\ch_{\cX/S,2}\oplus\ch_{\cX/S,3}$ and, if $n=3$, we let $\Lambda_0^\sharp\subset\Lambda_0$ be the one generated by the image of $\ch_{\cX/S,3}$.
We write $\eeta$ for the class of a point, which generates $\Lambda_0^\sharp$.
Finally, we denote by $\oLambda:=\Lambda/\Lambda_0$ and by $\oLambda^\sharp:=\Lambda/\Lambda_0^\sharp$.

\subsection{Rotating slope-stability in families}\label{subsec:rotating1}

Let $\beta\in\Q$.
To simplify the notation, we write
\index{chXSbeta@$\ch_{\cX/S}^\beta$, twisted Chern characters}
\[
\ch_{\cX/S}^\beta (\blank) =
\ch_{\cX/S}\left(e^{-\beta H}\cdot \blank\right) \in\Lambda_\Q.
\]

For every $s\in S$ we define
\[
\sigma_s^{\sharp \beta} := \left( Z_s=\mathfrak{i}\ch_{\cX_s,1}^\beta+\ch_{\cX_s,0}^\beta,\Coh^\beta \cX_s\right),
\]
where $\Coh^\beta \cX_s$ is the tilt of $\Coh \cX_s$ at the slope $\beta$, defined as in Section~\ref{subsec:tiltingweakstability}.
\index{Cohbeta(X)@$\Coh^\beta \cX_s$, tilted heart of $\Coh \cX_s$ at slope $\beta$}

\index{sigmauSlopeRot@$\us^{\sharp\beta} := (\sigma_s^{\sharp \beta}=( \mathfrak{i}\ch_{\cX_s,1}^\beta+\ch_{\cX_s,0}^\beta,\Coh^\beta \cX_s))$!(weak) stability condition on $\Db(\cX)$ over $S$ given by rotating slope-stability on each fiber}
\begin{Prop}\label{prop:RotatingSlopeStability}
The collection $\us^{\sharp\beta}:=(\sigma_s^{\sharp \beta})$ is a (weak) stability condition on $\Db(\cX)$ over $S$ with respect to $\oLambda$. 
Moreover, if $n=3$ and we fix $\vv\in\Lambda$, then $\fM_{\us^{\sharp\beta}}^\st(\vv+b\eeta) = \emptyset$ for $b \gg 0$.
\end{Prop}

\begin{Rem} \label{rem:trivialboundedness}
In the case of slope-stable torsion free sheaves, the analogue of the last claim follows from boundedness: for every slope-stable sheaf $E$ of class $\vv + b\eeta$, the kernel of any surjection $E \onto T$ for $T$ a sheaf of length $b$ will be stable of class $\vv$; thus the dimension of $M_\sigma^\st(\vv)$ is at least $3b$.
\end{Rem}

\begin{proof}
For $n=1$ this was already observed in Remark~\ref{rem:fiberwiseGL2action}, so we assume $n\geqslant2$.

As shown in Proposition~\ref{prop:RotateWeakStability}, $\sigma_s^{\sharp \beta}$ is a weak stability condition on $\Db(\cX_s)$ for all $s\in S$.
There is nothing to prove for Definition~\ref{def:fiberwisesupport}.\eqref{enum:fiberwisesupport}. 
Conditions \eqref{enum:A0noetheriantorsionfibers} and \eqref{enum:curveassumptionweak} in Definition~\ref{def:familyfiberstabilities} follow from Propositions~\ref{prop:RotateWeakStability} and \ref{prop:tiltingweakHNstructures}, respectively, and condition \eqref{enum:Zsindependentbasechange} is immediate. 

Recall from Lemma~\ref{lem:RotateWeakStabilitySemistableObjects} the classification of $\sigma_s^{\sharp \beta}$-semistable objects $E \in \Coh^\beta \cX_s$:
\begin{enumerate}[{\rm (1)}] 
\item\label{enum:RotatingSlopeStability1} if $\ch_{\cX_s,0}(E)\geqslant 0$, then $E$ is a torsion free slope-semistable sheaf, or a torsion sheaf with support either pure of codimension one, or of codimension $\geqslant 2$;
\item\label{enum:RotatingSlopeStability2} if $\ch_{\cX_s,0}(E)< 0$, then $E$ is an extension
\begin{equation} \label{eq:UEV-extension}
 U[1] \to E \to V
\end{equation}
where $U$ is a torsion free slope-semistable sheaf and $V$ is a torsion sheaf supported in codimension $\geqslant 2$.
Moreover, if either $\ch_{\cX_s,1}^\beta(E)>0$ or $E$ is $\sigma_s^{\sharp \beta}$-stable, then $\Hom(V',E)=0$, for all sheaves $V' \in \Coh \cX_s$ supported in codimension $\geqslant 2$; in particular, $U$ is reflexive.
\end{enumerate}

We use this to show openness in the sense of 
conditions \eqref{enum:stabilityopenbasechange} and \eqref{enum:genOpSstweak} for geometrically stable/semistable objects of the form \eqref{eq:UEV-extension}. Standard arguments show openness of the condition that $\rH^{-1}$ is torsion-free, and that $\rH^0$ is supported codimension two. Using a flattening stratification for $\rH^{i}(E)$ shows that the semistable locus is constructible, since being reflexive or semistable is open in flat families of sheaves. It remains to show that the unstable locus is closed under specialization. This is easy to show for both the Hom-vanishing from sheaves supported in codimension two and the slope-semistability of $\rH^{-1}$: in both cases, we first lift the destabilizing sheaf from the generic point to the appropriate DVR, and then extend the morphism to one that specializes to a non-zero morphism.

To finish the proof, we will simultaneously prove the second claim of the proposition and boundedness in the sense of Definition~\ref{def:fiberwisesupport}.\eqref{enum:bounded} by proving that $\coprod_{b \geqslant 0} \fM_{\us^{\sharp\beta}}^{\st}(\vv + b\eeta)$ is bounded. First we consider the case in which $\ch_{\cX_s, 0}(E) \geqslant 0$. Torsion sheaves can never be strictly $\sigma^{\sharp\beta}$-stable; thus the claim follows directly from boundedness of slope-stable sheaves, \cite[Theorem~4.2]{Langer:positive}, and Remark~\ref{rem:trivialboundedness}.

Thus we are left to consider the case $\ch_{\cX_s, 0}(E) < 0$. After modifying $\beta$ slightly if necessary (which will not affect stability of objects of class $\vv + b\eeta$), we may assume $\ch_{\cX_s, 1}^\beta(E) > 0$. Let $\mathbb{D}_s = \cHom (-,\cO_{\cX_s}) [1]$. By \cite[Lemma~2.19]{BLMS}, for such $E$ there is an exact triangle
\[
E^\sharp \to \mathbb{D}_s(E) \to W[-1]
\]
where $E^\sharp$ is a torsion free slope-stable sheaf and $W$ is supported in codimension $3$.
Moreover, for all $i=0,\ldots,3$, $\ch_{\cX_s,i}(\mathbb{D}_s(E))=(-1)^{i+1}\ch_{\cX_s,i}(E)$.
Hence
\begin{equation*}
 \begin{split}
 &\ch_{\cX_s,i}(E^\sharp) = (-1)^{i+1}\ch_{\cX_s,i}(E), \text{ for all $i=0,1,2$, and}\\
 &\ch_{\cX_s,3}(E^\sharp)=\ch_{\cX_s,3}(E)+\ch_{\cX_s,3}(W)\geqslant\ch_{\cX_s,3}(E).
 \end{split}
\end{equation*}

Since $E^\sharp$ is a torsion free slope-stable sheaf, the first case in \cite[Theorem~4.4]{Langer:positive} shows that it belongs to a bounded family. This also implies that $\ch_{\cX_s,3}(W)$ can only take finitely many values, and so $W$ also belongs to a bounded family.
By Lemma~\ref{lem:ExtensionOfBounded}, the same holds for $\D_s(E)$.
Since the morphism $g$ is smooth, the duality functor exists in families and commutes with base change; therefore, $E$ also belongs to a bounded family, as we wanted.
\end{proof}

\subsection{Tilt-stability in families}\label{subsect:tiltstabfamility}

Let $\alpha,\beta\in\Q$, $\alpha>0$.
We now deform the weak stability condition $\us^{\sharp\beta}$ of Section~\ref{subsec:rotating1} with respect to $\alpha$.

For every $s\in S$ we define
\index{Ztilt@$Z^{\alpha,\beta}_s:=\mathfrak{i}\ch_{\cX_s,1}^\beta+\frac{\alpha^2}{2}\ch_{\cX_s,0}^\beta-\ch_{\cX_s,2}^\beta$,!central charge for tilt-stability}
\[
\sigma_s^{\alpha,\beta} := \left( Z^{\alpha,\beta}_s=\mathfrak{i}\ch_{\cX_s,1}^\beta+\frac{\alpha^2}{2}\ch_{\cX_s,0}^\beta-\ch_{\cX_s,2}^\beta,\Coh^\beta \cX_s \right).
\]

We can now prove the following. 

\begin{Prop}\label{prop:tiltstabfamily}
Assume that the Bogomolov--Gieseker inequality holds for slope-stable sheaves on the fibers of $g\colon \cX \to S$, namely for all $s\in S$ and for all $\sigma_s$-stable sheaves $E$
\index{Deltas@$\Delta_s:=\ch_{\cX_s,1}^2-2\ch_{\cX_s,0}\ch_{\cX_s,2}$, discriminant}
\[
\Delta_s(E):=\ch_{\cX_s,1}(E)^2-2\ch_{\cX_s,0}(E)\ch_{\cX_s,2}(E)\geqslant 0.
\]
Then the collection $\us^{\alpha,\beta}:=(\sigma_s^{\alpha,\beta})$ is a (weak) stability condition on $\Db(\cX)$ over $S$ with respect to $\oLambda^\sharp$. Moreover $\sigma_s^{\alpha,\beta}$ has the tilting property, for all $s\in S$, and for any base change $C\to S$ from a Dedekind scheme $C$, the HN structure $\sigma^{\alpha,\beta}_C$ has the tilting property as well.
Finally, if $n=3$ and we fix $\vv\in\Lambda$, then $\fM_{\us^{\alpha,\beta}}^\st(\vv+b\eeta) = \emptyset$ for $b \gg 0$.
\index{sigmauTilt@$\us^{\alpha,\beta}:=(\sigma_s^{\alpha,\beta}=(Z^{\alpha,\beta}_s,\Coh^\beta \cX_s))$!(weak) stability condition on $\Db(\cX)$ over $S$ given by tilt-stability on each fiber}
\end{Prop}

Combined with Theorem~\ref{thm:deformfamiliystability} in case $n=1,2$, this in particular gives Theorem~\ref{mainthm:construction}.\eqref{enum:construction1}.

\begin{proof}
The proof is a relative version of \cite[Section~4.5]{PT15:bridgeland_moduli_properties}. The extension can be done analogously as in the proof of Theorem~\ref{thm:deformfamiliystability}; the difference is that at
$\sigma^{\infty, \beta} := \sigma^{\sharp\beta}$ we start with the weak stability condition given in Proposition~\ref{prop:RotatingSlopeStability}. The key observation is the following: 

\begin{claim}{1}
Given
$\overline{\vv} \in \overline{\Lambda}^\sharp$, then for $\sigma^{\alpha, \beta}$-semistable objects
$E \in \Coh^\beta \cX_s$ of class $\overline{\vv}$, there are only finitely many possible classes in $\overline{\Lambda}^\sharp$ of HN factors of $E$ with respect to $\sigma^{\infty, \beta}$.
\end{claim}

This claim is purely a statement on fibers, and is shown for example in the proof \cite[Theorem~7.3.1]{BMT:3folds-BG} or \cite[Theorem~3.5]{BMS:abelian3folds}; see also \cite[Lemma~2.26]{PT15:bridgeland_moduli_properties}. The implied finiteness of wall-crossing as $\alpha \to \infty$ also shows that the Bogomolov--Gieseker inequality is preserved: for each wall-crossing, it follows by simple linear algebra, see \cite[Lemma~A.6]{BMS:abelian3folds}. 

Applying the strong boundedness statement of Proposition~\ref{prop:RotatingSlopeStability} to each HN factor of $E$, in combination with the previous claim then immediately gives the following refinement.

\begin{claim}{2}\label{claim2} Given
$\vv \in \Lambda$, then for $\sigma^{\alpha, \beta}$-semistable objects
$E \in \Coh^\beta \cX_s$ of class $\vv$, there are only finitely many possible classes in $\Lambda$ of HN factors of $E$ with respect to $\sigma^{\infty, \beta}$.
\end{claim}

Boundedness as in Definition~\ref{def:fiberwisesupport}.\eqref{enum:bounded}, as well as the stronger boundedness claimed in the last statement of the Proposition, follow immediately;
see also \cite[Corollary~4.18]{PT15:bridgeland_moduli_properties}.

The condition on \eqref{enum:A0noetheriantorsionfibers} and the second part of \eqref{enum:curveassumptionweak} in Definition~\ref{def:familyfiberstabilities} follow trivially from the corresponding properties of $\us^{\sharp\beta}$, as
$Z^{\alpha, \beta}(\blank) = 0$ for objects of $\Coh^\beta \cX_s$ or
$\Coh^\beta \cX_C$ is a stronger condition than $Z^{\sharp\beta}(\blank) = 0$, and as the property of being a noetherian torsion subcategory is clearly preserved by passing to a smaller subcategory. 

We now want to show universal openness of geometric stability; we can restrict to phases $\phi$ with $0 < \phi < 1$ (as $\sigma^{\alpha, \beta}$-semistable objects of phase 1 are the same as $\sigma^{\sharp\beta}$-semistable objects). Lemma~\ref{lem:boundQsPs} applies literally in our situation, which means we can restrict to 
the situation where $E \in \Db(\cX_T)$ is contained in $\cP^{\sharp\beta}[\phi-\epsilon, \phi + \epsilon]$. After possibly replacing $T$ be an open subset, we can further assume $E_t\in\cP^{\sharp\beta}_t[\phi-\epsilon,\phi+\epsilon]$, for all $t\in T$.

To show openness of geometric stability, note that we have already verified assumption \eqref{enum:basechangeweakassumption1} of Proposition~\ref{prop:basechangeweakstabilityviaopenness} above; therefore, we can base change to the algebraic closure
$\overline{t}$
 and argue as in the proof of Theorem~\ref{thm:deformfamiliystability} to deduce it from boundedness of the Quot space $\Quot_T^{\leqslant \phi+\epsilon}(E)$.
This is defined analogously as in Grothendieck Lemma~\ref{lem:GrothendieckLemma} as the subfunctor of the Quot space such that $\Quot_T^{\leqslant \phi+\epsilon}(E)(T')$ parametrizes quotients $E_{T'} \to Q$ that satisfy $\phi(Q_t) \leqslant \phi+\epsilon$ for all $t \in T'$, where the phase is calculated with respect to the weak stability condition $\us^{\sharp\beta}$.
By Claim~\ref{claim2}, the classes of possible quotients are finite in $\Lambda$. Hence, the Quot space is bounded, and we can conclude the proof of \eqref{enum:stabilityopenbasechange} in Definition~\ref{def:familyfiberstabilities}.

Property \eqref{enum:Zsindependentbasechange} in Definition~\ref{def:familyfiberstabilities} is immediate.
The proof of universal generic openness, namely property \eqref{enum:genOpSstweak} in Definition~\ref{def:familyfiberstabilities}, now follows as in the proof of Lemma~\ref{lem:SemiStUnivGenOpenn} when the phase is in $(0,1)$ (and so JH filtrations exist), since as we observed $\sigma_s^{\alpha,\beta}$ can be base changed over any field extension and, for phase $1$, it follows from the corresponding property of $\sigma_s^{\sharp\beta}$.

To prove part \eqref{enum:curveassumption}, or the remaining part of \eqref{enum:curveassumptionweak} in Definition \ref{def:familyfiberstabilities}, we use semistable reduction and Theorem~\ref{thm:HNstructureviafibers}. Let $C$ be a Dedekind scheme.
Then, by Proposition~\ref{prop:RotatingSlopeStability} and Proposition~\ref{prop:yeswehaveopennessofflatness}, the tilted category $\Coh^\beta \cX_C$ universally satisfies openness of flatness. 
As we observed before, $\sigma_c^{\alpha,\beta}$ has the tilting property and so the assumptions of Theorem~\ref{thm:HNstructureviafibers} are met, thus giving a HN structure $\sigma_C^{\alpha,\beta}$.
Finally, as in the sheaf case discussed in Example~\ref{ex:slopestabilitycurveshastiltingproperty}, by using the dual functor $\mathbb{D}_C$ on $\Db(\cX_C)$ and Remark~\ref{rem:tiltingpropertycurvesses}, it is not hard to check that $\sigma_C^{\alpha,\beta}$ also has the tilting property.
\end{proof}

\section{Tilting tilt-stability in families of threefolds}\label{sect:tiltingtilt}

In this section we consider the case of families of threefolds and we show that the double-tilt construction from \cite{BMT:3folds-BG} works in families. 
As before, we first rotate tilt-stability, and show that this provides a family of weak stability conditions. 
Then we deform to complete the proof of part \eqref{enum:construction2} of Theorem~\ref{mainthm:construction}.

We keep the notation and setup of Section~\ref{sec:tiltingslope}, to which we add the following assumptions: the morphism $g\colon \cX \to S$ has relative dimension $n=3$ and the Bogomolov--Gieseker inequality holds for slope-stable sheaves on its fibers.

\subsection{Rotating tilt-stability in families of threefolds}\label{subsect:rotating2}

Let $\alpha,\beta,\gamma\in\Q$ with $\alpha>0$. Let $u_\gamma\in\C$ be the unit vector in the upper half plane such that $\gamma=-\frac{\Re u_\gamma}{\Im u_\gamma}$.

For every $s\in S$ we define
\[
\sigma_s^{\alpha,\beta\sharp\gamma} := \left( Z^{\alpha,\beta\sharp\gamma}_s=\frac{1}{u_\gamma}\cdot Z^{\alpha,\beta}_s,\cA^{\gamma}_{\alpha,\beta,s} \right),
\]
where $\cA^{\gamma}_{\alpha,\beta,s}$ is the tilted category of $\Coh^\beta \cX_s$ at tilt-slope $\gamma$.
\index{Aalphabetagamma@$\cA^{\gamma}_{\alpha,\beta,s}$, tilted heart of $\Coh^\beta \cX_s$ at tilt-slope $\gamma$}

\begin{Prop}\label{prop:RotatingTiltStability}
The collection $\us^{\alpha,\beta\sharp\gamma}:=(\sigma_s^{\alpha,\beta\sharp\gamma})$ is a weak stability condition on $\Db(\cX)$ over $S$ with respect to $\oLambda^\sharp$.
\index{sigmauTiltRot@$\us^{\alpha,\beta\sharp\gamma}:=(\sigma_s^{\alpha,\beta\sharp\gamma} = ( \tfrac{1}{u_\gamma}\cdot Z^{\alpha,\beta}_s,\cA^{\gamma}_{\alpha,\beta,s} ))$,!(weak) stability condition on $\Db(\cX)$ over $S$ given by rotating tilt-stability on each fiber}
\end{Prop}

\begin{proof}
The argument is very similar to Proposition~\ref{prop:RotatingSlopeStability}.

By \cite[Proposition~2.15]{BLMS}, for all $s\in S$, $\sigma_s^{\alpha,\beta\sharp\gamma}$ is a weak stability condition on $\Db(\cX_s)$.
The classification of $\sigma_s^{\alpha,\beta\sharp\gamma}$-stable objects is identical to the one in the proof of Proposition~\ref{prop:RotatingSlopeStability}, where the role of $\ch_{\cX_s,0}^\beta$ and $\ch_{\cX_s,1}^\beta$ is replaced by the real and imaginary parts of $Z^{\alpha,\beta\sharp\gamma}_s$, respectively. 
Moreover, the objects $V_s$ in case \eqref{enum:RotatingSlopeStability2} (in the proof of Proposition~\ref{prop:RotatingSlopeStability}) are torsion sheaves supported on points.

Now, the proof works exactly in the same way as in Proposition~\ref{prop:RotatingSlopeStability}.
Here we use the derived dual functor $\mathbb{D}_{2,s}:=\cHom (-,\cO_{\cX_s}) [2]$ and the fact that $\mathbb{D}_{2,s}(E_s)$ is directly a $\sigma_s^{\alpha,\beta\sharp\gamma}$-stable object in $\cA^{\gamma}_{\alpha,\beta,s}$ in the proof of property \eqref{enum:bounded} in Definition~\ref{def:fiberwisesupport}.
\end{proof}

Let $X$ be a Fano threefold of Picard rank $1$ over an algebraically closed field $k$.
Let us denote by $\Ku(X)$ its Kuznetsov component as defined in \cite{Kuz:Fano3folds} and \cite[Section~6]{BLMS}.
By Lemma~\ref{lem-E1Em-sod}, the definition behaves nicely for smooth families.

By \cite[Theorem~1.1]{BLMS}, if $X$ is not a complete intersection of a quadric and a cubic in $\P^5$, then Bridgeland stability conditions exist on $\Ku(X)$. The construction is done by rotating tilt-stability and by inducing stability. By Proposition~\ref{prop:RotatingTiltStability} and Theorem~\ref{thm:InducingStabilityOverBase}, these two steps work in families as well; see also Section~\ref{subsec:famstabcondKuz} where a similar (more involved) argument is used in the cubic fourfold case. Hence, we obtain the following result.

\begin{Cor}\label{cor:Fano3foldsPic1}
Let $\cX \to S$ be a smooth family of Fano threefolds of Picard rank $1$ which are not complete intersections of a quadric and a cubic in $\P^5$. Let $\Ku(\cX/S)$ denote the relative Kuznetsov component.
Then the space of numerical stability conditions on $\Ku(\cX/S)$ over $S$ is non-empty.
\end{Cor}

\subsection{Bridgeland stability in families of threefolds}\label{subsect:Bridgstabfamility}

Let $\alpha,\beta,a,b\in\Q$ such that $\alpha>0$ and
\[
a>\frac{1}{6}\alpha^2+\frac{1}{2}|b|\alpha.
\]
We keep the notation in the previous section and fix $\gamma=0$. For $s\in S$, we set
\[
\cA_{\alpha,\beta,s}:=\cA_{\alpha,\beta,s}^{\gamma=0},
\]
\[
Z_{\alpha,\beta,s}^{a,b}:=\mathfrak{i}\left(\ch_{\cX_s,2}^\beta-\frac{\alpha^2}{2}\ch_{\cX_s,0}^\beta\right)+a\ch_{\cX_s,1}^\beta+b\ch_{\cX_s,2}^\beta-\ch_{\cX_s,3}^\beta,
\]
and
\[
\us_{\alpha,\beta}^{a,b}:=\left( \sigma_{\alpha,\beta,s}^{a,b} = \left( Z_{\alpha,\beta,s}^{a,b},\cA_{\alpha,\beta,s} \right) \right).
\]

We can now prove the following which is the family version of \cite[Theorem~8.2]{BMS:abelian3folds}. 

\begin{Prop}\label{prop:Bridgstabfamily}
Assume that the generalized Bogomolov--Gieseker inequality holds for tilt-stable objects on the fibers of $g$, namely for all $s\in S$ and for all $\sigma_{\alpha,\beta,s}$-stable objects $E$
\begin{equation}\label{eqn:genBG}
\nabla_{\beta,s}(E):=4\ch_{\cX_s,2}^\beta(E)^2-6\ch_{\cX_s,1}^\beta(E)\ch_{\cX_s,3}^\beta(E)- \alpha^2\Delta_s(E)\geqslant 0.
\end{equation}
Then the collection $\us_{\alpha,\beta}^{a,b}$ is a stability condition on $\Db(\cX)$ over $S$ with respect to $\Lambda$.
\index{Nablas@$\nabla_{\beta,s}:=4(\ch^\beta_{\cX_s,2})^2-6\ch_{\cX_s,1}^\beta\ch_{\cX_s,3}^\beta- \alpha^2\Delta_s$}
\end{Prop}

\begin{proof}
The proof is the relative version of \cite[Section~4.6]{PT15:bridgeland_moduli_properties}. The extension can be done analogously as in the proof of Proposition~\ref{prop:tiltstabfamily}, the limit case being Proposition~\ref{prop:RotatingTiltStability}.
\end{proof}

By Theorem~\ref{thm:deformfamiliystability} this completes the proof of Theorem~\ref{mainthm:construction}.\eqref{enum:construction2}; 
in particular, we can take $\alpha,\beta,a,b\in\R$.
 
\begin{Rem}\label{rmk:generalizedBGexamplesholds}
The generalized Bogomolov--Gieseker inequality \eqref{eqn:genBG} was first proven for $\P^3$ in \cite{Macri:P3} and, soon after, for the smooth quadric hypersurface in $\P^4$ in \cite{Schmidt:Quadric}. 
The case of Fano threefolds of Picard rank one was treated in \cite{Li:FanoPic1}. 
The case of abelian threefolds is covered independently by \cite{MacPiy:ab3folds} and \cite{BMS:abelian3folds} (the full support property is now also known, see \cite{OPT:DTabelian}).
Recently, the case of the quintic threefolds has finally been settled in \cite{Li:Quintic3fold}.
For other cases with higher Picard rank, 
we refer to \cite{BMSZ:Fano,Dulip:Fano,koseki:1,koseki:2}.
Unfortunately, the inequality does not hold in general, at least for higher Picard rank, see \cite{Schmidt:counterexample}, as well as \cite[Appendix~A]{koseki:1} and \cite{MartinezSchmidt:counterexample}.
\end{Rem}

\section{Stability conditions from degeneration}\label{subsec:degenerations}

As an application of Proposition~\ref{prop:tiltstabfamily}, one can use degeneration to prove the generalized Bogomolov--Gieseker inequality.
The following is a variation of \cite[Proposition~3.2]{koseki:2}.

\begin{Prop}\label{prop:Koseki}
Let $g\colon \cX\to C$ be a smooth family of polarized threefolds over a Dedekind scheme $C$ of characteristic zero, and fix a point $0\in C$.
Consider an arbitrary $\Q$-divisor $B$ on $\cX$.
Let $B_0$ (resp. $B_\eta$) be the restriction of $B$ to the special fiber $\cX_0:=g^{-1}(0)$ (resp. the general fiber $\cX_\eta$).
If the generalized Bogomolov--Gieseker inequality holds for tilt-stable objects on $\cX_0$ with respect to $B_0$, i.e., for all $\sigma_{\alpha,B_0}$-stable objects $E$
\begin{equation*}
\nabla_{B_0}(E):=4\ch_{\cX_0,2}^{B_0}(E)^2-6\ch_{\cX_0,1}^{B_0}(E)\ch_{\cX_0,3}^{B_0}(E)- \alpha^2\Delta_0(E)\geqslant 0,
\end{equation*}
then it also holds for tilt-stable objects on $\cX_\eta$ with respect to $H_\eta$, $B_\eta$.
\end{Prop}

\begin{proof}
Assume that there exists a tilt-stable object $E_\eta$ on $\cX_\eta$ violating the inequality.
We consider the relative tilt stability condition $\us^{\alpha, B}$ coming from the analogue of Proposition~\ref{prop:tiltstabfamily} for non parallel classes $\alpha H$ and $B$ (see Remark~\ref{rem:BomegaNot||}) and the relative moduli space $M_{\us^{\alpha, B}}(\vv)$ over $C$ with Chern character $\vv$ as the Chern character of a $C$-flat lift $E$ of $E_\eta$ to $\cX$.

Since $M_{\us^{\alpha, B}}(\vv)$ is proper over $C$ (by Theorem~\ref{mainthm:construction}, based on Theorem~\ref{thm:modulispacesArtinstacks}.\eqref{enum:MAlgStackFT}) and non-empty at the generic point, it is non-empty on the special fiber $\cX_0$, which is a contradiction.
\end{proof}
\begin{Rem}
The proof of \cite[Proposition~3.2]{koseki:2} gives the following variants. If $C = \bA^1_k$, and if all fibers
$\cX_b$ for $b \neq 0$ are isomorphic (as in the case of a toric
degeneration), then the generalized Bogomolov-Gieseker inequality for $\cX_0$ implies the same inequality for \emph{all} $\cX_b$. Without this assumption, we obtain the result for \emph{very general} $b \in C$, as any counterexample lives in a moduli space $M_{\us}(\vv)$ proper over $C$, and there are countably many choices for $\vv$.
\end{Rem}

\section{Donaldson--Thomas invariants}\label{sec:DT}

As pointed out in \cite{PT15:bridgeland_moduli_properties}, an immediate application of properness of the relative moduli space is that counting invariants of Donaldson--Thomas type arising from moduli spaces of stable objects in the derived category are actually deformation-invariant.

Let $X$ be a smooth projective Calabi--Yau threefold with $H^1(X,\cO_X)=0$ over the complex numbers.
We assume that the generalized Bogomolov--Gieseker inequality holds for tilt-stable objects in $\Db(X)$;
for example, by \cite{Li:Quintic3fold}, this holds for the quintic threefold.
We consider the open subset $\Stab^\dagger(\Db(X))$ of the space of stability conditions on $\Db(X)$ constructed in \cite{BMT:3folds-BG,BMS:abelian3folds} via the generalized Bogomolov--Gieseker inequality (or a variant of it).

For a stability condition $\sigma\in\Stab^\dagger(\Db(X))$ and a numerical class $\vv\in \Knum(\Db(X))$, we consider the moduli stack $\fM_\sigma(\vv)$. In the case $\fM_\sigma(\vv)=\fM^{\st}_\sigma(\vv)$ the results of \cite{HuybrechtsThomas:defo} show that the coarse moduli space $M_{\sigma}(\vv)$ has a symmetric perfect obstruction theory, and so a zero-dimensional virtual class $[M_{\sigma}(\vv)]^{\mathrm{vir}}$ and a
\emph{Donaldson--Thomas invariant} 
\index{DTsigma@$\mathrm{DT}_{\sigma}(\vv)$, Donaldson--Thomas invariant}
\[
\mathrm{DT}_{\sigma}(\vv) := \int_{[M_{\sigma}(\vv)]^{\mathrm{vir}}} 1 \in\Z.
\]
We can use Theorem~\ref{thm:modulispacesArtinstacks}.\eqref{enum:MAlgStackFT} and Theorem~\ref{mainthm:construction} together with \cite[Remark~5.4]{PT15:bridgeland_moduli_properties} (which is based on \cite{BF:NormalCone} and \cite[Corollary~4.3]{HuybrechtsThomas:defo}) to deduce the invariance of $\mathrm{DT}_{\sigma}(\vv)$ under complex deformations of $X$.
In particular, by \cite{Li:Quintic3fold}, we get the following result:

\begin{Thm}\label{thm:DTquintic} 
The Donaldson--Thomas invariants $\mathrm{DT}_{\sigma}(\vv)$ counting stable objects on a smooth quintic threefold $X$, with respect to $\sigma \in \Stab^\dag(X)$, are deformation-invariant.
\end{Thm}

\newpage
\part{Moduli spaces for Kuznetsov components of cubic fourfolds}\label{part:CubicFourfolds}

\section{Main applications to cubic fourfolds}\label{sec:MainResultsCubics}

Let $X \subset \mathbb{P}^5$ be a smooth cubic fourfold. Its \emph{Kuznetsov component} is the admissible subcategory defined by 
\index{Ku(X)@$\Ku(X)$, Kuznetsov component of cubic fourfolds}
\[
\Ku(X) := \mathcal{O}_X^\perp \cap \mathcal{O}_X(H)^\perp \cap \mathcal{O}_X(2H)^\perp \subset \Db(X),
\]
where $H$ denotes the hyperplane class. 
The goal of this final part of the paper is to describe the structure of moduli spaces of stable objects in $\Ku(X)$. 
We will work over the complex numbers throughout. 
Many of our arguments can be adapted to positive characteristic (and have interesting applications in that setting that will be discussed elsewhere), but the strongest results can be proved over $\C$. 

The category $\Ku(X)$ shares many properties with the derived category of K3 surfaces. 
Its foundations were developed in \cite{Kuz:fourfold, AT:CubicFourfolds, Huy:cubics}; see \cite{Huy:survey,MS:survey} for surveys of those results. 
In particular, we recall: 
\begin{enumerate}[{\rm (1)}] 
\item $\Ku(X)$ is a $2$-Calabi--Yau category: $\Hom(E, F) = \Hom(F, E[2])^\vee$.
\item The topological $K$-theory of $\Ku(X)$, along with the faithful functor $\Ku(X) \to \Db(X)$ and the Hodge structure on $H^4(X,\Z)$ equips $\Ku(X)$ with an extended \emph{Mukai lattice}, which 
we denote by $\tH(\Ku(X), \Z)$: 
\index{Htilde(Ku(X),Z)@$\tH(\Ku(X), \Z)$, extended Mukai lattice}
as a lattice, it is isomorphic to $H^*(S,\Z)$, for any K3 surface $S$; it carries a weight two Hodge structure with $h^{2, 0} = 1$;
and it admits a Mukai vector $v \colon K(\Ku(X)) \to \tH(\Ku(X), \Z)$ satisfying $(v(E), v(F)) = - \chi(E,F)$.
\end{enumerate}

The Mukai lattice embeds into $K_{\mathrm{top}}(X) \subset H^*(X, \Q)$ as the right orthogonal complement of the classes of $\cO_X, \cO_X(H), \cO_X(2H)$.
We denote the sublattice of integral $(1,1)$-classes by $\tH_{\Hdg}(\Ku(X), \Z)$ and the image of the Mukai vector by $\tH_{\mathrm{alg}}(\Ku(X), \Z)$;
the latter is isomorphic to $K_{\mathrm{num}}(\Ku(X))$.
The rational Hodge conjecture for cubic fourfolds (proved in \cite{zucker:HodgeConjecture}; see also \cite{ConteMurre:RationalHodge} for a short proof) shows that $\tH_{\Hdg}(\Ku(X), \mathbb{Q})$ is isomorphic to $\tH_{\mathrm{alg}}(\Ku(X), \mathbb{Q})$.
By the integral Hodge conjecture \cite{Voisin:Hodgeaspects}, the two groups are actually isomorphic over $\Z$.
While the rational Hodge conjecture will be used later in the proof of Lemma~\ref{lem:MukaiHomomorphismClosedPoints}, the integral version is not needed in our argument, and in fact it will also follow from our results (see Corollary~\ref{cor:integralHdg}).
\index{HtildeHdg(Ku(X),Z)@$\tH_{\Hdg}(\Ku(X), \Z)$, sublattice of integral $(1,1)$-classes}
\index{Htildealg(Ku(X),Z)@$\tH_{\mathrm{alg}}(\Ku(X), \Z)\cong K_{\mathrm{num}}(\Ku(X))$, sublattice of algebraic classes}

By \cite[Theorem~1.2 and Remark~9.11]{BLMS}, one can explicitly describe a non-empty connected open subset $\Stab^\dagger(\Ku(X))$\index{StabDagger(KuX)@$\Stab^\dagger(\Ku(X))$, connected component of $\Stab(\Ku(X))$ containing geometric stability conditions} in the space of numerical Bridgeland stability conditions on $\Ku(X)$ (with respect to the lattice $\tH_{\Hdg}(\Ku(X),\Z)$ and the Mukai vector);
it is the covering of a certain period domain, which is defined analogously to the case of K3 surfaces, treated in \cite{Bridgeland:K3}.
We can then extend \cite[Theorem~1.1]{Bridgeland:K3} as follows.

\begin{Thm}\label{thm:ConnectedComponentStab}
The open subset $\Stab^\dagger(\Ku(X))$ is a connected component in $\Stab(\Ku(X))$.
\end{Thm}

The connected component $\Stab^\dagger(\Ku(X))$ is realized as a covering $\eta\colon\Stab^\dagger(\Ku(X))\to\fP_0^+$, where $\fP_0^+$ is a period domain defined as follows. 
We take $\fP\subset\tH_{\Hdg}(\Ku(X),\mathbb{C})$ as the open subset consisting of those vectors whose real and imaginary parts span positive-definite two-planes. 
\index{P@$\fP$, generalized period domain}
With $\Delta:=\{\delta\in\tH_\Hdg(\Ku(X), \Z)\colon (\delta,\delta)=-2\}$,
we set
\[
\fP_0 := \fP \setminus \bigcup_{\delta\in\Delta} \delta^\perp,
\]
which has two connected components; we let $\fP_0^+$ be the one containing the image under $\eta$ of the examples of stability conditions constructed in \cite[Theorem~1.2]{BLMS}.

Let $\vv \in \tH_{\Hdg}(\Ku(X), \Z)$ be a non-zero primitive class, and let $\sigma \in\mathrm{Stab}^\dagger(\Ku(X))$.
Theorem~\ref{thm:ConnectedComponentStab} relies on our second main result, which concerns the existence and non-emptiness of the moduli space $M_\sigma(\Ku(X), \vv)$ of $\sigma$-stable objects in $\Ku(X)$ with Mukai vector $\vv$. It is the analogue of a long series of results \cite{Beauville:HK,Mukai:Symplectic,Mukai:BundlesK3,O'Grady:weight2,Huybrechts:BirationaSymplectic,Yoshioka:Abelian} on moduli spaces of sheaves on K3 surfaces, the last one being \cite[Theorems 0.1 and 8.1]{Yoshioka:Abelian}.

\begin{Thm}\label{thm:YoshiokaMain}
Let $X$ be a cubic fourfold.
Then
\index{HtildeHdg(Ku(X),Z)@$\tH_{\Hdg}(\Ku(X), \Z)$, sublattice of integral $(1,1)$-classes}\index{Htildealg(Ku(X),Z)@$\tH_{\mathrm{alg}}(\Ku(X), \Z)\cong K_{\mathrm{num}}(\Ku(X))$, sublattice of algebraic classes}
\[
\tH_{\Hdg}(\Ku(X),\Z)=\tH_{\mathrm{alg}}(\Ku(X),\Z).
\]
Moreover, assume that $\vv\in \tH_{\Hdg}(\Ku(X),\Z)$ is a non-zero primitive vector and let $\sigma\in\Stab^\dagger(\Ku(X))$ be a stability condition on $\Ku(X)$ that is generic with respect to $\vv$.
Then
\begin{enumerate}[{\rm (1)}]
\item \label{enum:YoshiokaMain1} $M_{\sigma}(\Ku(X),\vv)$ is nonempty if and only if $\vv^2\geqslant-2$. 
Moreover, in this case, it is a smooth projective irreducible holomorphic symplectic variety of dimension $\vv^2 + 2$, deformation-equivalent to a Hilbert scheme of points on a K3 surface.
\item \label{enum:YoshiokaMain2} If $\vv^2\geqslant 0$, then there exists a natural Hodge isometry \[\theta\colon H^2(M_\sigma(\Ku(X),\vv),\Z)\xrightarrow{\quad\sim\quad}\begin{cases}\vv^\perp & \text{if }\vv^2>0\\ \vv^\perp/\Z\vv & \text{if } \vv^2=0,\end{cases}\]
where the orthogonal is taken in $\tH(\Ku(X),\Z)$. 
\end{enumerate}
\end{Thm}

Here \emph{generic} means that $\sigma$ is not on a wall: since $\vv$ is primitive, this means that stability and semistability coincide for objects of Mukai vector $\vv$.

The embedding $\vv^\perp \into \tH(\Ku(X), \Z)$ identifies the latter with the Markman--Mukai lattice of $M_\sigma(\Ku(X), \vv)$, which determines the birational class of $M_\sigma(\Ku(X), \vv)$ by Markman's global Torelli theorem \cite[Corollary~9.9]{Eyal:survey}; moreover, the Hodge classes of $\tH(\Ku(X, \Z)$ control the Mori cone by \cite{BHT, Mongardi:note}.

\begin{Rem}
Assume now that $\vv$ is not primitive, i.e., $\vv=m\vv_0$, for some $m>1$, and $\sigma$ a $\vv$-generic stability condition.
Then the previous theorem implies immediately that the moduli space $M_{\sigma}(\Ku(X),\vv)$ is non-empty if $\vv_0^2\geqslant-2$; conversely, if $\vv_0^2 < -2$, one can show easily by induction on $m$ that $M_{\sigma}(\Ku(X),\vv)$ is empty. 
If the good moduli space $M_\sigma(\Ku(X),\vv)$ is normal, one can prove further that $M_\sigma(\Ku(X),\vv)$ is an irreducible proper algebraic space (by using a similar argument as in \cite[Theorem~4.4]{KLS:SingSymplecticModuliSpaces}). 
Moreover, either $\dim M_\sigma(\Ku(X),\vv)=\vv^2+2$ and $M_{\sigma}^{\st}(\Ku(X),\vv)\neq\emptyset$, or $m>1$ and $\vv^2\leqslant0$.
\end{Rem}

Theorem~\ref{thm:YoshiokaMain} is proved by deformation to the case where $\Ku(X)$ is known to be equivalent to the derived category of a K3 surface. 
Such deformation arguments rely on relative moduli spaces of Bridgeland stable objects, given by the following result.

\begin{Thm}\label{thm:Msigmarelative}
Let $\cX \to S$ be a family of cubic fourfolds, where $S$ is a connected quasi-projective variety over $\C$. 
Let $\vv$ be a primitive section of the local system of the Mukai lattices $\tH(\Ku(\cX_s), \Z)$ of the fibers, such that $\vv$ is algebraic on all fibers. 
Assume that, for a closed point $s_0 \in S$ not contained in any Hodge locus, there exists a stability condition $\tau_{s_0} \in \mathrm{Stab}^\dagger(\Ku(\cX_{s_0}))$ that is generic with respect to $\vv$, and whose central charge $Z_{s_0} \colon \tH_{\Hdg}(\Ku(\cX_{s_0}),\Z)\to \mathbb{C}$ is invariant
under the monodromy action induced by the inclusion $\tH_{\Hdg}(\Ku(\cX_{s_0}),\Z) \subset H^*(X, \Q)$.
\begin{enumerate}[{\rm (1)}]
\item \label{enum:algspace}
If $S = C$ is a curve, then there exists an algebraic space $\tM(\vv)$, and a smooth proper morphism $\tM(\vv) \to C$ that makes $\tM(\vv)$ a relative moduli space over $C$: 
the fiber over any point $c \in C$ is a coarse moduli space $M_{\sigma_c}(\Ku(\cX_c), \vv)$ of stable objects in the Kuznetsov component of the corresponding cubic fourfold for some stability condition $\sigma_c$.
\item \label{enum:ProjMorph}
There exist a non-empty open subset $S^0 \subset S$, a quasi-projective variety $M^0(\vv)$, and a smooth \emph{projective} morphism $M^0(\vv) \to S^0$ that makes $M^0(\vv)$ a relative moduli space over $S^0$.
\item \label{enum:goodrelativemoduli}
There exists an algebraic space $M(\vv)$ and a proper morphism $M(\vv) \to S$ such that every fiber is a
good moduli space $M_{\sigma_s}(\Ku(\cX_s), \vv)$ of semistable objects.
\end{enumerate}
In all cases, we can choose $\us$ such that
$M_{\sigma_{s_0}}(\Ku(\cX_{s_0}), \vv) = M_{\tau_{s_0}}(\Ku(\cX_{s_0}), \vv)$.
\end{Thm}

Note that every fiber of the morphism $\tM(\vv) \to C$ is projective, but the morphism itself might not be. In contrast, we expect that the morphism
$M(\vv) \to S$ is always projective.

For a very general cubic fourfold, $\tH_\Hdg(\Ku(X), \Z)$ is isomorphic to the lattice $A_2$ generated by two roots $\llambda_1, \llambda_2$ with $(\llambda_1, \llambda_2) = -1$ and $\llambda_1^2=\llambda_2^2=2$; for example, we can set
\index{lambda1@$\llambda_1,\llambda_2$, roots of the $A_2$-lattice inside $\tH_\Hdg(\Ku(X), \Z)$}
\begin{equation} \label{eq:defl1l2}
\llambda_1=v(p(\cO_L(1)))\quad\text{and}\quad\llambda_2=v(p(\cO_L(2))),
\end{equation}
where $L$ is a line in $X$ and $p$ is the left adjoint of $\Ku(X)\hookrightarrow \Db(X)$.
Applying the theorem above to classes in $A_2$ yields the following result.

\begin{Cor}\label{cor:locfam20dim}
For any pair $(a,b)$ of coprime integers, there is a unirational locally complete $20$-dimensional family, over an open subset of the moduli space of cubic fourfolds, of smooth polarized irreducible holomorphic symplectic manifolds of dimension $2n+2$, where $n=a^2-ab+b^2$. 
The polarization has either degree $6n$ and divisibility $2n$ if $3$ does not divide $n$, or degree and divisibility $\frac{2}{3}n$ otherwise.
\end{Cor}

It may be worth pointing out that, as observed for example in \cite{ATwoRatConJ}, an even integer $2n$ has the form $2n=2(a^2-ab+b^2)$ if and only if it satisfies Hassett's condition:
\begin{enumerate}
\item[$(\ast\ast)$] $2n$ is not divisible by $4$, $9$ or any odd prime $p\equiv 2$ $(\mod 3)$. 
\end{enumerate}
By \cite[Theorem~1.0.2]{Hassett:specialcubics}, integers satisfying $(\ast\ast)$ are related to cubic fourfolds with Hodge-theoretically associated K3 surfaces discussed in Corollary~\ref{cor:completeAT}.

\begin{Ex} \label{ex:4fold8fold}
Let $S$ be the moduli space of cubic fourfolds.
If we choose $\vv = \llambda_1+\llambda_2$ in Theorem~\ref{thm:Msigmarelative}, then by \cite{LiPertusiZhao:TwistedCubics} $S^0 = S$, and $M(\vv)$ is the relative Fano variety of lines over $S$.
For $\vv = 2\llambda_1 + \llambda_2$, still by \cite{LiPertusiZhao:TwistedCubics} (see also \cite{LLMS}), we have that $S^0 \subset S$ is the complement of cubics containing a plane,
and $M^0(\vv)$ is the family of irreducible holomorphic symplectic eightfolds constructed by Lehn, Lehn, Sorger and van Straten \cite{LLSvS}.
Finally, for $\vv = 2\llambda_1 + 2 \llambda_2$, in \cite{LiPertusiZhao:EllipticQuintics} the authors obtain an algebraic construction of a $20$-dimensional family of $10$-dimensional O'Grady spaces compactifying the twisted intermediate Jacobian fibration of the cubic hyperplane sections, birational to the one constructed in \cite{Voisin:twisted}.
\end{Ex}

Recall from Hassett's work on cubic fourfolds, \cite{Hassett:specialcubics}, that there is a countable union of divisors of special cubics with a Hodge-theoretically associated K3 surface.
In our notation, a cubic is contained in one of Hassett's special divisors if and only if $\tH_{\Hdg}(\Ku(X), \Z)$ contains a hyperbolic plane.

\begin{Cor}\label{cor:completeAT}
Let $X$ be a cubic fourfold.
Then $X$ has a Hodge-theoretically associated K3 if and only if there exists a smooth projective K3 surface $S$ and an equivalence $\Ku(X)\cong \Db(S)$.
\end{Cor}

This (literally) completes a result by Addington and Thomas, \cite[Theorem~1.1]{AT:CubicFourfolds}, who proved that every divisor described by Hassett contains an open subset of cubics admitting a derived equivalence as above.
A version of the corollary also holds for K3 surfaces with a Brauer twist, completing a result by Huybrechts \cite[Theorem~1.4]{Huy:cubics};
the corresponding Hodge-theoretic condition is the existence of a square-zero class in $\tH_\Hdg(\Ku(X),\Z)$ (see Proposition~\ref{prop:KuzHass}).
Partial results were also obtained in \cite{Kuz:fourfold, Mosch:cubics}.

As pointed out to us by Voisin, the non-emptiness of moduli spaces also produces enough algebraic cohomology classes to reprove her result on the integral Hodge conjecture for cubic fourfolds; we also refer to \cite[Corollary~0.3]{MongardiOttem:HodgeConj} for a different recent proof:

\begin{Cor}[{\cite[Theorem~18]{Voisin:Hodgeaspects}}]\label{cor:integralHdg}
	The integral Hodge conjecture holds for $X$.
\end{Cor}

\begin{Rem}
In \cite{IHC-CY2} our argument for proving the integral Hodge conjecture for cubic fourfolds has been further developed and generalized. 
In particular, a version of the integral Hodge conjecture is proved for suitable CY2 categories, without the use of stability conditions. 
\end{Rem}

Our results also provide the full machinery of \cite{BM:MMP_K3}, describing the birational geometry of $M_\sigma(\Ku(X), \vv)$ in terms of wall-crossing, but we will not discuss the details here.

Let us discuss the line of argument in the proofs of Theorems~\ref{thm:ConnectedComponentStab}, \ref{thm:YoshiokaMain} and \ref{thm:Msigmarelative}.
The key point is to generalize a deformation argument by Mukai, and show that the deformation of simple objects in the Kuznetsov components along a deformation of cubic fourfolds is unobstructed as long as their Mukai vector remains algebraic; combined with openness of stability this shows smoothness of the relative
moduli space.
The existence of the family of stability conditions is explained in Section~\ref{sec:Kuzstabfam}, while the above deformation argument is in Section~\ref{sec:Mukai}.
The proofs of all results are then in Section~\ref{sec:ProofFirstTheoremsCubics}.

\section{Stability conditions on families of Kuznetsov components}\label{sec:Kuzstabfam}

In this section we recast Kuznetsov's work in \cite{Kuz:Quadric} in the relative setting along the lines of \cite{BLMS}. The aim is to construct stability conditions on families of Kuznetsov components of cubic fourfolds.

\subsection{The Kuznetsov component in families}\label{subsec:KuzFamilies}

In this section we study Kuznetsov components for families of smooth cubic fourfolds. We start by reconsidering the results in \cite{Kuz:Quadric} and \cite[Section~7]{BLMS} for families of cubic fourfolds.
The proofs discussed here are very close to those presented in \cite{BLMS}, hence we will be concise.

Let $g\colon\cX\to S$ be a family of cubic fourfolds, where $S$ is a quasi-projective variety over $\C$. We let $\cO_{\cX}(1)$ denote the relative very ample line bundle.

\begin{Lem}\label{lem:CategoricalKuzFamily}
There exist an admissible subcategory $\Ku(\cX)\hookrightarrow\Db(\cX)$ and a strong $S$-linear semiorthogonal decomposition of finite cohomological amplitude
\[
\Db(\cX)=\langle\Ku(\cX),\alpha_{\cO_{\cX}}(\Db(S)),\alpha_{\cO_{\cX}(1)}(\Db(S)),\alpha_{\cO_{\cX}(2)}(\Db(S))\rangle,
\]
where $\alpha_{\cO_{\cX}(n)}(\Db(S)):=g^*\Db(S)\otimes\cO_{\cX}(n)$.
\end{Lem}

\begin{proof}
This follows immediately from Lemma~\ref{lem-E1Em-sod}, since $\cO_\cX,\cO_\cX(1),\cO_\cX(2)$ is a relative exceptional collection in $\Db(\cX)$.
\end{proof}

Let us assume further that the family $g\colon \cX\to S$ comes with a family $\cL\subset\cX$ over $S$ of lines contained in the fibers of $g$
	\begin{equation}\label{eqn:FamilyLines}
		\xymatrix{
		\P_S(g_*\cO_\cL(1))=\cL\ar[rd]\ar@{^{(}->}[r]&\cX\ar[d]\ar@{^{(}->}[r]&\P_S(g_*\cO_\cX(1))\ar[ld]\\
		&S
		}
	\end{equation}
such that, for all $s\in S$, the line $\cL_s$ is not contained in a plane in $\cX_s$.

\begin{Ex}\label{ex:relFano}
Given any family $g\colon\cX\to S$ of cubic fourfolds, we can always find a base change $g'\colon\cX'\to S'$ satisfying this existence of an appropriate family of lines. For example, we can take $S'$ to be the open subset $F^0(\cX/S)$ of the relative Fano variety of lines $F(\cX/S)$ in $\cX$ consisting of all lines which are not contained in planes inside the fibers of $g$. The base-change $S'\to S$ can also be taken to be finite.
\end{Ex}

Let $\uPP_S\to \P_S(g_*\cO_\cX(1))$ be the blow-up along $\cL$.
We denote by $\varepsilon\colon \uXX\to \cX$ the strict transform of $\cX$ via this blow-up (or equivalently the blow-up of $\cX$ along $\cL$). Consider the projective bundle $\P_S(g_*\sI_\cL(1))$ which, for simplicity, we denote by $\P^3_S$, even though it is not trivial.
We let $\uPP_S \to \P^3_S$ be the $\P^2$-bundle induced by the projection from $\cL$, whose restriction to $\uXX$ induces a conic fibration $q\colon \uXX\to\P_S^3$.

As in \cite[Section~3]{Kuz:Quadric}, we denote by $\cB^S_0$ (resp.~$\cB^S_1$) the sheaf on $\P^3_S$ of even (resp.~odd) parts of Clifford algebras corresponding to this fibration and, for all $m\in\Z$,
\begin{equation*}
\cB^S_{2m+1}=\cB^S_{1}\otimes \cO_{\P^3_S}(m)\quad\text{and}\quad
\cB^S_{2m}=\cB^S_0\otimes \cO_{\P^3_S}(m).
\end{equation*}

According to~\cite[Theorem~4.2]{Kuz:Quadric} there is a strong $S$-linear semiorthogonal decomposition of finite cohomological amplitude
\[
\Db(\uXX) = \langle \Phi(\Db(\P^3_S,\cB^S_0)), q^*\Db(\P_S^3)\rangle.
\]
Here $\Phi\colon \Db(\P^3_S,\cB^S_0) \to \Db(\uXX)$ is the fully faithful Fourier--Mukai functor whose kernel is explicitly described in \cite[Section~4]{Kuz:Quadric} (it corresponds to $\Phi_{-1,0}$ in \cite[Proposition~4.9]{Kuz:Quadric}), but such an explicit description is not needed in this paper. 
Denote by $\Psi$ its left adjoint.

\begin{Prop}\label{prop:stabinfam}
The functor $\Psi\circ\varepsilon^*\colon\Ku(\cX)\to\Db(\P^3_S,\cB^S_0)$ is fully faithful and it induces a strong $S$-linear semiorthogonal decomposition of finite cohomological amplitude 
\[
 \Db(\P_S^3 ,\cB^S_0)=\gen{ \Psi ( \varepsilon^* \Ku(\cX)), \alpha_{\cB^S_1}(\Db(S)),\alpha_{\cB^S_2}(\Db(S)), \alpha_{\cB^S_3}(\Db(S))}.
\]
\end{Prop}

\begin{proof}
The result follows by repeating line by line the same proof as in \cite[Proposition~7.7]{BLMS} in the relative setting above.
\end{proof}

\subsection{Existence of stability conditions in families}\label{subsec:famstabcondKuz}

The next step consists in constructing stability conditions on $\Ku(\cX)$ over $S$ in the sense of Definitions~\ref{def:familyfiberstabilities} and \ref{def:fiberwisesupport}.

Let $g\colon\cX\to S$ be a family of cubic fourfolds over a connected quasi-projective variety $S$ over $\C$. 
Assume that $g \colon \cX \to S$ is equipped with a family 
$\cL\to S$ of lines which are not contained in planes in the fibers of $g$. 
Let $G:=\mathrm{Mon}(g)$ be the monodromy group of $g$. Its action on $H^*(\cX_s,\Q)$ preserves, and thus acts on, $\tH(\Ku(\cX_s),\Z)$, for all $s\in S$; for very general $s \in S$, it also preserves Hodge classes on $\cX_s$ (see \cite[Theorem~4.1]{Voisin:HodgeLoci}), and hence acts on $\tH_\Hdg(\Ku(\cX_s),\Z)$. 
The sublattice $\Fix(G):=\tH_\Hdg(\Ku(\cX_s),\Z)^G$, for $s\in S$ a very general point, is then naturally identified with a saturated sublattice of $\tH_\Hdg(\Ku(\cX_s),\Z)$, for all closed $s\in S$, by parallel transport.
%Note that $\Fix(G)$ is dual, over $\Q$, to the lattice $\cN(\Ku(\cX)/S)$ defined in Proposition and Definition~\ref{propdef:lattice}.

Let $M$ be a saturated sublattice of $\Fix(G)$ containing $A_2$ generated as in \eqref{eq:defl1l2}, and denote by $M^\vee$ the dual of $M$. 
By assumption, we have a sequence of natural homomorphisms 
\[
v_s\colon\Knum(\Ku(\cX_s))\hookrightarrow\tH_\Hdg(\Ku(\cX_s),\Z)\hookrightarrow\tH_\Hdg(\Ku(\cX_s),\Z)^\vee\to M^\vee,
\]
for all closed points $s\in S$. 
%This extends to a natural morphism
%\[
%v:=\prod v_s\colon\Knum(\Ku(\cX)/S)\to M^\vee
%\]
%defined over all points $s\in S$ since such a morphism is determined by %the closed points.
These homomorphisms determine a relative Mukai homomorphism in the sense of Definition~\ref{def-relative-mukai}.

\begin{Lem}\label{lem:MukaiHomomorphismClosedPoints}
There is a relative Mukai homomorphism $v \colon \Knum(\Ku(\cX)/S) \to M^{\vee}$, such that for any closed point $s \in S$ the composition 
\begin{equation*}
\Knum(\Ku(\cX_s)) \to \Knum(\Ku(\cX)/S) \xrightarrow{\, v \,} M^{\vee}
\end{equation*} 
is equal to $v_s$.
If $S$ is smooth, $v$ factors via a homomorphism $\cN(\Ku(\cX)/S) \to M^{\vee}$ out of the 
uniformly numerical relative Grothendieck%\footnote{See Definition \ref{propdef:lattice}.}
group of $\Ku(\cX)$.
\end{Lem} 

\begin{proof}
For any non-closed point $s \in S$ we define a homomorphism 
\begin{equation*}
v_s \colon \Knum(\Ku(\cX_s))_{\overline{s}} \to M^{\vee} 
\end{equation*} 
as follows: 
\begin{enumerate}
\item \label{vs-setup} For $[F] \in \Knum(\Ku(\cX_s))_{\overline{s}}$, choose a finite type morphism $f \colon T \to S$ from a connected $\C$-variety $T$ with a point $t \in T$ such that $f(t) = s$, and a $T$-perfect object $E \in \rD(\cX_T)$ such that $E_u \in \Ku(\cX_u)$ for all $u \in T$ and $\eta_{t/s}^{\vee}[E_t] = [F]$ (recall Proposition and Definition~\ref{propdef:KnumDell}). 

\item \label{vs-definition} 
For $[F]$, $T$, and $E$ as in~\eqref{vs-setup}, choose a closed point $q \in T$, let $p = f(q)$, and set 
\begin{equation*} 
v_s([F]) = v_{p}( \eta^{\vee}_{q/p} [E_q] ) . 
\end{equation*} 
\end{enumerate}

We first explain why the choice in~\eqref{vs-setup} is always possible.
Write $[F] = \eta^{\vee}_{K/\kappa(s)}[E]$ where $K/\kappa(s)$ is a finite field extension and $E_K \in \Ku(\cX_K)$;
choose a finite type morphism $f \colon T \to S$ from a smooth irreducible $\C$-variety with fraction field $K$ realising the given field extension $K/\kappa(s)$;
finally, extend $E_K \in \Ku(\cX_t)$ to an object $E \in \Ku(\cX_T)$ (Lemma~\ref{lem-open-restriction-es}), which by the smoothness of $T$ and Lemma~\ref{lem-relations-S-perfect}\eqref{db-S-perfect} is automatically $T$-perfect. 

To prove independence of all choices, and that the $v_s$ induce a homomorphism from the relative numerical $K$-group $\Knum(\Ku(\cX)/S)$, we instead prove the stronger statement that they induce a homomorphism out of the uniformly numerical relative Grothendieck group $\cN(\Ku(\cX)/S)$ when $S$ is smooth; the statement in the general case follows by replacing $S$ with a resolution of singularities.

Thus it is enough to show that for $S$ smooth, and $s \in S$ a very general closed point, the composition
\[
K(\Ku(\cX)_{\perf}) \otimes \Q = K(\Ku(\cX)) \otimes \Q \to K(\Ku(\cX_s))\otimes \Q \xrightarrow{v_s} 
\tH_\Hdg(\Ku(\cX_s),\Q) \to M \otimes \Q
\]
is surjective, where the last map is given by orthogonal projection. This in turn is follows from the surjectivity of the composition
\[
K(\Db(\cX)) \to K(D^b(\cX_s)) \to H^*_{\Hdg}(\cX_s, \Q),
\]
where we use the morphism $H^*_{\Hdg}(\cX_s, \Q) \to M \otimes \Q$ that is compatible with the projection $\Db(\cX) \to \Ku(\cX)$.
Finally, the surjectivity of the last map is a consequence of the rational Hodge conjecture \cite{zucker:HodgeConjecture} for the very general fiber of $\cX \to S$.
Indeed, the only non-trivial statement is the surjectivity onto $H^{2, 2}(\cX_s, \Q)$.
The rational Hodge conjecture implies that classes of the form $\ch_2(I)$, for $I$ an ideal sheaf of a surface in $\cX_s$, where $I$ is in a component of the relative Hilbert scheme that dominates $S$, generate $H^{2, 2}(\cX_s, \Q)$. 
Taking a finite base change $T \to S$ admitting a rational section to the component containing $I$, extending the pullback of the universal ideal sheaf to $\cX_T$ in an arbitrary manner, and pushing forward along $\cX_T \to \cX$ produces an object in $\Db(\cX)$ whose restriction to $\cX_s$ has class proportional to $[I] \in K(\Db(\cX_s))$. This proves the claim.
\end{proof}

\begin{Prop}\label{prop:exfamstabKuz}
In the above setup, there exists a stability condition $\us$ on $\Ku(\cX)$ over $S$ with respect to $M^\vee$. 
\end{Prop}

\begin{proof}
We first consider the case $M=A_2$. By Proposition~\ref{prop:stabinfam}, we can realize $\Ku(\cX)$ as an admissible subcategory of $\Db(\P_S^3,\cB_0^S)$. For any $s\in S$, as in \cite[Definition~9.1]{BLMS}, we can define a twisted Chern character $\ch_{\cB_{0}^s}\colon\Knum(\Db(\P^3_s,\cB_{0}^s))\to\Q^4$;
this induces a global morphism
\[
\ch_{\cB_{0}^S}:=\prod\ch_{\cB_{0}^s}\colon\Knum(\Db(\P^3_S,\cB_0^S)/S)\to\Q^4.
\]
In analogy to the notation in Section~\ref{sec:tiltingslope} we set:
\begin{gather*}
\Lambda:=\mathrm{im}\left(\ch_{\cB_{0}^S}\right) \qquad
\Lambda_0:=\mathrm{im}\left(\ch_{\cB_{0}^S,2}\oplus\ch_{\cB_{0}^S,3}\right)\qquad
\Lambda_0^\sharp:=\mathrm{im}\left(\ch_{\cB_{0}^S,3}\right)\\
\overline{\Lambda}:=\Lambda/\Lambda_0\qquad\overline{\Lambda}^\sharp:=\Lambda/\Lambda^\sharp_0.
\end{gather*}

By Example~\ref{ex:LambdaFamilyChernClasses} we have a weak stability condition $\us^1$ on $\Db(\P^3_S,\cB^S_0)$ with respect to the lattice $\overline{\Lambda}$. 
Propositions~\ref{prop:RotatingSlopeStability} and \ref{prop:tiltstabfamily} also apply in the twisted setting that we are considering here; the required Bogomolov inequality is given by \cite[Theorem~8.3]{BLMS}.
Thus we get a weak stability condition $\us^{\alpha,\beta}$ on $\Db(\P^3_S,\cB_0^S)$ over $S$ with respect to $\overline{\Lambda}^\sharp$, which coincides fiberwise with the weak stability conditions constructed in \cite[Proposition~9.3]{BLMS}. 
We can rotate one more time by Proposition~\ref{prop:RotatingTiltStability} getting the weak stability condition $\us^{\alpha,\beta\sharp\gamma}$.

Finally, we can apply Theorem~\ref{thm:InducingStabilityOverBase} and get a stability condition on $\Ku(\cX)$ over $S$ with respect to the lattice $\Lambda_1:=\ch_{\cB^S_0}(\Knum(\Ku(\cX)/S))$; the boundedness assumption \ref{thm:InducingStabilityOverBase}.\eqref{enum:boundednessforinducing} follows from the analogue of the boundedness statement in Proposition~\ref{prop:tiltstabfamily}. 
By \cite[Proposition~9.10]{BLMS}, we have $\Lambda_1=A_2$, thus completing the proof in the case $M=A_2$.

Now consider the case of a lattice $M$ containing $A_2$.
Let $\fP(M) \subset \Hom(M^\vee, \C)$ be the set of central charges such that $\Im Z$ and $\Re Z$ span a positive definite two-plane in $M \otimes \R$.
Let $\Delta \subset \tH(\Ku(\cX_s), \Z)$, for some arbitrary $s \in S$, be the set of $(-2)$-classes $\delta$ that are not orthogonal to $A_2$.
(This excludes classes that would become algebraic only over the Hassett divisor $\cC_2$, see \cite[Definition~3.1.3 and Section~4.4]{Hassett:specialcubics}.)
Let $\fP_0(M) \subset \cP(M)$ be the open subset where $\Im Z, \Re Z$ are not orthogonal to any $\delta \in \Delta$;
by standard arguments, e.g.~as in \cite[Section~8]{Bridgeland:Stab}, this is an open subset, and it contains the central charge of the stability condition we constructed above for $M = A_2$.
Now consider the Mukai vector $\vv' \in \tH_{\Hdg}(\Ku(\cX_{s'}), \Z)$ of any stable object in any fiber $\Ku(\cX_{s'})$; either parallel transport (which does not change its class in $M^\vee$) identifies it with a class in $\Delta$, or it satisfies $\vv'^2 \geqslant 0$.
In both cases, it is not contained in the kernel of any central charge in $\fP_0(M)$.
The same arguments as in \cite[Section~8]{Bridgeland:Stab} therefore imply the support property with respect to $M$ for any stability condition with central charge in $\fP_0(M)$.
\end{proof}

\begin{Rem} \label{rem:coveringproperty}
One can also show, just as in \cite[Section~8]{Bridgeland:Stab}, that Theorem~\ref{thm:deformfamiliystability} implies the existence of an open set in $\Stab_M(\Ku(\cX)/S)$ that covers a connected component of $\fP_0(M)$.
\end{Rem}

\begin{Rem}\label{rmk:changinglattice}
It will be crucial for us to be able to modify slightly the lattice $M$ at a closed subset, similar to the case of a one-dimensional base considered in Example~\ref{ex:modification}.
Let $S'\subseteq S$ be a closed subvariety of $S$ which is an irreducible component of the Hodge locus of the family $g$ (see \cite{CDK:absolute}).
Consider the base changed family $g'\colon\cX'\to S'$.
Let $G':=\mathrm{Mon}(g')$ and consider a saturated sublattice $M'\subseteq\mathrm{Fix}(G')$, such that $M'$ is orthogonal to all classes that remain algebraic along all of $S$. We can then modify the morphism $v$ to a map
\[
v'\colon\Knum(\Ku(\cX)/S)\to M^\vee \oplus M'^\vee,
\]
by setting the second component to be zero for $s \notin S'$. Since the full support property holds on each fiber of $g$, the statement in Proposition~\ref{prop:exfamstabKuz} holds true with respect to the image of $v'$ inside $(M\oplus M')^\vee$.

For example if $S$ is a curve and $S'$ is a finite set of closed points in the Hodge locus of $S$, we can choose $M':=\bigoplus_{s\in S'}\tH_\Hdg(\Ku(\cX_s),\Z)$.
\end{Rem}

Finally, we specialize to families over a one-dimensional base. Let $g\colon\cX \to C$ be a family of cubic fourfolds, where $C$ is a quasi-projective irreducible curve.
Assume that $g$ comes with a family $\cL\to C$ of lines not contained in a plane and that the fixed locus $M$ of the monodromy group $G$ of $g$ is $\tH_\Hdg(\Ku(\cX_c),\Z)$, for $c$ a very general point of $C$. Let $\vv$ be a primitive section of the local system of the Mukai lattices $\tH(\Ku(\cX_c), \Z)$ of the fibers, such that $\vv$ is of Hodge type on all fibers $\cX_c$.
In particular, in our case, this means that $A_2$ and $\vv$ are contained in $M$.

\begin{Prop}\label{prop:Msigmarelative1}
In the assumptions above, let $c_0 \in C$ be a closed point and let $\tau_{c_0} \in \mathrm{Stab}^\dagger(\Ku(\cX_{c_0}))$ be a stability condition which is $\vv$-generic.
Then there exist a lattice $M'\supseteq M$ and a stability condition $\us$ on $\Ku(\cX)$ over $C$ with respect to $(M')^\vee$ such that: 
\begin{enumerate}[{\rm (1)}] 
\item\label{Msigmarel1} $\sigma_c$ is $\vv$-generic for all $c\in C$; 
\item\label{Msigmarel2} $\sigma_{c_0}$ is a small deformation of $\tau_{c_0}$ satisfying $M_{\sigma_{c_0}}(\Ku(\cX_{c_0}), \vv) = M_{\tau_{c_0}}(\Ku(\cX_{c_0}), \vv)$;
\item\label{Msigmarel3} if $c_0 \in C$ is a smooth point and the central charge $Z_{c_0}$ of $\tau_{c_0}$ factors via $M^\vee$, then there is an open neighborhood $U \subset C$ of $c_0$ such that the central charge of $\us$ restricted over $U$ factors through $\cN(\Ku(\cX_U)/U)$. 
\end{enumerate}
\end{Prop}

Note that, if we are in the setting of Theorem \ref{thm:Msigmarelative}.\eqref{enum:algspace}, then up to taking a finite cover $C'\to C$ we can always assume that the assumptions in Proposition \ref{prop:Msigmarelative1} are satisfied, as well as the one in part~\eqref{Msigmarel3} (see \cite[Theorem~4.1]{Voisin:HodgeLoci}).

\begin{proof}
By Proposition~\ref{prop:exfamstabKuz} there exists a stability condition $\us'$ on $\Ku(\cX)$ over $C$ with respect to $M^\vee$ such that $\sigma'_c$ is $\vv$-generic for $c$ in a non-empty subset of $C$.
Let $C'$ be the subset of $C$ where $\sigma'_c$ is not $\vv$-generic.
By openness of geometric stability and boundedness of the relative moduli space $C'$ is a finite set.
For each $c \in C'$ there is (by the support property) a finite set in $\tH_\Hdg(\Ku(\cX_c), \Z)$ of Mukai vectors of Jordan--H\"older factors of objects in $\cM_{\us'}(\vv)$. We now apply
Remark~\ref{rmk:changinglattice} to enlarge the lattice $M$ to $M':=M \oplus_{c\in C'} M_c$ such that these classes are no longer proportional 
to $\vv$ in $M^\vee \oplus M_c^\vee$; for $c=c_0$ we just take $M_{c_0}$ such that the central charge $Z_{\tau_{c_0}}$ is defined over it. 
By Remark~\ref{rem:coveringproperty}, there exists a deformation $\us$ of $\us'$ such that \eqref{Msigmarel1} and \eqref{Msigmarel2} are satisfied.

Under the additional assumption in \eqref{Msigmarel3}, we do not need to modify $M$ in the fiber over $c_0$ and thus we can apply directly Lemma~\ref{lem:MukaiHomomorphismClosedPoints}. 
\end{proof}

\begin{Rem}\label{rmk:ExtendOverTrivialMonodromy}
Consider the analogous situation over a higher-dimensional base $S$.
Then $S' \subset S$ becomes a closed subset, contained in finite union of Hodge loci $S'_i$.
Assume that the monodromy group for each $S'_i$ acts trivially on the Mukai vectors of Jordan--H\"older factors of objects in $\cM_{\us}(\vv)$ occurring along $S'_i$.
Then we can apply exactly the same procedure as above and again produce a stability condition over $S$ that is $\vv$-generic on all fibers.
\end{Rem}

\section{Generalized Mukai's theorem}\label{sec:Mukai}

The result of this section is the following generalization to $\Ku(X)$ of well-known results by Mukai \cite{Mukai:Symplectic} and Inaba \cite{Inaba:Symplectic} for K3 surfaces.

Let $g\colon \cX\to S$ be a smooth family of cubic fourfolds, where $S$ is reduced and of finite type over $\C$. 
Let $\vv$ be a primitive section of the local system given by the Mukai lattices $\tH(\Ku(\cX_s),\Z)$ of the fibers over $s\in S(\C)$, such that $\vv$ is of Hodge type on all fibers.
Consider the locally of finite type algebraic stack $s\cM_\pug(\Ku(\cX)/S)(\vv)\to S$ parameterizing simple universally gluable $S$-perfect objects in $\Ku(\cX)$ with Mukai vector $\vv$. 
This is a $\mathbb{G}_m$-gerbe over a locally of finite type algebraic space $\rho \colon sM_{\pug}(\Ku(\cX)/S)(\vv) \to S$, see Lemma~\ref{Lem:SimpleAlgebraicSpace}. 

\begin{Thm}[Mukai]\label{thm:Mukai}
The morphism $\rho \colon sM_{\pug}(\Ku(\cX)/S)(\vv) \to S$ is smooth.
\end{Thm}

\begin{Rem}
A generalization of Theorem~\ref{thm:Mukai} was proved in \cite{IHC-CY2}, where $\Ku(\cX)$ is replaced by any ``family of CY2 categories" over $S$. 
In fact, the proof of Theorem~\ref{thm:Mukai} below incorporates some clarifications and corrections from \cite{IHC-CY2} to the argument from the first version of this paper, in particular the use of the $T^1$-lifting theorem to 
reduce to split square-zero extensions in the deformation argument. 
\end{Rem} 

The proof of Theorem~\ref{thm:Mukai} involves a combination of Mukai's original argument and the methods of \cite{KM:symplectic}, which in turn build on ideas of Buchweitz and Flenner.
We start by recalling a result on the deformation theory of objects in the derived category; 
for a more complete treatment of this subject and related arguments, see \cite{HuybrechtsThomas:defo, Lieblich:mother-of-all, HMS:Orientation,AT:CubicFourfolds}. 

Let $\cY \to \Spec R$ be a smooth projective morphism, and let $\cY_0 \to \Spec R_0$ be its base change along $R \to R_0 = R/I$ where $I \subset R$ is a square-zero ideal. 
Note that $\cY \to \Spec R$ is a deformation of $\cY_0 \to \Spec R_0$ over $R$, and recall that the set of isomorphism classes of such deformations of $\cY_0$ forms a torsor under $\rH^1(\mathrm{T}_{\cY_0/\Spec R_0} \otimes I)$. 
If $R \to R_0$ is a split square-zero extension, i.e. admits a section $R_0 \to R$, then there is a trivial deformation of $\cY_0$ over $R$ given by base change along the section, so the set of deformation classes is canonically identified with $\rH^1(\mathrm{T}_{\cY_0/\Spec R_0} \otimes I)$ with the trivial deformation corresponding to $0$. 
In this case, we write 
\begin{equation*}
\kappa(\cY) \in \rH^1(\mathrm{T}_{\cY_0/\Spec R_0} \otimes I) \cong 
\Ext^1(\Omega_{\cY_0/R_0}, I) 
\end{equation*}
for the element corresponding to the deformation $\cY \to \Spec R$, called the 
\emph{Kodaira--Spencer class}. 
Further, recall that for any object $E_0 \in \Dperf(\cY_0)$ there is a canonical element 
\begin{equation*} 
A(E_0) \in \Ext^1(E_0, E_0 \otimes \Omega_{\cY_0/R_0})
\end{equation*} 
called the \emph{Atiyah class}. 
The following is then the main result of \cite{HuybrechtsThomas:defo}, simplified to the case where the structure morphism is smooth and the square-zero extension is split. 

\begin{Thm}\label{thm:HTsquarezero}
Let $\cY \to \Spec R$ be a smooth projective morphism, let $R \to R_0 = R/I$ be a split square-zero extension, and let $\cY_0 \to \Spec R_0$ be the base change of $\cY$. 
For any $E_0 \in \Dperf(\cY_0)$, there exists an object $E \in \Dperf(\cY)$ such that $E_{\cY_0} \cong E_0$ if and only if 
\begin{equation*}
 \kappa(\cY) \cdot A(E_0) = 0 \in \Ext^2(E_0, E_0 \otimes I), 
\end{equation*} 
in which case the set of isomorphism classes of such $E$ forms a torsor under $\Ext^1(E_0, E_0 \otimes I)$. 
\end{Thm}

\begin{proof}[Proof of Theorem~\ref{thm:Mukai}]
First we claim that it is enough to prove the theorem when the base $S$ is smooth. 
Indeed, first note that for any morphism $S' \to S$ with based changed family $\cX'$, the relative moduli space $\rho' \colon sM_{\pug}(\Ku(\cX')/S')(\vv) \to S'$ is the base change of $\rho \colon sM_{\pug}(\Ku(\cX)/S)(\vv) \to S$ along $S' \to S$.
Applied to closed points of $S$ this shows that the fibers of $\rho$ are smooth. 
Applied to the spectrum of a DVR over $S$ (which factors via a smooth scheme of finite type over $S$), this proves the valuative criterion for flatness \cite[11.8.1]{EGA4-3}.
Combined, this shows that $\rho$ is smooth.

Thus from now on we may assume that $S$ is smooth. 
Assume we have proven that in this case the total space $sM_{\pug}(\Ku(\cX)/S)(\vv)$ is smooth. 
Let $0 \in S(\C)$ be a closed point, let $\cX_0$ be the fiber of $\cX \to S$, and let $\vv_0 \in \tH(\Ku(\cX_0), \Z)$ be the fiber of $\vv$. 
Then the moduli space $sM_{\pug}(\Ku(\cX_0))(\vv_0)$ ---i.e., the fiber of $\rho$ over $0$--- is smooth by our assumption (applied to the case where $S$ is a point), and of dimension $\vv_0^2 + 2$ (as this computes $\dim \Ext^1(E_0,E_0)$ for a simple universally gluable object in $\Ku(\cX_0)$ of class $\vv_0$). 
Therefore $\rho \colon sM_{\pug}(\Ku(\cX)/S)(\vv) \to S$ is a locally finite type morphism between smooth spaces whose closed fibers are smooth of constant dimension. 
It follows that $\rho$ is smooth. 

Thus to prove the theorem it suffices to show that $sM_{\pug}(\Ku(\cX)/S)(\vv)$ is smooth if $S$ is smooth. 
Since $sM_{\pug}(\Ku(\cX)/S)(\vv)$ is locally of finite type over $\C$, we just need to show that it is formally smooth at any $\C$-point. 
More precisely, let $E_0$ be a $\C$-point of $sM_{\pug}(\Ku(\cX)/S)(\vv)$ lying over a closed point $0 \in S$. 
Let $\Art_{\C}$ be the category of Artinian local $\C$-algebras with residue field $\C$, and let $F \colon \Art_{\C} \to \Sets$ be the deformation functor of the point $E_0$, i.e.\ $F(A)$ consists of pairs $(\Spec(A) \to S, E)$ where $\Spec(A) \to S$ takes the closed point $p \in \Spec(A)$ to $0 \in S$, and $E$ is an $A$-point of $sM_{\pug}(\Ku(\cX)/S)(\vv)$ whose restriction over $p \in S$ is isomorphic to $E_0$; 
for simplicity we often simply write $E \in F(A)$ suppressing the map $\Spec(A) \to S$ from the notation. 
(Note that in the definition of $F(A)$, instead of requiring $E$ to be an $A$-point of $sM_{\pug}(\Ku(\cX)/S)(\vv)$, it is equivalent to require $E \in \Dperf(\cX_{A})$;
indeed, an $A$-point $E$ of $sM_{\pug}(\Ku(\cX)/S)(\vv)$ must lie in $\Dperf(\cX_{A})$  since it is relatively perfect over $\Spec(A)$ and $\cX_A \to \Spec(A)$ is smooth, and conversely, the condition that an object $E \in \Dperf(\cX_{A})$ restricts over $p$ to $E_0$ guarantees that $E$ lies in $\Ku(\cX_{A})$, has class $\vv$, and is universally gluable, as $E_0$ has these properties.) 
To prove formal smoothness of $sM_{\pug}(\Ku(\cX)/S)(\vv)$ at $E_0$, we must show that $F$ is a smooth functor, i.e. for any surjection $A' \to A$ in $\Art_{\C}$, the map $F(A') \to F(A)$ is surjective. 

It will also be useful to consider the deformation functor $G \colon \Art_{\C} \to \Sets$ of the point $0 \in S$, whose value on $A \in \Art_{\C}$ consists of morphisms $\Spec(A) \to S$ taking the closed point to $0$. 
By definition, there is a morphism of functors $F \to G$. 

We start by considering a split square-zero extension $A' \to A$, that is $A' \cong A[\varepsilon]/(\varepsilon^2)$. 
We will prove that the fiber $\mathrm{Def}_E(A')$ of the map 
\begin{equation*}
F(A') \to G(A') \times_{G(A)} F(A)
\end{equation*} 
over any point $(\Spec(A') \to S, E)$ is a torsor under $\Ext^1(E,E)$. 
Let $\cX_{A'} \to \Spec A'$ and $\cX_{A} \to \Spec A$ be the base changes of our family $\cX \to S$ of cubic fourfolds. 
By Theorem~\ref{thm:HTsquarezero}, it suffices to show that 
\begin{equation} 
\label{kappa-AE-vanish}
\kappa(\cX_{A'}) \cdot A(E) = 0 \in \Ext^2(E, E). 
\end{equation} 
Note that $\Hom(E, E) \cong A$ since $E$ is simple, and relative Serre duality over $A$ gives that $\Ext^2(E, E)$ is a line bundle on $\Spec A$.
More precisely, a version of the arguments in \cite[Theorem~4.3]{KM:symplectic} relative to $A$ shows that
\begin{equation} \label{KMtrick}
\Ext^2(E, E ) \xrightarrow{\mathrm{Tr}_{E}(\blank \circ A(E))} 
\Ext^3(\cO_{\cX_A}, \Omega_{\cX_A/A}) 
\end{equation} 
is an isomorphism.
Now we claim that 
\begin{equation*}
 0 = \Tr_{E}\left(\kappa(\cX_{A'}) \cdot A(E)^2\right) \in \Ext^3(\cO_{\cX_A}, \Omega_{\cX_A/A}), 
\end{equation*}
so that by the isomorphism~\eqref{KMtrick} the desired vanishing~\eqref{kappa-AE-vanish} holds. 
Indeed, consider the formal exponential 
\begin{equation*}
 \exp(A(E))\colon E \to \bigoplus_{i\geqslant 0} E \otimes \Omega_{\cX_A/A}^i[i]. 
\end{equation*}
Due to the vanishing of $H^{i+2, i}$ for $i \neq 1$, we have
\[ \Tr_{E}\left(\kappa(\cX_{A'}) \cdot A(E)^2\right) = 2 \Tr_{E}(\kappa(\cX_{A'}) \cdot \exp(A(E))) = 2 \kappa(\cX_{A'}) \cdot \ch(E),
\]
which vanishes, as it is precisely the obstruction to $\vv$, or equivalently $\ch(E)$, remaining of Hodge type along $A'$.

Now we prove that $F$ is a smooth functor. 
We use the $T^1$-lifting theorem \cite{kawamata-T1, fantechi-T1} to reduce to the case of split square-zero extensions handled above. 
Namely, for any integer $n \geqslant 0$ set $A_n = \C[t]/(t^{n+1})$ and $A'_n = A_n[\varepsilon]/(\varepsilon^2)$. 
The $T^1$-lifting theorem says that $F$ is smooth if for every $n \geqslant 0$ the natural map 
\begin{equation*}
F(A'_{n+1}) \to F(A'_n) \times_{F(A_n)} F(A_{n+1}) 
\end{equation*} 
is surjective. 
Note that this map fits into a commutative diagram 
\begin{equation*}
\xymatrix{
F(A'_{n+1}) \ar[r] \ar[d] & F(A'_n) \times_{F(A_n)} F(A_{n+1}) \ar[d] \\ 
G(A'_{n+1}) \ar[r] & G(A'_n) \times_{G(A_n)} G(A_{n+1}) 
}
\end{equation*} 
where the bottom horizontal arrow is surjective because $0 \in S$ is a smooth point. 
Therefore, it suffices to prove the map 
\begin{equation} 
\label{T1-diagram} 
F(A'_{n+1}) \to G(A'_{n+1}) \times_{G(A'_n) \times_{G(A_n)} G(A_{n+1})} ( F(A'_n) \times_{F(A_n)} F(A_{n+1}) ) 
\end{equation} 
is surjective. Let $(\Spec(A'_{n+1}) \to S, E'_n, E_{n+1})$ be a point of the target of this map, 
and set $E_n = (E_{n+1})_{A_n} \cong (E'_n)_{A_n}$. 
By the previous paragraph, the set $\mathrm{Def}_{E_{n+1}}(A'_{n+1})$ of deformations 
of $E_{n+1}$ over $A'_{n+1}$ is an $\Ext^1(E_{n+1}, E_{n+1})$-torsor, 
and the set $\mathrm{Def}_{E_{n}}(A'_n)$ of deformations of $E_n$ over $A'_n$ is an 
$\Ext^1(E_n, E_n)$-torsor. 
Moreover, the restriction map $\mathrm{Def}_{E_{n+1}}(A'_{n+1}) \to \mathrm{Def}_{E_{n}}(A'_n)$ 
is compatible with the torsor structures under the restriction map 
\begin{equation}
\label{Ext1-restriction}
\Ext^1(E_{n+1}, E_{n+1}) \to \Ext^1(E_n, E_n). 
\end{equation}
We claim the map on $\Ext^1$ groups, and hence also $\mathrm{Def}_{E_{n+1}}(A'_{n+1}) \to \mathrm{Def}_{E_{n}}(A'_n)$, is surjective. 
Given this claim, we can choose an element $E'_{n+1} \in \mathrm{Def}_{E_{n+1}}(A'_{n+1})$ restricting to $E'_n \in \mathrm{Def}_{E_{n}}(A'_n)$; then $E'_{n+1}$ maps to $(\Spec(A'_{n+1}) \to S, E'_n, E_{n+1})$ under~\eqref{T1-diagram}, proving the required surjectivity. 

It remains to prove the map~\eqref{Ext1-restriction} is surjective. 
In fact, we will prove there is an isomorphism 
\begin{equation}
\label{Ext1-bc}
\Ext^1(E_{n+1}, E_{n+1}) \otimes_{A_{n+1}} A_n \cong 
\Ext^1(E_{n}, E_{n}) 
\end{equation} 
which implies the surjectivity of~\eqref{Ext1-restriction}. 
The isomorphism~\eqref{Ext1-bc} is a special case of \cite[Lemma~7.4(2)]{IHC-CY2}, but for completeness we include the argument. 
By base change (Lemma~\ref{Lem-cHom-bc}), we have 
\begin{equation*}
\cHom_{A_{n+1}}(E_{n+1}, E_{n+1}) \otimes_{A_{n+1}} A_n \cong 
\cHom_{A_{n}}(E_{n}, E_{n}) . 
\end{equation*} 
We will prove the cohomology groups $\Ext^i(E_{n+1}, E_{n+1})$ of $\cHom_{A_{n+1}}(E_{n+1}, E_{n+1})$ are locally free $A_{n+1}$-modules for all $i$, which by the above base change formula implies the isomorphism~\eqref{Ext1-bc}. 
Serre duality shows that $\Ext^i(E_{n+1}, E_{n+1})$ vanishes for $i \notin [0,2]$, and for $i = 0$ or $i = 2$ we already observed above that it is a line bundle. 
For $i = 1$, by the local criterion for flatness it suffices to show that $\rH^{-1}(j_0^*\Ext^1(E_{n+1}, E_{n+1})) = 0$, where $j_0 \colon \Spec(\C) \to \Spec(A_{n+1})$ is the closed point. 
Note that we have a spectral sequence with $E_2$-page 
\begin{equation*}
E_2^{i,j} = \rH^j(j^*_0 \Ext^i(E_{n+1}, E_{n+1})) \implies \Ext^{i+j}(E_0, E_0). 
\end{equation*} 
But $E_2^{0,j}$ and $E_2^{2,j}$ are $1$-dimensional for $j = 0$ and vanish for $j \neq 0$, and $\Ext^0(E_0, E_0)$ is $1$-dimensional, so the vanishing $\rH^{-1}(j_0^*\Ext^1(E_{n+1}, E_{n+1})) = 0$ follows. 
\end{proof}

\section{Specializing Kuznetsov components to twisted K3 surfaces} 

Combining Proposition~\ref{prop:Msigmarelative1}, Theorem~\ref{thm:modulispacesArtinstacks}, and Theorem~\ref{thm:Mukai}, we obtain the following.

\begin{Cor}\label{cor:famstabcurvescubics}
Let $X$ be a cubic fourfold, let $\vv\in\tH_\Hdg(\Ku(X),\Z)$ be a primitive vector, and let $\sigma\in\Stab^\dagger(\Ku(X))$ be a $\vv$-generic stability condition. 
Let $X'$ be another cubic fourfold in the Hodge locus where $\vv$ stays a Hodge class. 
Then there exist a family $g\colon\cX\to C$ of cubic fourfolds over a smooth connected quasi-projective curve and a stability condition $\us$ on $\Ku(\cX)$ over $C$ such that:
\begin{enumerate}[{\rm (1)}] 
\item\label{enum:staysalg} $\vv$ is a primitive vector in $\tH_\Hdg(\Ku(\cX_c),\Z)$ for all closed points $c\in C$;
\item\label{enum:atp0} there exists closed points $c_0, c_1 \in C$ such that $\cX_{c_0}=X$ and $\cX_{c_1}=X'$; 
\item \label{enum:smalldefo} $\sigma_{c_0}$ is a small deformation of $\sigma$ satisfying $M_{\sigma_{c_0}}(\Ku(\cX_{c_0}), \vv) = M_{\sigma}(\Ku(\cX_{c_0}), \vv)$;
\item\label{enum:vgeneric} $\sigma_c$ is $\vv$-generic for all $c\in C$;
\item\label{enum:famstabcurvescubics5} the relative moduli space $M_{\us}(\vv)$ is smooth and proper over $C$.
\end{enumerate}
Furthermore, if we assume that the central charge $Z$ of $\sigma$ is defined over $\Q[\ii]$ and that $X'$ is also in the Hodge locus where $\ell_\sigma\in\tH_\Hdg(\Ku(X),\Q)$ stays a Hodge class, then we can choose $\us$ so that there is an open neighborhood $U \subset C$ of $c_0$ 
such that the central charge of $\us$ restricted over $U$ factors through $\cN(\Ku(\cX_U)/U)$. 
\end{Cor}

\begin{proof}
Parts \eqref{enum:staysalg}--\eqref{enum:famstabcurvescubics5} follow directly from Proposition~\ref{prop:Msigmarelative1}.
To clarify the last statement, we recall that the class $\ell_{\sigma}\in N^1(M_{\sigma}(\vv))$ of Theorem~\ref{thm:PositivityLemmaFamily} corresponds via the Mukai morphism (see Theorem~\ref{thm:YoshiokaMain}.\eqref{enum:YoshiokaMain2}) to a class in $\tH(\Ku(X),\R)$, which is actually rational under our assumptions.
Then, up to slightly deforming $\sigma$, we can assume that its central charge factors through $M:=\tH_\Hdg(\Ku(\cX_c),\Z)$, for $c$ a very general point in $C$.
The conclusion follows then directly from Proposition~\ref{prop:Msigmarelative1}.\eqref{Msigmarel3}.
\end{proof}

%\begin{Cor}\label{cor:famstabcurvescubics}
%Let $X$ be a cubic fourfold, $\vv\in\tH_\Hdg(\Ku(X),\Z)$ a primitive vector, and $\sigma\in\Stab^\dagger(\Ku(X))$ $\vv$-generic. 
%Let $X'$ be another cubic fourfold in the Hodge locus where $\vv$ stays a Hodge class. 
%Then there exist a family $g\colon\cX\to C$ of cubic fourfolds over a smooth connected quasi-projective curve and a stability condition $\us$ on $\Ku(\cX)$ over $C$ such that:
%\begin{enumerate}[{\rm (1)}] 
%\item\label{enum:staysalg} $\vv$ is a primitive vector in $\tH_\Hdg(\Ku(\cX_c),\Z)$ for all $c\in C$;
%\item\label{enum:atp0} $\cX_{c_0}=X$, $\cX_{c_1}=X'$ and $\us_{c_0}=\sigma$;
%\item\label{enum:vgeneric} $\us_c$ is $\vv$-generic for all $c\in C$;
%\item the relative moduli space $M_{\us}(\vv)$ is smooth and proper over $C$.
%\end{enumerate}
%\end{Cor}

Corollary~\ref{cor:famstabcurvescubics} shows that to prove deformation invariant properties about $M_{\sigma}(\Ku(X), \vv)$, we may specialize the cubic fourfold $X$ within the Hodge locus for $\vv$. 
In the proofs of our results stated in Section~\ref{sec:MainResultsCubics}, we will use this observation to specialize to the case where $\Ku(X)$ is equivalent to the derived category of a twisted K3 surface $(S, \alpha)$, for which many results are already known \cite{BM:proj}. 
There is one subtlety: moduli spaces of $\sigma$-stable objects in $\Db(S, \alpha)$ are only well-understood when $\sigma$ lies in the connected component 
$\Stab^{\dagger}(S,\alpha)$ containing geometric stability conditions, i.e., those for which skyscraper sheaves of points are stable of the same phase. 
\index{StabDagger(S,alpha)@$\Stab^{\dagger}(S,\alpha)$, connected component of $\Stab(\Db(S,\alpha))$ containing geometric stability conditions}

\begin{Rem}
In the first version of this paper, the above subtlety was overlooked, but in the meantime it was addressed in \cite[Section~5.2]{GM-stability} in the context of Gushel--Mukai varieties. 
The arguments below closely follow those of \cite[Section~5.2]{GM-stability}, but are included for completeness. 
\end{Rem} 

\begin{Def}
Let $X$ be a cubic fourfold and $(S, \alpha)$ a twisted K3 surface. 
A \emph{$\dagger$-equivalence}\index{dagger-equivalence@$\dagger$-equivalence preserving geometric components} $\Ku(X) \simeq \Db(S, \alpha)$ is an equivalence under which $\Stab^{\dagger}(\Ku(X))$ maps to $\Stab^{\dagger}(S, \alpha)$. 
\end{Def} 

The following proposition gives the desired type of specialization, and is the main result of this section. 

\begin{Prop}
\label{proposition-specialize-twisted}
Let $X$ be a cubic fourfold and let $\vv\in\tH_\Hdg(\Ku(X),\Z)$. 
Then $X$ is deformation equivalent within the Hodge locus for $\vv$ to a cubic fourfold $X'$ such that there is a $\dagger$-equivalence $\Ku(X') \simeq \Db(S', \alpha')$ for a twisted K3 surface $(S', \alpha')$. 
\end{Prop} 

We will prove the proposition at the end of this section, after some preliminary results. 

\begin{Lem}
\label{Lem-Ku-twisted-K3}
Let $X$ be a cubic fourfold, let $\vv\in\tH_\Hdg(\Ku(X),\Z)$ be primitive vector with $\vv^2 = 0$, and let $\sigma \in \Stab^{\dagger}(\Ku(X))$ be a $\vv$-generic stability condition. 
Assume there exists a $\sigma$-stable object in $\Ku(X)$ of class $\vv$. 
Then $S = M_{\sigma}(\Ku(X), \vv)$ is a smooth K3 surface and there is a $\dagger$-equivalence $\Ku(X) \simeq \Db(S, \alpha)$ for a certain Brauer class $\alpha \in \Br(S)$. 
\end{Lem} 

\begin{proof}
The results of Section~\ref{subsec:famstabcondKuz} applied in the case where the base is a point show that $\sigma$ is a stability condition \emph{over $\Spec(\C)$}. 
Thus by Theorem~\ref{thm:modulispacesArtinstacks}.\eqref{enum:GoodModuliChar0}, 
$S$ is a proper algebraic space over $\Spec(\C)$. 
Moreover, by Theorem~\ref{thm:Mukai}, $S$ is also smooth over $\Spec(\C)$. 
Since $\Ku(X)$ is a 2-Calabi--Yau category, standard arguments show that $S$ has dimension $\vv^2 + 2 = 2$ and is equipped with a symplectic form (see \cite{KM:symplectic}). 
Since $S$ is a smooth proper $2$-dimensional algebraic space, it is in fact a smooth projective surface. 
Moreover, by \cite[Proposition~A.7]{BLMS}, $S$ is also connected. 

Now let $\cE$ be a quasi-universal family over $S \times X$ and $\alpha \in \Br(S)$ the associated Brauer class. 
By \cite{Bridgeland:EqFMT}, the exact functor $\Phi_\cE\colon\Db(S,\alpha)\to \Ku(X)$ is fully faithful. Since the $0$-th Hochschild cohomology $\mathrm{HH}^0(\Ku(X))$ is one-dimensional, $\Phi_{\cE}$ is also essentially surjective (Bridgeland's trick; see for example \cite[Proposition~5.1]{Kuz:CYcat}). 
By construction, the equivalence $\Phi_{\cE}$ is a $\dagger$-equivalence. 

As $S$ has a symplectic form, it is either a K3 or abelian surface. 
Since $\Db(S, \alpha) \simeq \Ku(X)$ has the same Hochschild homology as a K3 surface, $S$ must in fact be a K3 surface. 
\end{proof}

Next we observe that the existence of a $\dagger$-equivalence deforms along Hodge loci for square-zero classes. 

\begin{Lem}
\label{Lem-defo-preserves-dagger}
Let $X$ and $X'$ be cubic fourfolds which are deformation equivalent within the Hodge locus for a vector $\vv \in \tH_{\Hdg}(\Ku(X), \Z)$ with $\vv^2 = 0$. 
Then $\Ku(X)$ is $\dagger$-equivalent to the derived category of a twisted K3 surface if and only if $\Ku(X')$ is. 
\end{Lem} 

\begin{proof}
Assume $\Ku(X)$ is $\dagger$-equivalent to the derived category of a twisted K3 surface. 
Then if $\sigma \in \Stab^{\dagger}(\Ku(X))$ is $\vv$-generic, the moduli space 
$M_{\sigma}(\Ku(X), \vv)$ is a smooth K3 surface \cite{BM:proj}. 
By Corollary~\ref{cor:famstabcurvescubics}, it follows that for any $\vv$-generic $\sigma' \in \Stab^{\dagger}(\Ku(X'))$, the moduli space $M_{\sigma'}(\Ku(X'), \vv)$ is also a smooth K3 surface. 
By Lemma~\ref{Lem-Ku-twisted-K3}, we conclude $\Ku(X')$ is $\dagger$-equivalent to the derived category of a twisted K3 surface. 
\end{proof} 

The following is our key technical ingredient.
It allows us to replace a given equivalence $\Ku(X) \simeq \Db(S, \alpha)$ with a $\dagger$-equivalence, provided $X$ admits deformations with suitable Hodge-theoretic properties. 

\begin{Lem}
\label{Lem-defo-spherical-dagger}
Let $X$ be a cubic fourfold such that: 
\begin{enumerate}
\item There is an equivalence $\Ku(X) \simeq \Db(S, \alpha)$ for a twisted K3 surface $(S, \alpha)$. 
\item There exists a vector $\vv \in \tH_{\Hdg}(\Ku(X), \bZ)$ with $\vv^2 = 0$ such that $X$ is deformation equivalent within the Hodge locus for $\vv$ to a cubic fourfold $X'$ with the property that $\tH_{\Hdg}(\Ku(X'), \bZ)$ contains no vectors $\ww$ with $\ww^2 = -2$. 
\end{enumerate}
Then there exists a twisted K3 surface $(T, \beta)$, possibly different from $(S, \alpha)$, and a $\dagger$-equiva\-lence $\Ku(X) \simeq \Db(T, \beta)$. 
\end{Lem} 

\begin{proof}
Let $\cX \to C$ be a family of cubic fourfolds over a smooth connected quasi-projective curve $C$, such that $\cX_0 = X$ and $\cX_1 = X'$ for some points $0,1 \in C(\C)$, and $\vv$ remains a Hodge class along $C$. 
There exists a simple object $E \in \Ku(X)$ of class $\vv \in \tH_{\Hdg}(\Ku(X), \bZ)$ with $\Ext^{<0}(E, E) = 0$; 
indeed, even stronger, there exist $\sigma$-stable objects of class $\vv$ for a $\vv$-generic $\sigma \in \Stab^{\dagger}(S, \alpha)$ by \cite{BM:proj}. 
Thus by Theorem~\ref{thm:Mukai} there is a Zariski open $U \subset C$ such that for any $c \in U(\C)$ there exists a simple object $E_c \in \Ku(\cX_c)$ of class $\vv \in \tH_{\Hdg}(\Ku(\cX_c), \bZ)$ with $\Ext^{<0}(E_c, E_c) = 0$;
in particular, it follows that $\Ext^1(E_c, E_c) \cong \C^2$. 

The condition that $\tH_{\Hdg}(\Ku(\cX_c), \bZ)$ contains no vectors $\ww$ with $\ww^2 = -2$ holds for a very general $c \in U(\C)$, because it holds for $c = 1$ by assumption. 
Therefore, after possibly replacing $X'$ by a different fiber of $\cX \to C$, we may assume there is an object $E' \in \Ku(X')$ such that $\Ext^1(E', E') \cong \C^2$. 
It then follows from \cite[Lemma~A.4]{BLMS} that $E'$ is $\sigma$-stable for any $\sigma \in \Stab(\Ku(X'))$. 
Thus, Lemma~\ref{Lem-Ku-twisted-K3} shows $\Ku(X')$ is $\dagger$-equivalent to the derived category of a twisted K3 surface, and by Lemma~\ref{Lem-defo-preserves-dagger} we conclude the same for $\Ku(X)$. 
\end{proof} 

\begin{proof}[Proof of Proposition~\ref{proposition-specialize-twisted}]
By \cite[Theorem~4.1]{AT:CubicFourfolds}, $X$ is deformation equivalent within the Hodge locus for $\vv$ to a cubic 
fourfold in $\cC_8$, where $\cC_8$ is the discriminant $8$ Hassett divisor. 
Recall that $\cC_8$ can be described either as the irreducible divisor 
parameterizing cubic fourfolds containing a plane, or as the Hodge locus for a certain square-zero 
vector in $\tH_{\Hdg}(\Ku(X), \bZ)$, see \cite{Huy:cubics}. 
Therefore, to finish the proof, 
by Lemma~\ref{Lem-defo-preserves-dagger} it suffices to show there 
exists a $X \in \cC_{8}$ such that $\Ku(X)$ is $\dagger$-equivalent to the 
derived category of a twisted K3 surface. 

By \cite[Proposition~2.15]{Huy:cubics}, there exists a Hassett divisor $\cC_d$ such that $\cC_d$ is the Hodge locus for a square-zero vector in $\tH_{\Hdg}(\Ku(X), \bZ)$, and the very general point of $\cC_d$ parameterizes a cubic fourfold $X$ such that $\tH_{\Hdg}(\Ku(X), \bZ)$ contains no vectors $\ww$ with $\ww^2 = -2$. 
Moreover, again by \cite[Theorem~4.1]{AT:CubicFourfolds}, there exists a cubic fourfold $X \in \cC_d \cap \cC_8$ such that $\Ku(X) \simeq \Db(S, \alpha)$ for a twisted K3 surface $(S, \alpha)$ (and in fact we may take $\alpha = 0$). 
Thus, we conclude by Lemma~\ref{Lem-defo-spherical-dagger} that $\Ku(X)$ is $\dagger$-equivalent to the derived category of a twisted K3 surface, as required. 
\end{proof}

%\begin{Rem}
%Instead of invoking Lemma~\ref{Lem-defo-spherical-dagger} in the above proof, we could 
%argue as follows. 
%\end{Rem} 

\section{Proofs of the main results}\label{sec:ProofFirstTheoremsCubics}

In this section, we apply the previous results to prove Theorems~\ref{thm:ConnectedComponentStab}, \ref{thm:YoshiokaMain}, and \ref{thm:Msigmarelative}, and their consequences, Corollaries~\ref{cor:locfam20dim}, \ref{cor:completeAT}, and \ref{cor:integralHdg}.

\begin{proof}[Proof of Theorem~\ref{thm:Msigmarelative}.\eqref{enum:algspace}.]
Let $\cX\to C$ be a family of cubic fourfolds.
Let $\widetilde{C} \to C$ be a finite (not necessarily \'etale) Galois cover such that $\cX_{\widetilde C}$ admits a family of lines in the fibers that are not contained in a plane.
Let $\us$ be the stability condition over $\widetilde{C}$ given by Proposition~\ref{prop:Msigmarelative1} that is $\vv$-generic in every fiber;
in the construction it is easy to ensure that $\us$ is Galois-invariant, in the naive sense that in each orbit, the stability conditions in the corresponding Kuznetsov component are identical. 

Now consider the relative moduli space $M_{\us}(\vv)$. 
By construction, there are no properly semistable objects, and so
$M_{\us}(\vv)$ is an open subspace of $sM_\pug(\Ku(\cX_{\widetilde{C}})/\widetilde{C})$. 
By Theorem~\ref{thm:modulispacesArtinstacks}.\eqref{enum:MAlgStackFT}, it is an algebraic space, proper over $\widetilde{C}$. By Theorem~\ref{thm:Mukai}, it is also smooth over $\widetilde{C}$.

Finally, by the Galois invariance of $\us$, the Galois group of $\widetilde{C} \to C$ also acts on $M_{\us}(\vv)$; therefore it descends to a smooth and proper morphism $\widetilde{M}(\vv) \to C$ with the properties described in the statement of the theorem.
\end{proof}
\begin{proof}[Proof of Theorem~\ref{thm:Msigmarelative}.\eqref{enum:goodrelativemoduli}.]
In the case where $S$ admits a family of lines over $S$, none of which are contained in a plane, this is just Theorem~\ref{thm:modulispacesArtinstacks} in our context. The general case can be reduced to that situation using a cover of $S$, just as in the previous proof.
\end{proof}

Next we prove Corollary~\ref{cor:completeAT}, in the following general form. 

\begin{Prop}\label{prop:KuzHass}
Let $X$ be a cubic fourfold.
Then there exist a smooth projective K3 surface $S$ and a Brauer class $\alpha\in\Br(S)$ such that $\Ku(X) \simeq \Db(S,\alpha)$ if and only if there exists a non-zero primitive Mukai vector $\vv\in\tH_\Hdg(\Ku(X),\Z)$ such that $\vv^2=0$.
Moreover, the class $\alpha$ can be chosen to be trivial if and only if there exists another Mukai vector $\vv'\in\tH_\Hdg(\Ku(X),\Z)$ such that $(\vv,\vv')=1$.
\end{Prop}

Note the following generalization of \cite{Mukai:BundlesK3}, which can also be regarded as an elaboration on Lemma~\ref{Lem-Ku-twisted-K3} (with slightly different hypotheses). 

\begin{Lem}\label{Lem:MukaiEquivalenceFM}
Let $\cD$ be a 2-Calabi--Yau category, i.e. an admissible subcategory of the derived category of a smooth projective variety, with Serre functor $\rS_{\cD} = [2]$. Assume $\mathrm{HH}^0(\cD) = \C$.
\begin{enumerate}[{\rm (1)}]
 \item\label{enum:MukaiEquivalenceFM1} If there exists a K3 surface $S$ and a Brauer class $\alpha\in\Br(S)$ such that $\cD \simeq \Db(S,\alpha)$, then there exists a primitive Mukai vector $\vv\in\tH_\Hdg(\cD,\Z)$ with $\vv^2=0$.
 \item\label{enum:MukaiEquivalenceFM2} If there exists a primitive Mukai vector $\vv$ with $\vv^2=0$ and a Bridgeland stability condition $\sigma$ which is $\vv$-generic and such that $M_{\sigma}(\vv)$ exists and it is isomorphic to a K3 surface $S$, then $\cD \simeq \Db(S,\alpha)$, for a certain $\alpha\in\Br(S)$.
\end{enumerate}
Moreover, the Brauer class $\alpha$ can be chosen to be trivial if and only if there exists another Mukai vector $\vv'\in\tH_\Hdg(\cD,\Z)$ with $(\vv',\vv)=1$.
\end{Lem}

\begin{Rem}
The group of Hodge classes $\tH_{\Hdg}(\cD, \bZ)$ appearing in the lemma can be defined as in \cite{IHC-CY2};
however, in this paper we will only need the case where $\cD = \Ku(X)$ for a cubic fourfold $X$.
\end{Rem}

\begin{proof}
Part \eqref{enum:MukaiEquivalenceFM1} is clear, by taking $\vv$ to be the Mukai vector of a skyscraper sheaf (and, if the Brauer class is trivial, $\vv'$ to be that of the structure sheaf).
Part \eqref{enum:MukaiEquivalenceFM2} holds as in the proof of Lemma~\ref{Lem-Ku-twisted-K3}. 
%, we let $\cE$ be the quasi-universal family and $\alpha\in\Br(S)$ the associated Brauer class. By \cite{Bridgeland:EqFMT}, the exact functor $\Phi_\cE\colon\Db(S,\alpha)\to\cD$ is fully faithful. Since $\mathrm{HH}^0(\cD)$ is one-dimensional, it is also essentially surjective (Bridgeland's trick; see for example \cite[Proposition~5.1]{Kuz:CYcat}).
Finally, if there exists another Mukai vector $\vv'\in\tH_\Hdg(\cD,\Z)$ with $(\vv',\vv)=1$, then the quasi-universal family is universal, and so $\alpha$ can be chosen to be trivial.
\end{proof}

\begin{proof}[Proof of Proposition~\ref{prop:KuzHass}.]
By Lemma~\ref{Lem:MukaiEquivalenceFM}, we only need to show that, given a non-zero primitive Mukai vector $\vv\in\tH_\Hdg(\Ku(X),\Z)$ such that $\vv^2=0$, there exists a $\vv$-generic stability condition $\sigma$ for which the moduli space $M_\sigma(\vv)$ is a K3 surface. 
This follows from Proposition~\ref{proposition-specialize-twisted}, Corollary~\ref{cor:famstabcurvescubics}, 
and the analogous statement for moduli spaces for stability conditions in the distinguished component of the stability manifold of a twisted K3 surface \cite{BM:proj}. 
\end{proof}

The above proposition is enough to extend \cite[Theorem~1.5.ii)]{Huy:cubics}, by using the Derived Torelli Theorem for twisted K3 surfaces \cite{Orlov:representability,HuybrechtsStellari:CaldararuConj}.

\begin{Cor}\label{cor:Huy}
	Let $d$ be a positive integer such that $d\equiv 0,2 \pmod{6}$ and  $n_i\equiv 0\pmod{2}$ 
 for all primes $p_i\equiv 2\pmod{3}$ occuring in the factorization $2d=\prod p_i^{n_i}$. 
%	\begin{equation*}
%	    d\equiv 0,2 \pmod{6} \quad \text{and} \quad n_i\equiv 0\pmod{2}
%	\end{equation*}
%	for all primes $p_i\equiv 2\pmod{3}$ occuring in the factorization $2d=\prod p_i^{n_i}$. 
%	\[
%	d\equiv 0,2 \pmod{6} \quad \text{and }n_i\equiv 0\pmod{2}\text{ for all primes }p_i\equiv 2\pmod{3}\text{ in }2d=\prod p_i^{n_i}
%	\]
	Then for cubic fourfolds $X,X'\in\cC_d$, there exists an equivalence  $\Ku(X)\simeq \Ku(X')$ if and only if there exists a Hodge isometry $\tH(\Ku(X),\Z)\cong\tH(\Ku(X'),\Z)$.
\end{Cor}

\begin{proof}[Proof of Theorem~\ref{thm:YoshiokaMain}.]
We divide the proof in a few steps.
For the moment, we think of $\Stab^\dagger(\Ku(X))$ as a non-empty connected open subset. We postpone the proof that it is a whole connected component of the space of stability conditions to later on in the section.

\subsection*{Non-emptiness}
We first deal with the non-emptiness statement in part \eqref{enum:YoshiokaMain1} of the theorem.
One implication follows immediately from the properties of the Mukai pairing: since $\vv$ is primitive and $\sigma$ is $\vv$-generic, then $M_\sigma(\vv)=M_\sigma^{\st}(\vv)$ and any object $E$ in $M_\sigma(\vv)$ satisfies $\vv^2=-\chi(E,E)\geqslant-2$.

For the converse, let $\vv\in\tH_{\Hdg}(\Ku(X),\Z)$ be a primitive Mukai vector with $\vv^2\geqslant-2$ and $\sigma\in\Stab^\dagger(\Ku(X))$ be a $\vv$-generic stability condition.
We apply one more time Corollary~\ref{cor:famstabcurvescubics} as in the proof of Proposition~\ref{prop:KuzHass}: the Hodge locus where $\vv$ stays a Hodge class will intersect the divisor $\cC_{8}$. Since, by \cite[Corollary~6.9]{BM:proj}, the moduli space is non-empty there, it is non-empty on all fibers, in particular over $X$, as we wanted.

\subsection*{Hodge classes are algebraic}
The equality $\tH_{\Hdg}(\Ku(X),\Z)=\tH_{\mathrm{alg}}(\Ku(X),\Z)$ is now immediate: every Hodge class can be written as sum of Hodge classes having positive square.
The non-emptiness statement above guarantees that these are algebraic, as we wanted.

\subsection*{Projectivity}
The moduli space $M_\sigma(\vv)$ is smooth and proper, by Theorem~\ref{thm:Mukai} and Theorem~\ref{thm:modulispacesArtinstacks}.\eqref{enum:GoodModuliChar0}.
Moreover, by \cite[Theorem~2.2]{KM:symplectic}, since $M_\sigma(\vv)$ parameterizes stable objects in the K3 category $\Ku(X)$, Serre duality gives a non-degenerate closed symplectic $2$-form on $M_\sigma(\vv)$, and so it has trivial canonical bundle.
We want to prove that $M_\sigma(\vv)$ is projective.\footnote{The following argument was suggested to us by Giulia Sacc\`a; it greatly simplifies the one in an earlier version of this paper.}

We claim, in a slightly more precise version of Corollary~\ref{cor:famstabcurvescubics}, that there exists a family $g\colon\cX\to C$ of cubic fourfolds, parameterized by a smooth connected quasi-projective curve $C$, $c_0,c_1\in C$, and a stability condition $\us$ on $\Ku(\cX)$ over $C$ whose central charge factors via $\cN(\Ku(\cX)/C)$ such that:
\begin{enumerate}[{\rm (1)}] 
\item $\vv$ is a primitive vector in $\tH_\Hdg(\Ku(\cX_c),\Z)$ for all closed points $c\in C$;
\item $\cX_{c_0}=X$ and $M_{\sigma_{c_0}}(\Ku(\cX_{c_0}), \vv) = M_{\sigma}(\Ku(\cX_{c_0}), \vv)$;
\item\label{enum:projectivity2} $\Ku(\cX_{c_1})\cong\Db(S,\alpha)$, for some twisted K3 surface $(S,\alpha)$;
\item $\sigma_c$ is $\vv$-generic for all $c\in C$;
\item the relative moduli space $M_{\us}(\vv)$ is smooth and proper over $C$.
\end{enumerate}

If we assume the claim, we can conclude the proof as follows.
Let us consider the relatively nef real numerical Cartier divisor class $\ell_{\us}\in N^1(M_{\us}(\vv)/C)$ of Theorem~\ref{thm:PositivityLemmaFamily}.
By slightly deforming $\us$ and taking multiples if necessary, we can assume it is actually an integral class.
By Proposition~\ref{proposition-specialize-twisted} and \cite[Corollary~7.5]{BM:proj}, we have that the divisor class $\ell_{\sigma_{c_1}}$ is ample.
Since ampleness is an open property, $\ell_{\sigma_c}$ is ample for all $c$ in a Zariski open subset $U$ of $C$.
By relative Serre vanishing, we can further assume it has no higher cohomology, for all $c\in U$. 
Hence, $h^0(M_{\sigma_c}(\vv),\ell_{\sigma_{c}}^{\otimes m})=\chi(M_{\sigma_c}(\vv),\ell_{\sigma_{c}}^{\otimes m})$, for all $c\in U$ and $m>0$, and thus it is independent on $c$.

By semicontinuity, this shows that $h^0(M_\sigma(\vv),\ell_{\sigma}^{\otimes m})$ has maximal growth, for $m\gg0$, and therefore it is big on $M_\sigma(\vv)$.
Since $M_\sigma(\vv)$ has trivial canonical bundle and $\ell_{\sigma}$ is nef as well, the Base Point Free Theorem \cite[Theorem~3.3]{KollarMori} (for algebraic spaces see also \cite{Ancona:Vanishing}) implies that a multiple of $\ell_{\sigma}$ is globally generated on $M_\sigma(\vv)$. 
But $\ell_{\sigma}$ has the stronger positivity property of intersecting any curve strictly positively; this shows that $\ell_{\sigma}$ is actually ample, and thus $M_\sigma(\vv)$ is projective, as we wanted.

Finally, we prove the claim above.
By Corollary~\ref{cor:famstabcurvescubics}, there exists a family $g\colon\cX\to C$ and a stability condition $\us$ with all the properties except that the central charge might not factor via $\cN(\Ku(\cX)/C)$.
This works if we replace $C$ with an open subset $C'\subseteq C$ containing $c_0$ but $c_1$ might not be in $C'$. 

The claim follows if we show that we can take $C$ with infinitely many points satisfying \eqref{enum:projectivity2}.
The argument is rather elementary and we will briefly discuss it. Let $N$ be the smallest saturated sublattice of $\tH_\Hdg(\Ku(X),\Z)$ generated by the natural sublattice $A_2$, $\vv$ and $\ell_\sigma$.
By \cite[Proposition~2.3]{AT:CubicFourfolds}, the saturated sublattice $N^\perp$ of $\tH(\Ku(X),\Z)$ identifies with a saturated sublattice $N'$ in $H^4(X,\Z)$ (up to sign).
Thus $X$ is contained in the Hodge locus $\cC_M$ of the moduli space of cubic fourfolds $\cC$ determined by the lattice $M:=(N')^\perp$.
By Proposition~\ref{prop:KuzHass}, the set of cubic fourfolds satisfying \eqref{enum:projectivity2} is a countable collection of divisors in $\cC$ and each of them is a Hodge locus determined by a rank-$2$ sublattice of the fourth integral cohomology lattice of the cubic fourfold with non-trivial intersection with $M$, due to the signature of the lattice orthogonal to the saturated sublattice generated by $N$ and the additional class of square zero.
Thus the intersection of $\cC_M$ with any such divisor $\cC'$ has codimension at most $1$ in $\cC_M$.
In order to show that the intersection $\cC_M\cap\cC'$ is nonempty for infinitely many $\cC'$ as above, by the explicit description of the image of the period map for cubic fourfolds \cite{Loo:period,Laza:cubic4folds2}, we just need to show that this intersection is not contained in the two special divisors $\cC_2$ and $\cC_6$ in the period domain of cubic fourfolds.
But if, for infinitely many $\cC'$ the intersection is contained in either of the two divisors then, by the density of such Hodge loci, we would get that $\cC_M$ is contained either in $\cC_2$ or $\cC_6$ which contradicts the fact that $X$ is in $\cC_M$.
Hence we can pick the curve $C$ in $\cC_M$ so that all the properties above are satisfied and \eqref{enum:projectivity2} holds true for infinitely many points.
The same is then true when passing to an open subset $C'$ containing $c_0$.

\subsection*{The holomorphic symplectic structure}
Since $M$ is deformation equivalent to a Hilbert scheme of points on a K3 surface, it is a irreducible holomorphic symplectic manifold.
Moreover, the existence of a quasi-universal family on the relative moduli space guarantees that the morphism $\theta$ does behave well in family as well, and thus Theorem~\ref{thm:YoshiokaMain}.\eqref{enum:YoshiokaMain2} follows by again reducing to the case of (twisted) K3 surfaces and using \cite[Theorem~6.10]{BM:proj} therein.
Theorem~\ref{thm:YoshiokaMain} is proven.
\end{proof}

We can finally complete the proof of Theorem~\ref{thm:Msigmarelative}.

\begin{proof}[Proof of Theorem~\ref{thm:Msigmarelative}.\eqref{enum:ProjMorph}.]
First of all, as in the first part of the proof of Proposition~\ref{prop:Msigmarelative1}, we can base change with a finite cover $u\colon\widetilde{S}\to S$ to have a family of lines over $\widetilde{S}$ and so that the fixed locus $M$ of the monodromy group is $\tH_\Hdg(\Ku(\cX_{s_0}),\Z)$.
This gives a stability condition $\us$ on $\Ku(\cX)$ over $\widetilde{S}$ with respect to $M^\vee$ and with the property that $\us_{s_0}=\sigma_{s_0}$. 
We can shrink $S$ and assume further that the cover $\widetilde{S}\to S$ is \'etale.

Stability is open; in particular, the set of points $s\in\widetilde{S}$ such that $\us_s$ is $\vv$-generic is open.
We set $\widetilde{S}^0$ this open subset.
We consider the relative moduli space $\widetilde{M}^0(\vv):=M_{\us}(\vv)$ over $\widetilde{S}^0$.
This is proper and smooth over $\widetilde{S}^0$, and since now the lattice is fixed, comes also with a relatively ample divisor class $l_{\us}$, as defined in the proof of Theorem~\ref{thm:YoshiokaMain}.

By \citestacks{0D30}, $\widetilde{M}^0(\vv)$ is a smooth projective integral scheme over $\widetilde{S}^0$.
Let $S^0:=u(\widetilde{S}^0)\subset S$; it is open.
By Proposition~\ref{prop:ProductStabilityCondition}, $\widetilde{M}^0(\vv)$ gives a descent datum for $\widetilde{M}^0(\vv)/\widetilde{S}^0/S^0$ (see \citestacks{023U}).
Hence, $\widetilde{M}^0(\vv)$ does descend to an algebraic space $M^0(\vv)\to S^0$.
Since the central charge is monodromy-invariant, the relatively ample divisor class does descend as well.
Therefore $M^0(\vv)$ is a scheme which is smooth projective over $S$ and a relative moduli space, which is what we wanted.
\end{proof}

\begin{proof}[Proof of Corollary~\ref{cor:locfam20dim}.]
Let $S$ be an open subset of the $20$-dimensional moduli space of smooth cubic fourfolds admitting a universal family $g\colon\cX\to S$;
for example, we can choose the open set of cubics that have no automorphism.
For a very general point $s_0\in S$, we have $\Knum(\Ku(\cX_{s_0}))=A_2$, which is monodromy-invariant by construction.

For a pair of integers $(a,b)$ as in the statement, consider the vector $\vv:=a\llambda_1+b\llambda_2\in A_2$.
Take $\sigma_{s_0}\in\Stab^\dagger(\Ku(\cX_{s_0}))$ a $\vv$-generic stability condition.
Since $\eta(\sigma_{s_0})\in A_2\otimes\C$, the central charge of $\sigma_{s_0}$ is monodromy invariant.

By Theorem~\ref{thm:Msigmarelative}.\eqref{enum:ProjMorph}, there is a non-empty open subset $S^0\subseteq S$ and a relative moduli space $g\colon M^0(\vv)\to S^0$.
By Theorem~\ref{thm:YoshiokaMain}, the fibers of $g$ are irreducible holomorphic symplectic manifolds of dimension $\vv^2+2=2(a^2-ab+b^2)+2$.
This proves the first part of the statement.

As explained in the proof of Theorem~\ref{thm:YoshiokaMain}, $M^0(\vv)$ is endowed with a relative line bundle $l$ such that $l_s=l_{\sigma_s}$.
By Theorem~\ref{thm:YoshiokaMain}.\eqref{enum:YoshiokaMain2}, $l_s$ is orthogonal to $\vv$ for all $s\in S^0$ and must be a vector in $A_2$, since for a very general point $s_0$ we have $\Knum(\Ku(\cX_{s_0}))=A_2$.
Hence $l_s$ is proportional to $\ww=-(2b-a)\llambda_1+(2a-b)\llambda_2$.
Given $a,b$ coprime, the class $\ww$ is primitive if $a^2-ab+b^2$ is not divisible by $3$, and divisible by $3$ otherwise.
A simple computation yields the degree.
Another application of Theorem~\ref{thm:YoshiokaMain}.\eqref{enum:YoshiokaMain2} shows that the divisibility of the polarisation is equal to the divisibility of $\ww$ or $\frac{\ww}3$ in $\vv^\perp$.
The claim then follows from the fact that $\vv^\perp$ has discriminant $2n$, whereas $\vv^\perp \cap \ww^\perp = A_2^\perp$ has discriminant $3$.

To get a unirational family, we observe that $S^0$ is open in the moduli space of cubic fourfolds which is unirational by construction. 
\end{proof}

\begin{Rem}
Using a descent argument as in the proof of Theorem~\ref{thm:Msigmarelative}.\eqref{enum:ProjMorph}, one can extend this family to a larger subset of the DM stack of cubic fourfolds.
The precise locus $S \setminus S^0$ we are forced to omit can be determined as in \cite{DebarreMacri} as the locus where a stability condition $\sigma$ with $\eta(\sigma) \in A_2 \otimes \C$ cannot be $\vv$-generic, or, equivalently, where $l$ cannot extend to a polarisation.
\end{Rem}

\begin{Ex}\label{ex:extendLLSvS}
Let $S$ be the moduli space of smooth cubic fourfolds, and $\vv = 2\llambda_1 + \llambda_2$ as in Example~\ref{ex:4fold8fold}.
By \cite[Theorem~1.2]{LiPertusiZhao:TwistedCubics}, we can take $S^0$ to be the complement of the divisor $\cC_8$ of cubics containing a plane.
(Alternatively, this consequence could also be deduced from the extension of \cite[Theorem~5.7]{BM:MMP_K3} to our context; this would also show that along $\cC_8$, stability conditions with central charge in $A_2$ lie on a wall, induced by the additional Hodge classes;
in other words, $S^0$ is the maximal possible subset in which Theorem~\ref{thm:Msigmarelative}.\eqref{enum:ProjMorph} holds.)
It follows from the first statement of \cite[Proposition~5.2.1]{Hassett:specialcubics} that the monodromy on $\cC_8$ acts trivially on these additional Hodge classes.
Therefore, we are in the setting of Remark~\ref{rmk:ExtendOverTrivialMonodromy}, and can deform the stability conditions with central charge in $A_2$ to some that are generic on all fibers; the associated relative coarse moduli space extends $M^0(\vv) \to S^0$ to a proper morphism $M(\vv) \to S$ of algebraic spaces, with all fibers being smooth and projective.
Over $\cC_8$, it agrees with the moduli spaces of stable objects constructed by Ouchi in \cite{Ouchi:8fold}.
\end{Ex}

\begin{proof}[Proof of Theorem~\ref{thm:ConnectedComponentStab}.]
Let $\sigma$ be a stability condition in the boundary of the open subset $\Stab^\dag(\Ku(X)) \subset \Stab(\Ku(X))$; by the covering map property, its central charge is on the boundary of $\fP^+_0(\Ku(X))$.
This means that in the kernel of $Z$ there is either a root $\delta \in \tH_\Hdg(X,\Z), \delta^2 = -2$, or there is a \emph{real} class $\ww \in \tH_\Hdg(X,\Z) \otimes \R$ with $\ww^2 \geqslant 0$.
	
By definition of the topology on $\Stab(\Ku(X))$, the set where a given object is semistable is closed;
therefore, the non-emptiness of Theorem~\ref{thm:YoshiokaMain} still holds for $\sigma$.
In the former case, this is a direct contradiction to $Z(\delta) = 0$.
In the latter case, let $Q$ be the quadratic form giving the support property for $\sigma$, and consider a sequence $\ww_i \in \tH_\Hdg(X, \Z) \otimes \Q$ with $\ww_i^2 \geqslant 0$ and $\ww_i \to \ww$.
As an integral multiple of $\ww_i$ is the Mukai vector of a $\sigma$-semistable object, we have $Q(\ww_i) \geqslant 0$, a contradiction to $Q(\ww) < 0$.
\end{proof}

\begin{proof}[Proof of Corollary~\ref{cor:integralHdg}]
Take a class $\vv\in H^4(X,\Z)\cap H^{2,2}(X)$.
By \cite[Section~2.5]{AH:Ktop} (see also \cite[Theorem~2.1(3)]{AT:CubicFourfolds}), there is $\ww\in K_\mathrm{top}(X)$ such that $v(\ww)=\vv+\widetilde{\vv}$, where $\widetilde{\vv}\in H^6(X,\Q)\oplus H^8(X,\Q)$.

Consider the projection $\ww'$ of $\ww$ to $\tH(\Ku(X),\Z)$ which is induced by the projection functor.
It is clear that the identity $\ww'=\ww+a_0[\cO_W]+a_1[\cO_W(1)]+a_2[\cO_W(2)]$ holds in $K_\mathrm{top}(X)$, where for $a_0,a_1,a_2\in\Z$ (see, for example, \cite[Section~2.4]{AT:CubicFourfolds}).
Note that $c_2(\ww')$ and $c_2(\ww)=c_2(\vv)$ differ by a multiple of $h^2$, so one is algebraic if and only if the other one is.
The vector $\ww'$ is in $\tH_\Hdg(\Ku(X),\Z)=\tH_{\mathrm{alg}}(\Ku(X),\Z)$ since the projection preserves the Hodge structure (for the latter equality we used Theorem~\ref{thm:YoshiokaMain}).
Hence, there exists $E\in\Ku(X)$ such that $v(E)=\ww'$ and $c_2(\ww')$ is algebraic.
%
%Consider $E\in\Ku(X)$ such that $v(E)=\ww'$ and take
%\[
%F:=E\oplus\cO_X^{\oplus |a_0|}[\epsilon(a_0)]\oplus\cO_X(1)^{\oplus %|a_1|}[\epsilon(a_1)]\oplus\cO_X(2)^{\oplus |a_2|}[\epsilon(a_2)],
%\]
%where for an integer $a\in\Z$, we set $\epsilon(a)=0$ (resp., $=1$) if $a\geqslant0$ (resp., $a<0$).
%Then $c_2(F)=\vv$, which is algebraic.
\end{proof}

\newpage
\bibliography{all}
\bibliographystyle{alphaspecial}

\printindex

\end{document}